\documentclass[10pt,a4paper,twoside,openright]{book}

\usepackage[english]{babel}

\usepackage{mathrsfs}
\usepackage{amssymb}
\usepackage{amsmath}
\usepackage{latexsym}
\usepackage{amsthm}
\usepackage{amscd}
\usepackage{graphicx}

\usepackage[a-1b]{pdfx}
\usepackage{hyperref}

\theoremstyle{plain}
\newtheorem{teo}{Theorem}[chapter]
\newtheorem{prop}[teo]{Proposition}
\newtheorem{coroll}[teo]{Corollary}
\newtheorem{lem}[teo]{Lemma}

\theoremstyle{definition}
\newtheorem{defin}[teo]{Definition}
\newtheorem{ese}[teo]{Example}
\newtheorem{rmk}[teo]{Remark}

\newcommand{\R}{\mathbb{R}^n}
\newcommand{\Sp}{\mathbb{S}^{n-1}}

\newcommand{\bv}{BV(\mathbb{R}^n)}
\newcommand{\bvo}{BV(\Omega)}

\newcommand{\fracs}{W^{s,1}(\mathbb{R}^n)}
\newcommand{\fracso}{W^{s,1}(\Omega)}

\newcommand{\J}{\mathcal{J}}
\newcommand{\Ll}{\mathcal{L}}
\newcommand{\I}{\mathcal{I}}
\newcommand{\Co}{\mathcal{C}}
\newcommand{\Han}{\mathcal{H}^{n-1}}

\newcommand{\h}{\mathcal{H}}
\newcommand{\s}{\mathbb{S}}

\newcommand{\kers}{|x-y|^{n+s}}

\newcommand{\supp}{\textrm{supp }}

\newcommand{\Dim}{\textrm{Dim}}

\newcommand{\sig}{\textrm{sgn}}

\begin{document}

\begin{titlepage}
\begin{center}
\begin{figure}[htbp]
\begin{center}
\includegraphics[width=130mm]{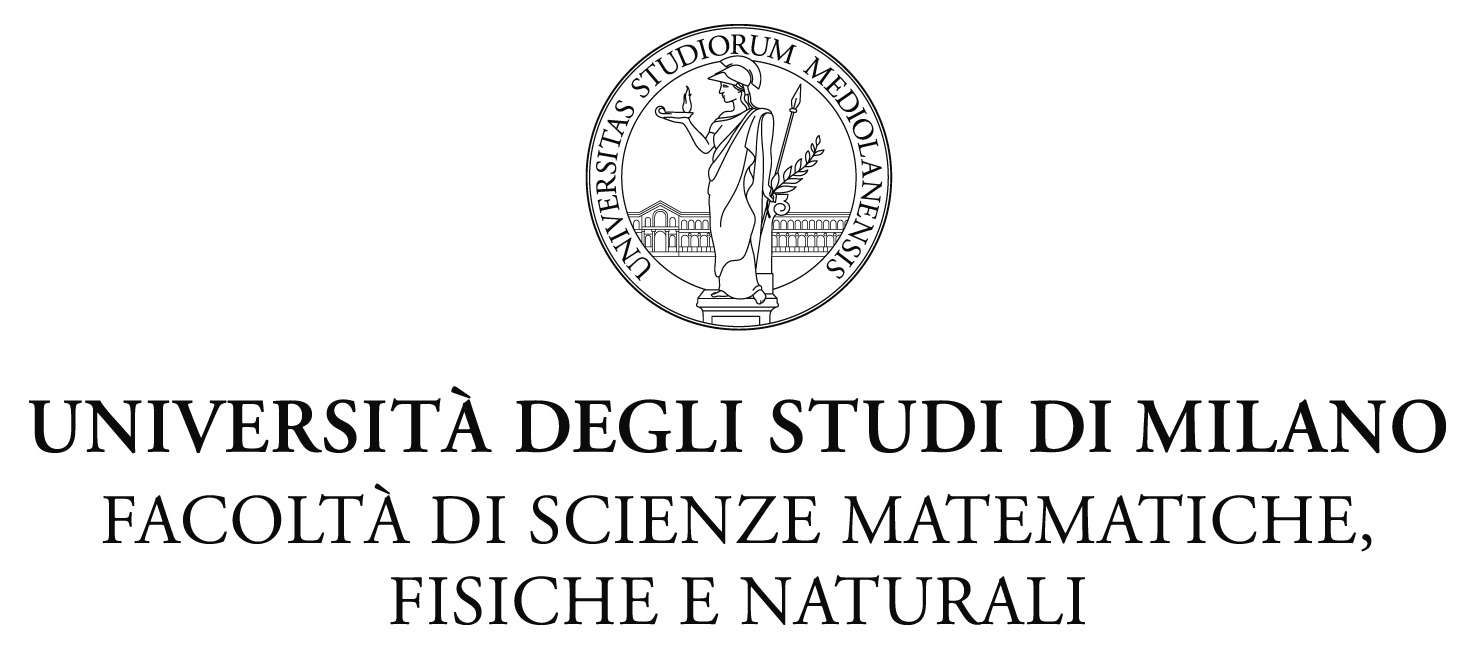}\
\end{center}
\end{figure}
\large{\scshape{Corso di Laurea Magistrale in Matematica}}

\vspace{\stretch{.8}}

\vspace{\stretch{.2}}
\linespread{1.4}
\textbf{\huge{Fractional Perimeter}}\\
\textbf{\huge{and}}\\
\textbf{\huge{Nonlocal Minimal Surfaces}}
\linespread{1.2}
\end{center}

\vspace{\stretch{2}}

\noindent Relatore: Prof. Enrico VALDINOCI

\hfill {\small TESI DI LAUREA DI}

\hfill Luca LOMBARDINI

\hfill Matr. 809016

\vspace{\stretch{1}}

\begin{center}
\large{ANNO ACCADEMICO 2014 - 2015}
\end{center}

\end{titlepage} 

\clearpage
\null
\thispagestyle{empty}
\clearpage

\pagenumbering{roman}
\chapter*{Introduction}

This thesis presents a study of the basic properties of the fractional $s$-perimeter
and of the regularity theory of the corresponding $s$-minimal sets.\\
The fractional $s$-perimeter arises naturally in nonlocal phase transition problems (e.g., as $\Gamma$-limit of
a nonlocal version of the Ginzburg-Landau energy,
\cite{Phase})
%where the authors study the limit, in the $\Gamma$-convergence sense, of a nonlocal version of the Ginzburg-Landau
%energy, namely replacing the classical term $\|\nabla u\|_{L^2}^2$ with the Gagliardo seminorm
%$[u]_{H^s}^2$. In particular the authors develop an analogue of the classical Modica-Mortola theory
and the related notion of fractional mean curvature appears in nonlocal evolution equations for surfaces
(e.g., in \cite{Cso} and \cite{CMP}).\\

Given an open set $\Omega\subset\R$ we can define the fractional $s$-perimeter
of a measurable set $E\subset\R$ in $\Omega$, with $s\in(0,1)$, as the functional
\begin{equation*}\begin{split}
P_s(E,\Omega)&:=\Ll_s(E\cap\Omega,\Co E\cap\Omega)
+\Ll_s(E\cap\Omega,\Co E\setminus\Omega)\\
&
\qquad+\Ll_s(E\setminus\Omega,\Co E\cap\Omega),
\end{split}
\end{equation*}
where
\begin{equation*}
\Ll_s(A,B):=\int_A\int_B\frac{1}{\kers}dx\,dy,
\end{equation*}
for every couple of disjoint sets $A,\,B\subset\R$.\\
We simply write $P_s(E)=P_s(E,\R)$, when $\Omega=\R$.

Formally, this coincides with
\begin{equation*}
P_s(E,\Omega)=\frac{1}{2}\big([\chi_E]_{W^{s,1}(\R)}-[\chi_E]_{W^{s,1}(\Co \Omega)}\big),
\end{equation*}
where $[u]_{W^{s,1}}$ denotes the Gagliardo seminorm of $u$ in the Sobolev space $W^{s,1}$.

We are neglecting the interactions coming from $\Co\Omega$ because these might be infinite and in the end we are interested in the minimization of
$P_s(F,\Omega)$ among all sets $F\subset\R$ with fixed `boundary data' $F\setminus\Omega=E_0\setminus\Omega$,
so they would not contribute to the minimization.

The $s$-perimeter is a nonlocal functional in the sense that $P_s(E,\Omega)$
is not determined by the behavior of $E$ in a neighborhood of $\Omega$.

Moreover, the $s$-perimeter can be thought of as a fractional perimeter, in the sense that $P_s(E,\Omega)$
can be finite even when the Hausdorff dimension of $\partial E$ is strictly bigger than $n-1$
(see below for more details).\\

\begin{section}*{Nonlocal Minimal Surfaces}

The main part of the thesis is devoted to the study of $s$-minimal sets and their regularity properties.
We followed the paper \cite{CRS}, where $s$-minimal sets were introduced and studied for the first time.
In particular, we give full detailed proofs for all the Theorems of \cite{CRS}.

A set $E\subset\R$ is $s$-minimal in $\Omega$ if
\begin{equation*}
P_s(E,\Omega)\leq P_s(F,\Omega)\quad\textrm{for every }F\subset\R\textrm{ s.t. }F\setminus\Omega=E\setminus\Omega.
\end{equation*}

Once we fix the exterior data $E_0\setminus\Omega$, the existence of an $s$-minimal set $E$ coinciding with $E_0$ outside $\Omega$
is obtained through the direct method of Calculus of Variations. Namely, a fractional Sobolev inequality guarantees the compactness of a minimizing sequence, while Fatou's Lemma is enough to have the inferior semicontinuity.

Then an interesting problem consists in studying the regularity of $\partial E\cap\Omega$.\\
This is done using techniques similar to those employed in the classical framework.
%with difficulties coming
%mainly from the nonlocal character of the functionals and the singularity of the kernels appearing in the integrals.\\

As a first step we obtain uniform density estimates for $s$-minimal sets.\\
An important consequence is the locally uniform convergence of minimizers, which is a fundamental tool in many proofs.

Moreover the uniform density estimates guarantee a clean ball condition.\\
To be more precise, this means that if $E$ is $s$-minimal in $\Omega$ and $x\in\partial E$, with
$B_r(x)\subset\Omega$, then there exist balls
\begin{equation*}
B_{cr}(y_1)\subset E\cap B_r(x),\qquad B_{cr}(y_2)\subset\Co E\cap B_r(x),
\end{equation*}
for some universal constant $c$.

\begin{subsection}*{Euler-Lagrange Equation}

We prove that a set $E$ which is $s$-minimal in $\Omega$ satisfies the Euler-Lagrange equation
\begin{equation*}
\I_s[E](x)=0,\quad x\in\partial E\cap\Omega
\end{equation*}
in the viscosity sense. Here $\I_s[E](x)$ denotes the $s$-fractional mean curvature of $\partial E$ in $x$,
\begin{equation*}
\I_s[E](x):=P.V.\int_{\R}\frac{\chi_E(y)-\chi_{\Co E}(y)}{\kers}dy.
\end{equation*}
Roughly speaking, if we think that $(\chi_E-\chi_{\Co E})(x_0)=0$ for every $x_0\in\partial E$, the Euler-Lagrange equation can be thought of as 
\begin{equation*}
(-\Delta)^\frac{s}{2}(\chi_E-\chi_{\Co E})=0\quad\textrm{along }\partial E\cap\Omega,
\end{equation*}
in the viscosity sense.

This is quite a difficult and delicate result because of the many estimates involved.\\
First of all we remark that we can define the fractional mean curvature only in the principal value sense,
\begin{equation*}
\I_s[E](x)=\lim_{\rho\to0}\I_s^\rho[E](x),\qquad\I_s^\rho[E](x):=\int_{\Co B_\rho(x)}\frac{\chi_E(y)-\chi_{\Co E}(y)}{\kers}dy,
\end{equation*}
since the integrand is
not in $L^1$.

Moreover we need to require some sort of `cancellation' between $E$ and $\Co E$ to guarantee that the limit exists.
In particular, following \cite{curvature}, we show that asking $E$ to have both an interior and an exterior tangent paraboloid in $x\in\partial E$ is enough.

Since, a priori, we do not know anything about the regularity of the boundary of an $s$-minimal set, this explains
why we obtain the equation only in the viscosity sense.

Namely, we prove the following
\begin{teo}
Let $E$ be $s$-minimal in the open set $\Omega$. If $x\in\partial E\cap \Omega$ and $E\cap\Omega$ has an interior tangent ball
at $x$, then
\begin{equation}\label{resume980}
\limsup_{\delta\to0}\I_s^\delta[E](x)\leq0.
\end{equation}
\end{teo}
Similarly with exterior tangent balls.

Therefore a first difficulty comes from the limit defining the principal value.\\
To obtain the Euler-Lagrange equation, a natural thing to do would be to look, for example, at the ratios
\begin{equation*}
\frac{1}{|A\setminus E|}\big(P_s(E\cup A,\Omega)-P_s(E,\Omega)\big),
\end{equation*}
where the `perturbating' set $A$ is a small neighborhood of
some point $x_0\in\partial E$. We can think for simplicity that $A=B_r(x_0)\subset\Omega$.

Then we expect that letting $|A|\to0$
gives the Euler-Lagrange equation.\\
Carrying out the computation of the ratio gives
\begin{equation*}
-\frac{1}{|A\setminus E|}\int_{A\setminus E}\Big(\int_{\R}\frac{\chi_E(y)-\chi_{\Co(E\cup A)}(y)}{\kers}\Big)dx.
\end{equation*}
Roughly speaking, since $A\searrow\{x_0\}$ as $|A|\to0$, we can think that the inner integral
converges to the fractional mean curvature, while the outer integral `disappears' in the limit.

However, carrying out all the estimates involved is really difficult, even when $A$ is a ball,
mainly because we can not control what sort of cancellation we have, if any, between $E$ and $\Co (E\cup A)$
in the inner integral. Also, as remarked above, we do not even know if the fractional mean curvature at $x_0$ is well defined.

Therefore to obtain inequality $(\ref{resume980})$ we
consider a very particular kind of perturbation.

Namely, we exploit the existence of an interior tangent ball $B$ to define a small perturbating set,
which is symmetric in an appropriate sense.\\
Exploiting the symmetry of this construction, we can control all the error terms.\\
We remark that, even for these particular perturbations, the estimates are really delicate.\\

In any case we also prove, following \cite{CMP}, that the fractional mean curvature gives the first variation of the fractional perimeter, at least
when we consider regular sets.

To be more precise, let  $E$ be a  bounded open set with $C^2$ boundary. If $\Phi_t:\R\to\R$ is a one-parameter family of $C^2$-diffeomorphisms 
which is $C^2$ also in $t$ and $\Phi_0=Id$, then
\begin{equation*}
\frac{d}{dt}P_s(\Phi_t(E))\Big|_{t=0}=-\int_{\partial E}\I_s[E](x)\nu_E(x)\cdot\phi(x)\,d\Han(x),
\end{equation*}
where $\phi(x):=\frac{\partial}{\partial t}\Phi_t(x)\big|_{t=0}$.

\end{subsection}

\begin{subsection}*{Regularity}
The remaining part of the thesis is a careful study of the `basic' regularity properties of $\partial E$.

We remark that if $E$ is $s$-minimal in $\Omega$, then it is $s$-minimal also in every $\Omega'\subset\Omega$.
Thus, when we want to study the regularity of $\partial E$ in the neighborhood of some point $x\in\partial E$, using
translations and dilations we can reduce to the case of a set $E$, which is $s$-minimal
in $B_1$ and s.t. $0\in\partial E$.

\begin{subsubsection}{Improvement of Flatness}

One of the fundamental results is the following
\begin{teo}\label{reg}
Let $\alpha\in(0,s)$. There exists $\epsilon_0=\epsilon_0(n,s,\alpha)>0$ s.t. if $E$ is $s$-minimal in $B_1$, with $0\in\partial E$ and
\begin{equation*}
\partial E\cap B_1\subset\{|x_n|\leq\epsilon_0\},
\end{equation*}
then
$\partial E\cap B_{1/2}$ is a $C^{1,\alpha}$ surface.
\end{teo}

\noindent
The proof (which is quite long and technical) relies on an improvement of flatness technique, in the style of De Giorgi.

Roughly speaking the idea consists in showing that if $\partial E$ is contained in some small cylinder,
in a neighborhood of $x_0\in\partial E$, then in a smaller neighborhood it is actually contained in a flatter cylinder,
up to a change of coordinates.

Then a compactness argument shows that, if the height of the first cylinder
is small enough, we can go on inductively, finding flatter and flatter cylinders.

In the end, since we are controlling the oscillation of $\partial E$ in smaller and smaller neighborhoods of $x_0$, we
obtain our $C^{1,\alpha}$ regularity.\\
Actually the proof is more delicate. Indeed, mainly because of the nonlocality of the fractional perimeter,
we need to control also what happens far from our point $x_0$.

%To be more precise, our inductive argument relies on the following

%CLAIM:$\quad$There exists a universal $k_0\in\mathbb{N}$ s.t. if $E\subset\R$ is $s$-minimal in $B_{2^{k+2}}$, with $k\geq k_0$,
%and
%\begin{equation*}
%\partial E\cap B_{2^j}\subset\{|x\cdot\nu_j|\leq a_k2^{j(\alpha+1)}\},\quad\textrm{for every } j\in\{0,\dots,k\},
%\end{equation*}
%then
%there exists $\nu_{-1}$ s.t.
%\begin{equation*}
%\partial F\cap B_{1/2}\subset\{|x\cdot\nu_{-1}|\leq a_k2^{-1-\alpha}\}.
%\end{equation*}
%Here $\alpha\in(0,s)$ is a fixed parameter and $a_k=2^{-\alpha k}$.

%Roughly speaking, requiring flatness of $\partial E\cap B_1$ of order $a_k$,
%but also flatness of order $a_k2^{i\alpha}$ for all diadic balls $B_{2^i}$ from $B_1$ to $B_{2^k}$, i.e.
%until flatness becomes of order one, gives flatness of order $a_k2^{-\alpha}$ in $B_{1/2}$.

\end{subsubsection}

\begin{subsubsection}{Monotonicity Formula}
The next step consists in proving a monotonicity formula for a `localized' version of the fractional perimeter functional, obtained through an extension technique
introduced in \cite{extension}.

To be more precise, let $u:=\chi_E-\chi_{\Co E}$ and consider the function $\tilde{u}:\mathbb{R}^{n+1}_+\to\mathbb{R}$ which solves
\begin{equation*}
\left\{\begin{array}{cc}
\textrm{div}(z^{1-s}\nabla\tilde{u})=0&\textrm{in }\mathbb{R}^{n+1}_+,\\
\tilde{u}=u&\textrm{on }\{z=0\},
\end{array}\right.
\end{equation*}
where
\begin{equation*}
\mathbb{R}^{n+1}_+=\{(x,z)\in\mathbb{R}^{n+1}\,|\,x\in\R,\,z>0\}.
\end{equation*}
Let $a:=1-s$. We use capital letters, like $X$, to
denote points in $\mathbb{R}^{n+1}$.

We remark that the first equation above
is the Euler-Lagrange equation for the functional
\begin{equation*}
\mathcal{E}(u)=\int_{\{z>0\}}|\nabla\tilde{u}|^2z^a\,dX.
\end{equation*}
%Given $\Omega\subset\mathbb{R}^{n+1}$, we denote
%\begin{equation*}
%\Omega_0:=\Omega\cap\{z=0\}\subset\R,\qquad\Omega_+:=\Omega\cap\{z>0\}.
%\end{equation*}

We relate this energy to the fractional perimeter, showing in particular the following
\begin{prop}
The set $E$ is $s$-minimal in $B_1$ if and only if the extension $\tilde{u}$ of $u=\chi_E-\chi_{\Co E}$ satisfies
\begin{equation*}
\int_{\Omega\cap\{z>0\}}|\nabla\bar{v}|^2z^a\,dX\geq\int_{\Omega\cap\{z>0\}}|\nabla\tilde{u}|^2z^a\,dX,
\end{equation*}
for all bounded open sets $\Omega$ with Lipschitz boundary s.t. $\Omega\cap\{z=0\}\subset\subset B_1$
and all functions $\bar{v}$ that equal $\tilde{u}$ in a neighborhood of $\partial\Omega$ and take the values $\pm1$ on $\Omega\cap\{z=0\}$.
\end{prop}
Notice that asking $\bar{v}$ to take only the values $\pm1$ on $\Omega\cap\{z=0\}$ corresponds to ask that
$\bar{v}$ is of the form $\chi_F-\chi_{\Co F}$, for some set $F$, on $\Omega\cap\{z=0\}$.

Roughly speaking, this means that using the extension $\tilde{u}$ we can reduce the minimization problem
of the fractional perimeter to a pde problem in $\mathbb{R}^{n+1}_+$.\\

Finally, exploiting the extension $\tilde{u}$ of $\chi_E-\chi_{\Co E}$ we can define the `localized' energy we were looking for.

To be more precise, if $E$ is $s$-minimal in $B_R$ we define the rescaled functional
\begin{equation*}
\Phi_E(r):=\frac{1}{r^{n+a-1}}\int_{\mathcal{B}_r^+}|\nabla\tilde{u}|^2z^a\,dX,\quad\textrm{for }r\in(0,R).
\end{equation*}
We remark that rescaling guarantees that $\Phi_{\lambda E}(\lambda r)=\Phi_E(r)$.

The monotonicity formula then says that $\Phi_E(r)$ is increasing.

\end{subsubsection}

\begin{subsubsection}{Blow-up and Cones}
The functional $\Phi_E$ is a fundamental tool to study the regularity of $\partial E$.

Indeed, exploiting the monotonicity formula, we can study the blow-up limit $\lambda E$ as $\lambda\to\infty$,
showing that it is a cone $C$, which is locally $s$-minimal in $\R$.\\
We call $C$ a tangent cone.\\
To be more precise, we prove that $\Phi_E$ is constant if and only if $\tilde{u}$ is homogeneous of degree 0.
In particular, since
the trace of $\tilde{u}$ on $\{z=0\}$ is $\chi_E-\chi_{\Co E}$,
this implies that $E$ is a cone.
Now suppose that $\lambda_k E\to C$.
Exploiting the scaling property, we can prove that the functional $\Phi_C$ is constant, so $C$ is indeed a cone.\\

Roughly speaking, considering the blow-up $\lambda_k E$ corresponds to zooming in on a neighborhood of $0\in\partial E$.

If we see the boundary become flatter and flatter, tending to a plane,
then $\partial E$ must be $C^{1,\alpha}$ in a neighborhood of 0.

Indeed, using
Theorem $\ref{reg}$ and the locally uniform convergence of minimizers,
we obtain the following
\begin{teo}
Let $E\subset\R$ be $s$-minimal in $B_1$ with $0\in\partial E$. If $E$ has a half-space as a tangent cone,
then $\partial E$ is a $C^{1,\alpha}$ surface in a neighborhood of 0.
\end{teo}

On the other hand, if $C$ is not a half-space, our point is singular.\\
Notice that if $C$ is not a half-space, then $\partial C$ is singular in 0.

\end{subsubsection}

\begin{subsubsection}{Singular Set}

The last part of the thesis studies the dimension of the singular set of $\partial E$,
i.e. of the subset $\Sigma_E\subset\partial E\cap\Omega$ of points
having a singular cone as tangent cone.
Adapting the classical dimension reduction argument by Federer, we prove that the singular set has Hausdorff dimension at most $n-3$.

\begin{teo}
Let $E$ be $s$-minimal in $\Omega$. Then
\begin{equation*}
\h^d(\Sigma_E)=0\quad\textrm{for every }d>n-3.
\end{equation*}
\end{teo}

\noindent
In particular we see that, as we would expect, $\partial E\cap\Omega$ has Hausdorff dimension at most $n-1$,
\begin{equation*}
\h^d(\partial E\cap\Omega)=0\quad\textrm{for every }d>n-1.
\end{equation*}

The idea of the dimension reduction argument is the following.\\
Suppose that $C\subset\R$ is a singular $s$-minimal cone,
having a singularity also in some $x_0\in\partial C$, $x_0\not=0$. Then, if we blow-up at $x_0$ we get
another singular $s$-minimal cone $C'$.\\
Now the delicate part consists in showing that $C'=K\times\mathbb{R}$ (up to rotation).
Then it is easily seen that also $K\subset\mathbb{R}^{n-1}$ is a singular $s$-minimal cone.

Proceeding inductively in this way we reduce the dimension of the ambient space until we end up with
a singular $s$-minimal cone $\tilde{K}\subset\mathbb{R}^k$, which is singular only in 0. Finally, since
in \cite{cones} it is shown that there are no singular $s$-minimal cones in dimension 2, we obtain our estimate.

\end{subsubsection}

\end{subsection}

\end{section}

\begin{section}*{Original Contributions}

In the thesis we  study the asymptotics of the $s$-perimeter as $s\to1^-$,
obtaining an original result which, in particular,  improves a previous Theorem of \cite{cafenr}.\\
To be more precise, we obtain
the asymptotics for $P_s(E,\Omega)$, for any bounded open set $\Omega$ with Lipschitz boundary,
asking minimal regularity on $E$, namely we only require $E$ to have finite (classical) perimeter in a neighborhood of $\Omega$.
On the other hand, the result obtained in \cite{cafenr} holds only when $\Omega=B_R$ is a ball and requires
$\partial E$ to be $C^{1,\alpha}$.\\

We provide an original example of a set which has finite $s$-perimeter for every $s\in(0,\sigma)$ and
infinite perimeter for every $s\in(\sigma,1)$. To be more precise
we consider the von Koch snowflake $S\subset\mathbb{R}^2$ and we show that its Minkowski dimension
coincides with the fractal dimension which can be defined using the fractional perimeter.\\

We use a formula proved in \cite{unifor} to prove in an original way that the fractional curvature is continuous
with respect to $C^{1,\alpha}$-convergence of sets.\\

We also remark that we provide full details for all the proofs of \cite{CRS}.
In particular we added a lot of details to the proof of the Flatness Improvement.

\begin{subsection}*{Asymptotics as $s\to1$}

We state our result, then we sketch an explanation,
but
first we need to introduce some notation.\\
We split the fractional perimeter in the following two parts
\begin{equation*}
P_s(E,\Omega)=P_s^L(E,\Omega)+P_s^{NL}(E,\Omega),
\end{equation*}
where
\begin{equation*}\begin{split}
&P^L_s(E,\Omega):=\Ll_s(E\cap\Omega,\Co E\cap\Omega)=\frac{1}{2}[\chi_E]_{W^{s,1}(\Omega)},\\
&
P^{NL}_s(E,\Omega):=\Ll_s(E\cap\Omega,\Co E\setminus\Omega)+\Ll_s(E\setminus\Omega,\Co E\cap\Omega).
\end{split}
\end{equation*}
We can think of $P_s^L(E,\Omega)$ as the local contribution to the fractional perimeter, in the sense that it is determined by the behavior of $E$ inside $\Omega$.

Let $\Omega\subset\R$ be a bounded open set with Lipschitz boundary and let $\bar{d}_\Omega$ denote the signed distance function from $\Omega$, negative inside. Define for any $\rho\in\mathbb{R}$ with $|\rho|$ small,
the open set
\begin{equation*}
\Omega_\rho:=\{\bar{d}_\Omega<\rho\}.
\end{equation*}
It can be shown that $\Omega_\rho$ has Lipschitz boundary for every $|\rho|<\alpha$,
for some $\alpha>0$ small enough.
Notice that $\Omega_\rho\subset\subset\Omega$ when $\rho<0$ and
$\Omega\subset\subset\Omega_\rho$ when $\rho>0$.
Also notice that for $\rho>0$
\begin{equation*}
N_\rho(\partial\Omega)=\Omega_\rho\setminus\overline{\Omega_{-\rho}}=\{-\rho<\bar{d}_F<\rho\},
\end{equation*}
is an open tubular neighborhood of $\partial\Omega$.

Our result is the following
\begin{teo}
Let $\Omega\subset\R$ be a bounded open set with Lipschitz boundary. Then
%be a set having finite perimeter in $D_1$.\\
%Then $P_s(E,\Omega)<\infty$ for every $s\in(0,1)$,

$(i)\quad\quad E\subset\R$ has finite perimeter in $\Omega$ if and only if $P_s(E,\Omega)<\infty$
for every $s\in(0,1)$, and
\begin{equation}\label{resume9}
\liminf_{s\to1}(1-s)P_s^L(E,\Omega)<\infty.
\end{equation}
In this case we have
\begin{equation}\label{resume6}
\lim_{s\to1}(1-s)P_s^L(E,\Omega)=\omega_{n-1}P(E,\Omega).
\end{equation}

$(ii)\quad$ Suppose that $E$ has finite perimeter in $\Omega_\beta$, for some $0<\beta<\alpha$. Then
\begin{equation}\label{resume10}
\limsup_{s\to1}(1-s)P_s^{NL}(E,\Omega)
\leq2\omega_{n-1}\lim_{\rho\to0^+}P(E,N_\rho(\partial\Omega)).
\end{equation}
In particular, if $P(E,\partial\Omega)=0$, then
\begin{equation}
\lim_{s\to1}(1-s)P_s(E,\Omega)=\omega_{n-1}P(E,\Omega).
\end{equation}

$(iii)\quad$ Let $E$ be as in $(ii)$; then there exists a set $S\subset(-\alpha,\beta)$, at most countable,
s.t.
\begin{equation}\label{resume123}
\lim_{s\to1}(1-s)P_s(E,\Omega_\delta)=\omega_{n-1}P(E,\Omega_\delta),
\end{equation}
for every $\delta\in(-\alpha,\beta)\setminus S$.
\end{teo}

In \cite{cafenr} the authors obtained point (iii) only
for $\Omega=B_R$ a ball, asking $C^{1,\alpha}$ regularity of $\partial E$
in $B_R$. They proved the convergence in every ball $B_r$ with $r\in(0,R)\setminus S$,
with $S$ at most countable,
exploiting uniform estimates.

On the other hand, asking $E$ to have finite perimeter in a neighborhood (as small as we want) of the open set $\Omega$
is optimal.\\
Indeed if $E\subset\R$ is s.t. $(\ref{resume123})$ holds true, then point $(i)$ guarantees that $E$ has finite perimeter
in $\Omega_\delta$.

In \cite{Gamma} the authors studied the asymptotics as $s\to1$ in the $\Gamma$-convergence sense.
In particular, for the proof of a $\Gamma$-limsup inequality, which is typically constructive and by density, they show that if $\Pi$ is a polyhedron, then
\begin{equation*}
\limsup_{s\to1}(1-s)P_s(\Pi,\Omega)
\leq\Gamma_n^*P(\Pi,\Omega)+2\Gamma_n^*\lim_{\rho\to0^+}P(\Pi,N_\rho(\partial\Omega)),
\end{equation*}
which is $(\ref{resume10})$, once we sum the local part of the perimeter.

Their proof
%is based on quite technical
%estimates which rely
relies on the fact that $\Pi$ is a polyhedron to obtain the convergence of the local part of the perimeter,
which is then used, like we do (see below), also in the estimate of the nonlocal part.  Moreover to prove that
the constant is $\Gamma_n^*=\omega_{n-1}$ they need a delicate approximation result.

They also prove, in particular
\begin{equation*}
\Gamma-\liminf_{s\to1}(1-s)P_s^L(E,\Omega)\geq\omega_{n-1}P(E,\Omega),
\end{equation*}
which is a stronger result than our point $(i)$.\\
%They also prove
%

Our proof relies only on a convergence result by Davila which says, roughly speaking,
\begin{equation*}
(1-s)[u]_{W^{s,1}(\Omega)}\xrightarrow{s\to1}C_n[u]_{BV(\Omega)},
\end{equation*}
when $\Omega$ is a bounded open set with Lipschitz boundary.

In the thesis we explicitly compute the constant in an elementary way, showing that
\begin{equation*}
C_n=2\omega_{n-1}=2\Ll^{n-1}(B_1),
\end{equation*}
twice the volume of the $(n-1)$-dimensional unit ball $B_1\subset\mathbb{R}^{n-1}$.

Using this result we immediately get the convergence of the rescaled `local' part of the fractional perimeter, $(\ref{resume6})$.

Then we approximate the nonlocal part of the perimeter showing that
\begin{equation*}
P^{NL}_s(E,\Omega)\leq 2P_s^L(E,N_\rho(\partial\Omega))+O(1),\quad\textrm{as }s\to1.
\end{equation*}
This gives $(ii)$ and $(iii)$ is a simple consequence based on the fact that
\begin{equation*}
P(E,A)=\h^{n-1}(\partial^*E,A),
\end{equation*}
for every $A\subset\R$, where $\partial^*E$ denotes the reduced boundary of $E$.\\
Now, if we ask $P(E,\Omega_\beta)<\infty$, the set of $\delta\in(-\alpha,\beta)$ s.t.
$P(E,\{\bar{d}_\Omega=\delta\})>0$
can be at most countable, proving $(iii)$.\\

We also provide an original example to show that condition ($\ref{resume9}$) is necessary. Namely, we
construct a bounded set $E\subset\mathbb{R}$ s.t.
$P_s(E)<\infty$ for every $s$, but $P(E)=\infty$.

\end{subsection}

\begin{subsection}*{Von Koch Snowflake}
Before stating our result, we briefly define the Minkowski dimension in an informal way and we sketch
the result obtained in \cite{Visintin}

Roughly speaking, to define the Minkowski dimension of
a set $\Gamma\subset\R$ in an open set $\Omega\subset\R$, we consider the $\rho$-neighborhoods $N_\rho(\Gamma)$ and we look at the limits
\begin{equation*}
m_r=\lim_{\rho\to0}\frac{|N_\rho(\Gamma)\cap\Omega|}{\rho^{n-r}},\qquad r\in(0,n].
\end{equation*}
Then the dimension is defined as $\Dim_{\mathcal{M}}(\Gamma,\Omega):=\inf\{r\,|\,m_r=0\}$.\\
(we remark that a correct definition is much more delicate: we would have to consider the limsup and the liminf of the ratios,
then look at the inf and sup of the quantities we obtain).\\
As usual, when $\Omega=\R$ we drop it in the formulas.

Following \cite{Visintin}, we can introduce a notion of fractal dimension by setting
\begin{equation*}
\Dim_F(\partial E,\Omega):=n-\sup\{s\in(0,1)\,|\,P_s(E,\Omega)<\infty\},
\end{equation*}
whenever $\Omega$ is a bounded open set with Lipschitz boundary or $\Omega=\R$.\\
As shown in \cite{Visintin}, we can relate this dimension to the Minkowski dimension showing, roughly, the following.

Suppose that
$E\subset\R$ is s.t.
$\Dim_\mathcal{M}(\partial E,\Omega)\in[n-1,n)$. Then
\begin{equation}\label{resume776}
P_s(E,\Omega)<\infty\qquad\textrm{for every }s\in\left(0,n-\Dim_\mathcal{M}(\partial E,\Omega)\right),
\end{equation}
i.e.
\begin{equation}\label{resume1}
\Dim_F(\partial E,\Omega)\leq \Dim_\mathcal{M}(\partial E,\Omega).
\end{equation}

It would be interesting to have also a lower bound on $\Dim_F$.

To be more precise, $(\ref{resume776})$ guarantees that a set $E$ can have finite $s$-perimeter even when
the boundary $\partial E$ is really irregular, at least for every $s$ below some treshold $\sigma$.
However we do not know what happens above this treshold, when $s>\sigma$.

We provide an example of a set for which this treshold is sharp.

\begin{prop}
Let $S\subset\mathbb{R}^2$
be the von Koch snowflake. Then
\begin{equation*}
\Dim_\mathcal{M}(\partial S)=\Dim_F(\partial S)=\frac{\log 4}{\log 3},
\end{equation*}
i.e.
\begin{equation*}
P_s(S)<\infty,\qquad\forall\,s\in\Big(0,2-\frac{\log4}{\log3}\Big)
\end{equation*}
and
\begin{equation*}
P_s(S)=\infty,\qquad\forall\,s\in\Big(2-\frac{\log4}{\log3},1\Big).
\end{equation*}

\end{prop}
To show that $\Dim_F(\partial S)\geq\Dim_{\mathcal{M}}(\partial S)$, we exploited the self-similarity of $S$ and the scaling property of the fractional perimeter to prove that
\begin{equation*}
P_s(S)\geq\sum_{k=1}^\infty a_k(s),
\end{equation*}
which is a divergent series precisely when $s$ is bigger than $2-\frac{\log 4}{\log 3}$.

It is also worth noting that to find an exmple of a set $E$ s.t. $\Dim_F(\partial E)\geq\Dim_{\mathcal{M}}(\partial E)$,
we did not construct an ad hoc `pathological' set. Indeed the von Koch snowflake is a quite classical and well known example of fractal set.

\end{subsection}

\begin{subsection}*{Continuity of Fractional Curvature}

In \cite{unifor} the authors proved a formula to compute the `local' contribution to the fractional mean curvature $\I_s[E](x)$,
when $\partial E$ is a $C^{1,\alpha}$ graph in a neighborhood of $x$, for some $\alpha>s$.

To be more precise, let $K_1$ be the cylinder $B_1'\times(-1,1)$ and suppose that $x=0$ and
\begin{equation*}
E\cap K_1=\{(x',x_n)\in\R\,|\,x'\in B_1',\,-1<x_n<u(x')\},
\end{equation*}
for some $u\in C^{1,\alpha}(B_1')$ s.t. $u(0)=0$, $\nabla u(0)=0$ and $\alpha>s$.
Then
\begin{equation*}
P.V.\int_{K_1}\frac{\chi_E(y)-\chi_{\Co E}(y)}{|y|^{n+s}}\,dy
=2\int_{B'_1}\Big(\int_0^{\frac{u(y')}{|y'|}}\frac{dt}{(1+t^2)^\frac{n+s}{2}}\Big)\frac{dy'}{|y'|^{n+s-1}}.
\end{equation*}

Exploiting this formula we prove that
\begin{prop}
Let $E$ and $E_k$ be bounded open sets with $C^{1,\alpha}$ boundary for some $\alpha>0$ s.t. $E_k\longrightarrow E$ in $C^{1,\alpha}$ and let
$x_k\in\partial E_k$, $x\in\partial E$ s.t. $x_k\longrightarrow x$. Then
\begin{equation}\label{resume889}
\I_s[E_k](x_k)\longrightarrow\I_s[E](x),
\end{equation}
for every $s\in(0,\alpha)$.\\
In particular, if we ask $C^2$ regularity of the boundaries and the convergence to be in $C^2$ sense, this holds true for every $s\in(0,1)$.
\end{prop}
By $C^{1,\alpha}$ convergence of sets we mean that our sets can locally be described as the graphs
of functions which converge in $C^{1,\alpha}$.

We remark that this result is stated for $C^2$ convergence only, without a proof, in \cite{CMP}.
We provide an original proof and lower the requested regularity.\\

Actually, if we are interested in the convergence only in the neighborhood of some point $x_0\in\partial E$, we can lower our regularity requests and still get the convergence of the curvatures.

To be more precise, let $E$ and $E_k$ be defined in $K_1$ as the subgraphs of $u$ and $u_k$ respectively,
with $0\in\partial E$ and $0\in\partial E_k$ (up to translations we can always reduce to this case).
Then, using the formula above we get
\begin{equation*}
|\I_s[E](0)-\I_s[E_k](0)|\leq C\|u-u_k\|_{C^{1,\alpha}(B_1')}+|(E\Delta E_k)\setminus K_1|.
\end{equation*}
This shows that, if we are looking for convergence of the curvatures only in a fixed neighborhood $U$ of $x_0\in\partial E$,
we need not ask $C^{1,\alpha}$ regularity for the whole boundaries. For example,
we can ask $E_k$ and $E$ to be $C^{1,\alpha}$ subgraphs in $U$ and only ask
them to be measurable in $\Co U$. Then we obtain $(\ref{resume889})$ by asking $C^{1,\alpha}$ convergence of $E_k$ to $E$ in $U$ and
only convergence in measure in $\Co U$.\\

A similar problem is studied also in \cite{matteo}, where the author estimates the
difference between the fractional mean curvature of a set $E$ with $C^{1,\alpha}$
boundary and that of the set
$\Phi(E)$, where $\Phi$ is a $C^{1,\alpha}$ diffeomorphism of $\R$.\\
The estimates obtained there are much more precise than ours.\\
However, as far as only convergence is
considered, our result is more general in that the sets involved need not be diffeomorphic.

Moreover, as remarked above, our convergence is somewhat local, while the setting in \cite{matteo} is global. Indeed, the author estimates the difference between the curvatures in terms of the $C^{0,\alpha}$ norm of the Jacobian of the diffeomorphism $\Phi$.

Thus, even if we want the convergence of the curvatures only in a neighborhood $U$ of $x_0$,
to use \cite{matteo}
we still need to ask $C^{1,\alpha}$ regularity for the whole boundaries.

We remark that in \cite{matteo} the author studies also the stability of these estimates as $s\to1$.

\end{subsection}

\end{section}

\chapter*{Notation}

\begin{itemize}

\item All sets and functions considered are assumed to be Lebesgue measurable.

\item We will usually write $D\varphi$ to denote the distributional gradient of $\varphi$.\\We will
write $\nabla\varphi$ when the gradient exists (at least) in the weak sense of Sobolev spaces.

\item We denote $\mathcal{L}^k$ the $k$-dimensional Lebesgue measure.\\In $\R$ we will usually write
$|E|=\mathcal{L}^n(E)$ for the $n$-dimensional Lebesgue measure of a set $E\subset\R$.\\
We write $\h^d$ for the $d$-dimensional Hausdorff measure, for any $d\geq0$.

\item Equality and inclusions of sets will usually be considered in the measure sense, e.g. $E=F$ will usually mean
$|E\Delta F|=0$.

\item We define the dimensional constants
\begin{equation*}
\omega_d:=\frac{\pi^\frac{d}{2}}{\Gamma\big(\frac{d}{2}+1\big)},\qquad d\geq0.
\end{equation*}
In particular, we remark that $\omega_k=\mathcal{L}^k(B_1)$ is the volume of the $k$-dimensional unit ball $B_1\subset\mathbb{R}^k$
and $k\,\omega_k=\h^{k-1}(\mathbb{S}^{k-1})$ is the surface area of the $(k-1)$-dimensional sphere
\begin{equation*}
\mathbb{S}^{k-1}=\partial B_1=\{x\in\mathbb{R}^k\,|\,|x|=1\}.
\end{equation*}

%%%    NOTATION FOR DERIVATIVES: WEAK AND DISTRIBUTIONAL

%%%   ASSUMPTION: ALL SETS ARE LEBESGUE MEASURABLE

%%%   NOTATION FOR LEBESGUE MEASURE

%%%   NOTATION FOR MEASURE OF UNITARY BALLS AND SPHERES

\end{itemize}

\tableofcontents

\begin{chapter}{Caccioppoli Sets}
\pagenumbering{arabic}

\begin{section}{Caccioppoli Sets}

In this section we recall the definitions and the main properties of $BV$-functions and of Caccioppoli sets.

For all the details we refer to \cite{GiaSou}, \cite{Giusti} and \cite{Maggi}.

\begin{defin}
Let $\Omega\subset\R$ be an open set. We say that $f\in L^1_{loc}(\Omega)$ is a function of locally bounded variation in $\Omega$ if
\begin{equation}\label{variation}
V(f,A):=\sup\left\{\int_Af\textrm{ div }\varphi\,dx\Big|\varphi\in C_c^1(A,\R), |\varphi|\leq1\right\}<\infty,
\end{equation}
for every open set $A\subset\subset\Omega$. We write $BV_{loc}(\Omega)$ for the space of functions with locally bounded variation.\\
If $f\in L^1(\Omega)$ and $V(f,\Omega)<\infty$, we say that $f$ is a function of bounded variation in $\Omega$ and we write
$BV(\Omega)$ for the space of such functions.
\end{defin}

By using Riesz representation Theorem, it is easy to see that a function $u\in L^1_{loc}(\Omega)$ has locally bounded variation if and only if its distributional gradient $Du=(D_1u,\dots,D_nu)$ is a vector valued Radon measure. Then
\begin{equation*}\begin{split}
\int_\Omega u\textrm{ div }\varphi\,dx&=-\sum_{i=1}^n\int_\Omega\varphi_i\,dD_iu\\
&
=-\int_\Omega\varphi\cdot\sigma\,d|Du|,\qquad\forall\varphi\in C_c^1(\Omega,\R),
\end{split}\end{equation*}
for some $\sigma:\Omega\to\R$ s.t. $|\sigma(y)|=1$ for $|Du|-$almost every $y\in\Omega$.

Since $|Du|$ is a Radon measure on $\Omega$, we can consider $|Du|(A)$ for every $A\subset\Omega$ open (actually Borel),
not necessarily bounded, and
it is immediate to see that
\begin{equation*}
|Du|(A)=V(u,A)\qquad\forall A\subset\Omega\textrm{ open.}
\end{equation*}
Moreover, if $u\in L^1(\Omega)$, then $u\in BV(\Omega)$ if and only if  $Du$ is a vector valued Radon measure with finite total variation
\begin{equation*}
|Du|(\Omega)<\infty.
\end{equation*}

The Sobolev space $W^{1,1}(\Omega)$ is contained in $BV(\Omega)$. Indeed, for any $u\in W^{1,1}(\Omega)$ the distributional derivative is $\nabla u\Ll^n\llcorner\Omega$ and
\begin{equation*}
|Du|(\Omega)=\int_\Omega|\nabla u|\,dx<\infty.
\end{equation*}
Notice that the inclusion is strict, as is shown by considering e.g. the Heavside function $\chi_{[a,\infty)}$.

Since for every fixed vector field $\varphi\in C_c^1(\Omega,\R)$ the mapping
\begin{equation*}
u\longmapsto\int_\Omega u\textrm{ div }\varphi\,dx
\end{equation*}
is continuous in the $L^1_{loc}(\Omega)$ topology, the functional
\begin{equation*}
V(\cdot,\Omega):L^1_{loc}(\Omega)\longrightarrow[0,\infty]
\end{equation*}
is lower semicontinuous
and hence in particular we have the following result.
\begin{prop}
Let $\{u_k\}\subset BV(\Omega)$ s.t. $u_k\to u$ in $L^1_{loc}(\Omega)$. Then
\begin{equation}\label{bv_semicont}
V(u,\Omega)\leq\liminf_{k\to\infty}|Du_k|(\Omega).
\end{equation}
%In particular, if $|Du_k|(\Omega)\leq C$ for every $k$, then $u\in BV(\Omega)$.
\end{prop}

Before giving the definitions of perimeter and Caccioppoli set, we recall the following useful approximation result.
\begin{prop}\label{bv_approx}
Let $u\in BV(\Omega)$. Then $\exists\{u_k\}\subset C^\infty(\Omega)\cap BV(\Omega)$ s.t.
\begin{equation*}\begin{split}
&(i)\quad u_k\longrightarrow u,\qquad\textrm{in } L^1(\Omega),\\
&(ii)\quad\lim_{k\to\infty}\int_\Omega|\nabla u_k|\,dx=|Du|(\Omega).
\end{split}\end{equation*}
\end{prop}

We can define a norm by
\begin{equation*}
\|u\|_{BV(\Omega)}:=\|u\|_{L^1(\Omega)}+|Du|(\Omega),
\end{equation*}
which makes $BV(\Omega)$ a Banach space.\\
Notice that for an approximating sequence we have
\begin{equation*}
\lim_{k\to\infty}\|u_k\|_{BV(\Omega)}=\|u\|_{BV(\Omega)},
\end{equation*}
but we don't have convergence in the $BV$-norm.

\begin{defin}
We say that a set $E\subset\R$ has locally finite perimeter in $\Omega$ if $\chi_E\in BV_{loc}(\Omega)$.
It has finite perimeter if $\chi_E\in BV(\Omega)$.\\
A set $E$ of locally finite perimeter in $\R$ is called Caccioppoli set.\\
The perimeter of $E$ in $\Omega$ is defined as
\begin{equation}\begin{split}
P(E,\Omega)&:=|D\chi_E|(\Omega)=V(\chi_E,\Omega)\\
&
=\sup\left\{\int_E\textrm{div }\varphi\,dx\Big|\varphi\in C_c^1(\Omega,\R), |\varphi|\leq1\right\}.
\end{split}
\end{equation}
If $\Omega=\R$, we simply write $P(E)=P(E,\R)$ for the perimeter.
\end{defin}

If $E$ is a Caccioppoli set, we write $\nu_E:=\sigma$ for the function obtained in the polar decomposition of the Radon measure $D\chi_E$. Then for every bounded open set $\Omega\subset\R$
\begin{equation}
\int_E\textrm{div }\varphi\,dx=-\int_\Omega\varphi\cdot\nu_E\,d|D\chi_E|,
\end{equation}
for all vector fields $\varphi\in C^1_c(\Omega,\R)$.
This formula can be viewed as a generalization of the Gauss-Green formula, so the measure
\begin{equation*}
D\chi_E=\nu_E|D\chi_E|
\end{equation*}
is sometimes called Gauss-Green measure.

Unlike Sobolev spaces, the space $BV(\Omega)$
contains the characteristic functions of all sufficiently regular sets.

Indeed, consider a bounded open set $E\subset\R$ with $C^2$ boundary $\partial E$.\\
Let $\Omega$ be an open set; since $E$ is bounded, $|E\cap\Omega|<\infty$, i.e. $\chi_E\in L^1(\Omega)$.\\
Now take a vector field
$\varphi\in C^1_c(\Omega,\R)$ s.t. $|\varphi|\leq1$. Then the classical Gauss-Green formula gives
\begin{equation*}
\int_E\textrm{div }\varphi\,dx=-\int_{\partial E}\varphi\cdot\nu\,d\Han=-\int_{\partial E\cap\Omega}\varphi\cdot\nu\,d\Han,
\end{equation*}
where $\nu$ is the inner unit normal to $\partial E$, and hence
\begin{equation*}
P(E,\Omega)\leq\Han(\partial E\cap\Omega)<\infty,
\end{equation*}
so $\chi_E\in BV(\Omega)$.

Actually we have
\begin{equation}\label{smooth_perimeter}
D\chi_E=\nu\Han\llcorner\partial E
\end{equation}
and
\begin{equation}\label{smooth_perimeter2}
P(E,\Omega)=\Han(\partial E\cap\Omega).
\end{equation}

This shows that the perimeter coincides with the $(n-1)-$dimensional Hausdorff measure, at least when the set is regular.

However, unlike the Hausdorff measure, we have lower semicontinuity and compactness properties which make this definition of perimeter very useful when dealing with variational problems.\\
For example, we can write the semicontinuity property for Caccioppoli sets as
\begin{prop}[Semicontinuity]\label{smicont_Cacc}
Let $\{E_k\}$ be a sequence of Caccioppoli sets s.t. $E_k\xrightarrow{loc}E$. Then for every $\Omega\subset\R$ open
(bounded or not)
\begin{equation*}
P(E,\Omega)\leq\liminf_{k\to\infty}P(E_k,\Omega).
\end{equation*}
\end{prop}

By $E_k\xrightarrow{loc}E$ we mean the local convergence of sets in measure i.e. of their characteristic functions in 
$L^1_{loc}(\R)$.

Since $|\chi_E-\chi_F|=\chi_{E\Delta F}$, this means
\begin{equation*}
|(E_k\Delta E)\cap K|\longrightarrow0,
\end{equation*}
for every compact set $K\subset\R$.

Moreover we have the following compactness property.
\begin{prop}[Compactness]\label{compact_Cacc}
Let $\{E_k\}$ be a sequence of Caccioppoli sets s.t.
\begin{equation*}
\sup_{k\in\mathbb{N}}P(E_k,\Omega)\leq c(\Omega)<\infty,
\end{equation*}
for any bounded open set $\Omega$. Then there exists a Caccioppoli set $E$ and a subsequence
$\{E_{k_i}\}$ of $\{E_k\}$ s.t.
\begin{equation*}
E_{k_i}\xrightarrow{loc}E.
\end{equation*}
\end{prop}

Since $P(E,\Omega)=V(\chi_E,\Omega)$, we also have the following property.\\
(Locality)$\quad$The mapping $E\longmapsto P(E,\Omega)$ is local i.e.
\begin{equation}\label{locality_perimeter}
P(E,\Omega)=P(F,\Omega),\qquad\textrm{whenever }|(E\Delta F)\cap\Omega|=0,
\end{equation}
(even if $E\not= F$ in measure outside $\Omega$).\\
As we will see, this will not be true for the fractional perimeter.

Actually,  since the characteristic functions of the two sets are equal in $L^1_{loc}(\Omega)$, the equality is at the level of measures i.e.
\begin{equation*}
D\chi_E\llcorner\Omega=D\chi_F\llcorner\Omega.
\end{equation*}

As a consequence we can modify a Caccioppoli set with a set of negligible Lebesgue measure
without changing its perimeter. Therefore the notion of topological boundary is not very useful when dealing with Caccioppoli sets (as is shown by the example below) and we cannot expect formula $(\ref{smooth_perimeter2})$ to hold for every Caccioppoli set.

\begin{ese}
Consider a bounded open set $E\subset\R$ with $C^2$ boundary, as above; now let $F:=E\cup\mathbb{Q}^n$.
Then $|(E\Delta F)\cap\Omega|=0$ for every $\Omega\subset\R$ open and hence $D\chi_F=D\chi_E=\nu\Han\llcorner\partial E$.
However $\partial F=\R\setminus E$.
\end{ese}

\end{section}

\begin{section}{Regularity of the Boundary}

We saw that we can modify a set of (locally) finite perimeter with a set of zero Lebesgue measure, making its topological boundary as big as we want, without changing its perimeter. For this reason one introduces measure theoretic notions of interior, exterior and boundary. We will see below that in some sense we can also minimize the size of the topological boundary.
\begin{defin}
Let $E\subset\R$. For every $t\in[0,1]$ define the set
\begin{equation}\label{density_t}
E^{(t)}:=\left\{x\in\R\,\big|\,\exists\lim_{r\to0}\frac{|E\cap B_r(x)|}{\omega_nr^n}=t\right\},
\end{equation}
of points density $t$ of $E$. The sets $E^{(0)}$ and $E^{(1)}$ are respectively the measure theoretic exterior and interior of the set $E$. The set
\begin{equation}\label{ess_bdry}
\partial_eE:=\R\setminus(E^{(0)}\cup E^{(1)})
\end{equation}
is the essential boundary of $E$.
\end{defin}

Using the Lebesgue points Theorem for the characteristic function $\chi_E$, we see that the limit in $(\ref{density_t})$ exists for a.e. $x\in\R$ and
\begin{equation*}
\lim_{r\to0}\frac{|E\cap B_r(x)|}{\omega_nr^n}=\left\{\begin{array}{cc}1,&\textrm{a.e. }x\in E,\\
0,&\textrm{a.e. }x\in\Co E.
\end{array}
\right.
\end{equation*}
So
\begin{equation*}
|E\Delta E^{(1)}|=0,\qquad|\Co E\Delta E^{(0)}|=0\qquad\textrm{and }|\partial_eE|=0.
\end{equation*}
In particular every set $E$ is equivalent to its measure theoretic interior.
Notice that $E^{(1)}$ in general is not open.

Recall that the support of a Radon measure $\mu$ on $\R$ is defined as the set
\begin{equation*}
\supp\mu:=\{x\in\R\,|\,\mu(B_r(x))>0\textrm{ for every }r>0\}.
\end{equation*}
Notice that, being the complementary of the union of all open sets of measure zero, it is a closed set.
In particular, if $E$ is a Caccioppoli set, we have
\begin{equation*}
\supp|D\chi_E|=\{x\in\R\,|\,P(E,B_r(x))>0\textrm{ for every }r>0\}.
\end{equation*}

\begin{defin}
The reduced boundary of a Caccioppoli set $E$ is the set
\begin{equation}
\partial^*E:=\left\{x\in\supp|D\chi_E|\,\Big|\,\exists\,\lim_{r\to0}\frac{D\chi_E(B_r(x))}{|D\chi_E|(B_r(x))}=:\nu_E(x)\in\Sp\right\}.
\end{equation}
The function $\nu_E:\partial^*E\longrightarrow\Sp$ is called measure theoretic inner unit normal to $E$.
\end{defin}

\noindent
Notice that
the function $\nu_E$ is (by definition) the Radon-Nikodym derivative
\begin{equation*}
\nu_E=\frac{d\,D\chi_E}{d\,|D\chi_E|}.
\end{equation*}
%i.e. the function obtained in the polar decomposition $D\chi_E=\nu_E|D\chi_E|$ of the vector valued Radon measure $D\chi_E$.
%and hence is well defined for $|D\chi_E|$-a.e. $x\in\R$.
Since $|D\chi_E|(\R\setminus\partial^*E)=0$, we have
\begin{equation*}
|D\chi_E|=|D\chi_E|\llcorner\partial^*E\qquad\textrm{and}\qquad D\chi_E=\nu_E|D\chi_E|\llcorner\partial^*E.
\end{equation*}

As the following Theorem shows, the reduced boundary of a Caccioppoli set is quite regular and allows us to generalize in some sense formula $(\ref{smooth_perimeter2})$.
\begin{teo}[De Giorgi]
Let $E\subset\R$ be a Caccioppoli set. Then

(i) The reduced boundary $\partial^*E$ is locally $(n-1)$-rectifiable and its approximate tangent plane at $x$ is normal to $\nu_E(x)$ for $\Han$-a.e. $x\in\partial^*E$.

(ii) $|D\chi_E|=\Han\llcorner\partial^*E$ and $D\chi_E=\nu_E\Han\llcorner\partial^*E$.

(iii) We have the following Gauss-Green formula
\begin{equation}
\int_E\textrm{div }\varphi\,dx=-\int_{\partial^*E}\varphi\cdot\nu_E\,d\Han\qquad\forall\,\varphi\in C^1_c(\R,\R).
\end{equation}
\end{teo}

The following Proposition shows the relation between the essential and the reduced boundary.

\begin{prop}
Let $E\subset\R$ be a Caccioppoli set. Then
\begin{equation*}
\partial^*E\subset E^{(1/2)}\subset\partial_eE,
\end{equation*}
and
\begin{equation}
\Han(\partial_eE\setminus\partial^*E)=0.
\end{equation}
\end{prop}
In particular we can as well take the essential boundary in the Gauss-Green formula or when calculating the perimeter.\\
Actually we have also the following characterization
\begin{teo}
A set $E\subset\R$ is a Caccioppoli set if and only if
\begin{equation*}
\Han(\partial_e E\cap K)<\infty,
\end{equation*}
for every $K\subset\R$ compact.
\end{teo}

\begin{rmk}
In particular this implies that any bounded open set with Lipschitz boundary has finite perimeter.
\end{rmk}

We have yet another natural way to define a measure theoretic boundary.
\begin{defin}
Let $E\subset\R$ and define the sets
\begin{equation*}\begin{split}
&E_1:=\{x\in\R\,|\,\exists r>0,\,|E\cap B_r(x)|=\omega_nr^n\},\\
&
E_0:=\{x\in\R\,|\,\exists r>0,\,|E\cap B_r(x)|=0\}.
\end{split}\end{equation*}
Then we define
\begin{equation*}\begin{split}
\partial^-E&:=\R\setminus(E_0\cup E_1)\\
&
=\{x\in\R\,|\,0<|E\cap B_r(x)|<\omega_nr^n\textrm{ for every }r>0\}.
\end{split}
\end{equation*}
\end{defin}
Notice that $E_0$ and $E_1$ are open sets and hence $\partial^-E$ is closed. Moreover, since
\begin{equation}\label{density_subsets}
E_0\subset E^{(0)}\qquad\textrm{and}\qquad E_1\subset E^{(1)},
\end{equation}
we have
\begin{equation*}
\partial_eE\subset\partial^-E.
\end{equation*}
We have
\begin{equation}\label{ess_bdry_top1}
F\subset\R\textrm{ s.t. }|E\Delta F|=0\quad\Longrightarrow\quad\partial^-E\subset\partial F.
\end{equation}
Indeed, if $|E\Delta F|=0$, then $|F\cap B_r(x)|=|E\cap B_r(x)|$ for every $r>0$. In particular for any $x\in\partial^-E$ we have
\begin{equation*}
0<|F\cap B_r(x)|<\omega_nr^n,
\end{equation*}
which implies
\begin{equation*}
F\cap B_r(x)\not=\emptyset\quad\textrm{and}\quad\Co F\cap B_r(x)\not=\emptyset\quad\textrm{for every }r>0,
\end{equation*}
and hence $x\in\partial F$.\\
In particular, $\partial^-E\subset\partial E$.

Moreover
\begin{equation}\label{ess_bdry_top2}
\partial^-E=\partial E^{(1)}.
\end{equation}
Indeed, since $|E\Delta E^{(1)}|=0$, we already know that $\partial^-E\subset\partial E^{(1)}$ and
the converse inclusion is clear from $(\ref{density_subsets})$.\\
From $(\ref{ess_bdry_top1})$ and $(\ref{ess_bdry_top2})$ we see that
\begin{equation*}
\partial^-E=\bigcap_{F\sim E}\partial F,
\end{equation*}
where the intersection is taken over all sets $F\subset\R$ s.t. $|E\Delta F|=0$,
so we can think of 
$\partial^-E$ as a way to minimize the size of the topological boundary of $E$.
In particular
\begin{equation*}
F\subset\R\textrm{ s.t. }|E\Delta F|=0\quad\Longrightarrow\quad\partial^-F=\partial^-E.
\end{equation*}

If $E$ is a Caccioppoli set, then it is easy to verify that
\begin{equation*}
\partial^-E=\supp|D\chi_E|=\overline{\partial^*E}.
\end{equation*}
However notice that in general the inclusions
\begin{equation*}
\partial^*E\subset\partial_eE\subset\partial^-E\subset\partial E
\end{equation*}
are all strict and in principle we could have
\begin{equation*}
\Han(\partial^-E\setminus\partial^*E)>0.
\end{equation*}

\begin{rmk}\label{gmt_assumption}
Let $E\subset\R$. From what we have seen above, up to modifying $E$ on a set of measure zero, we can assume that
\begin{equation}\label{gmt_assumption_eq}
\begin{split}
&E_1\subset E,\qquad E\cap E_0=\emptyset\\
\textrm{and}\quad\partial E=\partial^-E&=\{x\in\R\,|\,0<|E\cap B_r(x)|<\omega_nr^n,\,\forall\,r>0\}.
\end{split}
\end{equation}
We will make this assumption in later chapters.
\end{rmk}

\end{section}

\begin{section}{Minimal Surfaces}

In this section we recall the definition of minimal surfaces and we present the sketch of a proof due to Savin for the regularity of their reduced boundary.\\
The classical proof relies on the monotonicity formula and the approximation of a minimal surface with appropriate harmonic functions.
The proof by Savin makes use of viscosity solutions and a suitable Harnack inequality, which allows to prove an improvement of flatness Theorem; from this one easily obtains $C^{1,\alpha}$ regularity (and hence also smoothness) of the reduced boundary of a minimal surface.\\
In this section we suppose that every set satisfies $(\ref{gmt_assumption_eq})$. In particular every Caccioppoli set $E$ satisfies
\begin{equation*}
\partial E=\partial^-E=\supp |D\chi_E|.
\end{equation*}

\begin{defin}
Let $\Omega\subset\R$ be a bounded open set. We say that a Caccioppoli set $E$ has minimal perimeter in $\Omega$, or that $\partial E$ is a minimal surface in $\Omega$, if it has minimal perimeter among the sets which agree with $E$ outside a compact subset of $\Omega$, i.e.
\begin{equation}
P(E,\Omega)\leq P(F,\Omega),
\end{equation}
for every Caccioppoli set $F$ s.t. $E\Delta F\subset\subset\Omega$.\\
If $\Omega$ is not bounded, in particular if $\Omega=\R$, we say that $E$ has minimal perimeter in $\Omega$ if
the above inequality holds for every bounded open subset $\Omega'\subset\Omega$ s.t. $E\Delta F\subset\subset\Omega'$.
\end{defin}

Using the compactness and semicontinuity theorems, it is immediate to get the existence of minimal surfaces, via the direct method of Calculus of Variation.
\begin{prop}
Let $\Omega\subset\R$ be a bounded open set and fix a set $E_0$ of finite perimeter. Then there exists a set of finite perimeter $E$ s.t.
$E=E_0$ in $\Co\Omega$ and
\begin{equation}\label{local_Plateau}
P(E)=\inf\left\{P(F)\,|\,F\setminus\Omega=E_0\setminus\Omega\right\}.
\end{equation}
\end{prop}

The problem of finding a minimal set as in $(\ref{local_Plateau})$ is known as the Plateau problem. Since the perimeter is local, it is quite clear that the behavior of the set $E_0$ far from $\Omega$ doesn't matter. Roughly speaking $E_0$ has the role of boundary data and we want to find the surface which minimizes the area
among all surfaces having as boundary $\partial E_0\cap\partial\Omega$.
\begin{rmk}
Notice that a set solving the Plateau problem $(\ref{local_Plateau})$ has minimal perimeter in $\Omega$.
\end{rmk}

Now there are two main problems: to show that the reduced boundary of a set of minimal perimeter is actually smooth
and to understand how big the singular set $\partial E\setminus\partial^*E$ can be.\\

First of all notice that if $E$ has minimal perimeter in $\Omega$, then $\lambda E$ has minimal perimeter in $\lambda\Omega$ for every $\lambda>0$ thanks to the scaling of the perimeter.\\
Moreover if $E$ has minimal perimeter in $\Omega$, then it has minimal perimeter also in every $\Omega'\subset\Omega$ open.\\
Indeed, let $F$ be a Caccioppoli set s.t. $E\Delta F\subset\subset\Omega'$. Then also $E\Delta F\subset\subset\Omega$ and
\begin{equation*}\begin{split}
P(E,\Omega')+P(E,\Omega\setminus\Omega')&=P(E,\Omega)\\
&
\leq P(F,\Omega)=P(F,\Omega')+P(F,\Omega\setminus\Omega')\\
&
=P(F,\Omega')+P(E,\Omega\setminus\Omega').
\end{split}\end{equation*}
In particular when we want to study the regularity of a point $x\in\partial E$, using translations and dilations we can reduce to the study of a set $E$ of minimal perimeter in $B_1$ s.t. $0\in\partial E$.

Using the isoperimetric inequality we can obtain the following uniform density estimates
\begin{prop}
Let $E$ be a set of minimal perimeter in $B_1$ s.t. $0\in\partial B_1$. There exists a constant $c=c(n)>0$ s.t. for every $r\in (0,1)$ we have
\begin{equation*}
|E\cap B_r|\geq cr^n,\qquad|\Co E\cap B_r|\geq cr^n.
\end{equation*}
\end{prop}

These and similar estimates holding for the perimeter instead of the Lebesgue measure yield the first regularity result for minimal surfaces.
\begin{prop}
If $E$ is a set of minimal perimeter in $\Omega$, then
\begin{equation*}
\Han((\partial E\setminus\partial^*E)\cap\Omega)=0.
\end{equation*}
\end{prop}

Moreover we have the following compactness property
\begin{prop}
Let $\{E_k\}$ be a sequence of sets of minimal perimeter in $\Omega$. Then there exists a subsequence $\{E_{k_i}\}$ converging to
a set $E$ of minimal perimeter in $\Omega$
\begin{equation*}
E_{k_i}\xrightarrow{loc}E.
\end{equation*}
\end{prop}

Actually using the uniform estimates above we can show that the minimal surfaces $\partial E_k$ converge
in Hausdorff sense to $\partial E$ on any compact subset of $\Omega$, i.e. for every $\epsilon>0$ the surfaces
$\partial E_k$ lie in an $\epsilon$-neighborhood of $\partial E$ for every $k$ big enough (inside a compact subset of $\Omega$)

The case which interests us the most is the following. Let $E$ be a set of minimal perimeter in $B_1$ s.t. $0\in\partial E$ and consider the blow-ups
\begin{equation*}
E_r:=\{x\in\R\,|\,rx\in E\}=\frac{1}{r}E,
\end{equation*}
for $r\to0$.\\
As remarked above, these are sets of minimal perimeter in $B_r$, and hence also in $B_1$,
if $r\leq1$.\\
The compactness property implies that we can find a sequence $r_k\to0$ s.t.
\begin{equation*}
E_{r_k}\xrightarrow{loc} E_\infty,
\end{equation*}
for a set $E_\infty$ of minimal perimeter in $B_1$.\\
Actually the set $E_\infty$ has minimal perimeter in $\R$ and it is easily seen that it is a cone i.e. there exists a set $A$ s.t.
\begin{equation*}
E_\infty=\{tx\,|\,t\geq0,\,x\in A\}.
\end{equation*}
The set $E_\infty$ is called tangent cone to $E$ at $0$. We say that it is a singular cone if it is not a half-space.\\
If the point $0$ belongs to the reduced boundary of $E$, $0\in\partial^*E$, then the blow-up limit is actually the approximate tangent plane at $0$,
\begin{equation*}
\begin{split}
&\qquad\qquad\partial E_\infty=H(0)=\nu_E(0)^\bot,\\
&E_\infty=H^-(0)=\{x\in\R\,|\,x\cdot\nu_E(0)\leq0\}.
\end{split}
\end{equation*}
\begin{rmk}
This is true for a general Caccioppoli set $E$, not necessarily of minimal perimeter. Let $x\in\partial^*E$; up to translation we can suppose $x=0\in\partial^*E$. Then
\begin{equation*}
E_r\xrightarrow{loc}H^-(0),\qquad\textrm{as }r\to0.
\end{equation*}
\end{rmk}

Also notice that $0\in\partial E_r$ for every $r>0$. Roughly speaking we are zooming in on $0$ and if it is a regular point, then we see the boundary of $E$ becoming flatter and flatter, untill it becomes a plane.\\
Up to rotation we can suppose $\nu_E(0)=e_n$, so that
\begin{equation*}
H^-(0)=\{x\in\R\,|\,x_n\leq0\}.
\end{equation*}
As remarked above we see that if $E$ is of minimal perimeter in $B_1$ and $0\in\partial E$ is s.t. $E_\infty=\{x_n\leq0\}$ (in particular if $0\in\partial^*E$), then for every $\epsilon>0$
\begin{equation*}
\partial E_r\cap B_1\subset\{x\in\R\,|\,|x_n|<\epsilon\},
\end{equation*}
for $r$ small enough.\\

The fundamental result for the regularity of a minimal surface is the following flatness Theorem
\begin{teo}[De Giorgi]\label{original_flatness}
Let $E$ be a set of minimal perimeter in $B_1$ with $0\in\partial E$ s.t.
\begin{equation}\label{loc_min_per}
\partial E\cap B_1\subset\{|x_n|\leq\epsilon_0\},
\end{equation}
where $\epsilon_0=\epsilon_0(n)>0$ is a small constant depending only on $n$. Then $\partial E$ is an analytic surface in
$B_{1/2}$.
\end{teo}
It is clear from the discussion above that we can apply this Theorem to every point $x\in\partial E\cap\Omega$ of a surface of minimal perimeter in $\Omega$ which has as blow-up limit a half-space.\\
Indeed let $x$ be such a point; after translation we can suppose that $x=0$. Then for $r$ small enough the set $E_r$ satisfies the hypothesis of the Theorem, so scaling back we see that $\partial E$ is an analytic surface in $B_{r/2}$.\\
Also notice that every such point $x$ must belong to the reduced boundary of $E$, since $\partial E$ is smooth
in a neighborhood of $x$.

This implies that the reduced boundary $\partial^*E\cap\Omega$ of a set of minimal perimeter in $\Omega$ is an analytic surface
and the singular set $(\partial E\setminus\partial^*E)\cap\Omega$ coincides with the set of points of $\partial E\cap\Omega$ having as blow-up limit a singular minimal cone.\\
Notice that such a cone must indeed be singular (at least) in $0$ by construction.

A result of Simons shows that in dimension $n\leq7$ the only global minimal surfaces are planes. Since there cannot be singular minimal cones, then the singular set is empty if $n\leq7$.

Moreover, using a dimension reduction argument, Federer proved that in dimension $n\geq8$ the singular set can have Hausdorff dimension at most equal to $n-8$. Also notice that there exist examples of minimal sets for which this estimate is sharp. To sum up, we have
\begin{teo}\label{local_min_reg}
Let $\Omega\subset\R$ be a bounded open set. If $E$ is a set of minimal perimeter in $\Omega$, then $\partial^*E\cap\Omega$
is an analytic hypersurface, which is relatively open in $\partial E\cap\Omega$. Moreover the singular set $(\partial E\setminus\partial^*E)\cap\Omega$
satisfies the following properties:

(i) if $2\leq n\leq7$, then $(\partial E\setminus\partial^*E)\cap\Omega=\emptyset$,

(ii) if $n=8$, then $(\partial E\setminus\partial^*E)\cap\Omega$ has no accumulation points,

(iii) if $n\geq9$, then $\mathcal{H}^s((\partial E\setminus\partial^*E)\cap\Omega)=0$ for every $s>n-8$.\\
There exists a perimeter minimizer $E$ in $\mathbb{R}^8$ s.t. $\mathcal{H}^0(\partial E\setminus\partial^*E)=1$
and if $n\geq9$ there exists a perimeter minimizer $E$ in $\R$
s.t. $\mathcal{H}^{n-8}(\partial E\setminus\partial^*E)=\infty$.
\end{teo}

The main difficulty in proving Theorem $\ref{original_flatness}$ is the fact that a priori $\partial E$ cannot be written as a graph.\\
Define for every $r>0$ the cylinder
\begin{equation*}
C_r:=B'_r\times(-r,r)=\{(x',x_n)\in\R\,|\,|x'|<r,\,|x_n|<r\}.
\end{equation*}
Now let $E$ be a surface of minimal perimeter in $C_1$, with $0\in\partial E$. Suppose that $\partial E$ is the graph of a Lipschitz function $u:B'_1\longrightarrow\mathbb{R}$, with $u(0)=0$ and Lip$(u)=1$, i.e.
\begin{equation*}
\partial E\cap C_1=\{(x',u(x'))\in\R\,|\,x'\in B_1'\}
\end{equation*}
and that $E$ is the subgraph of $u$
\begin{equation*}
E\cap C_1=\{(x',x_n)\in\R\,|\,x'\in B_1',\,-1<x_n<u(x')\}.
\end{equation*}
Then for every Borel set $A\subset B_1'$ the perimeter of $E$ is computed as
\begin{equation*}
P(E,C_1\cap p^{-1}A)=\mathcal{A}(u,A):=\int_A\sqrt{1+|\nabla'u(x')|^2}\,dx',
\end{equation*}
where $p:\R\longrightarrow\mathbb{R}^{n-1}$, $(x',x_n)\mapsto x'$.

It is easily seen that since $E$ is a perimeter minimizer in $C_1$, then $u$ is a local minimizer of the area functional $\mathcal{A}(\cdot,B_1')$, meaning that for every compact subset $K\subset B_1'$ there exists $\epsilon>0$ s.t.
\begin{equation*}
\mathcal{A}(u,B_1')\leq\mathcal{A}(u+\varphi,B_1'),
\end{equation*}
for every $\varphi\in C^\infty_c(B_1')$ with $\supp\varphi\subset K$ and $|\varphi|<\epsilon$.\\
Moreover $u$ is a Lipschitz local minimizer for the area functional $\mathcal{A}(\cdot,B_1')$ if and only if
\begin{equation*}
\int_{B_1'}\frac{\nabla'u(x')}{\sqrt{1+|\nabla'u(x')|^2}}\cdot\nabla'\varphi(x')\,dx'=0,
\end{equation*}
for every $\varphi\in C_c^\infty(B_1')$ i.e. if and only if it is a weak solution in $B_1'$ of the Euler-Lagrange equation
\begin{equation}\label{min_surf_eq}
\textrm{div}\left(\frac{\nabla'u}{\sqrt{1+|\nabla'u|^2}}\right)=0,
\end{equation}
which is the so called minimal surface equation.\\
Notice that this equation expresses in local coordinates the vanishing of the mean curvature of the hypersurface $\partial E\cap C_1$, defined as the graph of $u$ over $B_1'$.

Moreover notice that once we prove that $u\in C^{1,\alpha}(B_1')$, then classical bootstrapping arguments of the elliptic regularity theory imply that $u$ actually is analytic.\\

As we said earlier the problem is that we do not know if $\partial E$ can be written as a graph.\\
In any case it is proved in \cite{CC} that if $E$ is a set of minimal perimeter in $B_1$, then $\partial E\cap B_1$ is a viscosity solution of
equation $(\ref{min_surf_eq})$.
\begin{defin}
The boundary $\partial E$ of a set $E$ satisfies the minimal surface equation $(\ref{min_surf_eq})$ in the viscosity sense if
for every smooth function $\varphi$ which has the subgraph
\begin{equation*}
S=\{x_n<\varphi(x')\}
\end{equation*}
included
in $E$ or in $\Co E$ in a small ball $B_r(y)$ centered in some point
\begin{equation*}
y\in\partial S\cap\partial E
\end{equation*}
we have
\begin{equation*}
\textrm{div}\left(\frac{\nabla'\varphi}{\sqrt{1+|\nabla'\varphi|^2}}\right)\leq0,
\end{equation*}
and if we consider supergraphs then the opposite inequality holds.
\end{defin}
This means that if we touch the boundary $\partial E$ with the graph of a smooth function from the inside (outside) of $E$, then the corresponding inequality holds.
For some details about viscosity solutions see the Appendix.\\
Now the idea is to adapt the methods used in the proof of the Harnack inequality for viscosity solutions (see Theorem $\ref{Harnack_vi}$ in Appendix A)
in order to get the following
\begin{teo}[Harnack inequality]
Let $E$ be a set of minimal perimeter in $B_1$ s.t. $0\in\partial E$ and
\begin{equation*}
\partial E\cap B_1\subset\{|x_n|\leq\epsilon\},
\end{equation*}
with $\epsilon\leq\epsilon_0(n)$. Then
\begin{equation*}
\partial E\cap B_{1/2}\subset\{|x_n|\leq(1-\eta)\epsilon\},
\end{equation*}
where $\eta>0$ is a small universal constant.
\end{teo}

\noindent
Then, arguing by contradiction, we obtain the following improvement of flatness Theorem for $\partial E$
\begin{teo}[Improvement of flatness]
Let $E$ be a set of minimal perimeter in $B_1$ s.t. $0\in\partial E$ and
\begin{equation*}
\partial E\cap B_1\subset\{|x_n|\leq\epsilon\},
\end{equation*}
with $\epsilon\leq\epsilon_0(n)$. Then there exists $\nu_1\in\Sp$ s.t.
\begin{equation}
\partial E\cap B_{r_0}\subset\left\{|x\cdot\nu_1|\leq\frac{\epsilon}{2}r_0\right\},
\end{equation}
where $r_0$ is a small universal constant.
\end{teo}
Roughly speaking this means that once we know that $\partial E$ in a neighborhood of $0$ is contained in
a cylinder of small heigth, then in a smaller neighborhood it is actually contained in a flatter cylinder, up to changing the coordinates.\\
Applying this Theorem inductively one can then show that $\partial E$ is actually a $C^{1,\alpha}$ graph in $B_{3/4}$ and hence, as remarked above, it is analytic.\\

We conclude this section recalling the
monotonicity formula, since we will later need to find an analogue for fractional perimeters.
\begin{teo}[Monotonicity Formula]
Let $E$ be a set of minimal perimeter in $\Omega$ and let $x_0\in\partial E\cap\Omega$. Then the density ratios
\begin{equation*}
\frac{P(E,B_r(x_0))}{n\omega_nr^{n-1}}
\end{equation*}
are increasing for $r\in(0,d(x_0,\partial\Omega))$.
\end{teo}

\noindent
Notice that if $E$ is a cone, then the above ratio is constant.

\end{section}
\end{chapter}

%%%%%%%%%%%%%%%%%%%%%%%%%%%%%%%%%%%%%%%%
%%%%%%%%%%%%%%%%%%%%%%%%%%%%%%%%%%%%%%%%
%%%%%%%%%%%%%%%%%%%%%%%%%%%%%%%%%%%%%%%%

%%%%%%%%%%%%%%%%%%%%%%%%%%%%%%%%%%%%%%%%
%%%%%%%%%%%%%%%%%%%%%%%%%%%%%%%%%%%%%%%%
%%%%%%%%%%%%%%%%%%%%%%%%%%%%%%%%%%%%%%%%

%%%%%%%%%%%%%%%%%%%%%%%%%%%%%%%%%%%%%%%%
%%%%%%%%%%%%%%%%%%%%%%%%%%%%%%%%%%%%%%%%
%%%%%%%%%%%%%%%%%%%%%%%%%%%%%%%%%%%%%%%%

%%%%%%%%%%%%%%%%%%%%%%%%%%%%%%%%%%%%%%%%
%%%%%%%%%%%%%%%%%%%%%%%%%%%%%%%%%%%%%%%%
%%%%%%%%%%%%%%%%%%%%%%%%%%%%%%%%%%%%%%%%

%%%%%%%%%%%%%%%%%%%%%%%%%%%%%%%%%%%%%%%%
%%%%%%%%%%%%%%%%%%%%%%%%%%%%%%%%%%%%%%%%
%%%%%%%%%%%%%%%%%%%%%%%%%%%%%%%%%%%%%%%%

%%%%%%%%%%%%%%%%%%%%%%%%%%%%%%%%%%%%%%%%
%%%%%%%%%%%%%%%%%%%%%%%%%%%%%%%%%%%%%%%%
%%%%%%%%%%%%%%%%%%%%%%%%%%%%%%%%%%%%%%%%

%%%%%%%%%%%%%%%%%%%%%%%%%%%%%%%%%%%%%%%%
%%%%%%%%%%%%%%%%%%%%%%%%%%%%%%%%%%%%%%%%
%%%%%%%%%%%%%%%%%%%%%%%%%%%%%%%%%%%%%%%%

%%%%%%%%%%%%%%%%%%%%%%%%%%%%%%%%%%%%%%%%
%%%%%%%%%%%%%%%%%%%%%%%%%%%%%%%%%%%%%%%%
%%%%%%%%%%%%%%%%%%%%%%%%%%%%%%%%%%%%%%%%

%%%%%%%%%%%%%%%%%%%%%%%%%%%%%%%%%%%%%%%%
%%%%%%%%%%%%%%%%%%%%%%%%%%%%%%%%%%%%%%%%
%%%%%%%%%%%%%%%%%%%%%%%%%%%%%%%%%%%%%%%%

\begin{chapter}{Fractional Perimeter}

\begin{section}{Fractional Sobolev Spaces}

We recall (see \cite{HitGuide}) the definition of fractional Sobolev space and some embedding properties which will be used in the sequel.

\begin{defin}
Let $\Omega\subset\R$ be an open set and fix $p\in[1,\infty)$, $s\in(0,1)$. Then we define the fractional Sobolev space
\begin{equation*}
W^{s,p}(\Omega):=\left\{u\in L^p(\Omega)\Big|\int_\Omega\int_\Omega\frac{|u(x)-u(y)|^p}{|x-y|^{n+sp}}\,dx\,dy<\infty\right\}.
\end{equation*}
The term
\begin{equation*}
[u]_{W^{s,p}(\Omega)}:=\left(\int_\Omega\int_\Omega\frac{|u(x)-u(y)|^p}{|x-y|^{n+sp}}\,dx\,dy\right)^\frac{1}{p}
\end{equation*}
is called Gagliardo seminorm of $u$.
\end{defin}

Endowed with the norm
\begin{equation*}
\|u\|_{W^{s,p}(\Omega)}:=\left(\|u\|_{L^p(\Omega)}^p+[u]_{W^{s,p}(\Omega)}^p\right)^\frac{1}{p},
\end{equation*}
$W^{s,p}(\Omega)$ is a Banach space.

When $p=2$, we write $H^s(\Omega)$ for the Hilbert space $W^{s,2}(\Omega)$.

For a fixed $p$, the fractional Sobolev spaces are intermediate between $L^p(\Omega)$ and $W^{1,p}(\Omega)$, as is shown by the following Propositions.

\begin{prop}\label{cont_scale}
Let $\Omega\subset\R$ be an open set and let $p\in[1,\infty)$, $0<s\leq t<1$. Then $\exists C=C(n,s,p)\geq1$ s.t.
for every measurable $u:\Omega\longrightarrow\mathbb{R}$
\begin{equation*}
\|u\|_{W^{s,p}(\Omega)}\leq C\|u\|_{W^{t,p}(\Omega)}.
\end{equation*}
In particular we have the continuous embedding
\begin{equation*}
W^{t,p}(\Omega)\hookrightarrow W^{s,p}(\Omega).
\end{equation*}
\end{prop}

In order to prove the embedding $W^{1,p}(\Omega)\hookrightarrow W^{s,p}(\Omega)$, we need to impose some
regularity condition on the boundary of $\Omega$, because in the proof we make use of an extension property.\\
To be more precise, we say that an open set $\Omega\subset\R$ is an extension domain for $W^{s,p}$ if
$\exists C=C(s,p,\Omega)\geq0$ s.t. for every $u\in W^{s,p}(\Omega)$ there exists $\tilde{u}\in W^{s,p}(\R)$ with
$\tilde{u}_{|\Omega}=u$ and $\|\tilde{u}\|_{W^{s,p}(\R)}\leq C\|u\|_{W^{s,p}(\Omega)}$.

We say that $\Omega$ is an extension domain if it is an extension domain for $W^{s,p}$ for every $p\in[1,\infty)$ and $s\in(0,1)$.
It can be proved that any open set $\Omega$ with bounded $C^{0,1}$ boundary $\partial\Omega$ is an extension domain (see \cite{HitGuide} for a proof and a counterexample).
We consider $\R$ itself as an extension domain (for terminology semplicity).

\begin{prop}
Let $\Omega\subset\R$ be an extension domain and let $p\in[1,\infty)$, $0<s<1$. Then $\exists C=C(n,s,p)\geq1$ s.t.
for every measurable $u:\Omega\longrightarrow\mathbb{R}$
\begin{equation*}
\|u\|_{W^{s,p}(\Omega)}\leq C\|u\|_{W^{1,p}(\Omega)}.
\end{equation*}
In particular we have the continuous embedding
\begin{equation*}
W^{1,p}(\Omega)\hookrightarrow W^{s,p}(\Omega).
\end{equation*}
\end{prop}

We have the following Sobolev-type inequality.
\begin{teo}\label{fractional_sobolev}
Let $p\in[1,\infty)$ and $0<s<1$ s.t. $sp<n$. Then $\exists C=C(n,s,p)\geq0$ s.t. for every measurable
$u:\R\longrightarrow\mathbb{R}$ with compact support we have
\begin{equation}\label{sobolev_inequality}
\|u\|_{L^{p^*}(\R)}^p\leq C\int_{\R}\int_{\R}\frac{|u(x)-u(y)|^p}{|x-y|^{n+sp}}\,dx\,dy=C[u]_{W^{s,p}(\R)}^p,
\end{equation}
where $p^*=\frac{np}{n-sp}$ is the fractional critical exponent.
\end{teo}

As a consequence, using Holder inequality we get the following embeddings.
\begin{coroll}
Let $p\in[1,\infty)$ and $0<s<1$ s.t. $sp<n$. Then we have the continuous embedding
\begin{equation*}
W^{s,p}(\R)\hookrightarrow L^q(\R),\qquad\textrm{for every }q\in[p,p^*].
\end{equation*}
\end{coroll}

Exploiting the extension property and the above results, we find
\begin{teo}
Let $\Omega\subset\R$ be an extension domain and let $p\in[1,\infty)$, $s\in(0,1)$ s.t. $sp<n$. Then $\exists C=C(s,p,\Omega)\geq0$ s.t. for every $u\in W^{s,p}(\Omega)$
\begin{equation}
\|u\|_{L^q(\Omega)}\leq C\|u\|_{W^{s,p}(\Omega)}, \qquad\textrm{for every }q\in[p,p^*],
\end{equation}
i.e. we have the continuous embedding
\begin{equation*}
W^{s,p}(\Omega)\hookrightarrow L^q(\Omega),\qquad\textrm{for every }q\in[p,p^*].
\end{equation*}
Moreover, if $\Omega$ is bounded, then
\begin{equation*}
W^{s,p}(\Omega)\hookrightarrow L^q(\Omega),\qquad\textrm{for every }q\in[1,p^*].
\end{equation*}
\end{teo}

Actually in a bounded extension domain the embedding is compact, except the case of the critical exponent.
\begin{teo}\label{compact_embd_th}
Let $\Omega\subset\R$ be a bounded extension domain and let $p\in[1,\infty)$, $s\in(0,1)$ s.t. $sp<n$. Then
we have the compact embedding
\begin{equation*}
W^{s,p}(\Omega)\hookrightarrow\hookrightarrow L^q(\Omega),\qquad\textrm{for every }q\in[1,p^*).
\end{equation*}
\end{teo}

\end{section}

\begin{section}{Fractional Perimeter}

First of all we fix an index $s\in(0,1)$.

Now for every couple of disjoint sets $E$, $F\subset\R$ we define the functional
\begin{equation}
\Ll_s(E,F):=\int_E\int_F\frac{1}{\kers}\,dx\,dy=\int_{\R}\int_{\R}\frac{\chi_E(x)\chi_F(y)}{\kers}\,dx\,dy.
\end{equation}

\begin{defin}
Let $\Omega\subset\R$ be an open set. Then for every $E\subset\R$ we define the fractional $s$-perimeter of $E$ in $\Omega$ as
\begin{equation}
P_s(E,\Omega):=\Ll_s(E\cap\Omega,\Co E)+\Ll_s(E\setminus\Omega,\Co E\cap\Omega).
\end{equation}
If $\Omega=\R$, then we write $P_s(E):=P_s(E,\R)$ for the (global) $s$-perimeter of $E$.
%We will drop the index $s$ when it is clear from the context that it is fixed.
\end{defin}

Notice that, if $E\subset\Omega$, then $E\setminus\Omega=\emptyset$ and $E\cap\Omega=E$, so
\begin{equation}
P_s(E,\Omega)=\Ll_s(E,\Co E)=P_s(E).
\end{equation}
Moreover
\begin{equation*}
\begin{split}
\Ll_s(E,\Co E)&=\int_E\int_{\Co E}\frac{1}{\kers}\,dx\,dy\\
&
=\frac{1}{2}\int_{\R}\int_{\R}\frac{|\chi_E(x)-\chi_E(y)|}{\kers}\,dx\,dy=\frac{1}{2}[\chi_E]_{W^{s,1}(\R)}
\end{split}
\end{equation*}

\begin{rmk}
Since $|\chi_E(x)-\chi_E(y)|^p=|\chi_E(x)-\chi_E(y)|$, we could as well consider the $W^{t,p}$ norm, with
$t=\frac{s}{p}$. For this reason in the literature the index $\sigma\in(0,\frac{1}{2})$ is sometimes used in place of $s$.
In the sequel we will consider the index $s\in(0,1)$ and define $\sigma:=\frac{s}{2}\in(0,\frac{1}{2})$, which is the natural index when considering $H^\sigma$ norms.
\end{rmk}

For a general open set $\Omega\subset\R$ we can split the $s$-perimeter in the three terms
\begin{equation*}
P_s(E,\Omega)=\Ll_s(E\cap\Omega,\Co E\cap\Omega)+\Ll_s(E\cap\Omega,\Co E\setminus\Omega)+\Ll_s(E\setminus\Omega,\Co E\cap\Omega),
\end{equation*}
and regroup them as
\begin{equation*}\begin{split}
&P^L_s(E,\Omega):=\Ll_s(E\cap\Omega,\Co E\cap\Omega)=\frac{1}{2}[\chi_E]_{W^{s,1}(\Omega)},\\
&
P^{NL}_s(E,\Omega):=\Ll_s(E\cap\Omega,\Co E\setminus\Omega)+\Ll_s(E\setminus\Omega,\Co E\cap\Omega).
\end{split}
\end{equation*}

The term $P^L_s(E,\Omega)$ can be considered as the local contribution to the fractional perimeter of $E$ in $\Omega$ in the sense that, given two sets $E$, $F\subset\R$ s.t. $|(E\Delta F)\cap\Omega|=0$, we clearly have $P^L_s(E,\Omega)=
P_s^L(F,\Omega)$.\\
However if $|(E\Delta F)\cap\Co\Omega|>0$, then (in general) we will have $P_s(E,\Omega)\not=P_s(F,\Omega)$, unlike what happens with the classical perimeter.\\
This means that the fractional perimeter is nonlocal.

In the following Proposition we collect some elementary properties of the $s$-perimeter.
\begin{prop}\label{elementary_properties}
Let $s\in(0,1)$ and $\Omega\subset\R$ open.

(i) (Subadditivity)$\quad$ Let $E,\,F\subset\R$ s.t. $|E\cap F|=0$. Then
\begin{equation}\label{subadditive}
P_s(E\cup F,\Omega)\leq P_s(E,\Omega)+P_s(F,\Omega).
\end{equation}

(ii) (Translation invariance)$\quad$ Let $E\subset\R$ and $x\in\R$. Then
\begin{equation}\label{translation_invariance}
P_s(E+x,\Omega+x)=P_s(E,\Omega).
\end{equation}

(iii) (Rotation invariance)$\quad$ Let $E\subset\R$ and $\mathcal{R}\in SO(n)$ a rotation. Then
\begin{equation}\label{rotation_invariance}
P_s(\mathcal{R}E,\mathcal{R}\Omega)=P_s(E,\Omega).
\end{equation}

(iv) (Scaling)$\quad$ Let $E\subset\R$ and $\lambda>0$. Then
\begin{equation}\label{scaling}
P_s(\lambda E,\lambda\Omega)=\lambda^{n-s}P_s(E,\Omega).
\end{equation}
\begin{proof}
(i) follows from the following observations. Let $A_1,\,A_2,\,B\subset\R$. If $|A_1\cap A_2|=0$, then
\begin{equation*}\begin{split}
\Ll_s(A_1\cup A_2,B)&=\int_{A_1\cup A_2}\int_B\frac{dx\,dy}{\kers}\\
&=
\int_{A_1}\int_B\frac{dx\,dy}{\kers}+\int_{A_2}\int_B\frac{dx\,dy}{\kers}\\
&
=\Ll_s(A_1,B)+\Ll_s(A_2,B).
\end{split}
\end{equation*}
Moreover
\begin{equation}
A_1\subset A_2\quad\Longrightarrow\quad\Ll_s(A_1,B)\leq\Ll_s(A_2,B),
\end{equation}
and
\begin{equation*}
\Ll_s(A,B)=\Ll_s(B,A).
\end{equation*}
Therefore
\begin{equation*}\begin{split}
P_s(E\cup F,\Omega)&=\Ll_s((E\cup F)\cap\Omega,\Co(E\cup F))+\Ll_s((E\cup F)\setminus\Omega,\Co(E\cup F)\cap\Omega)\\
&
=\Ll_s(E\cap\Omega,\Co(E\cup F))+\Ll_s(F\cap\Omega,\Co(E\cup F))\\
&
\qquad+\Ll_s(E\setminus\Omega,\Co(E\cup F)\cap\Omega)+\Ll_s(F\setminus\Omega,\Co(E\cup F)\cap\Omega)\\
&
\leq\Ll_s(E\cap\Omega,\Co E)+\Ll_s(F\cap\Omega,\Co F)\\
&
\qquad+\Ll_s(E\setminus\Omega,\Co E\cap\Omega)+\Ll_s(F\setminus\Omega,\Co F\cap\Omega)\\
&
=P_s(E,\Omega)+P_s(F,\Omega).
\end{split}\end{equation*}

(ii), (iii) and (iv) follow simply by a change of variables in $\Ll_s$ and the following observations:
\begin{equation*}\begin{split}
&(x+A_1)\cap(x+A_2)=x+A_1\cap A_2,\qquad x+\Co A=\Co(x+A),\\
&
\mathcal{R}A_1\cap\mathcal{R}A_2=\mathcal{R}(A_1\cap A_2),\qquad\mathcal{R}(\Co A)=\Co(\mathcal{R}A),\\
&
(\lambda A_1)\cap(\lambda A_2)=\lambda(A_1\cap A_2),\qquad\lambda(\Co A)=\Co(\lambda A).
\end{split}\end{equation*}
For example, for claim (iv) we have
\begin{equation*}\begin{split}
\Ll_s(\lambda A,\lambda B)&=\int_{\lambda A}\int_{\lambda B}\frac{dx\,dy}{\kers}
=\int_A\lambda^n\,dx\int_B\frac{\lambda^n\,dy}{\lambda^{n+s}\kers}\\
&
=\lambda^{n-s}\Ll_s(A,B).
\end{split}
\end{equation*}
Then
\begin{equation*}\begin{split}
P_s(\lambda E,\lambda\Omega)&=\Ll_s(\lambda E\cap\lambda\Omega,\Co(\lambda E))+
\Ll_s(\lambda E\cap\Co(\lambda\Omega),\Co(\lambda E)\cap\lambda\Omega)\\
&
=\Ll_s(\lambda(E\cap\Omega),\lambda\Co E)+\Ll_s(\lambda(E\setminus\Omega),\lambda(\Co E\cap\Omega))\\
&
=\lambda^{n-s}\left(\Ll_s(E\cap\Omega,\Co E)+\Ll_s(E\setminus\Omega,\Co E\cap\Omega)\right)\\
&
=\lambda^{n-s}P_s(E,\Omega).
\end{split}\end{equation*}

\end{proof}
\end{prop}

Now a natural question is: what kind of sets (if any) have finite fractional perimeter?

The following embedding implies that any set with finite perimeter has also finite $s$-perimeter for any $s\in(0,1)$.

\begin{prop}\label{bv_embd}
Let $\Omega\subset\R$ be an extension domain and let $s\in(0,1)$. Then $\exists C=C(n,s)\geq1$ s.t. for every measurable $u:\Omega\longrightarrow\mathbb{R}$
\begin{equation}
\|u\|_{W^{s,1}(\Omega)}\leq C\left(\|u\|_{L^1(\Omega)}+V(u,\Omega)\right)=C\|u\|_{BV(\Omega)}.
\end{equation}
In particular we have the continuous embedding
\begin{equation*}
BV(\Omega)\hookrightarrow W^{s,1}(\Omega).
\end{equation*}
\begin{proof}
The claim is trivially satisfied if the right hand side is infinite, so let $u\in\bvo$.We only need to check that the Gagliardo seminorm
of $u$ is bounded by its $BV$-norm.\\
Since $\Omega$ is
an extension domain, we know that $\exists C=C(n,s)\geq1$ s.t.
\begin{equation*}
\|v\|_{\fracso}\leq C\|v\|_{W^{1,1}(\Omega)}.
\end{equation*}
Take an approximating sequence $\{u_k\}\subset C^\infty(\Omega)\cap\bvo$ as in Proposition $\ref{bv_approx}$.\\
Then
\begin{equation*}
[u_k]_{\fracso}\leq\|u_k\|_{\fracso}\leq C\|u_k\|_{W^{1,1}(\Omega)}=C\|u_k\|_{\bvo},
\end{equation*}
for each $k\in\mathbb{N}$.\\
Now using Fatou's Lemma we get
\begin{equation*}\begin{split}
[u]_{\fracso}&\leq\liminf_{k\to\infty}[u_k]_{\fracso}\leq C\liminf_{k\to\infty}\|u_k\|_{\bvo}=C\lim_{k\to\infty}\|u_k\|_{\bvo}\\
&
=C\|u\|_{\bvo}
\end{split}
\end{equation*}
and hence the claim.

\end{proof}
\end{prop}

\begin{rmk}
In particular we have
\begin{equation*}
W^{1,1}(\Omega)\subset\bvo\subset\bigcap_{s\in(0,1)}\fracso
\end{equation*}
and we know that the first inclusion is strict. We will show below that (in general) the second inclusion is strict too.
\end{rmk}

As a consequence, since $P_s(E)=\frac{1}{2}[\chi_E]_{\fracs}$, we have
\begin{coroll}
Let $E\subset\R$ be a set of finite perimeter i.e. $\chi_E\in\bv$. Then
$E$ has finite $s$-perimeter for every $s\in(0,1)$.
\end{coroll}

Acually, in case $E$ is bounded we can say more.
\begin{teo}\label{cacc_equiv}
Let $E\subset\R$ be bounded. Then the following are equivalent:

(i) $E$ has finite perimeter,

(ii) $E$ has finite $s$-perimeter for every $s\in(0,1)$ and
\begin{equation*}
\liminf_{s\to1}(1-s)P_s(E)<\infty,
\end{equation*}

(iii) $\exists\{s_k\}\subset(0,1),s_k\nearrow1$ s.t. $E$ has finite $s_k$-perimeter
for each $k\in\mathbb{N}$ and
\begin{equation*}
\sup_{k\in\mathbb{N}}(1-s_k)P_{s_k}(E)<\infty.
\end{equation*}
Moreover in this case we have
\begin{equation}
\lim_{s\to1}(1-s)P_s(E)=\omega_{n-1}P(E).
\end{equation}
\end{teo}

This is a consequence of the following results (see \cite{BBM} and \cite{Davila})
\begin{teo}[Bourgain, Brezis, Mironescu]\label{bb}
Let $\Omega\subset\R$ be a smooth bounded domain. Let $u\in L^1(\Omega)$. Then
$u\in\bvo$ if and only if
\begin{equation*}
\liminf_{n\to\infty}\int_\Omega\int_\Omega\frac{|u(x)-u(y)|}{|x-y|}\rho_n(x-y)\,dxdy<\infty,
\end{equation*}
and then
\begin{equation}\label{rough}
\begin{split}
C_1|Du|(\Omega)&\leq\liminf_{n\to\infty}\int_\Omega\int_\Omega\frac{|u(x)-u(y)|}{|x-y|}\rho_n(x-y)\,dxdy\\
&
\leq\limsup_{n\to\infty}\int_\Omega\int_\Omega\frac{|u(x)-u(y)|}{|x-y|}\rho_n(x-y)\,dxdy\leq C_2|Du|(\Omega),
\end{split}
\end{equation}
for some constants $C_1$, $C_2$ depending only on $\Omega$.
\end{teo}

This result was refined by Davila
\begin{teo}[Davila]
Let $\Omega\subset\R$ be a bounded open set with Lipschitz boundary. Let $u\in\bvo$. Then
\begin{equation}\label{correct}
\lim_{k\to\infty}\int_\Omega\int_\Omega\frac{|u(x)-u(y)|}{|x-y|}\rho_k(x-y)\,dxdy=K_{1,n}|Du|(\Omega),
\end{equation}
where
\begin{equation*}
K_{1,n}=\frac{1}{n\omega_n}\int_{\mathbb{S}^{n-1}}|v\cdot e|\,d\sigma(v),
\end{equation*}
with $e\in\R$ any unit vector.
\end{teo}

In the above Theorems $\rho_k$ is any sequence of radial mollifiers i.e. of functions satisfying
\begin{equation}\label{rule1}
\rho_k(x)\geq0,\quad\rho_k(x)=\rho_k(|x|),\quad\int_{\R}\rho_k(x)\,dx=1
\end{equation}
and
\begin{equation}\label{rule2}
\lim_{k\to\infty}\int_\delta^\infty\rho_k(r)r^{n-1}dr=0\quad\textrm{for all }\delta>0.
\end{equation}

In particular, for $R>\textrm{diam}(\Omega)$, we can consider
\begin{equation*}
\rho(x):=\chi_{[0,R]}(|x|)\frac{1}{|x|^{n-1}}
\end{equation*}
and define for any sequence $\{s_k\}\subset(0,1),s_k\nearrow1$,
\begin{equation*}
\rho_k(x):=(1-s_k)\rho(x)c_{s_k}\frac{1}{|x|^{s_k}},
\end{equation*}
where the $c_{s_k}$ are normalizing constants. Then
\begin{equation*}\begin{split}
\int_{\R}\rho_k(x)\,dx&=(1-s_k)c_{s_k}n\omega_n\int_0^R\frac{1}{r^{n-1+s_k}}r^{n-1}\,dr\\
&
=(1-s_k)c_{s_k}n\omega_n\int_0^R\frac{1}{r^{s_k}}\,dr=c_{s_k}n\omega_nR^{1-s_k},
\end{split}
\end{equation*}
and hence taking $c_{s_k}:=\frac{1}{n\omega_n}R^{s_k-1}$ gives $(\ref{rule1})$; notice that
$c_{s_k}\to\frac{1}{n\omega_n}$.\\
Also
\begin{equation*}\begin{split}
\lim_{k\to\infty}\int_\delta^\infty\rho_k(r)r^{n-1}\,dr&=
\lim_{k\to\infty}(1-s_k)c_{s_k}\int_\delta^R\frac{1}{r^{s_k}}\,dr\\
&
=\lim_{k\to\infty}c_{s_k}(R^{1-s_k}-\delta^{1-s_k})=0,
\end{split}
\end{equation*}
giving $(\ref{rule2})$.\\
With this choice we get
\begin{equation*}
\int_\Omega\int_\Omega\frac{|u(x)-u(y)|}{|x-y|}\rho_k(x-y)\,dxdy=c_{s_k}(1-s_k)[u]_{W^{s_k,1}(\Omega)}.
\end{equation*}
Then, if $u\in\bvo$, Davila's Theorem gives
\begin{equation}\label{limitperimeter}\begin{split}
\lim_{s\to1}(1-s)[u]_{W^{s,1}(\Omega)}&=\lim_{s\to1}\frac{1}{c_s}(c_s(1-s)[u]_{W^{s,1}(\Omega)})\\
&
=n\omega_nK_{1,n}|Du|(\Omega).
\end{split}
\end{equation}

Now we need to compute the constant $K_{1,n}$.\\
Notice that we can choose $e$ in such a way that $v\cdot e=v_n$.\\
Then using spheric coordinates for $\s^{n-1}$ we obtain $|v\cdot e|=|\cos\theta_{n-1}|$
and
\begin{equation*}
d\sigma=\sin\theta_2(\sin\theta_3)^2\dots(\sin\theta_{n-1})^{n-2}d\theta_1\dots d\theta_{n-1},
\end{equation*}
with $\theta_1\in[0,2\pi)$ and $\theta_j\in[0,\pi)$ for $j=2,\dots,n-1$.
Notice that
\begin{equation*}\begin{split}
\h^k(\s^k)&=\int_0^{2\pi}\,d\theta_1\int_0^\pi\sin\theta_2\,d\theta_2\dots
%\int_0^\pi(\sin\theta_{k-2})^{k-3}\,d\theta_{k-2}
\int_0^\pi(\sin\theta_{k-1})^{k-2}\,d\theta_{k-1}\\
&
=\h^{k-1}(\s^{k-1})\int_0^\pi(\sin t)^{k-2}\,dt.
\end{split}
\end{equation*}
Then we get
\begin{equation*}
\begin{split}
\int_{\s^{n-1}}|v\cdot e|\,d\sigma(v)&=\h^{n-2}(\s^{n-2})\int_0^\pi(\sin t)^{n-2}|\cos t|\,dt\\
&
=\h^{n-2}(\s^{n-2})\Big(\int_0^\frac{\pi}{2}(\sin t)^{n-2}\cos t\,dt-\int_\frac{\pi}{2}^\pi(\sin t)^{n-2}\cos t\,dt\Big)\\
&
=\frac{\h^{n-2}(\s^{n-2})}{n-1}\Big(\int_0^\frac{\pi}{2}\frac{d}{dt}(\sin t)^{n-1}\,dt-\int_\frac{\pi}{2}^\pi\frac{d}{dt}(\sin t)^{n-1}\,dt\Big)\\
&
=\frac{2\h^{n-2}(\s^{n-2})}{n-1}.
\end{split}
\end{equation*}
Therefore
\begin{equation}
n\omega_nK_{1,n}=2\frac{\h^{n-2}(\s^{n-2})}{n-1}=2\mathcal{L}^{n-1}(B_1(0))=2\omega_{n-1},
\end{equation}
and hence $(\ref{limitperimeter})$ becomes
\begin{equation*}
\lim_{s\to1}(1-s)[u]_{W^{s,1}(\Omega)}=2\omega_{n-1}|Du|(\Omega),
\end{equation*}
for any $u\in\bvo$.

Putting everything together we obtain
\begin{coroll}
Let $\Omega\subset\R$ be a bounded open set with Lipschitz boundary. Let $u\in L^1(\Omega)$. Then $u\in\bvo$ if and only if
\begin{equation*}
\liminf_{s\to1}(1-s)[u]_{\fracso}<\infty.
\end{equation*}
Moreover in that case
\begin{equation}\label{davila_conv_def}
\lim_{s\to1}(1-s)[u]_{W^{s,1}(\Omega)}=2\omega_{n-1}|Du|(\Omega).
\end{equation}
\end{coroll}

We can rewrite this proposition for sets as
\begin{prop}\label{local_frac_converge}
Let $\Omega\subset\R$ be a bounded open set with Lipschitz boundary. Let $E\subset\R$. Then $P(E,\Omega)<\infty$ if and only if
\begin{equation*}
\liminf_{s\to1}(1-s)P_s^L(E,\Omega)<\infty.
\end{equation*}
Moreover in that case
\begin{equation}\label{perimeter_conv_def}
\lim_{s\to1}(1-s)P_s^L(E,\Omega)=\omega_{n-1}P(E,\Omega).
\end{equation}
\end{prop}

Now we can give the proof of Theorem $\ref{cacc_equiv}$, but first it is convenient to point out the following easy but useful estimate.
\begin{lem}\label{positive_distance}
Let $B\subset\R$ and $x\in\R$ s.t. $d(x,B)\geq d>0$; then
\begin{equation}\label{punctual_eq_positive_distance}
\int_B\frac{dy}{\kers}\leq\frac{n\omega_n}{s}\frac{1}{d^s}.
\end{equation}
In particular,
if $A\subset\R$ is s.t. $|A|<\infty$ and $d(A,B)\geq d>0$, then
\begin{equation}\label{positive_distance_eq}
\Ll_s(A,B)\leq\frac{n\omega_n}{s}|A|\frac{1}{d^s}.
\end{equation}
\begin{proof}
Since $d(x,B)\geq d$, we have $B\subset\Co B_d(x)$ and hence
\begin{equation*}\begin{split}
\int_B\frac{dy}{\kers}&\leq\int_{\Co B_d(x)}\frac{dy}{\kers}=\int_{\Co B_d}\frac{dz}{|z|^{n+s}}\\
&
=\int_{\Sp}d\Han\int_d^\infty\frac{1}{\rho^{n+s}}\rho^{n-1}d\rho\\
&
=-\frac{n\omega_n}{s}\int_d^\infty\frac{d}{d\rho}\rho^{-s}d\rho=\frac{n\omega_n}{s}\frac{1}{d^s}.
\end{split}
\end{equation*}
The second inequality follows.

\end{proof}
\end{lem}

\begin{proof}[Proof of Theorem $\ref{cacc_equiv}$]
Clearly (ii) and (iii) are equivalent.\\
Let $\Omega:=B_R$ s.t. $R>1, E\subset\Omega$ and $\textrm{dist}(E,\partial\Omega)\geq d>0$; this is clearly a bounded open set with smooth boundary.

(i)$\Rightarrow$(ii): we only need to show that the liminf is finite. Actually we prove that the limit holds true.\\
Notice that we have
\begin{equation*}
P_s(E)=P^L_s(E,\Omega)+\Ll_s(E,\Co\Omega).
\end{equation*}
Now, as $\chi_E\in\bv$ and $E\subset\subset\Omega$, we have $\chi_E\in\bvo$ and
$P(E,\Omega)=P(E)$; in particular we also have
$\chi_E\in\fracso$ for every $s\in(0,1)$.\\
Since $d(E,\partial\Omega)\geq d>0$, we have
\begin{equation*}
\Ll_s(E,\Co\Omega)\leq\frac{n\omega_n}{s}|E|\frac{1}{d^s}.
\end{equation*}
Multiplying by $(1-s)$ we get
\begin{equation*}
0\leq(1-s)\Ll_s(E,\Co\Omega)\leq(1-s)\frac{|E|n\omega_n}{sd^s}\to0,
\end{equation*}
as $s\to1$. This and ($\ref{perimeter_conv_def}$) give
\begin{equation*}\begin{split}
\lim_{s\to1}(1-s)P_s(E)&=\lim_{s\to1}(1-s)P_s^L(E,\Omega)=\omega_{n-1}P(E,\Omega)\\
&
=\omega_{n-1}P(E).
\end{split}
\end{equation*}

(iii)$\Rightarrow$(i):
We have
\begin{equation*}\begin{split}
c_{s_k}(1-s_k)[\chi_E]_{W^{s_k,1}(\Omega)}&\leq\frac{R}{n\omega_n}(1-s_k)[\chi_E]_{W^{s_k,1}(\Omega)}\\
&
\leq\frac{2R}{n\omega_n}(1-s_k)P_{s_k}(E).
\end{split}
\end{equation*}
Thus the hypothesis implies
\begin{equation*}
\liminf_{k\to\infty}c_{s_k}(1-s_k)[\chi_E]_{W^{s_k,1}(\Omega)}<\infty,
\end{equation*}
and the Theorem of Bourgain, Brezis
and Mironescu gives $\chi_E\in\bvo$.\\
Finally, as $E\subset\subset\Omega$, we also get $\chi_E\in\bv$.

\end{proof}

\begin{rmk}
The condition that the liminf be bounded is necessary i.e. there exist bounded measurable sets having finite $s$-perimeter for every $s\in(0,1)$ which are not of finite perimeter.\\
This also shows that in general the inclusion $\bvo\subset\bigcap_{s\in(0,1)}\fracso$
is strict.
\end{rmk}

\begin{ese}\label{inclusion_counterexample}
Let $a\in(0,1)\subset\mathbb{R}$ and consider the open intervals $I_k:=(a^{k+1},a^k)$ for every $k\in\mathbb{N}$.
Define $E:=\bigcup_{k\in\mathbb{N}}I_{2k}$, which is a bounded (open) set.
Due to the infinite number of jumps $\chi_E\not\in BV(\mathbb{R})$. However it can be proved that
$E$ has finite $s$-perimeter for every $s\in(0,1)$. We postpone the proof to the end of the chapter.
\end{ese}

Notice that, since
\begin{equation*}
A_1\subset A_2\quad\Longrightarrow\quad\Ll_s(A_1,B)\leq\Ll_s(A_2,B),
\end{equation*}
we always have
\begin{equation}\begin{split}
P^{NL}_s(E,\Omega)&=\Ll_s(E\cap\Omega,\Co E\cap\Co\Omega)+\Ll_s(\Co E\cap\Omega,E\cap\Co\Omega)\\
&
\leq2\Ll_s(\Omega,\Co\Omega)=2P_s(\Omega).
\end{split}\end{equation}

\noindent
In particular, if $\Omega\subset\R$ is a bounded open set with Lipschitz boundary, then $\Omega$ has finite perimeter and hence also finite $s$-perimeter, as we saw above. Therefore in that case
\begin{equation*}
P^{NL}_s(E,\Omega)\leq2P_s(\Omega)<\infty,
\end{equation*}
for every $E\subset\R$, so we only need to check the local part of the fractional perimeter. In particular, Proposition $\ref{local_frac_converge}$ then implies that
\begin{equation*}
P(E,\Omega)<\infty\quad\Longrightarrow\quad P_s(E,\Omega)<\infty,
\end{equation*}
for every $s\in(0,1)$.

We now wish to estimate $P^{NL}_s(E,\Omega)$ and show that the convergence in $(\ref{perimeter_conv_def})$ holds for the whole fractional perimeter instead of only its local part.

First we have to introduce some notation.\\
Let $\Omega\subset\R$ be a bounded open set with Lipschitz boundary.
Then we can find two sequences of
bounded open sets
$A_k,\, D_k\subset\R$ with Lipschitz boundary strictly approximating $\Omega$ from the inside and from the outside respectively, that is

$(i)\quad A_k\subset A_{k+1}\subset\subset\Omega$ and $A_k\nearrow\Omega$, i.e. $\bigcup_k A_k=\Omega$,

$(ii)\quad \Omega\subset\subset D_{k+1}\subset D_k$ and $D_k\searrow\overline{\Omega}$, i.e. $\bigcap_k D_k=\overline{\Omega}$.\\
For a proof we refer to \cite{LipApprox} and the references cited therein.

We define for every $k$
\begin{equation*}\begin{split}
&\Omega_k^+:=D_k\setminus\overline{\Omega},\qquad\Omega_k^-:=\Omega\setminus\overline{A_k}
\qquad T_k:=\Omega_k^+\cup\partial\Omega\cup\Omega_k^-,\\
&\qquad\qquad d_k:=\min\{d(A_k,\partial\Omega),\,d(D_k,\partial\Omega)\}>0.
\end{split}
\end{equation*}

Now we can prove the following
\begin{teo}
Let $\Omega\subset\R$ be a bounded open set with Lipschitz boundary and let $E\subset\R$
be a set having finite perimeter in $D_1$.\\
Then $P_s(E,\Omega)<\infty$ for every $s\in(0,1)$,
\begin{equation*}
\omega_{n-1}P(E,\Omega)\leq\liminf_{s\to1}(1-s)P_s(E,\Omega)
\end{equation*}
and
\begin{equation}
\limsup_{s\to1}(1-s)P_s(E,\Omega)
\leq\omega_{n-1}P(E,\Omega)+2\omega_{n-1}\lim_{k\to\infty}P(E,T_k).
\end{equation}
In particular, if $P(E,\partial\Omega)=0$, then
\begin{equation}
\lim_{s\to1}(1-s)P_s(E,\Omega)=\omega_{n-1}P(E,\Omega).
\end{equation}

\begin{proof}
Since $\Omega$ is regular and $P(E,\Omega)<\infty$, we already know from Proposition $\ref{local_frac_converge}$ that
$P_s(E,\Omega)$ is finite for every $s$ and
\begin{equation*}
\lim_{s\to1}(1-s)P_s^L(E,\Omega)=\omega_{n-1}P(E,\Omega),
\end{equation*}
so we only need to prove the second inequality.

%We can suppose $P(E,B_1)<\infty$.

Notice that, since $|D\chi_E|$ is a finite Radon measure on $D_1$ and
$T_k\searrow\partial\Omega$ as $k\nearrow\infty$, we have
\begin{equation*}
\exists\lim_{k\to\infty}P(E,T_k)=P(E,\partial\Omega).
\end{equation*}
Consider the nonlocal part of the fractional perimeter
\begin{equation*}
P_s^{NL}(E,\Omega)=\Ll_s(E\cap\Omega,\Co E\setminus\Omega)+\Ll_s(\Co E\cap\Omega,E\setminus\Omega),
\end{equation*}
and take any $k$. Then
\begin{equation*}\begin{split}
\Ll_s(E\cap\Omega,\Co E\setminus\Omega)&=\Ll_s(E\cap\Omega,\Co E\cap\Omega_k^+)+\Ll_s(E\cap\Omega,\Co E\cap(\Co\Omega\setminus B_k))\\
&
\leq\Ll_s(E\cap\Omega,\Co E\cap\Omega_k^+)+\frac{n\omega_n}{s}|\Omega|\frac{1}{d_k^s}\\
&
\leq\Ll_s(E\cap\Omega_k^-,\Co E\cap\Omega_k^+)+2\frac{n\omega_n}{s}|\Omega|\frac{1}{d_k^s}\\
&
\leq\Ll_s(E\cap(\Omega_k^-\cup\Omega_k^+),\Co E\cap(\Omega_k^-\cup\Omega_k^+))+2\frac{n\omega_n}{s}|\Omega|\frac{1}{d_k^s}\\
&
=P^L_s(E,T_k)+2\frac{n\omega_n}{s}|\Omega|\frac{1}{d_k^s}.
\end{split}\end{equation*}
Since we can bound the other term in the same way, we get
\begin{equation}
P^{NL}_s(E,\Omega)\leq2P^L_s(E,T_k)+4\frac{n\omega_n}{s}|\Omega|\frac{1}{d_k^s}.
\end{equation}
By hypothesis we know that $T_k$ is a bounded open set with Lipschitz boundary
\begin{equation*}
\partial T_k=\partial A_k\cup\partial D_k.
\end{equation*}
Therefore for every $k$, using again Proposition $\ref{local_frac_converge}$ we have
\begin{equation*}
\lim_{s\to1}(1-s)P^L_s(E,T_k)=\omega_{n-1}P(E,T_k),
\end{equation*}
and hence
\begin{equation*}
\limsup_{s\to1}(1-s)P_s(E,\Omega)
\leq\omega_{n-1}P(E,\Omega)+2\omega_{n-1}P(E,T_k).
\end{equation*}
Since this holds true for any $k$, we get the claim.

\end{proof}
\end{teo}

\begin{rmk}
We remark that in the proof above we showed 
that we can bound the nonlocal part of the perimeter as
\begin{equation*}
P^{NL}_s(E,\Omega)\leq2P^L_s(E,T_k)+4\frac{n\omega_n}{s}|\Omega|\frac{1}{d_k^s},
\end{equation*}
for every $k$, without making assumptions on the set $E\subset\R$.
Then, using Theorem $\ref{bb}$, we get
\begin{equation*}
\limsup_{s\to1}(1-s)P^{NL}_s(E,\Omega)\leq CP(E,T_k),
\end{equation*}
and hence
\begin{equation}
\limsup_{s\to1}(1-s)P^{NL}_s(E,\Omega)\leq C\liminf_{k\to\infty}P(E,T_k).
\end{equation}
We remark that this estimate holds true for any set $E\subset\R$.\\
However, if we suppose that $P(E,D_1)<\infty$ then the liminf is actually a limit, which is equal to $P(E,\partial\Omega)$, and we can use the constant $C=2\omega_{n-1}$,
as we did above.
\end{rmk}

Actually we can use the signed distance function to find our approximating sets.\\
Let $\Omega\subset\R$ be a bounded open set with Lipschitz boundary and let $\bar{d}_\Omega$ denote the signed distance function from $\Omega$, negative inside. Define for any $\rho\in\mathbb{R}$ with $|\rho|$ small,
the open set
%\begin{equation*}\begin{split}
%&F_\rho^+:=\{y\in\Co F\,|\,d(y,F)<\rho\},\\
%&
%F_\rho^-:=\{y\in F\,|\,d(y,\Co F)<\rho\},\\
%&
%N_\rho(F):=\{y\in\R\,|\,d(y,F)<\rho\}.
%\end{split}\end{equation*}
\begin{equation*}
\Omega_\rho:=\{\bar{d}_\Omega<\rho\}.
\end{equation*}
It is proved in \cite{LipApprox} that $\Omega_\rho$ has Lipschitz boundary for every $|\rho|<\alpha$,
for some $\alpha>0$ small enough.

Notice that $\Omega_\rho\subset\subset\Omega$ when $\rho<0$ and
$\Omega\subset\subset\Omega_\rho$ when $\rho>0$.\\
Therefore we can consider the sets $\Omega_\rho$ as our approximating sequences $A_k,\, D_k$,
with $\rho<0$ and $\rho>0$ respectively.

Having 'continuous' approximating sequences rather than numerable ones allows us to improve the previous result.

Notice that for $\rho>0$
\begin{equation*}
N_\rho(\partial\Omega)=\Omega_\rho\setminus\overline{\Omega_{-\rho}}=\{-\rho<\bar{d}_F<\rho\},
\end{equation*}
is an open tubular neighborhood of $\partial\Omega$.
These take the place of the sets $T_k$.

Now suppose that $E$ has finite perimeter in $\Omega_\alpha$.
In particular this implies that $E$ has finite perimeter in $N_\alpha(\partial\Omega)$.\\
Notice that $P(E,B)=\Han(\partial^*E\cap B)$ for every $B\subset\R$, where $\partial^*E$ is the reduced boundary of $E$. In particular
\begin{equation*}
P(E,\{\bar{d}_\Omega=\delta\})=\Han(\partial^*E\cap\{\bar{d}_\Omega=\delta\}),
\end{equation*}
for every $\delta\in(-\alpha,\alpha)$.
Now, since
\begin{equation*}
\Han(\partial^*E\cap N_\alpha(\partial\Omega))=P(E,N_\alpha(\partial\Omega))<\infty,
\end{equation*}
the set
\begin{equation*}
S:=\left\{\delta\in(-\alpha,\alpha)\,|\,P(E,\{\bar{d}_\Omega=\delta\})>0\right\}
\end{equation*}
is at most countable.\\
Moreover for every $\delta\in(-\alpha,\alpha)\setminus S$ we have
\begin{equation*}
\lim_{s\to1}(1-s)P_s(E,\Omega_\delta)=\omega_{n-1}P(E,\Omega_\delta).
\end{equation*}
This shows that even if the limit doesn't hold for $\Omega$, if we slightly enlarge or restrict $\Omega$, then the limit holds true.
Moreover, since $(-\alpha,\alpha)\setminus S$ is dense in $(-\alpha,\alpha)$,
we can enlarge (or restrict) $\Omega$ as little as we want. To sum up
\begin{coroll}
Let $\Omega\subset\R$ be an open set with Lipschitz boundary and let $E\subset\R$ be a set having finite perimeter
in $\Omega_\beta$, for some $0<\beta<\alpha$.\\
Then there exists a set $S\subset(-\alpha,\beta)$, at most countable,
s.t.
\begin{equation*}
\lim_{s\to1}(1-s)P_s(E,\Omega_\delta)=\omega_{n-1}P(E,\Omega_\delta),
\end{equation*}
for every $\delta\in(-\alpha,\beta)\setminus S$.
\end{coroll}

This is an improvement of Theorem 1 in \cite{cafenr}, which was obtained through uniform estimates.
In particular, our result holds for any bounded Lipschitz domain $\Omega$, without requiring $C^{1,\alpha}$
regularity of $\partial E$.\\

For a complete analysis of the asymptotics of the fractional perimeter as $s\to1$ in the context of $\Gamma$-convergence, see
\cite{Gamma}.\\

Finally we remark that also the asymptotics as $s\to0$ has been studied. See \cite{asymptzero} for a complete analysis.

\end{section}

\begin{section}{(Ir)Regularity of the Boundary}

We give a definition of fractal dimension related to the fractional perimeter, which was introduced in \cite{Visintin}. In particular we show that a set can have finite fractional perimeter even if the dimension of its boundary is bigger than $n-1$.\\
Using Proposition $\ref{cont_scale}$ we immediately get the following
\begin{prop}
Let $\Omega$ be an open set. For every measurable $u:\Omega\longrightarrow\mathbb{R}$ there exists one and only one
$R(u)\in[0,1]$ s.t.
\begin{equation*}
[u]_{\fracso}\quad\left\{\begin{array}{cc}
<\infty,& \forall\,s\in(0,R(u))\\
=\infty, &\forall\,s\in(R(u),1)
\end{array}\right.
\end{equation*}
that is
\begin{equation}\begin{split}\label{frac_range}
R(u)&=\sup\left\{s\in(0,1)\,\big|\,[u]_{\fracso}<\infty\right\}\\
&
=\inf\left\{s\in(0,1)\,\big|\,[u]_{\fracso}=\infty\right\}.
\end{split}
\end{equation}
\end{prop}

In particular, using the above Proposition for characteristic functions, we can give the following definition of fractal dimension.
\begin{defin}
Let $E\subset\R$ s.t. $\partial^-E\not=\emptyset$. We define
\begin{equation}
\Dim_F(\partial^-E,\Omega):=n-R(\chi_E),
\end{equation}
the fractal dimension of $\partial^-E$ in $\Omega$ relative to the fractional perimeter.
\end{defin}

Notice that in the case of sets $(\ref{frac_range})$ becomes
\begin{equation}
\begin{split}\label{frac_range_sets}
R(\chi_E)&=\sup\left\{s\in(0,1)\,\big|\,P_s^L(E,\Omega)<\infty\right\}\\
&
=\inf\left\{s\in(0,1)\,\big|\,P_s^L(E,\Omega)=\infty\right\}.
\end{split}
\end{equation}

In particular we can take $\Omega$ to be the whole of $\R$, or a bounded open set with Lipschitz boundary.\\
In the first case the local part of the fractional perimeter coincides with the whole fractional perimeter, while in the second case we know that we can bound the nonlocal part with $2P_s(\Omega)<\infty$ for every $s\in(0,1)$. Therefore in both cases in $(\ref{frac_range_sets})$
we can as well take the whole fractional perimeter $P_s(E,\Omega)$ instead of just the local part.

Using the embedding of Proposition $\ref{bv_embd}$, we know that if $\Omega$ is an extension domain, then
\begin{equation*}
P(E,\Omega)<\infty\quad\Longrightarrow\quad\Dim_F(\partial^-E,\Omega)=n-1.
\end{equation*}
However the Example
$\ref{inclusion_counterexample}$
shows that (in general) the converse is false.

Now we recall the definition of Minkowski content and dimension.\\
For simplicity set
\begin{equation*}
\bar{N}_\rho^\Omega(E):=\overline{N_\rho(E)}\cap\Omega
=\{x\in\Omega\,|\,d(x,E)\leq\rho\},
\end{equation*}
for any $\rho>0$.

\begin{defin}
Let $\Omega\subset\R$ be an open set. For any $\Gamma\subset\R$ and $r\in[0,n]$ we define the inferior and superior $r$-dimensional Minkowski contents of $\Gamma$ relative to the set $\Omega$ as, respectively
\begin{equation*}
\underline{\mathcal{M}}^r(\Gamma,\Omega):=\liminf_{\rho\to0}\frac{|\bar{N}_\rho^\Omega(\Gamma)|}{\rho^{n-r}},\qquad
\overline{\mathcal{M}}^r(\Gamma,\Omega):=\limsup_{\rho\to0}\frac{|\bar{N}_\rho^\Omega(\Gamma)|}{\rho^{n-r}}.
\end{equation*}
Then we define the lower and upper Minkowski dimensions of $\Gamma$ in $\Omega$ as
\begin{equation*}\begin{split}
\underline{\Dim}_\mathcal{M}(\Gamma,\Omega)&:=\inf\left\{r\in[0,n]\,|\,\underline{\mathcal{M}}^r(\Gamma,\Omega)=0\right\}\\
&
=n-\sup\left\{r\in[0,n]\,|\,\underline{\mathcal{M}}^{n-r}(\Gamma,\Omega)=0\right\},
\end{split}\end{equation*}
\begin{equation*}\begin{split}
\overline{\Dim}_\mathcal{M}(\Gamma,\Omega)&:=\sup\left\{r\in[0,n]\,|\,\overline{\mathcal{M}}^r(\Gamma,\Omega)=\infty\right\}\\
&
=n-\inf\left\{r\in[0,n]\,|\,\overline{\mathcal{M}}^{n-r}(\Gamma,\Omega)=\infty\right\}.
\end{split}
\end{equation*}
If they agree, we write
\begin{equation*}
\Dim_\mathcal{M}(\Gamma,\Omega)
\end{equation*}
for the common value and call it the Minkowski dimension of $\Gamma$ in $\Omega$.\\
If $\Omega=\R$ or $\Gamma\subset\subset\Omega$, we drop the $\Omega$ in the formulas.
\end{defin}

\begin{rmk}
Let $\Dim_\mathcal{H}$ denote the Hausdorff dimension. In general one has
\begin{equation*}
\Dim_\mathcal{H}(\Gamma)\leq\underline{\Dim}_\mathcal{M}(\Gamma)\leq\overline{\Dim}_\mathcal{M}(\Gamma),
\end{equation*}
and all the inequalities might be strict. However for some sets (e.g. self-similar sets with some symmetric and regularity condition) they are all equal.
\end{rmk}

In \cite{Visintin} the following Proposition (not explicitly stated) is proved.
\begin{prop}
Let $\Omega\subset\R$ be a bounded open set. Then for every $E\subset\R$ s.t. $\partial^-E\not=\emptyset$ and $\overline{\Dim}_\mathcal{M}(\partial^-E,\Omega)\geq n-1$ we have
\begin{equation}
\Dim_F(\partial^-E,\Omega)\leq\overline{\Dim}_\mathcal{M}(\partial^-E,\Omega).
\end{equation}
\begin{proof}
By hypothesis we have
\begin{equation*}
\overline{\Dim}_\mathcal{M}(\partial^-E,\Omega)=n-\inf\left\{r\in(0,1)\,|\,\overline{\mathcal{M}}^{n-r}(\partial^-E,\Omega)=\infty\right\},
\end{equation*}
and we need to show that
\begin{equation*}
\inf\left\{r\in(0,1)\,|\,\overline{\mathcal{M}}^{n-r}(\partial^-E,\Omega)=\infty\right\}
\leq
\sup\{s\in(0,1)\,|\,P_s^L(E,\Omega)<\infty\}.
\end{equation*}
Up to modifying $E$ on a set of Lebesgue measure zero
we can suppose that $\partial E=\partial^-E$, as in Remark $\ref{gmt_assumption}$. Notice that this does not affect the
$s$-perimeter.
Now for any $s\in(0,1)$
\begin{equation*}\begin{split}
2P_s^L(E,\Omega)&=\int_\Omega\,dx\int_\Omega\frac{|\chi_E(x)-\chi_E(y)|}{\kers}\,dy\\
&
=\int_\Omega dx\int_0^\infty d\rho\int_{\partial B_\rho(x)\cap\Omega}\frac{|\chi_E(x)-\chi_E(y)|}{\kers}\,d\Han(y)\\
&
=\int_\Omega dx\int_0^\infty\frac{d\rho}{\rho^{n+s}}\int_{\partial B_\rho(x)\cap\Omega}|\chi_E(x)-\chi_E(y)|\,d\Han(y).
\end{split}
\end{equation*}
Notice that
\begin{equation*}
d(x,\partial E)>\rho\quad\Longrightarrow\quad\chi_E(y)=\chi_E(x),\quad\forall\,y\in\overline{B_\rho(x)},
\end{equation*}
and hence
\begin{equation*}\begin{split}
\int_{\partial B_\rho(x)\cap\Omega}|\chi_E(x)-\chi_E(y)|\,d\Han(y)&
\leq\int_{\partial B_\rho(x)\cap\Omega}\chi_{\bar{N}_\rho(\partial E)}(x)\,d\Han(y)\\
&
\leq n\omega_n\rho^{n-1}\chi_{\bar{N}_\rho(\partial E)}(x).
\end{split}
\end{equation*}
Therefore
\begin{equation}\label{visintin_pf}
2P_s^L(E,\Omega)\leq n\omega_n\int_0^\infty\frac{d\rho}{\rho^{1+s}}\int_\Omega
\chi_{\bar{N}_\rho(\partial E)}(x)
=n\omega_n\int_0^\infty\frac{|\bar{N}^\Omega_\rho(\partial E)|}{\rho^{1+s}}\,d\rho.
\end{equation}
We prove the following\\
CLAIM
\begin{equation}\label{visintin_proof}
\overline{\mathcal{M}}^{n-r}(\partial E,\Omega)<\infty\quad\Longrightarrow\quad P_s^L(E,\Omega)<\infty,\quad\forall\,s\in(0,r).
\end{equation}
Indeed
\begin{equation*}
\limsup_{\rho\to0}\frac{|\bar{N}^\Omega_\rho(\partial E)|}{\rho^r}<\infty\quad\Longrightarrow\quad\exists\,C>0\textrm{ s.t. }
\sup_{\rho\in(0,C]}\frac{|\bar{N}^\Omega_\rho(\partial E)|}{\rho^r}\leq M<\infty.
\end{equation*}
Then
\begin{equation*}\begin{split}
2P_s^L(E,\Omega)&\leq n\omega_n\left\{\int_0^C\frac{|\bar{N}^\Omega_\rho(\partial E)|}{\rho^{1-(r-s)+r}}\,d\rho
+\int_C^\infty\frac{|\bar{N}^\Omega_\rho(\partial E)|}{\rho^{1+s}}\,d\rho\right\}\\
&
\leq n\omega_n\left\{
M\int_0^C\frac{1}{\rho^{1-(r-s)}}\,d\rho+|\Omega|\int_C^\infty\frac{1}{\rho^{1+s}}\,d\rho
\right\}\\
&
=n\omega_n\left\{
\frac{M}{r-s}C^{r-s}+\frac{|\Omega|}{sC^s}
\right\}<\infty,
\end{split}\end{equation*}
proving the claim.\\
This implies
\begin{equation*}
r\leq\sup\{s\in(0,1)\,|\,P_s^L(E,\Omega)<\infty\},
\end{equation*}
for every $r\in(0,1)$ s.t. $\overline{\mathcal{M}}^{n-r}(\partial E,\Omega)<\infty$.\\
Thus for $\epsilon>0$ very small, we have
\begin{equation*}
\inf\left\{r\in(0,1)\,|\,\overline{\mathcal{M}}^{n-r}(\partial^-E,\Omega)=\infty\right\}-\epsilon
\leq\sup\{s\in(0,1)\,|\,P_s^L(E,\Omega)<\infty\}.
\end{equation*}
Letting $\epsilon$ tend to zero, we conclude the proof.

\end{proof}
\end{prop}

In particular, if $\Omega$ has a regular boundary, we obtain
\begin{coroll}
Let $\Omega\subset\R$ be a bounded open set with Lipschitz boundary. Let $E\subset\R$ s.t. $\partial^-E\not=\emptyset$ and
$\overline{\Dim}_\mathcal{M}(\partial^-E,\Omega)\in[n-1,n)$. Then
\begin{equation}\label{fractal_per}
P_s(E,\Omega)<\infty\qquad\textrm{for every }s\in\left(0,n-\overline{\Dim}_\mathcal{M}(\partial^-E,\Omega)\right).
\end{equation}
\end{coroll}

\noindent
This shows that a set $E$ can have finite fractional perimeter even if its boundary is really irregular (unlike what happens with Caccioppoli sets and their reduced boundary).

Now we give some equivalent definitions of the Minkowski dimensions, usually referred to as box-counting dimensions, which are easier to compute. For the details and the relation between the Minkowski and the Hausdorff dimensions, see \cite{Mattila} and \cite{Falconer} and the references cited therein.\\
For simplicity we only consider the case $\Gamma$ bounded and $\Omega=\R$ (or $\Gamma\subset\subset\Omega$).
\begin{defin}
Given a nonempty bounded set $\Gamma\subset\R$, define for every $\delta>0$
\begin{equation*}
\mathcal{N}(\Gamma,\delta):=\min\left\{k\in\mathbb{N}\,\big|\,\Gamma\subset\bigcup_{i=1}^kB_\delta(x_i),\textrm{ for some }x_i\in\R\right\},
\end{equation*}
the smallest number of $\delta$-balls needed to cover $\Gamma$, and
\begin{equation*}
\mathcal{P}(\Gamma,\delta):=\max\left\{k\in\mathbb{N}\,|\,\exists\,\textrm{disjoint balls }B_\delta(x_i),\,i=1,\dots,k\textrm{ with }x_i\in \Gamma\right\},
\end{equation*}
the greatest number of disjoint $\delta$-balls with centres in $\Gamma$.
\end{defin}

Then it is easy to verify that
\begin{equation}\label{counting}
\mathcal{N}(\Gamma,2\delta)\leq\mathcal{P}(\Gamma,\delta)\leq\mathcal{N}(\Gamma,\delta/2).
\end{equation}
Moreover, since any union of $\delta$-balls with centers in $\Gamma$ is contained in $N_\delta(\Gamma)$, and any
union of $(2\delta)$-balls covers $N_\delta(\Gamma)$ if the union of the corresponding $\delta$-balls covers $\Gamma$, we get
\begin{equation}\label{counting2}
\mathcal{P}(\Gamma,\delta)\omega_n\delta^n\leq|N_\delta(\Gamma)|\leq
\mathcal{N}(\Gamma,\delta)\omega_n(2\delta)^n.
\end{equation}
Using $(\ref{counting})$ and $(\ref{counting2})$ we see that
\begin{equation*}\begin{split}
&\underline{\Dim}_\mathcal{M}(\Gamma)=\inf\left\{r\in[0,n]\,\big|\,\liminf_{\delta\to0}\mathcal{N}(\Gamma,\delta)\delta^r=0\right\},\\
&
\overline{\Dim}_\mathcal{M}(\Gamma)=\sup\left\{r\in[0,n]\,\big|\,\limsup_{\delta\to0}\mathcal{N}(\Gamma,\delta)\delta^r=\infty\right\}.
\end{split}
\end{equation*}
Then it can be proved that
\begin{equation}\label{log_counting}\begin{split}
&\underline{\Dim}_\mathcal{M}(\Gamma)=\liminf_{\delta\to0}\frac{\log\mathcal{N}(\Gamma,\delta)}{-\log\delta},\\
&
\overline{\Dim}_\mathcal{M}(\Gamma)=\limsup_{\delta\to0}\frac{\log\mathcal{N}(\Gamma,\delta)}{-\log\delta}.
\end{split}
\end{equation}
Actually notice that, due to $(\ref{counting})$,
we can take $\mathcal{P}(\Gamma,\delta)$ in place of $\mathcal{N}(\Gamma,\delta)$ in the above formulas.\\
It is also easy to see that if in the definition of $\mathcal{N}(\Gamma,\delta)$ we take cubes of side $\delta$ instead of balls of radius $\delta$, then we get exactly the same dimensions.

Moreover in $(\ref{log_counting})$ it is enough to consider limits as $\delta\to0$
through any decreasing sequence $\delta_k$
s.t. $\delta_{k+1}\geq c\delta_k$ for some constant $c\in(0,1)$; in particular for $\delta_k=c^k$. Indeed
if $\delta_{k+1}\leq\delta<\delta_k$, then
\begin{equation*}\begin{split}
\frac{\log\mathcal{N}(\Gamma,\delta)}{-\log\delta}&\leq\frac{\log\mathcal{N}(\Gamma,\delta_{k+1})}{-\log\delta_k}
=\frac{\log\mathcal{N}(\Gamma,\delta_{k+1})}{-\log\delta_{k+1}+\log(\delta_{k+1}/\delta_k)}\\
&
\leq\frac{\log\mathcal{N}(\Gamma,\delta_{k+1})}{-\log\delta_{k+1}+\log c},
\end{split}\end{equation*}
so that
\begin{equation*}
\limsup_{\delta\to0}\frac{\log\mathcal{N}(\Gamma,\delta)}{-\log\delta}\leq
\limsup_{k\to\infty}\frac{\log\mathcal{N}(\Gamma,\delta_k)}{-\log\delta_k}.
\end{equation*}
The opposite inequality is clear and in a similar way we can treat the lower limits.

Now we can study the following example.

\begin{ese}[von Koch Snowflake]
The von Koch snowflake, whose construction we recall below, is a bounded open set $S\subset\mathbb{R}^2$ whose boundary
has fractal dimension $\Dim_\mathcal{M}(\partial S)=\frac{\log4}{\log3}$. Therefore we have
\begin{equation}\label{koch1}
P_s(S)<\infty,\qquad\forall\,s\in\left(0,2-\frac{\log4}{\log3}\right).
\end{equation}
Moreover
\begin{equation}\label{koch2}
P_s(S)=\infty,\qquad\forall\,s\in\left(2-\frac{\log4}{\log3},1\right),
\end{equation}
and hence
\begin{equation*}
\Dim_F(\partial S)=\Dim_\mathcal{M}(\partial S)=\frac{\log4}{\log3}.
\end{equation*}
\begin{proof}
First of all we construct the von Koch curve. Then the snowflake is made of three
von Koch curves.\\
Let $E_0$ be a line segment of unit length. 
The set $E_1$ consists of the four segments obtained by removing the middle third of $E_0$ and replacing it by the other two sides of the equilateral triangle based on the removed segment.
We construct $E_2$ by applying the same procedure to each of the segments in $E_1$ and so on.
Thus $E_k$ comes from replacing the middle third of each straight line segment of $E_{k-1}$ by the other two sides of an equilateral triangle.\\
As $k$ tends to infinity, the sequence of polygonal curves $E_k$ approaches a limiting curve $E$, called the von Koch curve.\\
If we start with an equilateral triangle with unit length side and perform the same construction on all three sides, we obtain the von Koch snowflake $\Sigma$.\\
Let $S$ be the bounded region enclosed by $\Sigma$, so that $S$ is open and $\partial S=\Sigma$.\\
Now we calculate the dimension of $E$ using the remarks above about the box-counting dimensions.\\
The idea is to exploit the self-similarity of $E$ and consider covers made of squares with side $\delta_k=3^{-k}$.\\
The key observation is that $E$ can be covered by three squares of length $1/3$ (and cannot be covered by only two),
so that $\mathcal{N}(E,1/3)=3$.\\
Then consider $E_1$. We can think of $E$ as being made of four von Koch curves starting from the set $E_1$ and with initial segments of length $1/3$ instead of 1. Therefore we can cover each of these four pieces with three squares of side $1/9$, so that $E$ can be covered with $3\cdot4$ squares of length $1/9$ (and not one less) and $\mathcal{N}(E,1/9)=4\cdot3$.\\
We can repeat the same starting from $E_2$ to get $\mathcal{N}(E,1/27)=4^2\cdot3$, and so on.
In general we obtain
\begin{equation*}
\mathcal{N}(E,3^{-k})=4^{k-1}\cdot3.
\end{equation*}
Then, taking logarithms we get
\begin{equation*}
\frac{\log\mathcal{N}(E,3^{-k})}{-\log3^{-k}}=\frac{\log3+(k-1)\log4}{k\log3}\longrightarrow\frac{\log4}{\log3},
\end{equation*}
so that $\Dim_\mathcal{M}(E)=\frac{\log4}{\log3}$.\\
Notice that the Minkowski dimensions of the snowflake and of the curve are the same, so we get the claim and $(\ref{koch1})$.
%Moreover it can be proved that the Hausdorff dimension of the von Koch curve is equal to its Minkowski dimension.

Now we prove $(\ref{koch2})$.\\
As starting point for the snowflake take the
equilateral triangle $T$ of side 1, with baricenter in the origin and a vertex on the $y$-axis. Then $T_1$ is made of three triangles of side $1/3$, $T_2$ of $3\cdot4$ triangles of side $1/3^2$ and so on. In general
$T_k$ is made of $3\cdot4^{k-1}$ triangles of side $1/3^k$; call them $T_k^1,\dots,T_k^{3\cdot4^{k-1}}$
and let the $x^i_k$'s be their baricenters. For each triangle
$T^i_k$ there exists a rotation $\mathcal{R}_k^i\in SO(n)$ s.t.
\begin{equation*}
\mathcal{R}_k^i\left(\frac{1}{3^k}T\right)+x_k^i=T_k^i.
\end{equation*}
Fix a ball $B_1(x)\subset\Co S$, far from $S$, e.g. $B_1(0,15)$. Then
\begin{equation}\label{koch3}
B_{3^{-k}}(x+x_k^i)=\mathcal{R}_k^i\left(\frac{1}{3^k}B_1(x)\right)+x_k^i\subset\Co S,
\end{equation}
for every $i,\,k$.\\
Notice that $T_k$ and $T_{k-1}$ touch only on a set of measure zero and
$S=T\cup\bigcup T_k$. Then
\begin{equation*}\begin{split}
P_s(S)&=\Ll_s(S,\Co S)=\Ll_s(T,\Co S)+\sum_{k=1}^\infty\Ll_s(T_k,\Co S)\\
&
=\Ll_s(T,\Co S)+\sum_{k=1}^\infty\sum_{i=1}^{3\cdot4^{k-1}}\Ll_s(T_k^i,\Co S)
\geq\sum_{k=1}^\infty\sum_{i=1}^{3\cdot4^{k-1}}\Ll_s(T_k^i,\Co S)\\
&
\geq\sum_{k=1}^\infty\sum_{i=1}^{3\cdot4^{k-1}}\Ll_s(T_k^i,B_{3^{-k}}(x+x_k^i))\qquad\textrm{(by }(\ref{koch3}))\\
&
=\sum_{k=1}^\infty\sum_{i=1}^{3\cdot4^{k-1}}\Big(\frac{1}{3^k}\Big)^{2-s}\Ll_s(T,B_1(x))\qquad\textrm{(by Proposition }\ref{elementary_properties})\\
&
=\frac{3}{3^{2-s}}\Ll_s(T,B_1(x))\sum_{k=0}^\infty\Big(\frac{4}{3^{2-s}}\Big)^k.
\end{split}\end{equation*}
To conclude, notice that the last series is divergent if $s>2-\frac{\log4}{\log3}$.

\end{proof}
\end{ese}

\end{section}

\begin{section}{Proof of Example $\ref{inclusion_counterexample}$}

Note that $E\subset (0,a^2]$.
Let $\Omega:=(-1,1)\subset\mathbb{R}$. Then $E\subset\subset\Omega$ and $\textrm{dist}(E,\partial\Omega)=1-a^2=:d>0$.
Now
\begin{equation*}
P_s(E)=\int_E\int_{CE\cap\Omega}\frac{dxdy}{|x-y|^{1+s}}+
\int_E\int_{C\Omega}\frac{dxdy}{|x-y|^{1+s}}.%=:I_1+I_2.
\end{equation*}
As for the second term, we have
\begin{equation*}
\int_E\int_{C\Omega}\frac{dxdy}{|x-y|^{1+s}}\leq\frac{2|E|}{sd^s}<\infty.
\end{equation*}
We split the first term into three pieces
\begin{equation*}\begin{split}
\int_E&\int_{CE\cap\Omega}\frac{dxdy}{|x-y|^{1+s}}\\
&
=\int_E\int_{-1}^0\frac{dxdy}{|x-y|^{1+s}}
+\int_E\int_{CE\cap(0,a)}\frac{dxdy}{|x-y|^{1+s}}+\int_E\int_a^1\frac{dxdy}{|x-y|^{1+s}}\\
&
=\mathcal{I}_1+\mathcal{I}_2+\mathcal{I}_3.
\end{split}
\end{equation*}
Note that $CE\cap(0,a)=\bigcup_{k\in\mathbb{N}}I_{2k-1}=\bigcup_{k\in\mathbb{N}}(a^{2k},a^{2k-1})$.\\
A simple calculation shows that, if $a<b\leq c<d$, then
\begin{equation}\label{rectangle_integral}\begin{split}
\int_a^b&\int_c^d\frac{dxdy}{|x-y|^{1+s}}=\\
&
\frac{1}{s(1-s)}\big[(c-a)^{1-s}+(d-b)^{1-s}-(c-b)^{1-s}-(d-a)^{1-s}\big].
\end{split}
\end{equation}
Also note that, if $n>m\geq1$, then
\begin{equation}\label{derivative_bound}\begin{split}
(1-a^n)^{1-s}-(1-a^m)^{1-s}&=\int_m^n\frac{d}{dt}(1-a^t)^{1-s}\,dt\\
&
=(s-1)\log a\int_m^n\frac{a^t}{(1-a^t)^s}\,dt\\
&
\leq a^m (s-1)\log a\int_m^n\frac{1}{(1-a^t)^s}\,dt\\
&
\leq(n-m)a^m\frac{(s-1)\log a}{(1-a)^s}.
\end{split}
\end{equation}
Now consider the first term
\begin{equation*}
\mathcal{I}_1=\sum_{k=1}^\infty\int_{a^{2k+1}}^{a^{2k}}\int_{-1}^0\frac{dxdy}{|x-y|^{1+s}}.
\end{equation*}
Use $(\ref{rectangle_integral}$) and notice that $(c-a)^{1-s}-(d-a)^{1-s}\leq0$ to get
\begin{equation*}
\int_{-1}^0\int_{a^{2k+1}}^{a^{2k}}\frac{dxdy}{|x-y|^{1+s}}
\leq\frac{1}{s(1-s)}\big[(a^{2k})^{1-s}-(a^{2k+1})^{1-s}\big]\leq\frac{1}{s(1-s)}(a^{2(1-s)})^k.
\end{equation*}
Then, as $a^{2(1-s)}<1$ we get
\begin{equation*}
\mathcal{I}_1\leq\frac{1}{s(1-s)}\sum_{k=1}^\infty(a^{2(1-s)})^k<\infty.
\end{equation*}
As for the last term
\begin{equation*}
\mathcal{I}_3=\sum_{k=1}^\infty\int_{a^{2k+1}}^{a^{2k}}\int_a^1\frac{dxdy}{|x-y|^{1+s}},
\end{equation*}
use $(\ref{rectangle_integral}$) and notice that $(d-b)^{1-s}-(d-a)^{1-s}\leq0$ to get
\begin{equation*}\begin{split}
\int_{a^{2k+1}}^{a^{2k}}\int_a^1\frac{dxdy}{|x-y|^{1+s}}&
\leq\frac{1}{s(1-s)}\big[(1-a^{2k+1})^{1-s}-(1-a^{2k})^{1-s}\big]\\
&
\leq\frac{-\log a}{s(1-a)^s}a^{2k}\quad\textrm{by }(\ref{derivative_bound}).
\end{split}
\end{equation*}
Thus
\begin{equation*}
\mathcal{I}_3\leq\frac{-\log a}{s(1-a)^s}\sum_{k=1}^\infty(a^2)^k<\infty.
\end{equation*}
Finally we split the second term
\begin{equation*}
\mathcal{I}_2=\sum_{k=1}^\infty\sum_{j=1}^\infty\int_{a^{2k+1}}^{a^{2k}}\int_{a^{2j}}^{a^{2j-1}}
\frac{dxdy}{|x-y|^{1+s}}
\end{equation*}
into three pieces according to the cases $j>k$, $j=k$ and $j<k$.

If $j=k$, using $(\ref{rectangle_integral})$ we get
\begin{equation*}\begin{split}
\int_{a^{2k+1}}^{a^{2k}}&\int_{a^{2k}}^{a^{2k-1}}
\frac{dxdy}{|x-y|^{1+s}}=\\
&
=\frac{1}{s(1-s)}\big[(a^{2k}-a^{2k+1})^{1-s}+(a^{2k-1}-a^{2k})^{1-s}-(a^{2k-1}-a^{2k+1})^{1-s}\big]\\
&
=\frac{1}{s(1-s)}\big[a^{2k(1-s)}(1-a)^{1-s}+a^{(2k-1)(1-s)}(1-a)^{1-s}\\
&
\quad\quad\quad\quad\quad-a^{(2k-1)(1-s)}(1-a^2)^{1-s}\big]\\
&
=\frac{1}{s(1-s)}(a^{2(1-s)})^k\Big[(1-a)^{1-s}+\frac{(1-a)^{1-s}}{a^{1-s}}-\frac{(1-a^2)^{1-s}}{a^{1-s}}\Big].
\end{split}
\end{equation*}
Summing over $k\in\mathbb{N}$ we get
\begin{equation*}\begin{split}
\sum_{k=1}^\infty&\int_{a^{2k+1}}^{a^{2k}}\int_{a^{2k}}^{a^{2k-1}}
\frac{dxdy}{|x-y|^{1+s}}=\\
&
=\frac{1}{s(1-s)}\frac{a^{2(1-s)}}{1-a^{2(1-s)}}\Big[(1-a)^{1-s}+\frac{(1-a)^{1-s}}{a^{1-s}}-\frac{(1-a^2)^{1-s}}{a^{1-s}}\Big]<\infty.
\end{split}
\end{equation*}
In particular note that
\begin{equation*}\begin{split}
(1-s)&P_s(E)\geq(1-s)\mathcal{I}_2\\
&
\geq\frac{1}{s(1-a^{2(1-s)})}\big[a^{2(1-s)}(1-a)^{1-s}+a^{1-s}(1-a)^{1-s}-a^{1-s}(1-a^2)^{1-s}\big],
\end{split}
\end{equation*}
which tends to $+\infty$ when $s\to1$. This shows that $E$ cannot have finite perimeter.

To conclude let $j>k$, the case $j<k$ being similar, and consider
\begin{equation*}
\sum_{k=1}^\infty\sum_{j=k+1}^\infty\int_{a^{2j}}^{a^{2j-1}}\int_{a^{2k+1}}^{a^{2k}}
\frac{dxdy}{|x-y|^{1+s}}.
\end{equation*}
Again, using $(\ref{rectangle_integral}$) and $(d-b)^{1-s}-(d-a)^{1-s}\leq0$, we get
\begin{equation*}\begin{split}
\int_{a^{2j}}^{a^{2j-1}}&\int_{a^{2k+1}}^{a^{2k}}
\frac{dxdy}{|x-y|^{1+s}}\\
&
\leq\frac{1}{s(1-s)}\big[(a^{2k+1}-a^{2j})^{1-s}-(a^{2k+1}-a^{2j-1})^{1-s}\big]\\
&
=\frac{a^{1-s}}{s(1-s)}(a^{2(1-s)})^k\big[(1-a^{2(j-k)-1})^{1-s}-(1-a^{2(j-k)-2})^{1-s}\big]\\
&
\leq\frac{a^{1-s}}{s(1-s)}(a^{2(1-s)})^k\frac{(s-1)\log a}{(1-a)^s}a^{2(j-k)-2}\quad\quad\textrm{by }(\ref{derivative_bound})\\
&
=\frac{-\log a}{s(1-a^s)a^{s+1}}(a^{2(1-s)})^k(a^2)^{j-k},
\end{split}
\end{equation*}
for $j\geq k+2$. Then
\begin{equation*}
\begin{split}
\sum_{k=1}^\infty&\sum_{j=k+2}^\infty\int_{a^{2j}}^{a^{2j-1}}\int_{a^{2k+1}}^{a^{2k}}
\frac{dxdy}{|x-y|^{1+s}}\\
&
\leq\frac{-\log a}{s(1-a^s)a^{s+1}}\sum_{k=1}^\infty(a^{2(1-s)})^k\sum_{h=2}^\infty(a^2)^h<\infty.
\end{split}
\end{equation*}
If $j=k+1$ we get
\begin{equation*}\begin{split}
\sum_{k=1}^\infty\int_{a^{2k+2}}^{a^{2k+1}}\int_{a^{2k+1}}^{a^{2k}}\frac{dxdy}{|x-y|^{1+s}}&
\leq\frac{1}{s(1-s)}\sum_{k=1}^\infty(a^{2k+1}-a^{2k+2})^{1-s}\\
&
=\frac{a^{1-s}(1-a)^{1-s}}{s(1-s)}\sum_{k=1}^\infty(a^{2(1-s)})^k<\infty.
\end{split}
\end{equation*}
This shows that also $\mathcal{I}_2<\infty$, so that $P_s(E)<\infty$ for every $s\in(0,1)$ as claimed.

\end{section}
\end{chapter}

%%%%%%%%%%%%%%%%%%%%%%%%%%%%%%%%%%%%%%%%
%%%%%%%%%%%%%%%%%%%%%%%%%%%%%%%%%%%%%%%%
%%%%%%%%%%%%%%%%%%%%%%%%%%%%%%%%%%%%%%%%

%%%%%%%%%%%%%%%%%%%%%%%%%%%%%%%%%%%%%%%%
%%%%%%%%%%%%%%%%%%%%%%%%%%%%%%%%%%%%%%%%
%%%%%%%%%%%%%%%%%%%%%%%%%%%%%%%%%%%%%%%%

%%%%%%%%%%%%%%%%%%%%%%%%%%%%%%%%%%%%%%%%
%%%%%%%%%%%%%%%%%%%%%%%%%%%%%%%%%%%%%%%%
%%%%%%%%%%%%%%%%%%%%%%%%%%%%%%%%%%%%%%%%

%%%%%%%%%%%%%%%%%%%%%%%%%%%%%%%%%%%%%%%%
%%%%%%%%%%%%%%%%%%%%%%%%%%%%%%%%%%%%%%%%
%%%%%%%%%%%%%%%%%%%%%%%%%%%%%%%%%%%%%%%%

%%%%%%%%%%%%%%%%%%%%%%%%%%%%%%%%%%%%%%%%
%%%%%%%%%%%%%%%%%%%%%%%%%%%%%%%%%%%%%%%%
%%%%%%%%%%%%%%%%%%%%%%%%%%%%%%%%%%%%%%%%

%%%%%%%%%%%%%%%%%%%%%%%%%%%%%%%%%%%%%%%%
%%%%%%%%%%%%%%%%%%%%%%%%%%%%%%%%%%%%%%%%
%%%%%%%%%%%%%%%%%%%%%%%%%%%%%%%%%%%%%%%%

%%%%%%%%%%%%%%%%%%%%%%%%%%%%%%%%%%%%%%%%
%%%%%%%%%%%%%%%%%%%%%%%%%%%%%%%%%%%%%%%%
%%%%%%%%%%%%%%%%%%%%%%%%%%%%%%%%%%%%%%%%

%%%%%%%%%%%%%%%%%%%%%%%%%%%%%%%%%%%%%%%%
%%%%%%%%%%%%%%%%%%%%%%%%%%%%%%%%%%%%%%%%
%%%%%%%%%%%%%%%%%%%%%%%%%%%%%%%%%%%%%%%%

%%%%%%%%%%%%%%%%%%%%%%%%%%%%%%%%%%%%%%%%
%%%%%%%%%%%%%%%%%%%%%%%%%%%%%%%%%%%%%%%%
%%%%%%%%%%%%%%%%%%%%%%%%%%%%%%%%%%%%%%%%

\begin{chapter}{Nonlocal Minimal Surfaces}
\begin{section}{Nonlocal Minimal Surfaces}
In this section we give the definition of nonlocal minimal surface and we prove existence and compactness results.
\begin{defin}
Let $\Omega\subset\R$ be an open set and let $s\in(0,1)$. The set $E\subset\R$ is said to be $s$-minimal in $\Omega$ if
$P_s(E,\Omega)<\infty$ and
\begin{equation}
P_s(E,\Omega)\leq P_s(F,\Omega),
\end{equation}
for every $F\subset\R$ s.t. $F\setminus\Omega=E\setminus\Omega$.
\end{defin}
As in the classical case, $E\setminus\Omega$ plays the role of boundary data. However here it is not enough to know how it behaves in a neighborhood of $\partial\Omega$; indeed, since $P_s$ is nonlocal, we need to know the whole of $E\setminus\Omega$.

\begin{rmk}
Since from now on the index $s\in(0,1)$ will be fixed, we will usually write $\J_\Omega(E):=P_s(E,\Omega)$.\\
If $E$ is $s$-minimal in $\Omega$, we will also say that it is a minimizer for $\J_\Omega$.
\end{rmk}

Even if the definition makes sense for every open set $\Omega$, we will usually consider bounded open sets with Lipschitz boundary.
This ensures that for any fixed set $E_0$ we have
\begin{equation}\label{inf_existence}\begin{split}
\inf\{\J_\Omega(F)\,|\,F\setminus\Omega=E_0\setminus\Omega\}&\leq \J_\Omega(E_0\setminus\Omega)\\
&
=\Ll_s((E_0\setminus\Omega)\setminus\Omega,\Co(E_0\setminus\Omega)\cap\Omega)\\
&
\leq\Ll_s(\Co\Omega,\Omega)=P_s(\Omega)<\infty
\end{split}
\end{equation}

\begin{lem}
If $E$ is $s$-minimal in $\Omega$, then it is also $s$-minimal in every open subset $\Omega'\subset\Omega$.
\begin{proof}
Indeed, let $\Omega'\subset\Omega$ and $F\subset\R$. Then
\begin{equation}\label{eq1}\begin{split}
P_s(F,\Omega)-P_s(F,\Omega')=P_s^L(F,&\Omega\setminus\Omega')
+\Ll_s(F\setminus\Omega,\Co F\cap(\Omega\setminus\Omega'))\\
&
+\Ll_s(F\cap(\Omega\setminus\Omega'),\Co F\setminus\Omega).
\end{split}
\end{equation}
Now notice that if $F\setminus\Omega'=E\setminus\Omega'$, then the corresponding right hand sides in $(\ref{eq1})$ are equal and clearly we also have $F\setminus\Omega=E\setminus\Omega$.
Therefore
\begin{equation*}
\J_{\Omega'}(F)-\J_{\Omega'}(E)=\J_\Omega(F)-\J_\Omega(E)\geq0.
\end{equation*}
\end{proof}
\end{lem}

\begin{rmk}\label{elem_properties_sets}
Using Proposition $\ref{elementary_properties}$ it is immediate to see that if $E$ is $s$-minimal in $\Omega$, then if we dilate, rotate or translate both $E$ and $\Omega$, then we end up with an $s$-minimal set in the corresponding open set.\\
For example, let $\lambda>0$ and take a set $F$ s.t. $F\setminus\lambda\Omega=\lambda E\setminus\lambda\Omega$. Then
$(\lambda^{-1}F)\setminus\Omega=E\setminus\Omega$ and
\begin{equation*}
\J_{\lambda\Omega}(F)=\lambda^{n-s}\J_\Omega(\lambda^{-1}F)\geq
\lambda^{n-s}\J_\Omega(E)=\J_{\lambda\Omega}(\lambda E).
\end{equation*}
\end{rmk}

\begin{defin}
We say that a set $E$ is a (variational) supersolution in $\Omega$ if
\begin{equation}\label{var_supersol}
A\subset\Co E\cap\Omega\qquad\Longrightarrow\qquad\Ll_s(A,E)-\Ll_s(A,\Co(E\cup A))\leq0.
\end{equation}
It is a subsolution if
\begin{equation}\label{var_subsol}
A\subset E\cap\Omega\qquad\Longrightarrow\qquad\Ll_s(A,E\setminus A)-\Ll_s(A,\Co E)\geq0.
\end{equation}
\end{defin}

These definitions are justified by the following observation, which is easily obtained once we explicit all the terms.
\begin{lem}
Let $E\setminus\Omega=F\setminus\Omega$. Denote $A^+:= F\setminus E$ and $A^-:=E\setminus F$. Then
\begin{equation}\label{sper_difference}\begin{split}
\J_\Omega(F)-\J_\Omega(E)&=\left\{\Ll_s(A^-,E\setminus A^-)-\Ll_s(A^-,\Co E)\right\}+2\Ll_s(A^-,A^+)\\
&
\qquad\quad-\left\{\Ll_s(A^+,E)-\Ll_s(A^+,\Co(E\cup A^+))\right\}.
\end{split}\end{equation}
\end{lem}

As a consequence we get
\begin{prop}
The set $E$ is $s$-minimal in $\Omega$ if and only if it is both a subsolution and a supersolution.
\begin{proof}
Suppose $E$ is $s$-minimal. Let $A\subset\Co E\cap\Omega$ and define $F:=A\cup(E\setminus\Omega)$.
Using the notation of the Lemma, we have $A^+=A$ and $A^-=\emptyset$. Therefore, since $E$ is $s$-minimal and $F\setminus\Omega=E\setminus\Omega$, the right hand side of $(\ref{sper_difference})$ reduces to
\begin{equation*}
-\left\{\Ll_s(A,E)-\Ll_s(A,\Co(E\cup A))\right\}=\J_\Omega(F)-\J_\Omega(E)\geq0,
\end{equation*}
proving $(\ref{var_supersol})$.
In the same way we get also $(\ref{var_subsol})$.\\
On the other hand, if $F\setminus\Omega=E\setminus\Omega$, then $A^+\subset\Co E\cap\Omega$ and $A^-\subset E\cap\Omega$.
If we suppose that $E$ is both a subsolution and a supersolution, then all the terms in the right hand side of $(\ref{sper_difference})$ are non-negative, and hence
\begin{equation*}
\J_\Omega(F)-\J_\Omega(E)\geq0,
\end{equation*}
proving that $E$ is $s$-minimal in $\Omega$.

\end{proof}
\end{prop}

\begin{rmk}\label{elem_properties_variational}
Notice that $E$ is a subsolution (supersolution) in $\Omega$ if and only if $\Co E$ is a supersolution (subsolution) in $\Omega$.\\
Moreover, if $E$ is a subsolution (supersolution) in $\Omega$, then it is also a subsolution (supersolution) in every open subset $\Omega'\subset\Omega$.\\
Analogues of the statements in Remark $\ref{elem_properties_sets}$ hold for subsolutions and supersolutions.
\end{rmk}

\begin{prop}[Lower semicontinuity]
Let $\Omega\subset\R$ be an open set and $\{E_k\}$ a sequence of sets s.t. $E_k\xrightarrow{loc}E$. Then
\begin{equation*}
\J_\Omega(E)\leq\liminf_{k\to\infty}\J_\Omega(E_k).
\end{equation*}
\begin{proof}
The claim is just a consequence of Fatou's Lemma.
Indeed, it is enough to notice that if $\chi_{A_k}\longrightarrow\chi_A$ and $\chi_{B_k}\longrightarrow\chi_B$ in $L^1_{loc}(\R)$, then we can find $\{k_i\}\subset\mathbb{N}$ s.t. $k_i\nearrow\infty$ strictly and
\begin{equation*}
\chi_{A_{k_i}}(x)\chi_{B_{k_i}}(y)\longrightarrow\chi_A(x)\chi_B(y),
\end{equation*}
for a.e. $(x,y)\in\R\times\R$. Then Fatou's Lemma implies
\begin{equation*}\begin{split}
\Ll_s(A,B)&=\int\int\frac{\chi_A(x)\chi_B(y)}{\kers}\,dx\,dy\\
&
\leq\liminf_{i\to\infty}\int\int\frac{\chi_{A_{k_i}}(x)\chi_{B_{k_i}}(y)}{\kers}\,dx\,dy\\
&
=\liminf_{i\to\infty}\Ll_s(A_{k_i},B_{k_i}).
\end{split}
\end{equation*}
Applying this inequality to both the terms in the definition of $\J_\Omega$ we get the claim.

\end{proof}
\end{prop}

Using this Proposition and a compactness result for fractional Sobolev spaces we can now prove the existence of $s$-minimal sets
using the direct method of the Calculus of Variations.

\begin{teo}[Existence of minimizers]
Let $\Omega\subset\R$ be a bounded open set with Lipschitz boundary, and fix a set $E_0\subset\Co\Omega$. Then there exists a set $E$ s.t. $E\setminus\Omega=E_0$ and
\begin{equation*}
\J_\Omega(E)=\inf_{F\setminus\Omega=E_0}\J_\Omega(F).
\end{equation*}
\begin{proof}
As remarked above, since $\Omega$ is bounded and has Lipschitz boundary, then
\begin{equation*}
\inf_{F\setminus\Omega=E_0}\J_\Omega(F)\leq\J_\Omega(E_0)\leq P_s(\Omega)<\infty.
\end{equation*}
Let $\{F_k\}$ be a minimizing sequence, i.e. s.t. $F_k\setminus\Omega=E_0$ and
\begin{equation*}
\lim_{k\to\infty}\J_\Omega(F_k)=\inf_{F\setminus\Omega=E_0}\J_\Omega(F).
\end{equation*}
We can suppose that
\begin{equation*}
\J_\Omega(F_k)\leq M<\infty,\qquad\textrm{for every }k,
\end{equation*}
since in any case this is true for $k$ big enough. In particular
\begin{equation*}
[\chi_{F_k}]_{\fracso}=2P_s^L(F_k,\Omega)\leq2\J_\Omega(F_k)\leq2M,
\end{equation*}
and
\begin{equation*}
\|\chi_{F_k}\|_{L^1(\Omega)}=|F_k\cap\Omega|\leq|\Omega|,
\end{equation*}
for every $k$.
Then, using the compactness of the embedding $\fracso\hookrightarrow L^1(\Omega)$ (see Theorem $\ref{compact_embd_th}$), we get $u\in L^1(\Omega)$ s.t. $\chi_{F_k}\longrightarrow u$ in $L^1(\Omega)$ (we relabel the subsequence).
It is clear that $u$ is equal (in $L^1$) to the characteristic function of a set $F\subset\Omega$. Then, if we define $E:=E_0\cup F$ we have $F_k\xrightarrow{loc}E$ and hence the semicontinuity result implies
\begin{equation*}
\J_\Omega(E)\leq\liminf_{k\to\infty}\J_\Omega(F_k)=\inf_{F\setminus\Omega=E_0}\J_\Omega(F).
\end{equation*}
\end{proof}
\end{teo}

It is convenient to have a estimate for the difference $\J_\Omega(E)-\J_\Omega(F)$ also in the case $E\cap\Omega=F\cap\Omega$.

\begin{lem}
Let $\Omega\subset\R$ be a bounded open set with Lipschitz boundary. Assume $E=F$ inside $\Omega$; then
\begin{equation}\label{confront_equal_inside}
|\J_\Omega(E)-\J_\Omega(F)|\leq\Ll_s(\Omega,(E\Delta F)\setminus\Omega).
\end{equation}
\begin{proof}
Since $E$ is equal to $F$ inside $\Omega$, 
\begin{equation*}\begin{split}
\J_\Omega(E)-\J_\Omega(F)&=\Ll_s(E\cap\Omega,\Co E\setminus\Omega)+\Ll_s(E\setminus\Omega,\Co E\cap\Omega)\\
&
\qquad\quad-\Ll_s(F\cap\Omega,\Co F\setminus\Omega)-\Ll_s(F\setminus\Omega,\Co F\cap\Omega)\\
&
=\Ll_s(E\cap\Omega,\Co E\setminus\Omega)-\Ll_s(E\cap\Omega,\Co F\setminus\Omega)\\
&
\qquad\quad+\Ll_s(E\setminus\Omega,\Co E\cap\Omega)-\Ll_s(F\setminus\Omega,\Co E\cap\Omega).
\end{split}\end{equation*}
Now if we take the absolute value and explicit all the terms we get
\begin{equation*}\begin{split}
|\J_\Omega(E)-\J_\Omega(F)|&=\left|\int_{E\cap\Omega}\int_{\Co\Omega}\frac{\chi_F(y)-\chi_E(y)}{\kers}+
\int_{\Co E\cap\Omega}\int_{\Co\Omega}\frac{\chi_E(y)-\chi_F(y)}{\kers}\right|\\
&
\leq
\int_{E\cap\Omega}\int_{\Co\Omega}\frac{|\chi_F(y)-\chi_E(y)|}{\kers}+
\int_{\Co E\cap\Omega}\int_{\Co\Omega}\frac{|\chi_E(y)-\chi_F(y)|}{\kers}\\
&
=
\int_{E\cap\Omega}\int_{\Co\Omega}\frac{\chi_{E\Delta F}(y)}{\kers}+
\int_{\Co E\cap\Omega}\int_{\Co\Omega}\frac{\chi_{E\Delta F}(y)}{\kers}\\
&
=\Ll_s(\Omega,(E\Delta F)\setminus\Omega).
\end{split}\end{equation*}

\end{proof}
\end{lem}

Now we can prove a compactness result for $s$-minimal sets.
\begin{teo}\label{nonlocal_compactness}
Let $\{E_k\}$ be a sequence of $s$-minimal sets in $B_1$ s.t.
$
E_k\xrightarrow{loc} E.
$
Then $E$ is $s$-minimal in $B_1$ and
\begin{equation*}
\J_{B_1}(E)=\lim_{k\to\infty}\J_{B_1}(E_k).
\end{equation*}

\begin{proof}
Assume $F=E$ outside $B_1$ and let
\begin{equation*}
F_k:=(F\cap B_1)\cup(E_k\setminus B_1).
\end{equation*}
Then, since $F_k=E_k$ outside $B_1$ and $E_k$ is a minimizer, we have
\begin{equation*}
\J_{B_1}(F_k)\geq\J_{B_1}(E_k).
\end{equation*}
On the other hand, since $F_k=F$ inside $B_1$, inequality $(\ref{confront_equal_inside})$ gives
\begin{equation*}
|\J_{B_1}(F_k)-\J_{B_1}(F)|\leq\Ll_s(B_1,(F_k\Delta F)\setminus B_1)
=\Ll_s(B_1,(E_k\Delta E)\setminus B_1)=:b_k.
\end{equation*}
Now we have
\begin{equation*}
\J_{B_1}(F)+b_k\geq\J_{B_1}(F_k)\geq\J_{B_1}(E_k).
\end{equation*}
If we prove that $b_k\to0$, then
\begin{equation*}
\J_{B_1}(F)\geq\limsup_{k\to\infty}\J_{B_1}(E_k)\geq\liminf_{k\to\infty}\J_{B_1}(E_k)
\geq\J_{B_1}(E),
\end{equation*}
by lower semicontinuity, proving that $E$ is a minimizer for $\J_{B_1}$.\\
Also notice that taking $F=E$ gives
\begin{equation*}
\lim_{k\to\infty}\J_{B_1}(E_k)=\J_{B_1}(E).
\end{equation*}

We're left to prove that $b_k\to0$.\\
Define
\begin{equation*}
a_k(r):=\Han((E_k\Delta E)\cap\partial B_r)
\end{equation*}
and take any $r_0>1$. Then
\begin{equation*}
b_k=\Ll_s(B_1,(E_k\Delta E)\cap(B_{r_0}\setminus B_1))+\Ll_s(B_1,(E_k\Delta E)\setminus B_{r_0})
\end{equation*}
and the second term is
\begin{equation*}
\Ll_s(B_1,(E_k\Delta E)\setminus B_{r_0})\leq\frac{n\omega_n}{s}|B_1|(r_0-1)^{-s}.
\end{equation*}
As for the first term, we have
\begin{equation*}\begin{split}
\Ll_s(B_1,(E_k\Delta E)&\cap(B_{r_0}\setminus B_1))=\int_{B_{r_0}\setminus B_1}\chi_{E_k\Delta E}(x)\left(\int_{B_1}\frac{dy}{\kers}\right)dx\\
&
=\int_1^{r_0}\left(\int_{\partial B_r}\chi_{E_k\Delta E}(x)\left(\int_{B_1}\frac{dy}{\kers}\right)d\Han(x)\right)dr\\
&
\leq\frac{n\omega_n}{s}\int_1^{r_0}\left(\int_{\partial B_r}\frac{\chi_{E_k\Delta E}(x)}{(r-1)^s}d\Han(x)\right)dr\\
&
=\frac{n\omega_n}{s}\int_1^{r_0}\frac{a_k(r)}{(r-1)^s}dr.
\end{split}\end{equation*}
Now notice that
\begin{equation*}
\int_1^{r_0}\frac{1}{(r-1)^s}dr=\frac{(r_0-1)^{1-s}}{1-s}<\infty
\end{equation*}
and
\begin{equation*}
\int_1^{r_0}a_k(r)dr=|(E_k\Delta E)\cap(\overline{B_{r_0}}\setminus B_1)|\longrightarrow0.
\end{equation*}
Then, since
\begin{equation*}
a_k(r)\leq\Han(\partial B_r)=n\omega_nr^{n-1}\leq n\omega_nr_0^{n-1},
\end{equation*}
for every $r\leq r_0$, Lebesgue's dominated convergence Theorem gives
\begin{equation*}
\int_1^{r_0}\frac{a_k(r)}{(r-1)^s}dr\longrightarrow0.
\end{equation*}
Therefore
\begin{equation*}
\limsup_{k\to\infty}b_k\leq \frac{n\omega_n^2}{s}(r_0-1)^{-s}.
\end{equation*}
Since $r_0$ was arbitrary, this concludes the proof.

\end{proof}
\end{teo}

\begin{rmk}
Notice that the same proof applies if we take any ball $B_r(x)$ in place of $B_1$.
\end{rmk}

\end{section}

\begin{section}{Uniform Density Estimates}
Now we prove an analogue of the density estimates which hold in the classical case and we derive some consequences which will be fundamental in the sequel.

\begin{rmk}
From now on we suppose that any set $E$ satisfies $(\ref{gmt_assumption_eq})$. In particular
\begin{equation*}
\partial E=\partial^-E=\{x\in\R\,|\,0<|E\cap B_r(x)|<\omega_nr^n,\,\forall\,r>0\}.
\end{equation*}
\end{rmk}

The estimates will follow easily from the following result
\begin{lem}
Let $E$ be a subsolution in $B_1$. There exists a universal constant $c=c(n,s)>0$ s.t.
\begin{equation*}
|E\cap B_1|\leq c\qquad\Longrightarrow\qquad|E\cap B_{1/2}|=0.
\end{equation*}
\begin{proof}
Define for every $r\in(0,1]$
\begin{equation*}
V_r:=|E\cap B_r|,\qquad a(r):=\Han(E\cap\partial B_r).
\end{equation*}
Using the fractional Sobolev inequality (see Theorem $\ref{fractional_sobolev}$) with $p=1$,
\begin{equation*}
\|u\|_{L^\frac{n}{n-s}(\R)}\leq C[u]_{\fracs},
\end{equation*}
for $u=\chi_{E\cap B_r}$, we get
\begin{equation*}
V_r^\frac{n-s}{n}\leq2C\Ll_s(E\cap B_r,\Co(E\cap B_r)).
\end{equation*}
Now we split
\begin{equation*}
\Ll_s(E\cap B_r,\Co(E\cap B_r))=\Ll_s(E\cap B_r,\Co E)+\Ll_s(E\cap B_r,E\setminus B_r).
\end{equation*}
Since $E$ is a subsolution in $B_1$, $(\ref{var_subsol})$ with $A=E\cap B_r\subset E\cap B_1$ implies
\begin{equation*}
\Ll_s(E\cap B_r,\Co E)\leq\Ll_s(E\cap B_r,E\setminus B_r),
\end{equation*}
and, since $E\setminus B_r\subset \Co B_r$, we get
\begin{equation*}
\Ll_s(E\cap B_r,\Co(E\cap B_r))\leq2\Ll_s(E\cap B_r,\Co B_r).
\end{equation*}
For every $x\in E\cap B_r$ we have $d(x,\Co B_r)\geq r-|x|>0$, and hence (see Lemma $\ref{positive_distance}$)
\begin{equation*}
\int_{\Co B_r}\frac{dy}{\kers}\leq\frac{n\omega_n}{s}(r-|x|)^{-s}.
\end{equation*}
Therefore
\begin{equation*}\begin{split}
\Ll_s(E\cap B_r,\Co B_r)&=\int_0^r\left(\int_{\partial B_\rho}\chi_E(x)\left(\int_{\Co B_r}\frac{dy}{\kers}\right)
\,d\Han(x)\right)\,d\rho\\
&
\leq\frac{n\omega_n}{s}\int_0^r\left(\int_{\partial B_\rho}\chi_E(x)\,d\Han(x)\right)\frac{d\rho}{(r-\rho)^s}\\
&
=\frac{n\omega_n}{s}\int_0^r\frac{a(\rho)}{(r-\rho)^s}\,d\rho.
\end{split}
\end{equation*}
Putting everything together we have
\begin{equation*}
V_r^\frac{n-s}{n}\leq C\int_0^r\frac{a(\rho)}{(r-\rho)^s}\,d\rho.
\end{equation*}
Integrating on $(0,t)$, with $t<1$, we obtain
\begin{equation*}\begin{split}
\int_0^t V_r^\frac{n-s}{n}\,dr&\leq C\int_0^t\left(\int_0^r\frac{a(\rho)}{(r-\rho)^s}\,d\rho\right)\,dr\\
&
=C\int_0^ta(\rho)\left(\int_\rho^t(r-\rho)^{-s}\,dr\right)\,d\rho\\
&
=\frac{C}{1-s}\int_0^ta(\rho)(t-\rho)^{1-s}\,d\rho\\
&
\leq C\frac{t^{1-s}}{1-s}\int_0^ta(\rho)\,d\rho=Ct^{1-s}V_t.
\end{split}\end{equation*}
Now we consider the above inequality with
\begin{equation*}
t_k:=\frac{1}{2}+\frac{1}{2^k},\qquad k\geq1.
\end{equation*}
Notice that $t_1=1$, $t_k$ is strictly decreasing with $t_k-t_{k+1}=\frac{1}{2^{k+1}}$, and $t_k\to\frac{1}{2}$.
We have
\begin{equation*}
\int_0^{t_k}V_r^\frac{n-s}{n}\,dr\leq C t_k^{1-s}V_{t_k}\leq CV_{t_k},
\end{equation*}
and
\begin{equation*}\begin{split}
\int_0^{t_k}V_r^\frac{n-s}{n}\,dr&=\int_0^{t_{k+1}}V_r^\frac{n-s}{n}\,dr+\int_{t_{k+1}}^{t_k}V_r^\frac{n-s}{n}\,dr
\geq\int_{t_{k+1}}^{t_k}V_r^\frac{n-s}{n}\,dr\\
&
\geq\left(t_k-t_{k+1}\right)V^\frac{n-s}{n}_{t_{k+1}}=\frac{1}{2^{k+1}}V^\frac{n-s}{n}_{t_{k+1}},
\end{split}
\end{equation*}
since $V_r$ is nondecreasing. If we set $v_k:=V_{t_k}$, then
\begin{equation*}
2^{-(k+1)}v_{k+1}^\frac{n-s}{n}\leq Cv_k,
\end{equation*}
i.e.
\begin{equation*}
v_{k+1}\leq C_0 \left(2^\frac{n}{n-s}\right)^kv_k^\frac{n}{n-s},
\end{equation*}
where the constant $C_0=(2C)^\frac{n}{n-s}$ can be supposed to be strictly bigger than 1 and depends only on $n$ and $s$.\\
We claim that
if $v_1$ is small enough, $v_1\leq c(n,s)$, then $v_k\longrightarrow0$. Notice that this concludes the proof.\\
Let $b:=2^\frac{n}{n-s}>1$ and $\alpha:=\frac{s}{n-s}>0$ so that our inequality reads
\begin{equation*}
v_{k+1}\leq C_0b^kv_k^{1+\alpha}.
\end{equation*}
Now, if $v_1\leq C_0^{-\frac{1}{\alpha}}b^{-\frac{1}{\alpha^2}-\frac{1}{\alpha}}$, then
\begin{equation}\label{ineq_iter}
v_k\leq b^{-\frac{k-1}{\alpha}}v_1,
\end{equation}
for every $k\geq1$ and hence in particular $v_k\longrightarrow0$.\\
We prove $(\ref{ineq_iter})$ by induction. It is trivially satisfied if $k=1$. Suppose it is true for $k$; then
\begin{equation*}\begin{split}
v_{k+1}&\leq C_0b^kv_k^{1+\alpha}\leq C_0b^k\left(b^{-\frac{k-1}{\alpha}}v_1\right)^{1+\alpha}
=C_0b^{1-\frac{k-1}{\alpha}}v_1^\alpha v_1\\
&
\leq b^{-\frac{k}{\alpha}}v_1.
\end{split}\end{equation*}

\end{proof}
\end{lem}

Now we can prove the density estimates.
\begin{teo}[Uniform density estimate]
Let $E$ be a subsolution in $\Omega$. There exists a universal constant $c=c(n,s)>0$ s.t. if $x\in\partial E$ and $B_r(x)\subset\Omega$ then
\begin{equation*}
|E\cap B_r(x)|\geq cr^n.
\end{equation*}
In particular, if $E$ is $s$-minimal in $\Omega$, then
\begin{equation}\label{uniform_density_estimate}
|E\cap B_r(x)|\geq cr^n,\qquad |\Co E\cap B_r(x)|\geq cr^n.
\end{equation}
\begin{proof}
The second statement is an immediate consequence of the first. Indeed, if $E$ is $s$-minimal, then $E$ is also a supersolution
i.e. $\Co E$ is a subsolution and hence it satisfies the hypothesis of the Theorem.

The first inequality follows from previous Lemma.\\
Since $B_r(x)\subset\Omega$, the set $E$ is a subsolution in $B_r(x)$ and hence $\frac{E-x}{r}$ is a subsolution in $B_1$.
Let $c$ be the constant in the Lemma. If we suppose that
\begin{equation*}
r^n\left|\frac{E-x}{r}\cap B_1\right|=|E\cap B_r(x)|<cr^n,
\end{equation*}
we have
\begin{equation*}
\left|\frac{E-x}{r}\cap B_1\right|<c,
\end{equation*}
and hence
\begin{equation*}
\left|\frac{E-x}{r}\cap B_\frac{1}{2}\right|=0,\qquad\textrm{i.e. }\left|E\cap B_\frac{r}{2}(x)\right|=0.
\end{equation*}
However $x\in\partial E$ and hence $|E\cap B_\rho(x)|>0$ for every $\rho>0$. This gives a contradiction, proving that
\begin{equation*}
|E\cap B_r(x)|\geq cr^n.
\end{equation*}

\end{proof}
\end{teo}

A first consequence of the uniform density estimate is that we can always find a small ball completely contained
in $E\cap B_1$ and one contained in $\Co E\cap B_1$, if $0\in\partial E$. 
\begin{coroll}[Clean ball condition]\label{clean_ball}
Let $E$ be $s$-minimal in $\Omega$, $x\in\partial E$ and $B_r(x)\subset\Omega$. There exist balls
\begin{equation*}
B_{cr}(y_1)\subset E\cap B_r(x),\qquad B_{cr}(y_2)\subset\Co E\cap B_r(x),
\end{equation*}
for some small universal constant $c=c(n,s)>0$.
\begin{proof}
We can assume that $x=0$ and $B_r(x)=B_1$.
Without loss of generality we can also suppose that $B_1\subset\subset\Omega$, otherwise we could consider $B_{1/2}$.
In particular we can suppose that $d(B_1,\partial\Omega)>1/2$ i.e. that $B_{3/2}\subset\Omega$.\\
We decompose the space into cubes of side $\delta$. We want to show that $N_\delta$, the number of cubes contained in $B_1$
which intersect $\partial E$, satisfies
\begin{equation}\label{cube_number}
N_\delta\leq C\delta^{s-n},
\end{equation}
for some constant $C=C(n,s)>0$. Then, since $0\in\partial E$, the density estimate for the ball $B_{1/2}$ gives
\begin{equation*}
|E\cap B_{1/2}|\geq c 2^{-n},
\end{equation*}
and hence at least $2^{-n}c\delta^{-n}$ of the cubes intersect $E\cap B_{1/2}$. Moreover, if $\delta$ is small, say $\delta<\frac{1}{12\sqrt{n}}$,
then all these cubes are contained in $B_1$.\\
If noone of these cubes is completely contained in $E\cap B_1$,
then they all intersect $\partial E$, so that
\begin{equation*}
N_\delta\geq2^{-n}c\delta^{-n}.
\end{equation*}
This and $(\ref{cube_number})$ give
\begin{equation*}
\frac{c}{2^n C}\leq\delta^s,
\end{equation*}
where the left hand side does not depend on $\delta$. Therefore if we take $\delta$ small enough we get a contradiction.
This proves that there is a cube of side $\delta$ completely contained in $E\cap B_1$, and hence also a ball of radius $\delta/2$.
To conclude, notice that we can take
\begin{equation*}
\delta=\delta(n,s)=\min\left\{\frac{1}{2}\left(\frac{c}{2^n C}\right)^\frac{1}{s},\,\frac{1}{12\sqrt{n}}\right\},
\end{equation*}
and hence get the claim with $c=\delta/2$. It is clear that we can apply the same argument for $\Co E$.

We are left to prove $(\ref{cube_number})$.\\
Let $Q_\delta\subset B_1$ be a cube of side $\delta$ s.t. $Q_\delta\cap\partial E\not=\emptyset$ and let $y$ be a point in the intersection. Now if $\delta$ is small enough, $\delta<\frac{1}{12\sqrt{n}}$, then the cube $Q_{3\delta}$ (with same center and side $3\delta$)
is contained in $B_{3/2}$ and hence in $\Omega$.\\
Therefore there is a ball $B_\delta(y)\subset Q_{3\delta}\subset\Omega$, with $y\in\partial E$, so the density estimates give
\begin{equation*}
|E\cap Q_{3\delta}|\geq c\delta^n,\qquad|\Co E\cap Q_{3\delta}|\geq c\delta^n.
\end{equation*}
If $x\in E\cap Q_{3\delta}$ and $y\in\Co E\cap Q_{3\delta}$, then $|x-y|\leq$diam$(Q_{3\delta})=\sqrt{n}\,3\delta$, so
\begin{equation*}\begin{split}
\Ll_s(E\cap Q_{3\delta},\,\Co E\cap Q_{3\delta})&=\int_{E\cap Q_{3\delta}}\int_{\Co E\cap Q_{3\delta}}\frac{dx\,dy}{\kers}\\
&
\geq\int_{E\cap Q_{3\delta}}\int_{\Co E\cap Q_{3\delta}}\frac{dx\,dy}{(\sqrt{n}\,3\delta)^{n+s}}\\
&
=\frac{|E\cap Q_{3\delta}|\cdot|\Co E\cap Q_{3\delta}|}{(\sqrt{n}\,3\delta)^{n+s}}\\
&
\geq\frac{c^2}{(3\,\sqrt{n})^{n+s}}\,\frac{\delta^n\cdot\delta^n}{\delta^{n+s}}=c_0\delta^{n-s}.
\end{split}\end{equation*}
Let $F'_\delta$ be the family of cubes of side $\delta$ contained in $B_{3/2}$ and $F_\delta$ the subfamily made of those contained in $B_1$; finally let $G_\delta\subset F_\delta$ be the subfamily of those
intersecting $\partial E$, so that $N_\delta$ is the cardinality of $G_\delta$.\\
Notice that if $Q_\delta\in F_\delta$, then the cube $Q_{3\delta}$ covers $3^n$ cubes of $F'_\delta$.\\
Therefore, since the intersection of two distinct cubes has zero Lebesgue measure, we get
\begin{equation*}\begin{split}
\sum_{Q_\delta,Q_\delta'\in F_\delta} \int_{Q_{3\delta}\cap E}&\int_{Q'_{3\delta}\cap \Co E}\frac{dx\,dy}{\kers}\\
&
\leq9^n \sum_{Q_\delta,Q_\delta'\in F'_\delta} \int_{Q_\delta \cap E}\int_{Q_\delta'\cap \Co E}\frac{dx\,dy}{\kers}\\
&
\leq 9^n \int_{B_{3/2}\cap E}\int_{B_{3/2}\cap \Co E}\frac{dx\,dy}{\kers}.
\end{split}\end{equation*}
On the other hand we have
\begin{equation*}\begin{split}
\sum_{Q_\delta,Q_\delta'\in F_\delta} \int_{Q_{3\delta}\cap E}\int_{Q_{3\delta}'\cap \Co E}\frac{dx\,dy}{\kers}
&\ge \sum_{Q_\delta\in F_\delta} \int_{Q_{3\delta}\cap E}\int_{Q_{3\delta}\cap\Co E}\frac{dx\,dy}{\kers}\\
&
\ge \sum_{Q_\delta\in G_\delta} \int_{Q_{3\delta}\cap E}\int_{Q_{3\delta}\cap\Co E}\frac{dx\,dy}{\kers}\\
&
\ge \sum_{Q_\delta\in G_\delta} c_0\delta^{n-s}
= c_0\delta^{n-s} N_\delta.
\end{split}\end{equation*}
These give
\begin{equation*}
c_0\delta^{n-s} N_\delta\leq9^n\Ll_s(E\cap B_{3/2},\Co E\cap B_{3/2}).
\end{equation*}
Finally from the minimality of $E$ we get
\begin{equation*}\begin{split}
\Ll_s(E\cap B_{3/2},\Co E\cap B_{3/2})&\leq\Ll_s(E\cap B_{3/2},\Co E)\leq\Ll_s(E\cap B_{3/2},E\cap\Co B_{3/2})\\
&
\leq\Ll_s(B_{3/2},\Co B_{3/2})=P_s(B_{3/2}).
\end{split}
\end{equation*}
This proves $(\ref{cube_number})$, concluding the proof.

\end{proof}
\end{coroll}

From the proof of this Corollary we can deduce an estimate on the
Hausdorff measure of $\partial E\cap\Omega$.

\begin{coroll}
If $E$ is $s$-minimal in $\Omega$, then
\begin{equation*}
\mathcal{H}^{n-s}(\partial E\cap\Omega)<\infty.
\end{equation*}
\end{coroll}

Actually we will prove that $\partial E\cap\Omega$ has Hausdorff dimension equal to $n-1$.\\
For this reason we can think of $\partial E\cap\Omega$ as a nonlocal minimal surface.\\

Another important consequence of the density estimates is the following improvement of the convergence of $s$-minimal sets.
\begin{coroll}\label{haus_conv_min}
Let $\{E_k\}$ be a sequence of $s$-minimal sets in $\Omega$ s.t.
$
E_k\xrightarrow{loc} E.
$
For every compact set $K\subset\Omega$ and every $\epsilon>0$ there exists $n_0$ s.t.
\begin{equation}\label{sminimal_hausdorff_convergence}
\partial E_k\cap K\subset N_\epsilon(\partial E)\cap K,\qquad\textrm{for }k\ge n_0,
\end{equation}
where $N_\epsilon(\partial E)$ is the $\epsilon$-neighborhood of $\partial E$.
\begin{proof}
Let $d:=d(K,\partial\Omega)>0$.
Suppose that there exists a sequence (eventually relabeled) $\{x_k\}$ and $\epsilon_0>0$ s.t.
\begin{equation*}
x_k\in\partial E_k\cap K\qquad\textrm{and}\quad d(x_k,\partial E)\geq\epsilon_0.
\end{equation*}
We can suppose that $\epsilon_0<d$. Since $d(x_k,\partial E)\geq\epsilon_0$, we have
\begin{equation*}
E_k\cap B_{\epsilon_0/2}(x_k)\subset E_k\setminus E,
\end{equation*}
and, since $\epsilon_0<d$, $B_{\epsilon_0/2}(x_k)\subset\Omega$, so that the density estimate gives
\begin{equation*}
|E_k\cap B_{\epsilon_0/2}(x_k)|\ge \frac{c}{2^n}\epsilon_0^n.
\end{equation*}
But this contradicts the $L^1_{loc}$ convergence. Indeed, for $R$ big enough we have $K\subset B_R$; now, since
$x_k\in K\subset B_R$, we have $B_{\epsilon_0/2}(x_k)\subset B_{R+\epsilon_0}=:B_{R'}$. Then
\begin{equation*}
|(E_k\Delta E)\cap B_{R'}|\geq|(E_k\setminus E)\cap B_{R'}|\ge|E_k\cap B_{\epsilon_0/2}(x_k)|\ge \frac{c}{2^n}\epsilon_0^n,
\end{equation*}
for every $k$.

\end{proof}
\end{coroll}

\end{section}

\end{chapter}

%%%%%%%%%%%%%%%%%%%%%%%%%%%%%%%%%%%%%%%%
%%%%%%%%%%%%%%%%%%%%%%%%%%%%%%%%%%%%%%%%
%%%%%%%%%%%%%%%%%%%%%%%%%%%%%%%%%%%%%%%%

%%%%%%%%%%%%%%%%%%%%%%%%%%%%%%%%%%%%%%%%
%%%%%%%%%%%%%%%%%%%%%%%%%%%%%%%%%%%%%%%%
%%%%%%%%%%%%%%%%%%%%%%%%%%%%%%%%%%%%%%%%

%%%%%%%%%%%%%%%%%%%%%%%%%%%%%%%%%%%%%%%%
%%%%%%%%%%%%%%%%%%%%%%%%%%%%%%%%%%%%%%%%
%%%%%%%%%%%%%%%%%%%%%%%%%%%%%%%%%%%%%%%%

%%%%%%%%%%%%%%%%%%%%%%%%%%%%%%%%%%%%%%%%
%%%%%%%%%%%%%%%%%%%%%%%%%%%%%%%%%%%%%%%%
%%%%%%%%%%%%%%%%%%%%%%%%%%%%%%%%%%%%%%%%

%%%%%%%%%%%%%%%%%%%%%%%%%%%%%%%%%%%%%%%%
%%%%%%%%%%%%%%%%%%%%%%%%%%%%%%%%%%%%%%%%
%%%%%%%%%%%%%%%%%%%%%%%%%%%%%%%%%%%%%%%%

%%%%%%%%%%%%%%%%%%%%%%%%%%%%%%%%%%%%%%%%
%%%%%%%%%%%%%%%%%%%%%%%%%%%%%%%%%%%%%%%%
%%%%%%%%%%%%%%%%%%%%%%%%%%%%%%%%%%%%%%%%

%%%%%%%%%%%%%%%%%%%%%%%%%%%%%%%%%%%%%%%%
%%%%%%%%%%%%%%%%%%%%%%%%%%%%%%%%%%%%%%%%
%%%%%%%%%%%%%%%%%%%%%%%%%%%%%%%%%%%%%%%%

%%%%%%%%%%%%%%%%%%%%%%%%%%%%%%%%%%%%%%%%
%%%%%%%%%%%%%%%%%%%%%%%%%%%%%%%%%%%%%%%%
%%%%%%%%%%%%%%%%%%%%%%%%%%%%%%%%%%%%%%%%

%%%%%%%%%%%%%%%%%%%%%%%%%%%%%%%%%%%%%%%%
%%%%%%%%%%%%%%%%%%%%%%%%%%%%%%%%%%%%%%%%
%%%%%%%%%%%%%%%%%%%%%%%%%%%%%%%%%%%%%%%%

\begin{chapter}{Fractional Mean Curvature}

In Section 2 we show that an $s$-minimal set $E$ satisfies the Euler-Lagrange equation
\begin{equation}\label{euler-lagrange_frac}
\I_s[E](x)=0,\qquad x\in\partial E
\end{equation}
in the viscosity sense, where $\I_s[E](x)$ denotes the 
fractional mean curvature of $\partial E$ at $x$, defined below.

This can be thought of as the fractional analogue
of the equation $(\ref{min_surf_eq})$ satisfied by classical minimal surfaces.

Moreover in Section 3 we show that the fractional mean curvature is the first variation of the fractional perimeter, at least when the set is regular enough. To be more precise, we show that if $\Phi_t:\R\to\R$ is a one-parameter family of $C^2$-diffeomorphisms 
which is $C^2$ also in $t$ and s.t. $\Phi_0=Id$, then
\begin{equation}\label{first_variation}
\frac{d}{dt}P_s(\Phi_t(E))\Big|_{t=0}=-\int_{\partial E}\I_s[E](x)\nu_E(x)\cdot\phi(x)\,d\Han(x),
\end{equation}
where $\phi(x):=\frac{\partial}{\partial t}\Phi_t(x)\big|_{t=0}$ and $E$ is any bounded open set with $C^2$ boundary.

\begin{section}{Definition}

\begin{rmk}
Again, in this chapter we suppose that any set $E$ satisfies $(\ref{gmt_assumption_eq})$. In particular
\begin{equation*}
\partial E=\partial^-E=\{x\in\R\,|\,0<|E\cap B_r(x)|<\omega_nr^n,\,\forall\,r>0\}.
\end{equation*}
\end{rmk}

\begin{defin}
Given a set $E$, the $s$-fractional mean curvature of $\partial E$ at a point $x\in\partial E$ is formally defined as
\begin{equation*}
\I_s[E](x):=\textrm{P.V.}\int_{\R}\frac{\chi_E(y)-\chi_{\Co E}(y)}{\kers}\,dy.
\end{equation*}
This means
\begin{equation}\label{frac_mc}
\I_s[E](x)=\lim_{\rho\to0}\I^\rho_s[E](x),
\end{equation}
where
\begin{equation*}
\I^\rho_s[E](x):=\int_{\R\setminus B_\rho(x)}\frac{\chi_E(y)-\chi_{\Co E}(y)}{\kers}\,dy.
\end{equation*}
\end{defin}

\begin{rmk}
The integral above has to be considered in the principal value sense because the integrand is not in the space $L^1(\R)$.
Moreover, in order for the limit in $(\ref{frac_mc})$ to be well defined, we need some sort of cancellation near the point $x$.
\end{rmk}

In \cite{curvature} it is shown that if the boundary of the set $E$ is $C^2$ near $x$, then the limit exists. 
The proof exploits the cancellation provided by the existence of tangent interior and exterior paraboloids in a neighborhood of $x$.

To be more precise, let $E$ be an open set s.t. $\partial E$ is $C^2$ in a neighborhood of $x\in\partial E$. We can suppose $x=0$. Then in normal coordinates we have
\begin{equation}\label{cancellation}\begin{split}
E\cap B_\rho\subset\left\{(y',y_n)\in\R\,|\,y_n\leq M|y'|^2\right\},\\
\Co E\cap B_\rho\subset\left\{(y',y_n)\in\R\,|\,y_n\geq- M|y'|^2\right\},
\end{split}
\end{equation}
for some $M>0$ and $\rho$ small.\\
Given $0<\delta'<\delta$ small,
\begin{equation*}\begin{split}
\left|\I_s^\delta[E](0)-\I_s^{\delta'}[E](0)\right|&
=\left|\int_{B_\delta\setminus B_{\delta'}}\frac{\chi_E(y)-\chi_{\Co E}(y)}{|y|^{n+s}}\,dy\right|\\
&
=\left|\int_{\delta'}^\delta\,d\rho\int_{\partial B_\rho}\frac{\chi_E(y)-\chi_{\Co E}(y)}{\rho^{n+s}}\,d\Han(y)\right|\\
&
=\left|\int_{\delta'}^\delta\frac{\Han(E\cap\partial B_\rho)-\Han(\Co E\cap\partial B_\rho)}{\rho^{n+s}}\,d\rho\right|.
\end{split}
\end{equation*}
Now thanks to $(\ref{cancellation})$ we get
\begin{equation*}\begin{split}
\Han(E\cap\partial B_\rho)-\Han(\Co E\cap\partial B_\rho)&=\Han(E\cap\Sigma_\rho)-\Han(\Co E\cap\Sigma_\rho)\\
&
\leq\Han(\Sigma_\rho),
\end{split}
\end{equation*}
where
\begin{equation*}
\Sigma_\rho:=\left\{(y',y_n)\in\R\,|\,|y_n|\leq M|y'|^2\right\}\cap\partial B_\rho,
\end{equation*}
and hence
\begin{equation*}
\left|\I_s^\delta[E](0)-\I_s^{\delta'}[E](0)\right|\leq\int_{\delta'}^\delta\frac{\Han(\Sigma_\rho)}{\rho^{n+s}}\,d\rho.
\end{equation*}
Therefore, if we show that $\Han(\Sigma_\rho)=O(\rho^n)$ as $\rho\to0$, then
the sequence $\I_s^\delta[E](0)$ is a Cauchy sequence and the fractional curvature is well defined.

Notice that $\Sigma_\rho=\left\{y\in \partial B_\rho\,|\,|y_n|\leq h\right\}$, with
\begin{equation*}
h=\frac{\sqrt{\frac{1}{M^2}+4\rho^2}-\frac{1}{M}}{2}.
\end{equation*}
If we take polar coordinates in such a way that $y_n=\rho\cos\theta_{n-1}$, then the condition
$|y_n|\leq h$ is translated in $\theta_{n-1}\in(\tau,\pi-\tau]$, with
$
\tau=\arccos\big(\frac{h}{\rho}\big).
$

\noindent
Therefore
\begin{equation*}\begin{split}
\Han(\Sigma_\rho)&=\rho^{n-1}\int_0^{2\pi}\,d\theta_1\dots\int_0^\pi\sin(\theta_{n-2})^{n-3}\,d\theta_{n-2}\int_\tau^{\pi-\tau}
\sin(\theta_{n-1})^{n-2}\,d\theta_{n-1}\\
&
=2\rho^{n-1}\h^{n-2}(\mathbb{S}^{n-2})\int_\tau^{\frac{\pi}{2}}(\sin t)^{n-2}\,dt\\
&
\leq2\rho^{n-1}\h^{n-2}(\mathbb{S}^{n-2})\Big(\frac{\pi}{2}-\tau\Big).
\end{split}
\end{equation*}
Using Taylor expansion we find $\tau=\frac{\pi}{2}-M\rho+o(\rho)$ and hence
$\Han(\Sigma_\rho)=O(\rho^n)$,
as $\rho\to0$.\\

Since the existence of tangent interior and exterior balls to $\partial E$ in $x$ is enough to get $(\ref{cancellation})$ (up to rotation and translation), we get the following

\begin{lem}
Suppose there exist two open sets $F_1\subset E$ and $F_2\subset\Co E$ with $x\in\partial E\cap\partial F_i$.
If $\partial F_i$ is $C^2$ in a neighborhood of $x$ for $i=1,2$, then $\I_s[E](x)$ is well defined.
\end{lem}

\begin{rmk}
In order to study the existence of the fractional mean curvature at $x\in\partial E$ it is enough to know the behavior of $E$
in a neighborhood of $x$.\\
However, unlike what happens with the classical mean curvature, we need to know the whole of $E$ to determine $\I_s[E](x)$, meaning that the fractional mean curvature is nonlocal.
\end{rmk}

\begin{lem}
If $E\subset F$ and $x\in\partial E\cap\partial F$, then
\begin{equation}\label{frac_conf}
\I_s^\delta[E](x)\leq\I_s^\delta[F](x),
\end{equation}
for every $\delta>0$.
\begin{proof}
It is enough to notice that
\begin{equation*}\begin{split}
E\subset F&\quad\Longrightarrow\quad\chi_E\leq\chi_F\quad\textrm{and}\quad
\chi_{\Co E}\geq\chi_{\Co F}\\
&
\quad\Longrightarrow\quad\chi_E-\chi_{\Co E}\leq\chi_F-\chi_{\Co F}.
\end{split}
\end{equation*}
\end{proof}
\end{lem}

%Notice that the fractional mean curvature at $x\in\partial E$ 

For more details about the fractional mean curvature, see \cite{curvature}.\\
In particular it is proved there that if $E$ is open and $\partial E$ is $C^2$ near $x$, then the $s$-fractional mean curvature approaches the classical mean curvature as $s\to1$, i.e.
\begin{equation*}
\lim_{s\to1}(1-s)\I_s[E](x)=(n-1)\omega_{n-1}H(x),
\end{equation*}
where $H(x)$ is the classical mean curvature of $\partial E$ at $x$, i.e. the arithmetic mean of the principal curvatures of $\partial E$ in $x$.\\
See also \cite{unifor} for the asymptotics.

\end{section}

\begin{section}{Euler-Lagrange Equation}

We begin by showing a comparison principle between the boundary $\partial E$ of an $s$-minimal set and the hyperplane $\{x_n=0\}$.\\
The same technique used in the proof, with some complications due to error terms, will be used in the proof of the Euler-Lagrange equation.

\begin{prop}
Let $E$ be an $s$-minimal set in $B_1$. Then
\begin{equation*}
\{x_n\leq0\}\setminus B_1\subset E\quad\Longrightarrow\quad\{x_n\leq0\}\subset E.
\end{equation*}
\begin{proof}

Define
\begin{equation*}
A^-:=\{x_n\leq0\}\setminus E,
\end{equation*}
and notice that $A^-\subset B_1\cap\Co E$.
%Therefore the minimality of $E$ implies
%\begin{equation*}
%0\geq\Ll_s(A^-,E)-\Ll_s(A^-,\Co(E\cup A^-)).
%\end{equation*}
We want to show that $|A^-|=0$.

To this end we define a new set as perturbation and exploit symmetry in order to obtain cancellation in the integrals.\\
Let $T$ be the reflection across $\{x_n=0\}$, i.e. $T(x',x_n)=(x',-x_n)$ and define
\begin{equation*}
A^+:=T(A^-)\setminus E,
\end{equation*}
and
\begin{equation*}
A:=A^-\cup A^+.
\end{equation*}
Decompose $A$ in two sets, $A_1$, which is symmetric with respect to $\{x_n=0\}$ and the remaining part $A_2\subset A^-$, i.e.
\begin{equation*}
A_1:=A^+\cup T(A^+),\qquad A_2:=A^-\setminus T(A^+).
\end{equation*}
Notice that
\begin{equation*}
\{x_n\leq0\}\subset E\cup A,
\end{equation*}
and define
\begin{equation*}
F:=T(\Co(E\cup A)).
\end{equation*}
Then
\begin{equation*}
F\subset\{x_n\leq0\}\setminus A^-\subset E.
\end{equation*}
From the minimality of $E$, since $A\subset\Co E\cap B_1$, we get
\begin{equation*}
0\geq\Ll_s(A,E)-\Ll_s(A,\Co(E\cup A))=\sum_{i=1,2}\left(\Ll_s(A_i,E)-\Ll_s(A_i,\Co(E\cup A))\right).
\end{equation*}
Using the reflection $T$, since $A_1$ is symmetric we get
\begin{equation*}
\Ll_s(A_1,\Co(E\cup A))=\Ll_s(A_1,F),
\end{equation*}
and hence
\begin{equation*}
\Ll_s(A_1,E)-\Ll_s(A_1,\Co(E\cup A))=\Ll_s(A_1,E\setminus F).
\end{equation*}
As for the second term,
\begin{equation*}
\Ll_s(A_2,E)-\Ll_s(A_2,\Co(E\cup A))=\Ll_s(A_2,E\setminus F)+\Ll_s(A_2,F)-\Ll_s(T(A_2),F).
\end{equation*}
For every $x,y\in\{x_n\leq0\}$ we have
\begin{equation*}
|x-y|\leq|T(x)-y|,
\end{equation*}
and hence
\begin{equation*}
\Ll_s(A_2,F)-\Ll_s(T(A_2),F)\geq0.
\end{equation*}
Putting everything together we get
\begin{equation*}
0\geq\Ll_s(A_1,E\setminus F)+\Ll_s(A_2,E\setminus F)+[\Ll_s(A_2,F)-\Ll_s(T(A_2),F)],
\end{equation*}
and all three terms are nonnegative, so they must all be equal to 0.\\
This can happen only if $|A_2|=0$ and either $|A_1|=0$ or $|E\setminus F|=0$.

If $|A_1|=0$, we're done. On the other hand, if $|E\setminus F|=0$, we can repeat the same argument
with the hyperplane $\{x_n=-\epsilon\}$ for every $\epsilon>0$ small and in this case we have
$|E\setminus F|>0$. Letting $\epsilon$ tend to 0, we get the claim.

\end{proof}
\end{prop}

Now we state and prove the main result.

\begin{teo}
Let $E$ be a supersolution in $B_R$, with $0\in\partial E$. Suppose that $B_2(-2e_n)\subset E$. Then
\begin{equation*}
\limsup_{\delta\to0}\I_s^\delta[E](0)\leq0.
\end{equation*}
\begin{proof}

Fix $\delta>0$ small and $0<\epsilon\ll\delta$.\\
Denote by $d_x$ the distance of $x$ from the sphere $\partial B_{1+\epsilon}(-e_n)$ and let $T$ be the radial reflection
with respect to the sphere $\partial B_{1+\epsilon}(-e_n)$ in the annulus $d_x<2\delta$, i.e.
\begin{equation*}
\frac{x+T(x)}{2}+e_n=(1+\epsilon)\frac{x+e_n}{|x+e_n|}.
\end{equation*}
It can be shown that
\begin{equation}\label{reflex1}
|\det DT(x)|\leq1+Cd_x\leq 2,
\end{equation}
since $d_x<2\delta$ and $\delta$ is small,
and
\begin{equation}\label{reflex2}
|T(x)-T(y)|\geq(1-c\max\{d_x,d_y\})|x-y|,
\end{equation}
for every $x,\,,y\in\{z\in\R\,|\,d_z<2\delta\}$.\\
We want to define a small perturbation $A$ near the point $0$, in such a way that we can exploit some sort of cancellation given by $T$, in order to control the difference $\J_{B_R}(E)-\J_{B_R}(A\cup E)$.\\
We define
\begin{equation*}
A^-:=B_{1+\epsilon}(-e_n)\setminus E,
\end{equation*}
\begin{equation*}
A^+:=T(A^-)\setminus E,\qquad A:=A^-\cup A^+.
\end{equation*}
Notice that
\begin{equation*}
A^-\subset B_{1+\epsilon}(-e_n)\setminus B_2(-2 e_n),
\end{equation*}
and hence, if $\epsilon$ is small enough,
\begin{equation*}
A\subset B_{2\sqrt{\epsilon}}\subset B_\delta\subset B_R.
\end{equation*}
We decompose $A$ in two disjoint sets, with $A_1=T(A_1)$, i.e.
\begin{equation*}
A_1:=T(A^+)\cup A^+,\qquad A_2:=A\setminus A_1\subset A^-.
\end{equation*}
Finally let
\begin{equation*}
F:=T(B_\delta\cap\Co(E\cup A)).
\end{equation*}
Notice that
\begin{equation*}
B_\delta\cap\Co(E\cup A)\subset B_\delta\cap \Co B_{1+\epsilon}(-e_n),
\end{equation*}
and
\begin{equation*}
T(B_\delta\cap \Co B_{1+\epsilon}(-e_n))\subset B_\delta\cap B_{1+\epsilon}(-e_n).
\end{equation*}
Therefore
\begin{equation*}
F\subset (B_\delta\cap B_{1+\epsilon}(-e_n))\setminus A^-\subset E\cap B_\delta.
\end{equation*}
Now, since $A\subset B_R\cap\Co E$ and $E$ is a supersolution in $B_R$, we get
\begin{equation*}\begin{split}
0&\geq\Ll_s(A,E)-\Ll_s(A,\Co(E\cup A))\\
&
=\Ll_s(A,E\setminus B_\delta)+\Ll_s(A,E\cap B_\delta)-\Ll_s(A,\Co(E\cup A)\cap B_\delta)\\
&
\qquad\qquad-\Ll_s(A,\Co(E\cup A)\setminus B_\delta)\\
&
=[\Ll_s(A,E\setminus B_\delta)-\Ll_s(A,\Co E\setminus B_\delta)]+\Ll_s(A,E\setminus F)\\
&
\qquad\qquad+[\Ll_s(A,F)-\Ll_s(A,T(F))]\\
&
\geq[\Ll_s(A,E\setminus B_\delta)-\Ll_s(A,\Co E\setminus B_\delta)]+[\Ll_s(A,F)-\Ll_s(A,T(F))]\\
&
=:I_1+I_2.
\end{split}
\end{equation*}

For simplicity in the following inequalities we will always write $C$ for the constants appearing, understanding that it changes when necessary.

We first estimate $I_1$
\begin{equation*}\begin{split}
&\left|\frac{1}{|A|}I_1-\I_s^\delta[E](0)\right|\\
&
\qquad\quad=\left|\frac{1}{|A|}\int_A\left(\int_{\Co B_\delta}(\chi_E(y)-\chi_{\Co E}(y))\left(\frac{1}{\kers}-
\frac{1}{|y|^{n+s}}\right)dy\right)dx
\right|\\
&
\qquad\quad=\left|\frac{1}{|A|}\int_{\Co B_\delta}(\chi_E(y)-\chi_{\Co E}(y))\left(\int_A\left(\frac{1}{\kers}-
\frac{1}{|y|^{n+s}}\right)dx\right)dy
\right|\\
&
\qquad\quad\leq\int_{\Co B_\delta}\left(\frac{1}{|A|}\int_A\left|\frac{1}{\kers}-
\frac{1}{|y|^{n+s}}\right|dx\right)dy.
\end{split}
\end{equation*}
We know that for every $x\in A$ and $y\in\Co B_\delta$
\begin{equation*}
\left|\frac{1}{\kers}-\frac{1}{|y|^{n+s}}\right|
=(n+s)\frac{1}{|\xi|^{n+s+1}}||x-y|-|y||,
\end{equation*}
for some point $\xi$ lying on the segment with endpoints $x$ and $x-y$.\\
Moreover, since $A\subset B_{2\sqrt{\epsilon}}$, we have
\begin{equation*}
||x-y|-|y||\leq2\sqrt{\epsilon},
\end{equation*}
and hence
\begin{equation*}
\left|\frac{1}{\kers}-\frac{1}{|y|^{n+s}}\right|
\leq C\sqrt{\epsilon}\,\min\{|y|,|x-y|\}^{-(n+s+1)}.
\end{equation*}
For every fixed $y\in\Co B_\delta$ we can split $A$ in the two sets
\begin{equation*}
S_1:=\{x\in A\,|\, |x-y|\geq|y|\},\qquad S_2:=\{x\in A\,|\, |x-y|<|y|\}.
\end{equation*}
On the second set we have
\begin{equation*}\begin{split}
\int_{\Co B_\delta}&\left(\frac{1}{|A|}\int_{S_2}\left|\frac{1}{\kers}-
\frac{1}{|y|^{n+s}}\right|dx\right)dy\\
&
\qquad\leq C\sqrt{\epsilon}\,\int_{\Co B_\delta}\left(\frac{1}{|A|}\int_{S_2}\frac{1}{|x-y|^{n+s+1}}\,dx\right)dy\\
&
\qquad\leq C\sqrt{\epsilon}\,\int_{\Co B_\delta}\left(\frac{1}{|A|}\int_A\frac{1}{|x-y|^{n+s+1}}\,dx\right)dy\\
&
\qquad= C\sqrt{\epsilon}\,\frac{1}{|A|}\int_A\left(\int_{\Co B_\delta}\frac{1}{|x-y|^{n+s+1}}\,dy\right)dx\\
&
\qquad\leq C\sqrt{\epsilon}\,\frac{1}{|A|}\int_A\left(\int_{\Co B_{\delta-2\sqrt{\epsilon}}}\frac{1}{|z|^{n+s+1}}\,dz\right)dx\\
&
\qquad =C\sqrt{\epsilon}\,\int_{\Co B_{\delta-2\sqrt{\epsilon}}}\frac{1}{|z|^{n+s+1}}\,dz
\end{split}
\end{equation*}
where in the last inequality we have simply translated $z=x-y$ in the inner integral. Since $\epsilon\ll\delta$, we get
\begin{equation*}
C\sqrt{\epsilon}\,\int_{\Co B_{\delta-2\sqrt{\epsilon}}}\frac{1}{|z|^{n+s+1}}\,dz
\leq C\sqrt{\epsilon}\,\int_{\Co B_{\delta/2}}\frac{1}{|z|^{n+s+1}}\,dz
\leq C \epsilon^{1/2}\delta^{-1-s}.
\end{equation*}
As for the first set $S_1$, we simply have
\begin{equation*}\begin{split}
\int_{\Co B_\delta}&\left(\frac{1}{|A|}\int_{S_1}\left|\frac{1}{\kers}-
\frac{1}{|y|^{n+s}}\right|dx\right)dy\\
&
\qquad\leq C\sqrt{\epsilon}\,\int_{\Co B_\delta}\left(\frac{1}{|A|}\int_{S_1}\frac{1}{|y|^{n+s+1}}\,dx\right)dy\\
&
\qquad\leq C\sqrt{\epsilon}\,\int_{\Co B_\delta}\frac{1}{|y|^{n+s+1}}\,dy\\
&
\qquad\leq C \epsilon^{1/2}\delta^{-1-s}.
\end{split}\end{equation*}
Therefore
\begin{equation}\label{first_fmc}
\left|\frac{1}{|A|}I_1-\I_s^\delta[E](0)\right|\leq C \epsilon^{1/2}\delta^{-1-s}.
\end{equation}
To estimate $I_2$ we write
\begin{equation*}
I_2=[\Ll_s(A_1,F)-\Ll_s(A_1,T(F))]+[\Ll_s(A_2,F)-\Ll_s(A_2,T(F))].
\end{equation*}
Changing variables via $T$, since $A_1=T(A_1)$, we get
\begin{equation*}\begin{split}
\Ll_s(A_1,T(F))&=\Ll_s(T(A_1),T(F))\\
&
=\int_{A_1}\int_F\frac{|\det DT(x)||\det DT(y)|}{|T(x)-T(y)|^{n+s}}\,dx\,dy\\
&
\leq\int_{A_1}\int_F\frac{1+C\max\{d_x,\,d_y\}}{\kers}\,dx\,dy\\
&
=\Ll_s(A_1,F)+C\int_{A_1}\int_F\frac{\max\{d_x,\,d_y\}}{\kers}\,dx\,dy,
\end{split}
\end{equation*}
where we have used $(\ref{reflex1})$, $(\ref{reflex2})$ and Taylor expansion.\\
Since for every $x\,,y\in B_{1+\epsilon}(-e_n)$
\begin{equation*}
|x-y|\leq|x-T(y)|,
\end{equation*}
and $A_2\subset A^-\subset B_{1+\epsilon}(-e_n),\quad F\subset B_{1+\epsilon}(-e_n)$, we have
\begin{equation*}\begin{split}
\Ll_s(A_2,T(F))&=\int_{A_2}dx\int_F\frac{|\det DT(y)|}{|x-T(y)|^{n+s}}\,dy
\leq\int_{A_2}dx\int_F\frac{1+C\,d_y}{\kers}\,dy\\
&
\leq\Ll_s(A_2,F)+C \int_{A_2}\int_F\frac{\max\{d_x,\,d_y\}}{\kers}\,dx\,dy.
\end{split}
\end{equation*}
Therefore
\begin{equation}\label{2fmc_int}
-I_2\leq C\int_A\int_F\frac{\max\{d_x,\,d_y\}}{\kers}\,dx\,dy.
\end{equation}

We want to show that
\begin{equation}\label{second_fmc}
-\frac{I_2}{|A|}\leq C\delta^{1-s}+o(\epsilon).
\end{equation}
Then, since $I_1+I_2\leq0$, we have
\begin{equation*}
\frac{I_1}{|A|}\leq-\frac{I_2}{|A|}\leq C\delta^{1-s}+o(\epsilon),
\end{equation*}
and hence from $(\ref{first_fmc})$ we get
\begin{equation*}
\I_s^\delta[E](0)\leq\frac{I_1}{|A|}+C\epsilon^{1/2}\delta^{-1-s}\leq C\delta^{1-s}+o(\epsilon)+C\epsilon^{1/2}\delta^{-1-s}.
\end{equation*}
Passing to the limit $\epsilon\to0$ we obtain
\begin{equation*}
\I_s^\delta[E](0)\leq C\delta^{1-s},
\end{equation*}
and this concludes the proof.

We are left to show $(\ref{second_fmc})$.\\
We begin by estimating the contribution in the integral in $(\ref{2fmc_int})$ given by $x$ outside
$B_{1+\epsilon}(-e_n)$, i.e. $x\in A^+$.\\
Recall that $T(A^+)\subset A^-$ and $|\det DT(x)|\leq2$. Therefore changing variables via $T$ we get
\begin{equation*}
\begin{split}
\int_{A^+}dx\int_F\frac{\max\{d_x\,,d_y\}}{\kers}\,dy&=\int_{T(A^+)}|\det DT(x)|dx\int_F\frac{\max\{d_x\,,d_y\}}{|T(x)-y|^{n+s}}\,dy\\
&
\leq2\int_{A^-}dx\int_F\frac{\max\{d_x\,,d_y\}}{|T(x)-y|^{n+s}}\,dy\\
&
\leq2\int_{A^-}dx\int_F\frac{\max\{d_x\,,d_y\}}{\kers}\,dy,
\end{split}\end{equation*}
and hence
\begin{equation*}
-I_2\leq C\int_{A^-}dx\int_F\frac{\max\{d_x\,,d_y\}}{\kers}\,dy.
\end{equation*}
For a fixed $x\in A^-$ we can distinguish the cases $y\in B_{2d_x}(x)$ and $y\in\Co B_{2d_x}(x)$.
Also recall that $F\subset B_\delta$ and $A^-\subset B_\delta$. Then
\begin{equation*}\begin{split}
\int_{F\setminus B_{2d_x}(x)}\frac{\max\{d_x\,,d_y\}}{\kers}\,dy&\leq
\int_{B_\delta\setminus B_{2d_x}(x)}\frac{\max\{d_x\,,d_y\}}{\kers}\,dy\\
&
\leq\int_{B_{2\delta}(x)\setminus B_{2d_x}(x)}\frac{\max\{d_x\,,d_y\}}{\kers}\,dy\\
&
=\int_{2 d_x}^{2\delta}dr\int_{\partial B_r(x)}\frac{\max\{d_x\,,d_y\}}{r^{n+s}}\,d\Han(y)\\
&
\leq\int_{2 d_x}^{2\delta}dr\int_{\partial B_r(x)}\frac{r}{r^{n+s}}\,d\Han(y)\\
&
=n\omega_n\int_{2 d_x}^{2\delta}\frac{r}{r^{n+s}}r^{n-1}\,dr\\
&
=\frac{n\omega_n}{1-s}\int_{2 d_x}^{2\delta}\frac{d}{dr}r^{1-s}dr\\
&
\leq C\delta^{1-s}.
\end{split}\end{equation*}
Integrating on $A^-$ we obtain the fist term of the right hand side in $(\ref{second_fmc})$
\begin{equation*}
\int_{A^-}dx\int_{F\setminus B_{2d_x}(x)}\frac{\max\{d_x\,,d_y\}}{\kers}\,dy\leq C|A|\delta^{1-s}.
\end{equation*}
On the other hand, if $y\in B_{2d_x}(x)$, then since $A^-\subset B_{1+\epsilon}(-e_n)\setminus B_2(-2e_n)$,
\begin{equation*}
\max\{d_x\,,d_y\}\leq3 d_x\leq 3\epsilon,
\end{equation*}
and hence
\begin{equation*}\begin{split}
\int_{A^-}dx\int_{F\cap B_{2d_x}(x)}\frac{\max\{d_x\,,d_y\}}{\kers}\,dy&\leq3\epsilon
\int_{A^-}\int_{F\cap B_{2d_x}(x)}\frac{1}{\kers}\,dx\,dy\\
&
\leq3\epsilon\,\Ll_s(A^-,F).
\end{split}\end{equation*}
Therefore
\begin{equation*}
-I_2\leq C|A|\delta^{1-s}+C\epsilon\,\Ll_s(A^-,F).
\end{equation*}
To conclude the proof it is now enough to prove the following
\begin{lem}
There exists a sequence $\epsilon\to0$ s.t.
\begin{equation*}
\epsilon\,\Ll_s(A^-,F)\leq C\epsilon^\eta\,|A^-|,
\end{equation*}
for some $\eta\in(0,1-s)$.
\begin{proof}
Since $E$ is a supersolution in $B_R$ and $A^-\subset \Co E\cap B_R$, we have
\begin{equation*}
\Ll_s(A^-,E)\leq\Ll_s(A^-,\Co(E\cup A^-)).
\end{equation*}
Therefore, since $B_{1+\epsilon}(-e_n)\subset E\cup A^-$, we get
\begin{equation*}\begin{split}
\Ll_s(A^-,F)&\leq\Ll_s(A^-,E)\leq\Ll_s(A^-,\Co(E\cup A^-))\\
&
\leq\Ll_s(A^-,\Co B_{1+\epsilon}(-e_n)).
\end{split}\end{equation*}
For every $x\in B_{1+\epsilon}(-e_n)$ we have
\begin{equation*}
\int_{\Co B_{1+\epsilon}(-e_n)}\frac{1}{\kers}\,dy\leq\int_{\Co B_{d_x}(x)}\frac{1}{\kers}\,dy\leq C\,d_x^{-s}.
\end{equation*}
We denote
\begin{equation*}
a(r):=\Han(\Co E\cap\partial B_{1+r}(-e_n)),
\end{equation*}
for every $r\in[0,\epsilon)$. Then
\begin{equation*}\begin{split}
\Ll_s(A^-,\Co B_{1+\epsilon}(-e_n))&=\int_{A^-}dx\int_{\Co B_{1+\epsilon}(-e_n)}\frac{1}{\kers}\,dy\\
&
\leq C\int_{A^-}d_x^{-s}\,dx\\
&
=C\int_0^\epsilon dr\int_{\Co E\cap\partial B_{1+r}(-e_n)}(\epsilon-r)^{-s}\,d\Han(x)\\
&
=C\int_0^\epsilon a(r) (\epsilon-r)^{-s}\,dr.
\end{split}
\end{equation*}
In order to prove the claim, we show that for a sequence $\epsilon\to0$ we have
\begin{equation*}
\epsilon\,\int_0^\epsilon a(r) (\epsilon-r)^{-s}\,dr\leq\epsilon^\eta\,\int_0^\epsilon a(r)\,dr=\epsilon^\eta\,|A^-|.
\end{equation*}
Assume by contradiction that for all $\epsilon$ small we have the opposite inequality
\begin{equation*}
\int_0^\epsilon a(r)(\epsilon-r)^{-s}\,dr>\epsilon^{\eta-1}\int_0^\epsilon a(r)\,dr.
\end{equation*}
Integrating this inequality in $\epsilon$ between $0$ and $\lambda$ we get
\begin{equation}\label{lemma52}
\lambda^{1-s}\int_0^\lambda a(r)\,dr\geq c(s,\eta)\lambda^\eta\int_0^{\frac{\lambda}{2}}a(r)\,dr.
\end{equation}
Indeed the left hand side gives
\begin{equation*}\begin{split}
\int_0^\lambda\left(\int_0^\epsilon a(r)(\epsilon-r)^{-s}\,dr\right)d\epsilon&=
\int_0^\lambda a(r)\left(\frac{1}{1-s}\int_r^\lambda\frac{d}{d\epsilon}(\epsilon-r)^{1-s}\,d\epsilon\right)dr\\
&
=\frac{1}{1-s}\int_0^\lambda a(r)(\lambda-r)^{1-s}\,dr\\
&
\leq \frac{1}{1-s}\lambda^{1-s}\int_0^\lambda a(r)\,dr.
\end{split}
\end{equation*}
As for the right hand side we have
\begin{equation*}\begin{split}
\int_0^\lambda\left(\epsilon^{\eta-1}\int_0^\epsilon a(r)\,dr\right)d\epsilon&
\geq \int_{\frac{\lambda}{2}}^\lambda\left(\epsilon^{\eta-1}\int_0^\epsilon a(r)\,dr\right)d\epsilon\\
&
\geq \int_{\frac{\lambda}{2}}^\lambda\left(\epsilon^{\eta-1}\int_0^{\frac{\lambda}{2}} a(r)\,dr\right)d\epsilon\\
&
=\frac{1}{\eta}\int_{\frac{\lambda}{2}}^\lambda\frac{d}{d\epsilon}\,\epsilon^\eta\,d\epsilon\int_0^{\frac{\lambda}{2}}a(r)\,dr\\
&
=\frac{1}{\eta}(1-2^{-\eta})\lambda^\eta\int_0^{\frac{\lambda}{2}}a(r)\,dr,
\end{split}
\end{equation*}
and we get $(\ref{lemma52})$. Let $\alpha:=1-s-\eta>0$; then $(\ref{lemma52})$ reads
\begin{equation*}
\int_0^\lambda a(r)\,dr\geq c\lambda^{-\alpha}\int_0^{\frac{\lambda}{2}}a(r)\,dr,
\end{equation*}
and hence for every $M>0$, if $\lambda$ is small enough, $\lambda<\lambda_0$, we get
\begin{equation*}
\int_0^\lambda a(r)\,dr\geq M\int_0^{\frac{\lambda}{2}}a(r)\,dr.
\end{equation*}
If we take $\lambda=2^{-k}$, with $k\geq k_0$,
\begin{equation*}\begin{split}
\int_0^{2^{-k}}a(r)\,dr&\leq M^{-1}\int_0^{2^{-k+1}}a(r)\,dr
\leq M^{-2}\int_0^{2^{-k+2}}a(r)\,dr\leq\dots\\
&
\leq M^{k_0-k}\int_0^{2^{-k_0}}a(r)\,dr,
\end{split}
\end{equation*}
and we have
\begin{equation*}
\int_0^{2^{-k_0}}a(r)\,dr=\left|\Co E\cap B_{1+2^{-k_0}}(-e_n)\right|\leq |B_2(-e_n)\setminus B_1(-e_n)|,
\end{equation*}
for every $k_0\in\mathbb{N}$. Therefore
\begin{equation*}
\left|\Co E\cap B_{1+2^{-k}}(-e_n)\right|\leq C M^{k_0-k}.
\end{equation*}
However, since $E$ is a supersolution in $B_R$, with $0\in\partial E$ and $B_{2^{-k}}\subset B_R$, the uniform density estimate gives
\begin{equation*}
\left|\Co E\cap B_{1+2^{-k}}(-e_n)\right|\geq\left|\Co E\cap B_{2^{-k}}\right|\geq c2^{-nk}.
\end{equation*}
Choosing $M=2^{n+1}$ we obtain
\begin{equation*}
c2^{-nk}\leq C 2^{(n+1)(k_0-k)},
\end{equation*}
for every $k\geq k_0$, i.e.
\begin{equation*}
2^{-k}\geq\frac{c}{C2^{(n+1)k_0}},
\end{equation*}
which yields a contradiction once we choose $k$ big enough.

This concludes the proof.

\end{proof}
\end{lem}

\end{proof}
\end{teo}

Scaling and traslating we get
\begin{coroll}\label{Euler_Lag_ball_eq}
Let $E$ be a supersolution in the open set $\Omega$. If $x\in\partial E\cap \Omega$ and $E\cap\Omega$ has an interior tangent ball
at $x$, then
\begin{equation*}
\limsup_{\delta\to0}\I_s^\delta[E](x)\leq0.
\end{equation*}
\end{coroll}

Now let $E$ be a supersolution in $\Omega$, with $x\in\partial E\cap\Omega$.\\
Suppose we have an open set $F$ which is contained in $E$ and touches $E$ in $x$, i.e. $F\subset E$ s.t. $x\in\partial F$, and suppose $\partial F$ is $C^2$ in a neighborhood of $x$.\\
Then, since $\partial F$ is $C^2$ near $x$, we can find an interior tangent ball at $x$, i.e.
\begin{equation*}
B_r(y)\subset F\qquad\textrm{s.t. }x\in\partial B_r(y)\cap\partial F.
\end{equation*}
Taking a smaller ball if necessary, we can suppose that $B_r(y)\subset\Omega$.\\
Clearly, since $F\subset E$ and $x\in\partial E\cap\partial F$,
$B_r(y)$ is also an interior tangent ball to $E$ in $x$. Therefore previous Corollary gives
\begin{equation*}
\limsup_{\delta\to0}\I_s^\delta[E](x)\leq0.
\end{equation*}
Since $F$ is regular near $x$, we know that the fractional mean curvature of $F$ at $x$ is well defined. Moreover
$(\ref{frac_conf})$ gives
\begin{equation*}
\I_s^\delta[F](x)\leq\I_s^\delta[E](x),
\end{equation*}
and hence passing to the limit $\delta\to0$ we get
\begin{equation*}
\I_s[F](x)\leq0.
\end{equation*}
This proves that a supersolution is also a viscosity supersolution, in the following sense
\begin{coroll}
Let $E$ be a supersolution in the open set $\Omega$, with $x\in\partial E\cap\Omega$. If $F$ is an open set contained in $E$
with $x\in\partial F$ and s.t. $\partial F$ is $C^2$ near $x$, then
\begin{equation*}
\I_s[F](x)\leq0.
\end{equation*}
\end{coroll}

\begin{rmk}
Notice that if $E$ is a subsolution we get the analogous statements just by considering $\Co E$, which is then a supersolution.\\
For example, if we have an exterior tangent ball at $x\in\partial E\cap\Omega$, then
\begin{equation*}
\liminf_{\delta\to0}\I_s^\delta[E](x)=-\limsup_{\delta\to0}\I_s^\delta[\Co E](x)\geq0.
\end{equation*}
\end{rmk}
In particular, when $E$ is minimal we have both inequalities and hence we get the following
\begin{coroll}
Let $E$ be an $s$-minimal set in the open set $\Omega$. If $x\in\partial E\cap \Omega$ and $E$ has an interior and an exterior tangent ball at $x$, both contained in $\Omega$, then
\begin{equation*}
\I_s[E](x)=0.
\end{equation*}
\end{coroll}
The above Corollary says that an $s$-minimal set is a classical solution of the zero fractional mean curvature equation
$(\ref{euler-lagrange_frac})$
in every regular enough point $x\in\partial E\cap\Omega$.\\

As a consequence of the Euler-Lagrange equation we can also improve the comparison principle shown earlier.
If the boundary of $E$ is contained in a strip outside of $\Omega$, then it is contained in the same strip also inside $\Omega$.
\begin{coroll}%[Comparison Principle]
Let $E$ be an $s$-minimal set in the bounded open set $\Omega$. If
\begin{equation*}
\{x_n\leq a\}\setminus\Omega\subset E\subset\{x_n\leq b\}\setminus\Omega,
\end{equation*}
then
\begin{equation*}
\{x_n\leq a\}\subset E\subset\{x_n\leq b\}.
\end{equation*}
\begin{proof}
We only show that
\begin{equation*}
\{x_n\leq a\}\subset E,
\end{equation*}
the other inclusion being analogous.

It is enough to prove
\begin{equation*}
\inf_{x\in\partial E}x_n\geq a.
\end{equation*}
Notice that by hypothesis we know
\begin{equation*}
a\leq\inf_{x\in\partial E\setminus\Omega}x_n\leq b,
\end{equation*}
and, since $\Omega$ is bounded,
\begin{equation*}
\inf_{x\in\Omega}x_n\leq\inf_{x\in\partial E\cap\Omega}x_n\leq\sup_{x\in\Omega}x_n,
\end{equation*}
so that $\inf_{x\in\partial E}x_n$ is finite.\\
%Moreover, since $\partial E\cap\overline{\Omega}$ is compact, we have a point $z\in\partial E\cap\overline{\Omega}$ which realizes the inf. Suppose
By contradiction suppose that
\begin{equation*}
\inf_{x\in\partial E}x_n< a.
\end{equation*}
Then we can traslate an hyperplane $\{x_n=t\}$ until we touch $\partial E$.\\
We can suppose that the contact point is $x=0\in\partial E\cap\Omega$ and that the tangent hyperplane is $\{x_n=0\}$,
with $0<a<b$. Since $P:=\{x_n\leq0\}\subset E$ and $0\in\partial E\cap \partial P$,
\begin{equation*}
(\chi_E(y)-\chi_{\Co E}(y))-(\chi_P(y)-\chi_{\Co P}(y))\geq0,
\end{equation*}
for every $y\in\R$.
Let $T$ denote reflection across $\{x_n=0\}$; then changing coordinates via $T$ we get
\begin{equation*}
\int_{\Co B_\delta}\frac{\chi_{\Co P}(y)}{|y|^{n+s}}\,dy=
\int_{\Co B_\delta}\frac{\chi_{\Co P}(T(y))}{|T(y)|^{n+s}}\,dy=
\int_{\Co B_\delta}\frac{\chi_P(y)}{|y|^{n+s}}\,dy,
\end{equation*}
so that $\I_s^\delta[P](0)=0$ for every $\delta>0$.\\
Moreover the Euler-Lagrange equation for $E$ gives
\begin{equation*}
\limsup_{\delta\to0}\I_s^\delta[E](0)\leq0.
\end{equation*}
Therefore
\begin{equation*}\begin{split}
0&\leq \limsup_{\delta\to0} \int_{C B_\delta} \frac{(\chi_E(y)-\chi_{C E}(y)) - (\chi_P(y)-\chi_{C P}(y))}{|y|^{n+s}}
\,dx\\
&
\qquad\qquad=\limsup_{\delta\to0}\I_s^\delta[E](0)\leq 0,
\end{split}
\end{equation*}
which implies $\chi_E(y)=\chi_P(y)$ a.e. $y\in\R$, i.e. $E=P$, but this contradicts the hypothesis
\begin{equation*}
\{x_n\leq a\}\setminus\Omega\subset E.
\end{equation*}

\end{proof}
\end{coroll}

Taking thinner and thinner strips we get the following
\begin{coroll}
 An hyperplane is locally $s$-minimal, meaning that it is $s$-minimal in every bounded open set $\Omega\subset\R$.
\end{coroll}

\begin{rmk}
We can think of equation $(\ref{euler-lagrange_frac})$ as saying that
\begin{equation*}
-(-\Delta)^\frac{s}{2}(\chi_E-\chi_{\Co E})=0\qquad\textrm{along}\quad\partial E,
\end{equation*}
if we think that a point $x_0\in\partial E$ belongs both to $E$ and $\Co E$. To be more precise, following the notation of Section 1.2, we define
\begin{equation*}
\tilde{\chi}_E(y):=\left\{\begin{array}{cc}
1,&y\in E_1,\\
0,&y\in\partial E,\\
-1,&y\in E_0.
\end{array}\right.
\end{equation*}
Then, if $|\partial E|=0$ we have $\tilde{\chi}_E(y)=\chi_E(y)-\chi_{\Co E}(y)$
a.e. $y\in\R$.\\
Therefore for every $x_0\in\partial E$ and $\delta>0$
\begin{equation*}\begin{split}
\I_s^\delta[E](x_0)&=\int_{\Co B_\delta(x_0)}\frac{\chi_E(y)-\chi_{\Co E}(y)}{|x_0-y|^{n+s}}\,dy
=\int_{\Co B_\delta(x_0)}\frac{\tilde{\chi}_E(y)}{|x_0-y|^{n+s}}\,dy\\
&
=-\int_{\Co B_\delta(x_0)}\frac{\tilde{\chi}_E(x_0)-\tilde{\chi}_E(y)}{|x_0-y|^{n+s}}\,dy,
\end{split}
\end{equation*}
and hence passing to the limit $\delta\to0$ formally yields
\begin{equation*}
\I_s[E](x_0)=-(-\Delta)^\frac{s}{2}(\tilde{\chi}_E)(x_0).
\end{equation*}
\end{rmk}

\end{section}

\begin{section}{First Variation of the Fractional Perimeter}

First of all we show that the fractional mean curvature is continuous with respect to $C^2$ convergence of sets.
\begin{defin}
If $E$ and $E_k$ are bounded open sets with $C^2$ boundary, we say that $E_k\longrightarrow E$ in $C^2$ if
$|E_k\Delta E|\longrightarrow0$ and
the boundaries converge in the $C^2$ sense, meaning for example that they can be described locally with a finite number of graphs
of functions which converge in $C^2$.
\end{defin}

\begin{prop}\label{continuity_curv_prop}
Let $E$ and $E_k$ be bounded open sets with $C^2$ boundary s.t. $E_k\longrightarrow E$ in $C^2$ and let
$x_k\in\partial E_k$, $x\in\partial E$ s.t. $x_k\longrightarrow x$. Then
\begin{equation}
\I_s[E_k](x_k)\longrightarrow\I_s[E](x).
\end{equation}

\begin{proof}
Without loss of generality we can simplify a little bit our situation.\\
We can suppose that $x=0\in\partial E$ and that the normal vector at $\partial E$ in $0$ is $\nu_E(0)=e_n$.
Moreover notice that the translated sets $E_k-x_k$ still converge in $C^2$ to $E$ and
\begin{equation*}
\I_s^\rho[E_k-x_k](0)=\I_s^\rho[E_k](x_k),
\end{equation*}
for every $\rho>0$, so that we can assume $x_k=0\in\partial E_k$ for every $k$.\\
We can also rotate the sets $E_k$ so that the normal vector in $0$ is $e_n$. Indeed, for every $k\in\mathbb{N}$ let $\mathcal{R}_k\in SO(n)$ be a rotation s.t. $\nu_{\mathcal{R}_kE_k}(0)=e_n$. Then we still have $\mathcal{R}_kE_k\longrightarrow E$ in $C^2$ and
\begin{equation*}
\I_s^\rho[\mathcal{R}_kE_k](0)=\I_s^\rho[E_k](0),
\end{equation*}
for every $\rho>0$.

We begin by showing that far from $0$ the $L^1$ convergence of the sets is enough.\\
Indeed, let $A\subset\R$ be s.t. $B_r\subset A$ for some $r>0$. Then
\begin{equation*}
\left|\int_{\Co A}\left\{(\chi_{E_k}(y)-\chi_{\Co E_k}(y))-(\chi_E(y)-\chi_{\Co E}(y))\right\}\,\frac{dy}{|y|^{n+s}}\right|
\leq\frac{|(E_k\Delta E)\cap\Co A|}{r^{n+s}},
\end{equation*}
which tends to 0 as $k\to\infty$.

Therefore we only need to study what happens near 0.\\
We work in the cylinder
\begin{equation*}
A=K_r:=B'_r\times (-r,r)=\{(y',y_n)\in\R\,|\,|y'|<r\textrm{ and }y_n\in(-r,r)\}.
\end{equation*}
Taking $r$ small enough we can write
\begin{equation*}
E\cap K_r=\{(y',y_n)\in\R\,|\,y'\in B'_r,\,-r<y_n<u(y')\},
\end{equation*}
with $u\in C^2(B'_r)$ s.t. $u(0)=0$ and also $\nabla u(0)=0$, since $\nu_E(0)=0$.\\
Using our assumptions we can also write for $k$ big enough
\begin{equation*}
E_k\cap K_r=\{(y',y_n)\in\R\,|\,y'\in B'_r,\,-r<y_n<u_k(y')\},
\end{equation*}
with $u_k\in C^2(B'_r)$ s.t. $u_k(0)=0$, $\nabla u_k(0)=0$ and $u_k\longrightarrow u$ in $C^2(B'_r)$.

As shown in \cite{unifor}, we have an explicit formula to calculate the contribution to the fractional mean curvature of $E$ in 0 coming from $K_r$,

\begin{equation}\label{curv_graph}
P.V.\int_{K_r}\frac{\chi_E(y)-\chi_{\Co E}(y)}{|y|^{n+s}}\,dy
=2\int_{B'_r}\left(\int_0^{\frac{u(y')}{|y'|}}\frac{dt}{(1+t^2)^\frac{n+s}{2}}\right)\frac{dy'}{|y'|^{n+s-1}}.
\end{equation}
We give a proof of this formula in the Lemma below.

Notice that the right hand side is well defined in the classical sense. Indeed using Taylor expansion and our hypothesis on $u$ we see that
\begin{equation*}
\left|\frac{u(y')}{|y'|}\right|\leq\frac{1}{2}\|D^2u\|_{C^0(B'_r)}|y'|,
\end{equation*}
and hence

\begin{equation*}\begin{split}
\int_{B'_r}\Big|\int_0^{\frac{u(y')}{|y'|}}\frac{dt}{(1+t^2)^\frac{n+s}{2}}\Big|&\frac{dy'}{|y'|^{n+s-1}}
\leq\frac{1}{2}\|D^2u\|_{C^0}\int_{B'_r}\frac{|y'|}{|y'|^{n+s-1}}\,dy'\\
&
=\frac{1}{2}\|D^2u\|_{C^0}\h^{n-2}(\s^{n-2})\int_0^r\frac{\rho^{n-2}}{\rho^{n+s-2}}\,d\rho\\
&
=\frac{1}{2}\|D^2u\|_{C^0}\frac{\h^{n-2}(\s^{n-2})}{1-s}r^{1-s}.
\end{split}
\end{equation*}

Clearly formula $(\ref{curv_graph})$ holds also for every $E_k$.
Moreover, using Taylor again, we see that
\begin{equation*}
\left|\frac{u(y')-u_k(y')}{|y'|}\right|\leq\frac{1}{2}\|D^2u-D^2u_k\|_{C^0(B'_r)}|y'|
\leq\frac{1}{2}\|u-u_k\|_{C^2(B'_r)}|y'|.
\end{equation*}
Therefore
\begin{equation*}\begin{split}
&\left|P.V.\int_{K_r}\frac{\chi_E(y)-\chi_{\Co E}(y)}{|y|^{n+s}}\,dy
-P.V.\int_{K_r}\frac{\chi_{E_k}(y)-\chi_{\Co E_k}(y)}{|y|^{n+s}}\,dy\right|\\
&
\qquad\quad
=2\left|\int_{B'_r}\Big(\int_0^{\frac{u(y')}{|y'|}}\frac{dt}{(1+t^2)^\frac{n+s}{2}}
-\int_0^{\frac{u_k(y')}{|y'|}}\frac{dt}{(1+t^2)^\frac{n+s}{2}}\Big)\frac{dy'}{|y'|^{n+s-1}}\right|\\
&
\qquad\quad
\leq2\int_{B'_r}\left|\frac{u(y')-u_k(y')}{|y'|}\right|\frac{dy'}{|y'|^{n+s-1}}\\
&
\qquad\quad
\leq\frac{\h^{n-2}(\s^{n-2})}{1-s}\,r^{1-s}\,\|u-u_k\|_{C^2(B'_r)},
\end{split}
\end{equation*}
which goes to 0 as $k\to\infty$.

Putting the two estimates together we get,
\begin{equation}\label{curv_cont_last}
\left|\I_s[E](0)-\I_s[E_k](0)\right|\leq\frac{|(E_k\Delta E)\cap\Co K_r|}{r^{n+s}}
+C\,r^{1-s}\|u-u_k\|_{C^2(B'_r)},
\end{equation}
and hence the claim.

\end{proof}
\end{prop}

\begin{rmk}
Actually, to get $(\ref{curv_cont_last})$ it is enough to suppose that the sets $E_k$ converge in the $L^1$ sense to $E$ outside some small ball $B$ centered at $0$ and that inside $B$ we can represent $E_k$ and $E$ as $C^2$ graphs, with the graphs
of $E_k$ converging to that of $E$ in the $C^2$ sense.\\
In particular, if we are interested in the continuity of the fractional curvature only in 0, there is no need to ask $C^2$-regularity of the boundaries far from
0.
\end{rmk}

Setting $E_k=E$ for every $k$ we trivially get the following

\begin{coroll}
Let $E$ be a bounded open set with $C^2$ boundary. Then the function
\begin{equation*}
\I_s[E](-):\partial E\longrightarrow\R,\qquad x\longmapsto\I_s[E](x)
\end{equation*}
is continuous.
\end{coroll}

\begin{rmk}
Notice that we can weaken our regularity assumptions on the sets involved and ask their boundaries to be only $C^{1,\alpha}$ for some $\alpha>s$ and the convergence to be in the $C^{1,\alpha}$ sense.\\
Indeed, suppose we can write
\begin{equation*}
E\cap K_r=\{(y',y_n)\in\R\,|\,y'\in B'_r,\,-r<y_n<u(y')\},
\end{equation*}
with $u\in C^{1,\alpha}(B'_r)$ s.t. $u(0)=0$ and $\nabla u(0)=0$ and $\alpha>s$. Then the mean value theorem gives
\begin{equation*}\begin{split}
|u(y')|&\leq|\nabla u(\xi)|\,|y'|=|\nabla u(\xi)-\nabla u(0)|\,|y'|
\leq\|\nabla u\|_{C^{0,\alpha}}|\xi|^\alpha\,|y'|\\
&
\leq\|u\|_{C^{1,\alpha}}|y'|^{1+\alpha}.
\end{split}
\end{equation*}
Since we are asking $\alpha>s$, this inequality is enough to guarantee that the right hand side of $(\ref{curv_graph})$
is well defined,

\begin{equation*}\begin{split}
\int_{B'_r}&\Big|\int_0^{\frac{u(y')}{|y'|}}\frac{dt}{(1+t^2)^\frac{n+s}{2}}\Big|\frac{dy'}{|y'|^{n+s-1}}
\leq\|u\|_{C^{1,\alpha}}\int_{B'_r}\frac{|y'|^\alpha}{|y'|^{n+s-1}}\,dy'\\
&
=\|u\|_{C^{1,\alpha}}\h^{n-2}(\s^{n-2})\int_0^r\frac{\rho^{n-2}}{\rho^{n+s-\alpha-1}}\,d\rho
=\|u\|_{C^{1,\alpha}}\frac{\h^{n-2}(\s^{n-2})}{\alpha-s}r^{\alpha-s}.
\end{split}
\end{equation*}
Once we have formula $(\ref{curv_graph})$, arguing as above we find
\begin{equation*}\begin{split}
&\left|P.V.\int_{K_r}\frac{\chi_E(y)-\chi_{\Co E}(y)}{|y|^{n+s}}\,dy
-P.V.\int_{K_r}\frac{\chi_{E_k}(y)-\chi_{\Co E_k}(y)}{|y|^{n+s}}\,dy\right|\\
&
\qquad\qquad\qquad
\leq C\,r^{\alpha-s}\|u-u_k\|_{C^{1,\alpha}(B'_r)},
\end{split}\end{equation*}
and hence the convergence.
\end{rmk}

Now we prove $(\ref{curv_graph})$.
\begin{lem}\label{explicit_curv_formula}
Let $E$ be an open set with $0\in\partial E$. Suppose that
\begin{equation*}
E\cap K_r=\{(y',y_n)\in\R\,|\,y'\in B'_r,\,-r<y_n<u(y')\},
\end{equation*}
with $u\in C^{1,\alpha}(B'_r)$ s.t. $u(0)=0$, $\nabla u(0)=0$ and $\alpha>s$. Then
\begin{equation}
P.V.\int_{K_r}\frac{\chi_E(y)-\chi_{\Co E}(y)}{|y|^{n+s}}\,dy
=2\int_{B'_r}\left(\int_0^{\frac{u(y')}{|y'|}}\frac{dt}{(1+t^2)^\frac{n+s}{2}}\right)\frac{dy'}{|y'|^{n+s-1}}.
\end{equation}
In particular the $s$-fractional mean curvature of $E$ in 0 is well defined.

\begin{proof}
%First of all, as remarked above, the right hand side is well defined in the classical sense.\\
 We take $0<\rho<r$ and split the set $K_r\setminus B_\rho$ as
\begin{equation*}\begin{split}
K_r\setminus B_\rho&
=(B'_r\setminus B'_\rho)\times(-r,r)\cup\{(y',y_n)\,|\,y'\in B'_\rho,\,y_n\in(-r,r)\setminus\pi(y')\}\\
&
=:S_1\cup S_2,
\end{split}\end{equation*}
where
\begin{equation*}
\pi(y'):=\left[-\sqrt{\rho^2-|y'|^2},\sqrt{\rho^2-|y'|^2}\right]%=:[-h(y'),h(y')]
\subset\mathbb{R}.
\end{equation*}

\noindent
We first calculate the contribution coming from $S_1$. We have
\begin{equation*}\begin{split}
\int_{K_r\setminus B_\rho}&\frac{\chi_E(y)-\chi_{\Co E}(y)}{|y|^{n+s}}\,dy\\
&
=\int_{S_1}\frac{\chi_E(y)-\chi_{\Co E}(y)}{|y|^{n+s}}\,dy=
+\int_{S_2}\frac{\chi_E(y)-\chi_{\Co E}(y)}{|y|^{n+s}}\,dy\\
&
=:I_1+I_2,
\end{split}
\end{equation*}

and

\begin{equation*}\begin{split}
I_1&
=\int_{B'_r\setminus B'_\rho}\left(\int_{-r}^{u(y')}\frac{dy_n}{(|y'|^2+y_n^2)^\frac{n+s}{2}}
-\int^{r}_{u(y')}\frac{dy_n}{(|y'|^2+y_n^2)^\frac{n+s}{2}}\right)dy'\\
&
=\int_{B'_r\setminus B'_\rho}\Big(\int_{-r}^{u(y')}\Big(1+\Big(\frac{y_n}{|y'|}\Big)^2\Big)^{-\frac{n+s}{2}}\,dy_n\\
&
\qquad\qquad\qquad
-\int^{r}_{u(y')}\Big(1+\Big(\frac{y_n}{|y'|}\Big)^2\Big)^{-\frac{n+s}{2}}\,dy_n\Big)\frac{dy'}{|y'|^{n+s}}.
\end{split}
\end{equation*}

\noindent
We change variables $y_n=|y'|t$ and we write for simplicity $g(t):=(1+t^2)^{-\frac{n+s}{2}}$
%and $\tilde{\pi}(y'):=|y'|^{-1}\pi(y')$,
obtaining

\begin{equation*}
I_1=\int_{B'_r\setminus B'_\rho}\left(\int_{-\frac{r}{|y'|}}^\frac{u(y')}{|y'|}g(t)\,dt
-\int^\frac{r}{|y'|}_\frac{u(y')}{|y'|}g(t)\,dt\right)\frac{dy'}{|y'|^{n+s-1}}.
\end{equation*}

\noindent
Since the function $g$ is even, we get
\begin{equation*}
I_1
=2\int_{B'_r\setminus B'_\rho}\Big(\int_0^\frac{u(y')}{|y'|}g(t)\,dt\Big)\frac{dy'}{|y'|^{n+s-1}}.
\end{equation*}

In a similar way, we see that the contribution coming from $S_2$ gives
\begin{equation*}\begin{split}
I_2&=\int_{B'_\rho}\Big(\int_\mathbb{R}\chi_{|y'|^{-1}(-r,u(y'))}(t)\big(1-\chi_{|y'|^{-1}\pi(y')}(t)\big)g(t)\,dt\\
&
\qquad
-\int_\mathbb{R}\chi_{|y'|^{-1}(u(y'),r)}(t)\big(1-\chi_{|y'|^{-1}\pi(y')}(t)\big)g(t)\,dt\Big)\frac{dy'}{|y'|^{n+s-1}}.
\end{split}
\end{equation*}

\noindent
If we let
\begin{equation*}
U:=\{y'\in B'_\rho\,|\,u(y')\in\pi(y')\},
\end{equation*}
then playing with the characteristic functions and using again that $g$ is even, we see that

\begin{equation*}
I_2=2\int_{B'_\rho}\Big(\int_0^\frac{u(y')}{|y'|}g(t)\,dt\Big)\frac{dy'}{|y'|^{n+s-1}}
-2\int_U\Big(\int_0^\frac{u(y')}{|y'|}g(t)\,dt\Big)\frac{dy'}{|y'|^{n+s-1}}.
\end{equation*}

Summing gives
\begin{equation*}
I_1+I_2=2\int_{B'_r}\Big(\int_0^\frac{u(y')}{|y'|}g(t)\,dt\Big)\frac{dy'}{|y'|^{n+s-1}}
-2\int_U\Big(\int_0^\frac{u(y')}{|y'|}g(t)\,dt\Big)\frac{dy'}{|y'|^{n+s-1}},
\end{equation*}
and, as remarked above
\begin{equation*}\begin{split}
\left|\int_U\Big(\int_0^\frac{u(y')}{|y'|}g(t)\,dt\Big)\frac{dy'}{|y'|^{n+s-1}}\right|&
\leq\int_{B'_\rho}\Big|\int_0^\frac{u(y')}{|y'|}g(t)\,dt\Big|\frac{dy'}{|y'|^{n+s-1}}\\
&
\leq C\,\|u\|_{C^{1,\alpha}(B'_r)}\rho^{\alpha-s},
\end{split}
\end{equation*}
which goes to 0 as $\rho\to0$.\\
Therefore passing to the limit concludes the proof.

\end{proof}
\end{lem}

%Also notice that  $(\ref{curv_graph})$ provides another way to show the existence of the fractional curvature of $E$ in 0.

\begin{rmk}
In a similar context, we might want to have estimates for the difference between the fractional mean curvature of a set $E$ with $C^{1,\alpha}$ boundary and that of the set $\Phi(E)$, where $\Phi$ is a $C^{1,\alpha}$-diffeomorphism of $\R$.
For a detailed study of the estimates involved see \cite{matteo}.
\end{rmk}

Now it is convenient to switch our attention from graphs to level sets.

We consider a function $\varphi\in C^2_c(\R)$ s.t. $\nabla \varphi\not=0$ in $\{t_1\leq \varphi\leq t_2\}$, for some $0<t_1<t_2$.
That $\varphi$ has compact support guarantees that the sets $\{\varphi\geq t\}$ are compact for every $t>0$. Moreover 
$\nabla \varphi\not=0$ in $\{t_1\leq \varphi\leq t_2\}$ implies that $\{\varphi=t\}$ is a $C^2$ hypersurface for every $t\in[t_1,t_2]$.\\
In particular $\I_s[\{\varphi\geq \varphi(x)\}](x)$ is well defined for every $x\in\{t_1\leq \varphi\leq t_2\}$.\\
For simplicity we write $E_t:=\{\varphi\geq t\}$, for $t\in\mathbb{R}$.\\

Given $x\in\{t_1\leq \varphi\leq t_2\}$, we have for every $\rho>0$
\begin{equation*}\begin{split}
\I^\rho_s[E_{\varphi(x)}](x)&
=\int_{\Co B_\rho(x)}\frac{\chi_{\{\varphi\geq\varphi(x)\}}(y)-\chi_{\{\varphi<\varphi(x)\}}(y)}{\kers}dy\\
&
=\int_{\Co B_\rho(x)}\frac{\sig(\varphi(y)-\varphi(x))}{\kers}dy,
\end{split}
\end{equation*}
where sgn is the sign function
\begin{equation*}
\sig(t):=\left\{\begin{matrix}1, & t\geq0\\-1,&t<0\end{matrix}\right..
\end{equation*}

\noindent
Since by symmetry
\begin{equation*}
\int_{\Co B_\rho(x)}\frac{\sig(\nabla\varphi(x)\cdot (y-x))}{\kers}dy=0,
\end{equation*}
we get
\begin{equation*}\begin{split}
\qquad\I^\rho_s[E_{\varphi(x)}](x)&
=\int_{\Co B_\rho(x)}\frac{\sig(\varphi(y)-\varphi(x))-\sig(\nabla\varphi(x)\cdot (y-x))}{\kers}dy\\
&
=2\int_{\Co B_\rho(x)}\Big(\frac{\chi_{\{y|\varphi(y)\geq\varphi(x),\nabla\varphi(x)\cdot (y-x)\leq0\}}(y)}
{\kers}\\
&\qquad\qquad\qquad
-\frac{\chi_{\{y|\varphi(y)<\varphi(x),\nabla\varphi(x)\cdot (y-x)>0\}}(y)}{\kers}\Big)dy.
\end{split}
\end{equation*}
We want to exploit this formula to show that the limit
\begin{equation}\label{unif_curv}
\I^\rho_s[E_{\varphi(x)}](x)\xrightarrow{\rho\to0^+}\I_s[E_{\varphi(x)}](x)
\end{equation}
is uniform in $x\in\{t_1\leq \varphi\leq t_2\}$.\\
To be more precise, since
\begin{equation*}\begin{split}
|\I_s^\rho[E_{\varphi(x)}](x)-\I_s[E_{\varphi(x)}](x)|&=
\lim_{\delta\to0}|\I_s^\rho[E_{\varphi(x)}](x)-\I_s^\delta[E_{\varphi(x)}](x)|\\
&
=\lim_{\delta\to0}\Big|\int_{B_\rho(x)\setminus B_\delta(x)}\frac{\chi_{E_{\varphi(x)}}(y)-\chi_{\Co E_{\varphi(x)}}(y)}{\kers}dy\Big|,
\end{split}
\end{equation*}
we want to use the above formula to bound
\begin{equation*}
\Big|\int_{B_\rho(x)\setminus B_\delta(x)}\frac{\chi_{E_{\varphi(x)}}(y)-\chi_{\Co E_{\varphi(x)}}(y)}{\kers}dy\Big|
\end{equation*}
independently on $\delta$, so that we can let $\delta\to0$, then show that what we obtain goes to 0 as $\rho\to0$, independently on $x$.

Roughly speaking, we are going to show that the sets $E_{\varphi(x)}$ satisfy a uniform paraboloid condition, meaning that
we have tangent inner and outer paraboloids with the same opening width for every $x$.\\
Let $x\in\{t_1\leq\varphi\leq t_2\}$. Using Taylor expansion,
\begin{equation*}
|\varphi(y)-\varphi(x)-\nabla\varphi(x)\cdot(y-x)|\leq\frac{1}{2}\|D^2\varphi\|_{C^0}|y-x|^2,
\end{equation*}
and hence we have
\begin{equation*}\begin{split}
\{y\in\R\,|&\,\varphi(y)\geq\varphi(x),\nabla\varphi(x)\cdot (y-x)\leq0\}\\
&
=\{y\in\R\,|\,0\leq-\nabla\varphi(x)\cdot(y-x)\leq\varphi(y)-\varphi(x)-\nabla\varphi(x)\cdot(y-x)\}\\
&
\subset\left\{y\in\R\,|\,0\leq-\nabla\varphi(x)\cdot(y-x)\leq\frac{1}{2}\|D^2\varphi\|_{C^0}|y-x|^2\right\}.
\end{split}
\end{equation*}
We write
\begin{equation*}
e:=-\frac{\nabla\varphi(x)}{|\nabla\varphi(x)|}
\quad\textrm{ and }\quad
p_e(z):=z-e\cdot z,
\end{equation*}
so that
\begin{equation*}
-\nabla\varphi(x)\cdot(y-x)=|\nabla\varphi(x)|\,e\cdot(y-x)
\end{equation*}
and, if we consider $y\in B_\rho(x)$,
\begin{equation*}
|y-x|^2=\big(e\cdot(y-x)\big)^2+\big(p_e(y-x)\big)^2
\leq \rho\, e\cdot(y-x)+\big(p_e(y-x)\big)^2.
\end{equation*}
Moreover, if we let
\begin{equation*}
\beta:=\inf_{\{t_1\leq\varphi\leq t_2\}}|\nabla\varphi|>0,
\end{equation*}
and we take
\begin{equation*}
0<\rho<\frac{\beta}{\|D^2\varphi\|_{C^0}},
\end{equation*}
then
\begin{equation*}
|\nabla\varphi(x)|-\frac{1}{2}\|D^2\varphi\|_{C^0}\,\rho>\frac{\beta}{2}>0.
\end{equation*}
Therefore
\begin{equation*}\begin{split}
&\Big\{y\in B_\rho(x)\setminus B_\delta(x)\,\big|\,0\leq-\nabla\varphi(x)\cdot(y-x)\leq\frac{1}{2}\|D^2\varphi\|_{C^0}|y-x|^2\Big\}\\
&\subset\Big\{y\in B_\rho(x)\setminus B_\delta(x)\,\big|\,0\leq\Big(|\nabla\varphi(x)|-\frac{\|D^2\varphi\|_{C^0}}{2}\rho\Big)e\cdot(y-x)\\
&
\qquad\qquad\qquad\qquad\qquad\qquad\qquad
\leq\frac{\|D^2\varphi\|_{C^0}}{2}\big(p_e(y-x)\big)^2\Big\}\\
&
=\Big\{y\in B_\rho(x)\setminus B_\delta(x)\,\big|\,0\leq e\cdot(y-x)\leq
\frac{\|D^2\varphi\|_{C^0}}{2|\nabla\varphi(x)|-\|D^2\varphi\|_{C^0}\,\rho}\big(p_e(y-x)\big)^2\Big\}\\
&
\subset\Big\{y\in B_\rho(x)\setminus B_\delta(x)\,\big|\,0\leq e\cdot(y-x)\leq\frac{\|D^2\varphi\|_{C^0}}{\beta}\big(p_e(y-x)\big)^2\Big\},
\end{split}
\end{equation*}
which is a paraboloid whose opening width is independent of $x$, as wanted.

Now we obtain
\begin{equation*}\begin{split}
\int_{B_\rho(x)\setminus B_\delta(x)}&\frac{\chi_{\{y|\varphi(y)\geq\varphi(x),\nabla\varphi(x)\cdot (y-x)\leq0\}}(y)}
{\kers}dy\\
&
\leq\int_{B_\rho(x)\setminus B_\delta(x)}\frac{\chi_{\{y|0\leq e\cdot(y-x)\leq\frac{\|D^2\varphi\|_{C^0}}{\beta}(p_e(y-x))^2\}}(y)}{\kers}dy\\
&
\leq\int_{B'_\rho}\Big(\int_0^{\frac{\|D^2\varphi\|_{C^0}}{\beta}|z'|}\frac{dt}{(1+t^2)^\frac{n+s}{2}}\Big)\frac{dz'}{|z'|^{n+s-1}}\\
&
\leq\frac{\h^{n-2}(\s^{n-2})}{1-s}\frac{\|D^2\varphi\|_{C^0}}{\inf_{\{t_1\leq\varphi\leq t_2\}}|\nabla\varphi|}\,\rho^{1-s},
\end{split}
\end{equation*}
which is uniform in $\delta$ and does not depend on $x$.\\
Reasoning in a similar way yields the same inequality for
\begin{equation*}
\int_{B_\rho(x)\setminus B_\delta(x)}\frac{\chi_{\{y|\varphi(y)<\varphi(x),\nabla\varphi(x)\cdot (y-x)>0\}}(y)}
{\kers}dy,
\end{equation*}
and hence we can estimate
\begin{equation}
|\I_s^\rho[E_{\varphi(x)}](x)-\I_s[E_{\varphi(x)}](x)|
\leq4\frac{\h^{n-2}(\s^{n-2})}{1-s}\frac{\|D^2\varphi\|_{C^0}}{\inf_{\{t_1\leq\varphi\leq t_2\}}|\nabla\varphi|}\,\rho^{1-s},
\end{equation}
which tends to 0 uniformly in $x\in\{t_1\leq\varphi\leq t_2\}$ as $\rho\to0$.

We have just proved the following

\begin{lem}
Let $\varphi\in C_c^2(\R)$ s.t. $\nabla\varphi\not=0$ in $\{t_1\leq\varphi\leq t_2\}$, for some $0<t_1<t_2$.
Then the limit
\begin{equation}
\I^\rho_s[E_{\varphi(x)}](x)\xrightarrow{\rho\to0^+}\I_s[E_{\varphi(x)}](x)
\end{equation}
is uniform in $x\in\{t_1\leq \varphi\leq t_2\}$.\\
In particular
\begin{equation}\label{L1_curv_conv}
\int_S\I_s^\rho[E_{\varphi(x)}](x)\,dx\xrightarrow{\rho\to0^+}
\int_S\I_s[E_{\varphi(x)}](x)\,dx,
\end{equation}
for every $S\subset\{t_1\leq\varphi\leq t_2\}$.
\begin{proof}
We have proved the first assertion above. In more fancy language this means that the functions
\begin{equation*}
\I_s^\rho[E_{\varphi(-)}](-):\{t_1\leq\varphi\leq t_2\}\longrightarrow\mathbb{R}
\end{equation*}
converge in $L^\infty$ to $\I_s[E_{\varphi(-)}](-)$, and hence, since $\{t_1\leq\varphi\leq t_2\}$ is bounded,
we have also the $L^1$ convergence.

\end{proof}
\end{lem}

We can now relate the difference between the $s$-perimeter of the superlevel set $E_{t_1}$ and that of $E_{t_2}$,
with the $s$-fractional mean curvature
\begin{prop}
Let $\varphi\in C_c^2(\R)$ s.t. $\nabla\varphi\not=0$ in $\{t_1\leq\varphi\leq t_2\}$, for some $0<t_1<t_2$.
Then
\begin{equation}\label{fracurveq}
P_s(E_{t_1})=P_s(E_{t_2})-\int_{\{t_1<\varphi<t_2\}}\I_s[E_{\varphi(x)}](x)\,dx.
\end{equation}
Moreover for every $W\subset\R$ s.t. $\{\varphi\geq t_2\}\subset W\subset\{\varphi\geq t_1\}$
\begin{equation}\label{fracurvineq}
P_s(W)\geq P_s(E_{t_2})-\int_{W\setminus E_{t_2}}\I_s[E_{\varphi(x)}](x)\,dx.
\end{equation}
\begin{proof}
We remark that with our hypothesis $\I_s[E_{\varphi(x)}](x)$ is well defined for every $x\in\{t_1\leq\varphi\leq t_2\}$ and
Proposition $\ref{continuity_curv_prop}$ implies that the function $\I_s[E_{\varphi(-)}](-)$ is continuous there.\\
On the other hand recall that $\I_s^\rho[E_{\varphi(x)}](x)$ is defined for every $x\in\R$.

Let $A\subset\R$ be a bounded set. Then
\begin{equation*}\begin{split}
\int_A&\I_s^\rho[E_{\varphi(x)}](x)\,dx
=\int_A\Big(\int_{\R\setminus B_\rho(x)}\frac{\chi_{\{\varphi\geq\varphi(x)\}}(y)-\chi_{\{\varphi<\varphi(x)\}}(y)}{\kers}dy\Big)dx\\
&
=\int_{\R\times\R}\chi_A(x)\big(\chi_{\{\varphi\geq\varphi(x)\}}(y)-\chi_{\{\varphi<\varphi(x)\}}(y)\big)
\frac{\chi_{\Co(0,\rho)}(|x-y|)}{\kers}dx\,dy\\
&
=\frac{1}{2}\int_{\R\times\R}\big(\chi_A(x)-\chi_A(y)\big)\big(\chi_{\{\varphi\geq\varphi(x)\}}(y)-\chi_{\{\varphi<\varphi(x)\}}(y)\big)\\
&
\qquad\qquad\qquad\qquad\qquad\cdot
\frac{\chi_{\Co(0,\rho)}(|x-y|)}{\kers}dx\,dy.
\end{split}
\end{equation*}
As for the last equality, simply notice that
\begin{equation*}
\chi_{\{\varphi\geq\varphi(x)\}}(y)-\chi_{\{\varphi<\varphi(x)\}}(y)
=-\big(\chi_{\{\varphi\geq\varphi(y)\}}(x)-\chi_{\{\varphi<\varphi(y)\}}(x)\big),
\end{equation*}
for every $(x,y)\in\R\times\R$ s.t. $\varphi(x)\not=\varphi(y)$, and then just exchange $x$ and $y$ in the integral.

For a general bounded set $A$ this gives
\begin{equation*}
-\int_A\I_s^\rho[E_{\varphi(x)}](x)\,dx\leq\frac{1}{2}
\int_{\R\times\R}|\chi_A(x)-\chi_A(y)|\frac{\chi_{\Co(0,\rho)}(|x-y|)}{\kers}dx\,dy,
\end{equation*}
while if we take $A=E_t$ for some $t>0$, we have
\begin{equation*}
\big(\chi_{\{\varphi\geq t\}}(x)-\chi_{\{\varphi\geq t\}}(y)\big)\big(\chi_{\{\varphi\geq\varphi(x)\}}(y)-\chi_{\{\varphi<\varphi(x)\}}(y)\big)
=-|\chi_{E_t}(x)-\chi_{E_t}(y)|,
\end{equation*}
and hence
\begin{equation}
-\int_{E_t}\I_s^\rho[E_{\varphi(x)}](x)\,dx
=\frac{1}{2}
\int_{\R\times\R}|\chi_{E_t}(x)-\chi_{E_t}(y)|\frac{\chi_{\Co(0,\rho)}(|x-y|)}{\kers}dx\,dy.
\end{equation}
Now let $\{\varphi\geq t_2\}\subset W\subset\{\varphi\geq t_1\}$. If $P_s(W)=\infty$, then
$(\ref{fracurvineq})$ is clear.\\
On the other hand, if $P_s(W)<\infty$, then Lebesgue's dominated convergence Theorem implies that
\begin{equation*}
\lim_{\rho\to0}\frac{1}{2}
\int_{\R\times\R}|\chi_W(x)-\chi_W(y)|\frac{\chi_{\Co(0,\rho)}(|x-y|)}{\kers}dx\,dy
=P_s(W).
\end{equation*}
Therefore
\begin{equation*}\begin{split}
\frac{1}{2}\int_{\R\times\R}&|\chi_W(x)-\chi_W(y)|\frac{\chi_{\Co(0,\rho)}(|x-y|)}{\kers}dx\,dy
\geq-\int_W\I_s^\rho[E_{\varphi(x)}](x)\,dx\\
&
=-\int_{W\setminus E_{t_2}}\I_s^\rho[E_{\varphi(x)}](x)\,dx
-\int_{E_{t_2}}\I_s^\rho[E_{\varphi(x)}](x)\,dx\\
&
=-\int_{W\setminus E_{t_2}}\I_s^\rho[E_{\varphi(x)}](x)\,dx\\
&
\qquad\qquad
+\frac{1}{2}
\int_{\R\times\R}|\chi_{E_{t_2}}(x)-\chi_{E_{t_2}}(y)|\frac{\chi_{\Co(0,\rho)}(|x-y|)}{\kers}dx\,dy.
\end{split}\end{equation*}
Since $W\setminus E_{t_2}\subset\{t_1\leq\varphi\leq t_2\}$, previous Lemma guarantees that
\begin{equation*}
\lim_{\rho\to0}-\int_{W\setminus E_{t_2}}\I_s^\rho[E_{\varphi(x)}](x)\,dx
=-\int_{W\setminus E_{t_2}}\I_s[E_{\varphi(x)}](x)\,dx,
\end{equation*}
and hence passing to the limit $\rho\to0$ proves $(\ref{fracurvineq})$ (notice that $E_{t_2}$ is a bounded set with $C^2$ boundary and hence $P_s(E_{t_2})<\infty$).\\
To get $(\ref{fracurveq})$ simply take $W=E_{t_1}$, so that
\begin{equation*}\begin{split}
\frac{1}{2}\int_{\R\times\R}&|\chi_{E_{t_1}}(x)-\chi_{E_{t_1}}(y)|\frac{\chi_{\Co(0,\rho)}(|x-y|)}{\kers}dx\,dy
=-\int_{E_{t_1}}\I_s^\rho[E_{\varphi(x)}](x)\,dx\\
&
=-\int_{\{t_1<\varphi<t_2\}}\I_s^\rho[E_{\varphi(x)}](x)\,dx\\
&
\qquad\qquad
+\frac{1}{2}
\int_{\R\times\R}|\chi_{E_{t_2}}(x)-\chi_{E_{t_2}}(y)|\frac{\chi_{\Co(0,\rho)}(|x-y|)}{\kers}dx\,dy,
\end{split}\end{equation*}
and pass to the limit $\rho\to0$.

\end{proof}
\end{prop}

As shown in \cite{CMP} in the broader context of generalized perimeters and curvatures, previous Proposition is
enough (actually equivalent) to show that the fractional mean curvature $\I_s$ is the first variation of the fractional perimeter $P_s$, i.e. to prove $(\ref{first_variation})$.

\begin{teo}[First Variation of the Perimeter]
Let $E$ be a bounded open set with $C^2$ boundary and let $\Phi_t:\R\longrightarrow\R$ be a one-parameter family of diffeomorphisms of class
$C^2$ both in $x$ and in $t$, with $\Phi_0=$Id. Then
\begin{equation}
\frac{d}{dt}P_s(\Phi_t(E))\Big|_{t=0}=-\int_{\partial E}\I_s[E](x)\nu_E(x)\cdot\phi(x)\,d\Han(x),
\end{equation}
where $\phi(x):=\frac{\partial}{\partial t}\Phi_t(x)\big|_{t=0}$ and $\nu_E(x)$ is the outer unit normal at $\partial E$ in $x$.
\begin{proof}

We can write $E=\{\varphi\geq\frac{1}{2}\}$ for some $\varphi\in C^2_c(\R)$ with $\nabla\varphi\not=0$ in $\{\frac{1}{8}\leq\varphi\leq\frac{7}{8}\}$ (see Appendix C).
Moreover notice that, since $\Phi_t$ is $C^2$ in $t$ and $\Phi_0=$Id, for $|t|$ small we have
\begin{equation}\label{symmdifffirstvar}
\Phi_t(E)\Delta E\subset N_{M|t|}(\partial E),
\end{equation}
the $(M|t|)$-neighborhood of $\partial E$, for some $M>0$.

Therefore, for $|t|$ sufficiently small, we can construct a $C^2$ diffeomorphism $\tilde{\Phi}_t$ s.t.
$\tilde{\Phi}_t=\textrm{Id}$ outside $N_{2M|t|}(\partial E),$
and in particular out of $\{\frac{1}{4}\leq\varphi\leq\frac{3}{4}\}$,
$\tilde{\Phi}_t(E)=\Phi_t(E)$ and $\|\tilde{\Phi}_t-\textrm{Id}\|_{C^2}\longrightarrow0$ as $t\to0$.\\
In particular we have $P_s(\tilde{\Phi}_t(E))=P_s(\Phi_t(E))$.
Moreover, since
\begin{equation*}
\Phi_t(E)=\tilde{\Phi}_t(E)=\Big\{\varphi\circ\tilde{\Phi}_t^{-1}\geq\frac{1}{2}\Big\},
\end{equation*}
and
\begin{equation*}
\Big\{\varphi\circ\tilde{\Phi}_t^{-1}\geq\frac{7}{8}\Big\}=\Big\{\varphi\geq\frac{7}{8}\Big\},
\end{equation*}
using $(\ref{fracurveq})$ we find, for $|t|$ small enough,
\begin{equation*}\begin{split}
P_s&(\Phi_t(E))-P_s(\{\varphi\geq7/8\})=P_s(\{\varphi\circ\tilde{\Phi}_t^{-1}\geq1/2\})-P_s(\{\varphi\circ\tilde{\Phi}_t^{-1}\geq7/8\})\\
&
=-\int_{\{1/2<\varphi\circ\tilde{\Phi}_t^{-1}<7/8\}}\I_s\big[E_{\varphi\circ\tilde{\Phi}_t^{-1}(x)}\big](x)\,dx\\
&
=-\int_{\Phi_t(E)\setminus\{\varphi\geq7/8\}}\I_s\big[E_{\varphi\circ\tilde{\Phi}_t^{-1}(x)}\big](x)\,dx\\
&
=-\int_{N_{2M|t|}(\partial E)\cap\Phi_t(E)}
\Big(\I_s\big[E_{\varphi\circ\tilde{\Phi}_t^{-1}(x)}\big](x)-\I_s[E_{\varphi(x)}](x)\Big)dx\\
&
\qquad\qquad\qquad\qquad
-\int_{\Phi_t(E)\setminus\{\varphi\geq7/8\}}\I_s[E_{\varphi(x)}](x)\,dx.
\end{split}\end{equation*}
Now
\begin{equation*}\begin{split}
\Big|\int_{N_{2M|t|}(\partial E)\cap\Phi_t(E)}&
\Big(\I_s\big[E_{\varphi\circ\tilde{\Phi}_t^{-1}(x)}\big](x)-\I_s[E_{\varphi(x)}](x)\Big)dx\Big|\\
&
\leq\int_{N_{2M|t|}(\partial E)}
\Big|\I_s\big[E_{\varphi\circ\tilde{\Phi}_t^{-1}(x)}\big](x)-\I_s[E_{\varphi(x)}](x)\Big|dx,
\end{split}\end{equation*}
and by Proposition $\ref{continuity_curv_prop}$
\begin{equation*}
\|\I_s\big[E_{\varphi\circ\tilde{\Phi}_t^{-1}(x)}\big](x)-\I_s[E_{\varphi(x)}](x)\|_{L^\infty(N_{2M|t|}(\partial E))}
\xrightarrow{t\to0}0,
\end{equation*}
so that
\begin{equation*}
\int_{N_{2M|t|}(\partial E)\cap\Phi_t(E)}
\Big(\I_s\big[E_{\varphi\circ\tilde{\Phi}_t^{-1}(x)}\big](x)-\I_s[E_{\varphi(x)}](x)\Big)dx=o(t),
\end{equation*}
as $t\to0$. Therefore for $|t|$ small we get
\begin{equation*}
P_s(\Phi_t(E))=P_s(\{\varphi\geq7/8\})-\int_{\Phi_t(E)\setminus\{\varphi\geq7/8\}}\I_s[E_{\varphi(x)}](x)\,dx+o(t),
\end{equation*}
and hence
\begin{equation*}\begin{split}
\frac{d}{dt}P_s(\Phi_t(E))\Big|_{t=0}&=-\frac{d}{dt}\Big(\int_{\Phi_t(E)\setminus\{\varphi\geq7/8\}}\I_s[E_{\varphi(x)}](x)\,dx\Big)\Big|_{t=0}\\
&
=-\int_{\partial E}\I_s[E](x)\nu_E(x)\cdot\phi(x)\,d\Han(x),
\end{split}
\end{equation*}
where the last equality is classical, see for example Proposition 17.8 of \cite{Maggi}.

We give a sketch of the proof.

We know that $\I_s[E_{\varphi(-)}](-)$ is a continuous function in $\{1/8\leq\varphi\leq7/8\}$. We can find 
a continuous function $g\in C^0(\R)$ s.t. $g(x)=\I_s[E_{\varphi(x)}](x)$ for every $x\in\{1/4\leq\varphi\leq3/4\}$.
Then, since we have $(\ref{symmdifffirstvar})$ and $\partial E=\{\varphi=1/2\}$, for $|t|$ small enough we obtain
\begin{equation*}\begin{split}
\int_{\Phi_t(E)\setminus\{\varphi\geq7/8\}}&\I_s[E_{\varphi(x)}](x)\,dx
-\int_{E\setminus\{\varphi\geq7/8\}}\I_s[E_{\varphi(x)}](x)\,dx\\
&
=\int_{\Phi_t(E)}g(x)\,dx-\int_Eg(x)\,dx.
\end{split}\end{equation*}

We suppose for simplicity that $g\in C^1(\R)$; in general we would need to approximate $g$ with $C^1$ functions and then show that we can pass to the limit.\\
Let $\Omega\subset\R$ be a bounded open set s.t. $E\subset\subset\Omega$. Then we can write
\begin{equation*}
\Phi_t(x)=x+t\phi(x)+O(t^2),\qquad D_x\Phi_t(x)=\textrm{I}_n+t D\phi(x)+O(t^2),
\end{equation*}
as $t\to0$, uniformly in $x\in\Omega$. As a consequence it can be shown that
\begin{equation*}
|\det D_x\Phi_t(x)|=1+t\textrm{ div }\phi(x)+O(t^2),
\end{equation*}
uniformly in $x\in\Omega$, as $t\to0$.\\
Then changing variables and using the divergence Theorem we find
\begin{equation*}\begin{split}
\int_{\Phi_t(E)}&g(x)\,dx-\int_Eg(x)\,dx\\
&
=\int_E\big(g(x+t\phi(x)+O(t^2))|\det D_x\Phi_t(x)|-g(x)\big)dx\\
&
=\int_E\Big( \big(g(x)+t\nabla g(x)\cdot\phi(x)\big)\big(1+t\textrm{ div }\phi(x)\big)-g(x)+O(t^2)\Big)dx\\
&
=t\int_E\textrm{div}(g(x)\phi(x))dx+ O(t^2)\\
&
=t\int_{\partial E}g(x)\nu_E(x)\cdot\phi(x)\,d\Han(x)+O(t^2).
\end{split}
\end{equation*}

\end{proof}
\end{teo}

\end{section}

\end{chapter}

%%%%%%%%%%%%%%%%%%%%%%%%%%%%%%%%%%%%%%%%
%%%%%%%%%%%%%%%%%%%%%%%%%%%%%%%%%%%%%%%%
%%%%%%%%%%%%%%%%%%%%%%%%%%%%%%%%%%%%%%%%

%%%%%%%%%%%%%%%%%%%%%%%%%%%%%%%%%%%%%%%%
%%%%%%%%%%%%%%%%%%%%%%%%%%%%%%%%%%%%%%%%
%%%%%%%%%%%%%%%%%%%%%%%%%%%%%%%%%%%%%%%%

%%%%%%%%%%%%%%%%%%%%%%%%%%%%%%%%%%%%%%%%
%%%%%%%%%%%%%%%%%%%%%%%%%%%%%%%%%%%%%%%%
%%%%%%%%%%%%%%%%%%%%%%%%%%%%%%%%%%%%%%%%

%%%%%%%%%%%%%%%%%%%%%%%%%%%%%%%%%%%%%%%%
%%%%%%%%%%%%%%%%%%%%%%%%%%%%%%%%%%%%%%%%
%%%%%%%%%%%%%%%%%%%%%%%%%%%%%%%%%%%%%%%%

%%%%%%%%%%%%%%%%%%%%%%%%%%%%%%%%%%%%%%%%
%%%%%%%%%%%%%%%%%%%%%%%%%%%%%%%%%%%%%%%%
%%%%%%%%%%%%%%%%%%%%%%%%%%%%%%%%%%%%%%%%

%%%%%%%%%%%%%%%%%%%%%%%%%%%%%%%%%%%%%%%%
%%%%%%%%%%%%%%%%%%%%%%%%%%%%%%%%%%%%%%%%
%%%%%%%%%%%%%%%%%%%%%%%%%%%%%%%%%%%%%%%%

%%%%%%%%%%%%%%%%%%%%%%%%%%%%%%%%%%%%%%%%
%%%%%%%%%%%%%%%%%%%%%%%%%%%%%%%%%%%%%%%%
%%%%%%%%%%%%%%%%%%%%%%%%%%%%%%%%%%%%%%%%

%%%%%%%%%%%%%%%%%%%%%%%%%%%%%%%%%%%%%%%%
%%%%%%%%%%%%%%%%%%%%%%%%%%%%%%%%%%%%%%%%
%%%%%%%%%%%%%%%%%%%%%%%%%%%%%%%%%%%%%%%%

%%%%%%%%%%%%%%%%%%%%%%%%%%%%%%%%%%%%%%%%
%%%%%%%%%%%%%%%%%%%%%%%%%%%%%%%%%%%%%%%%
%%%%%%%%%%%%%%%%%%%%%%%%%%%%%%%%%%%%%%%%

\begin{chapter}{Regularity of Nonlocal Minimal Surfaces}

\begin{rmk}
Again, in this chapter we suppose that every set satisfies $(\ref{gmt_assumption_eq})$.
\end{rmk}

\begin{section}{Flatness Improvement}

In this section we exploit an improvement of flatness technique, similar to the one used in the classical case (see Chapter 1), in order to show that the boundary of an $s$-minimal set is a $C^{1,\gamma}$ graph in a neighborhood of every point which has an interior tangent ball.

The main result is the following Theorem, from which we can easily obtain our $C^{1,\gamma}$ regularity, see Theorem
$\ref{flat_reg_teo1}$ below.

Roughly speaking, the idea consists in showing that if $\partial E$ is contained in some cylinder, in a neighborhood of $x_0\in\partial E$, then in a smaller neighborhood it is actually contained in a flatter cylinder, up to a change of coordinates.

\begin{teo}[Improvement of Flatness]\label{imp_flat_teo1}
Let $s\in(0,1)$ and fix $\alpha\in(0,s)$.
There exists $k_0=k_0(n,s,\alpha)
%\in\mathbb{N}
$
s.t. the following result holds.\\
Let $E\subset\R$ be $s$-minimal in $B_1$, with $0\in\partial E$, and assume that
\begin{equation*}
\partial E\cap B_{2^{-i}}\subset\{|x\cdot\nu_i|\leq 2^{-i(\alpha+1)}\},
\end{equation*}
for every $i\in\{0,\dots,k_0\}$, for some $\nu_i\in\mathbb{S}^{n-1}$.\\
Then there exist vectors $\nu_i\in\mathbb{S}^{n-1}$ for every $i>k_0$ s.t. the above inclusion remains valid, i.e.
\begin{equation*}
\partial E\cap B_{2^{-i}}\subset\{|x\cdot\nu_i|\leq 2^{-i(\alpha+1)}\},
\end{equation*}
for every $i\in\mathbb{N}$.
\end{teo}

If we dilate everything by a factor $2^k$, we get
\begin{equation*}
\partial (2^kE)\cap B_{2^{k-i}}\subset\{|x\cdot\nu_i|\leq 2^k2^{-i(\alpha+1)}\}=\{|x\cdot\nu_i|\leq 2^{-\alpha k}2^{(k-i)(\alpha+1)}\}.
\end{equation*}
Also notice that we can start with a set $E$ which is $s$-minimal in $B_2$, rather than only $B_1$, and this guarantees that $2^kE$ is $s$-minimal in $B_{2^{k+2}}$, so if we slightly translate $E$ we still have an $s$-minimal set in $B_{2^{k+1}}$.

Let $a_k:=2^{-\alpha k}$. The situation can then be reduced to the following.\\

CLAIM:$\quad$There exists a universal $k_0\in\mathbb{N}$ s.t. if $F\subset\R$ is $s$-minimal in $B_{2^{k+2}}$, with $k\geq k_0$,
and
\begin{equation*}
\partial F\cap B_{2^j}\subset\{|x\cdot\nu'_j|\leq a_k2^{j(\alpha+1)}\},\quad\textrm{for every } j\in\{0,\dots,k\},
\end{equation*}
then
there exists $\nu'_{-1}$ s.t.
\begin{equation*}
\partial F\cap B_{1/2}\subset\{|x\cdot\nu_{-1}'|\leq a_k2^{-1-\alpha}\}.
\end{equation*}

Notice that up to a rotation we can always suppose that $\nu'_0=e_n$.\\
We define the flatness of a cylinder to be the ratio between its height and the diameter of the base.\\
Roughly speaking, requiring flatness of $\partial E\cap B_1$ of order $a_k$,
but also flatness of order $a_k2^{i\alpha}$ for all diadic balls $B_{2^i}$ from $B_1$ to $B_{2^k}$, i.e.
until flatness becomes of order one, gives flatness of order $a_k2^{-\alpha}$ in $B_{1/2}$.

 If we manage to prove this, then scaling back and forth we get Theorem $\ref{imp_flat_teo1}$ by induction on $k\geq k_0$.

The proof is by contradiction. Suppose that for every $k$ there is a set $E_k\subset\R$ which is $s$-minimal in $B_{2^{k+2}}$, s.t. $0\in\partial E_k$
and
\begin{equation*}
\partial E_k\cap B_{2^j}\subset\{|x\cdot\nu^k_j|\leq a_k2^{j(\alpha+1)}\},\quad\textrm{for every } j\in\{0,\dots,k\},
\end{equation*}
for some $\nu^k_j\in\mathbb{S}^{n-1}$, with $\nu^k_0=e_n$, but s.t.
$\partial E_k\cap B_{1/2}$ cannot fit in any cylinder of flatness $a_k2^{-\alpha}$.

Then we show that the rescaled sets
\begin{equation*}
\partial E^*_k:=\Big\{\big(x',\frac{x_n}{a_k}\big)\,\big|\,(x',x_n)\in\partial E_k\Big\}
\end{equation*}
converge
(up to a subsequence) to a plane $P=\{x\cdot\nu=0\}$, uniformly on compact sets, reaching a contradiction.

To do this we first show that there is a limiting set $P$, which is the graph of an Holder function $u$. Then we control the growth of
$u$ at infinity and we show that it is a $\frac{s+1}{2}$-harmonic function. This will imply that it is actually linear, concluding the proof.\\

One of the main tools is the following geometric Harnack-type inequality.

\begin{lem}
Let $s\in(0,1)$ and $\alpha\in(0,s)$. There exist $k_0\in\mathbb{N}$ and $\delta\in(0,1)$ which only depend on $n,\,s$ and $\alpha$,
for which the following result holds.\\
Let $k\geq k_0$ and let $a:=2^{-k\alpha}$.\\
Let $E\subset\R$ be $s$-minimal in $B_{2^{k+2}}$ and assume that
\begin{equation}
\partial E\cap B_1\subset\{|x_n|\leq a\}
\end{equation}
and, for every $i\in\{0,\dots,k\}$,
\begin{equation}
\partial E\cap B_{2^i}\subset\{|x\cdot\nu_i|\leq a2^{i(1+\alpha)}\},
\end{equation}
for some $\nu_i\in\mathbb{S}^{n-1}$. Then
\begin{equation}\begin{split}\label{harnack_inclusions}
&\textrm{either}\quad\partial E\cap B_\delta\subset\{x_n\leq a(1-\delta^2)\},\\
&
\textrm{or}\quad\partial E\cap B_\delta\subset\{x_n\geq a(-1+\delta^2)\}.
\end{split}
\end{equation}

\begin{proof}
Given $y\in\partial E\cap B_{1/2}$ we have for every $i\in\{0,\dots,k-1\}$
\begin{equation*}
\partial E\cap B_{2^i}(y)\subset\partial E\cap B_{2^1+\frac{1}{2}}\subset\partial E\cap B_{2^{i+1}}
\subset\big\{|x\cdot\nu_{i+1}|\leq a2^{(i+1)(\alpha+1)}\big\},
\end{equation*}
and also
\begin{equation*}
|y\cdot\nu_{i+1}|\leq a2^{(i+1)(\alpha+1)}.
\end{equation*}
Thus
\begin{equation*}
\partial E\cap B_{2^i}(y)\subset\big\{|(x-y)\cdot\nu_{i+1}|\leq 2a2^{(i+1)(\alpha+1)}\big\}.
\end{equation*}

This provides some cancellation for the integral of
the contribution coming from $\Co B_{1/2}(y)$ to the $s$-fractional
mean curvature of $E$ in $y$, yielding
\begin{equation*}\begin{split}
\big|\I_s^\frac{1}{2}[E](y)\big|&=\Big|\int_{\Co B_\frac{1}{2}(y)}\frac{\chi_E(x)-\chi_{\Co E}(x)}{\kers}dx\Big|\\
&
\leq\Big|\int_{B_{2^{k-1}}(y)\setminus B_\frac{1}{2}(y)}\frac{\chi_E(x)-\chi_{\Co E}(x)}{\kers}dx\Big|
+\big|\I_s^{2^{k-1}}[E](y)\big|\\
&
\leq\sum_{i=0}^{k-1}\Big|\int_{B_{2^i}(y)\setminus B_{2^{i-1}}(y)}\frac{\chi_E(x)-\chi_{\Co E}(x)}{\kers}dx\Big|
+\frac{n\omega_n}{s}2^s2^{-ks}\\
&
\leq C\Big\{\sum_{i=0}^{k-1}\int_{B_{2^i}(y)\setminus B_{2^{i-1}}(y)}\frac{\chi_{\{|(x-y)\cdot\nu_{i+1}|\leq 2a2^{(i+1)(\alpha+1)}\}}(x)}{\kers}dx+2^{-ks}\Big\}\\
&
\leq C\Big\{\sum_{i=0}^{k-1}\int_{2^{i-1}}^{2^i}\frac{a2^{(i+1)(\alpha+1)}r^{n-2}}{r^{n+s}}dr
+2^{-ks}\Big\}\\
&
\leq C_1 a,
\end{split}
\end{equation*}
since $\alpha<s$, for some $C_1=C_1(n,s,\alpha)>0$.

Since by hypothesis $\partial E\cap B_1\subset\{|x_n|\leq a\}$, we can assume that
\begin{equation*}
\{x_n<-a\}\cap B_1\subset E.
\end{equation*}
We also assume that $E$ contains
more than half of the measure of
the cylinder
\begin{equation*}
D:=\{|x'|\leq\delta\}\times\{|x_n|\leq a\},
\end{equation*}
i.e. that
\begin{equation}\label{E_cyl_meas_ass}
|E\cap D|\geq\frac{1}{2}|D|=\omega_{n-1}\delta^{n-1}a.
\end{equation}
Then we show that
\begin{equation}\label{Har_claim_eq}
\{x_n< a(-1+\delta^2)\}\cap B_\delta\subset E,
\end{equation}
which implies
\begin{equation*}
\partial E\cap B_\delta\subset\{x_n\geq a(-1+\delta^2)\}.
\end{equation*}
%The symmetric statement holds taking $\Co E$ in place of $E$.

Suppose that $(\ref{Har_claim_eq})$ doesn't hold.
Then there is a portion of $\partial E\cap B_\delta$ trapped in the strip
$\{-a\leq x_n\leq (-1+\delta^2)a\}$.\\
Now we slide the plane $x_n=t$ upwards, starting from $t=-a$ until we first touch $\partial E$. Let $y\in\partial E
\cap B_\delta$ be a contact point; then
\begin{equation*}
|y'|\leq\delta\quad\textrm{and}\quad-a\leq y_n\leq(-1+\delta^2)a.
\end{equation*}
Since
\begin{equation*}
\{x_n<y_n\}\cap B_\delta\subset E,
\end{equation*}
we can touch $\partial E$ in $y$ with an interior tangent paraboloid of opening $-\frac{a}{2}$.\\
To be more precise, let $P$ be the (interior of the) subgraph of the paraboloid
\begin{equation*}
x_n=-\frac{a}{2}|x'-y'|^2+y_n.
\end{equation*}
Then
\begin{equation*}
P\cap B_\delta\subset\{x_n<y_n\}\cap B_\delta\subset E
\end{equation*}
and
\begin{equation*}
(\partial P\cap\partial E)\cap B_\delta=\{y\}.
\end{equation*}
In particular from Corollary $\ref{Euler_Lag_ball_eq}$ we have
\begin{equation}\label{Eu_La_eq_flat}
\limsup_{\rho\to0}\I_s^\rho[E](y)\leq0.
\end{equation}

On the other hand
\begin{equation*}\begin{split}
P.V.\int_{B_\frac{1}{2}(y)}&\frac{\chi_E(x)-\chi_{\Co E}(x)}{\kers}dx\\
&
=
P.V.\int_{B_\frac{1}{2}(y)}\frac{\chi_P(x)-\chi_{\Co P}(x)}{\kers}dx
+2P.V.\int_{B_\frac{1}{2}(y)}\frac{\chi_{E\setminus P}(x)}{\kers}dx\\
&
=:I_1+I_2,
\end{split}
\end{equation*} 
and it is readily seen that
\begin{equation*}
I_1\geq -C_2\,a,
\end{equation*}
for some $C_2=C_2(n,s)>0$
(see the calculations in Section 4.1 and also Lemma $\ref{explicit_curv_formula}$).
Moreover, since $\delta$ and $a$ are very small, we have $D\subset B_{1/2}(y)$. Also, taking $k_0$ big enough,
we assume that $a<\delta$.\\
Using $(\ref{Har_claim_eq})$ we can estimate
\begin{equation*}\begin{split}
|(E\setminus P)\cap D|&=|E\cap D|-|P\cap D|
\geq\frac{1}{2}|D|-|B'_\delta\times\{-a\leq x_n\leq(-1+\delta^2)a\}|\\
&
=\omega_{n-1}\delta^{n-1}a-\omega_{n-1}\delta^{n-1}\delta^2 a\geq\frac{1}{2}\omega_{n-1}\delta^{n-1}a
=\frac{1}{4}|D|,
\end{split}
\end{equation*}
provided that $\delta^2<\frac{1}{2}$. Now we have
\begin{equation*}
I_2\geq2\int_D\frac{\chi_{E\setminus P}(x)}{\kers}dx
\geq2\int_D\frac{\chi_{E\setminus P}(x)}{(4\delta)^{n+s}}dx
\geq C_3\delta^{-1-s}a,
\end{equation*}
for some $C_3=C_3(n,s)>0$, where we have estimated $|x-y|\leq|x|+|y|<4\delta$, since both $x$ and $y$ belong to $D\subset B_{2\delta}$.\\
Putting the three estimates together we obtain
\begin{equation*}
\liminf_{\rho\to0}\I_s^\rho[E](y)\geq (-C_1-C_2+C_3\delta^{-1-s})a>0,
\end{equation*}
once we choose $\delta$ small enough. But this contradicts $(\ref{Eu_La_eq_flat})$, concluding the proof.

Notice that if $E$ doesn't satisfy $(\ref{E_cyl_meas_ass})$, then $\Co E$ does, and arguing as above with $\Co E$ in place of $E$
yields
\begin{equation*}
\partial E\cap B_\delta\subset\{x_n\leq a(1-\delta^2)\}.
\end{equation*}

\end{proof}
\end{lem}

In any case this provides flatness of order $a(1-\delta^2/2)/\delta$ for
$\partial E\cap B_\delta$.\\
Indeed, suppose e.g. that
the second inclusion in $(\ref{harnack_inclusions})$ is satisfied. Then
\begin{equation}\label{Harnack_ind_eq0}
\partial E\cap B_\delta\subset\{(-1+\delta^2)a\leq x_n\leq a\},
\end{equation}
which is a cylinder with base diameter $2\delta$ and height $(2-\delta^2)a$.

Now we want to apply Harnack inequality again.\\
Suppose we have $(\ref{Harnack_ind_eq0})$ with $k\gg k_0$ and let $t:=\frac{\delta^2}{2}\,a$. If we translate $E$ downwards by $t$, then
\begin{equation*}
\partial(E-t\,e_n)\cap B_{\delta/2}
\subset\Big\{a\Big(-1+\frac{\delta^2}{2}\Big)\leq x_n\leq a\Big(1-\frac{\delta^2}{2}\Big)\Big\},
\end{equation*}
hence if we dilate by a factor $2/\delta$, we get
\begin{equation*}
\partial\Big(\frac{2}{\delta}(E-te_n)\Big)\cap B_1
\subset\Big\{\frac{-2+\delta^2}{\delta}\,a\leq x_n\leq\frac{2-\delta^2}{\delta}\,a\Big\}.
\end{equation*}
Notice that
\begin{equation*}
\frac{2-\delta^2}{\delta}\,a> a,
\end{equation*}
and let
\begin{equation}\label{K_har}
k':=\max\Big\{j\in\mathbb{N}\,|\,2^{-\alpha j}\geq\frac{2-\delta^2}{\delta}\,a\Big\},
\end{equation}
so that
\begin{equation*}
a':=2^{-\alpha k'}> a\quad\Longrightarrow\quad k'<k,
\end{equation*}
and
\begin{equation*}
\partial F\cap B_1\subset\{-a'\leq x_n\leq a'\},
\end{equation*}
where $F:=\frac{2}{\delta}(E-t\,e_n)$.

Notice that we can take $\delta$ of the form $\delta=2^{-M_0}$.\\
Now for $i\geq1$
\begin{equation*}
\partial F\cap B_{2^i}\subset\partial\Big(\frac{2}{\delta}E\Big)\cap B_{2^{i+1}}
=2^{M_0+1}\big(\partial E\cap B_{2^{i-M_0}}\big).
\end{equation*}
If $i\leq M_0$, then $B_{2^{i-M_0}}\subset B_1$ and hence
\begin{equation*}
\partial F\cap B_{2^i}\subset\{|x_n|\leq 2^{M_0+1}a\}.
\end{equation*}
Since for every $i\geq1$
\begin{equation}\label{est_har}
2^{M_0+1}a=\frac{2}{\delta}\,a\leq\frac{2-\delta^2}{\delta}\,a\,2^{1+\alpha}\leq \,2^{1+\alpha}a'\leq
2^{i(1+\alpha)}a',
\end{equation}
we obtain
\begin{equation*}
\partial F\cap B_{2^i}\subset\{|x\cdot\nu_i'|\leq2^{i(1+\alpha)}a'\},\quad\textrm{for }0\leq i\leq M_0,
\end{equation*}
with $\nu_i'=e_n$.

On the other hand, for $M_0<i\leq k$ we get using $(\ref{est_har})$
\begin{equation*}\begin{split}
\partial F\cap B_{2^i}&\subset 2^{M_0+1}\big\{|x\cdot \nu_{M_0-i}|\leq 2^{(M_0-i)(1+\alpha)}a\big\}\\
&
\subset\{|x\cdot\nu_i'|\leq 2^{i(1+\alpha)}a'\},
\end{split}
\end{equation*}
with $\nu'_i=\nu_{M_0-i}$.

Notice that these inclusions hold for $0\leq i\leq k$ and hence in particular for $0\leq i\leq k'$. Therefore, if $k'$ as defined
in $(\ref{K_har})$ is s.t. $k'\geq k_0$, we can apply Harnack inequality to the set $F$ and get
\begin{equation*}\begin{split}
&\textrm{either}\quad\partial F\cap B_\delta\subset\{x_n\leq a'(1-\delta^2)\},\\
&
\textrm{or}\quad\partial F\cap B_\delta\subset\{x_n\geq a'(-1+\delta^2)\}.
\end{split}
\end{equation*}
%Suppose we have the first inclusion.
%Scaling back, we get
%\begin{equation*}
%\partial(E-t\,e_n)\cap B_{\delta^2/2}=\frac{\delta}{2}\big(\partial F\cap B_\delta\big)
%\subset\Big\{-a'\,\frac{\delta}{2}\leq x_n\leq a'\,\frac{\delta(1-\delta^2)}{2}\Big\}.
%\end{equation*}
%Traslating back, we get in the smaller ball $B_{\delta^2/4}$
%\begin{equation*}
%\partial E\cap B_{\delta^2/4}%\subset\partial(E-t\,e_n)\cap B_{\delta^2/2}+t\,e_n
%\subset\Big\{-a'\,\frac{\delta}{2}+\frac{\delta^2}{2}\,a\leq x_n\leq a'\,\frac{\delta(1-\delta^2)}{2}+\frac{\delta^2}{2}\,a\Big\}.
%\end{equation*}
%Since
%\begin{equation*}
%a'\sim \frac{2-\delta^2}{\delta}\,a,
%\end{equation*}
%we have
%\begin{equation*}
%\partial E\cap B_{(\frac{\delta}{2})^2}
%\subset\Big\{(-1+\delta^2)a\leq x_n\leq\frac{(1-\delta^2)^2+1}{2}\,a\Big\}
%\end{equation*}
Since
\begin{equation*}
a'\sim \frac{2-\delta^2}{\delta}\,a,
\end{equation*}
(actually we can take this as an equality, since Harnack inequality would still hold),
scaling and traslating back, we get flatness of order $\big(\frac{2}{\delta}\big)^2\big(1-\frac{\delta^2}{2}\big)^2a$
for $\partial E\cap B_{(\delta/2)^2}$.

Notice that the flatness increases but the height of the cylinder, and hence the oscillation of $\partial E$ in the $e_n$ direction, decreases.

We can repeat the same argument and go on appliying Harnack inequality as long as the hypothesis
are satisfied, that is until the flatness becomes of the order of $a_0:=2^{-k_0\alpha}$.\\
This gives flatness of order $\big(\frac{2}{\delta}\big)^j\big(1-\frac{\delta^2}{2}\big)^ja$
for $\partial E\cap B_{(\delta/2)^j}$, until
\begin{equation}\label{limit_exp_har}
j\sim c_0(\delta)\log\frac{a_0}{a},\qquad\textrm{with }c_0(\delta):=\Big(\log\frac{2}{\delta}\Big(1-\frac{\delta^2}{2}\Big)\Big)^{-1}.
\end{equation}
Notice that if $a\to0$, then $j\to\infty$.

Clearly after slightly dilating and traslating the set $E$, we can repeat the above analysis and get the same estimate for every
$x_0\in\partial E\cap B_{1/2}$, that is, we have flatness of order 
$c\,\big(\frac{2}{\delta}\big)^j\big(1-\frac{\delta^2}{2}\big)^ja$
for $\partial E\cap B_{(\delta/2)^j}(x_0)$, until
$j$ becomes as in $(\ref{limit_exp_har})$. Here $c$ is a small constant appearing as a consequence of the scaling and
does not depend on $E$ nor $a$.\\
%Moreover for every fixed
%$x_0\in\partial E\cap B_{1/2}$, we have
%\begin{equation*}
%x\in\partial E\textrm{ s.t. }|x'-x_0'|\leq \frac{1}{2}\Big(\frac{\delta}{2}\Big)^i\quad\Longrightarrow\quad
%|(x-x_0)\cdot e_n|\leq 2c\Big(1-\frac{\delta^2}{2}\Big)^ia,
%\end{equation*}
%We remark that the radii $r_i$ do not depend on $E$ nor $a$ but only on $\delta$ and $i$.
%for every index $i$ until the treshold index $j$.\\

Now we want to prove the CLAIM, so we consider our sequence of sets $E_k$ as above.
That is,
for every $k$ the set $E_k\subset\R$ is $s$-minimal in $B_{2^{k+2}}$, we have $0\in\partial E_k$,
and
\begin{equation*}
\partial E_k\cap B_{2^j}\subset\{|x\cdot\nu^k_j|\leq a_k2^{j(\alpha+1)}\},\quad\textrm{for every } j\in\{0,\dots,k\},
\end{equation*}
for some $\nu^k_j\in\mathbb{S}^{n-1}$, with $\nu^k_0=e_n$.
Moreover
$\partial E_k\cap B_{1/2}$ cannot fit in any cylinder of flatness $a_k2^{-\alpha}$.

We want to show that the flatness hypothesis on the sets $E_k$ imply that the vectors $\nu_j^k$ cannot oscillate too much
and must remain close to $e_n$.

Consider a set $E_k$ and fix any index $j$. From the two inclusions
\begin{equation*}\begin{split}
\partial E_k\cap B_{2^j}\subset\{|x\cdot\nu^k_j|\leq 2^j2^{(j-k)\alpha}\},\\
\partial E_k\cap B_{2^{j+1}}\subset\{|x\cdot\nu^k_{j+1}|\leq 2^{j+1}2^{(j+1-k)\alpha}\},
\end{split}
\end{equation*}
we deduce that
\begin{equation}
|\nu_j^k-\nu_{j+1}^k|\leq C\,2^{\alpha(j-k)},
\end{equation}
for some constant $C>0$ independent of $k$ and $j$.\\
Therefore we get for every $j\ge1$
\begin{equation}\label{contr_osc_ineq}
|\nu_j^k-e_n|\leq|\nu_j^k-\nu_{j-1}^k|+\dots+|\nu_1^k-e_n|=C\Big(\sum_{i=0}^{j-1}2^{\alpha i}\Big)2^{-\alpha k}.
\end{equation}
In particular, for every fixed $j$ we have
\begin{equation*}
\nu_j^k\xrightarrow{k\to\infty}e_n.
\end{equation*}

Now we stretch our sets in the $e_n$ direction and consider the sets
\begin{equation*}
%\partial E^*_k
\partial E^*_k:=\Big\{\big(x',\frac{x_n}{a_k}\big)\,\big|\,(x',x_n)\in\partial E_k%\cap B_1
\Big\}.
\end{equation*}
%We have $A_k\cap\{|x'|\leq1/2\}\subset\{|x_n|\leq1\}$ for every $k$.

\begin{lem}
There exists a Holder continuous function $u:\mathbb{R}^{n-1}\longrightarrow\mathbb{R}$
and a sequence $k_i\nearrow\infty$ s.t.
if we define
\begin{equation*}
A_\infty:=\{(x',u(x'))\,|\,x'\in\mathbb{R}^{n-1}\},
\end{equation*}
then $\partial E^*_{k_i}\longrightarrow A_\infty$ uniformly on compact sets, in the following sense.
For every fixed $K\subset\R$ compact, for any
$\epsilon>0$,
\begin{equation*}
\partial E^*_{k_i}\cap K\subset N_\epsilon(A_\infty)\cap K,\quad\textrm{for }k_i\geq k(\epsilon).
\end{equation*}
Moreover we have
\begin{equation*}
u(0)=0,\qquad|u(x')|\leq C(1+|x'|^{1+\alpha}).
\end{equation*}
\begin{proof}
We use a diagonal argument to prove the existence of such a function $u$. Then we estimate the growth of $u$ at infinity
using the flatness estimates of the sets $E_k$.

We first
consider the sets
\begin{equation*}
%\partial E^*_k
A_k:=\Big\{\big(x',\frac{x_n}{a_k}\big)\,\big|\,(x',x_n)\in\partial E_k\cap B_1
\Big\},
\end{equation*}
which are contained in $\{|x_n|\leq1\}$,
and
show that there exist a Holder function $u$ and a sequence $\{k_i\}$ s.t. $A_{k_i}\cap\{|x'|\leq1/2\}$ lies in
$N_\epsilon(A_\infty)\cap\{|x'|\leq1/2\}$ for $k_i$ big enough.\\
Suppose that
\begin{equation*}
y_0=(y'_0,y_{0n})\in A_k,\quad\textrm{with}\quad|y'_0|\leq1/2.
\end{equation*}
From the discussion above about Harnack inequality, we know that
\begin{equation}\label{comp_oscil_eq}
A_k\cap\Big\{|y'-y'_0|<\frac{1}{2}\Big(\frac{\delta}{2}\Big)^j\Big\}
\subset\Big\{|y_n-y'_{0n}|<2c\Big(1-\frac{\delta^2}{2}\Big)^j\Big\},
\end{equation}
for every $j$ s.t.
\begin{equation*}
j<j_k\sim c(\delta)\log\frac{a_0}{a_k}.
\end{equation*}

For the moment we fix an index $j_0$ and consider $j\leq j_0$. Notice that for $k$ big enough,
say $k\geq k(j_0)$, inclusion $(\ref{comp_oscil_eq})$
is satisfied.\\
We show that $A_k\cap\{|x'|\leq1/2\}$ is above the graph of
\begin{equation}
\Psi_{y_0,k}(y'):=y_{0n}-2c\Big(1-\frac{\delta^2}{2}\Big)^{j_0}-\theta|y'-y'_0|^\beta,
\end{equation}
where $\theta$ and $\beta\,>0$ depend only on $\delta$.\\
Let $(y',y'_n)\in A_k\cap\{|x'|\leq1/2\}$, so that $|y'-y'_0|\leq1$. Now we distinguish three cases:
\begin{equation*}\begin{split}
&(i)\quad|y'-y'_0|<\frac{1}{2}\Big(\frac{\delta}{2}\Big)^{j_0},\\
&
(ii)\quad \frac{1}{2}\Big(\frac{\delta}{2}\Big)^{j_0}\leq|y'-y'_0|\leq\frac{1}{2},\\
&
(iii)\quad\frac{1}{2}<|y'-y'_0|\leq1.
\end{split}
\end{equation*}
In case $(i)$ our claim follows immediately from $(\ref{comp_oscil_eq})$ with $j=j_0$.\\
In case $(ii)$ we argue as follows.
Notice that in this case there exists $0\leq j\leq j_0$ s.t.
\begin{equation}\label{eq_osc_proof}
\frac{1}{2}\Big(\frac{\delta}{2}\Big)^{j+1}\leq|y'-y'_0|\leq\frac{1}{2}\Big(\frac{\delta}{2}\Big)^j.
\end{equation}
From 
$(\ref{comp_oscil_eq})$
we obtain
\begin{equation*}
2c\Big(1-\frac{\delta^2}{2}\Big)^j\geq|y_n-y_{0n}|.
\end{equation*}
By 
$(\ref{eq_osc_proof})$ and the fact that $0<\delta/2<1$ we find
\begin{equation*}
j\leq\frac{-\log(2|y'-y'_0|)}{\log\frac{2}{\delta}}\leq j+1,
\end{equation*}
and hence
\begin{equation*}\begin{split}
\Big(1-\frac{\delta^2}{2}\Big)^j&\leq
\Big(1-\frac{\delta^2}{2}\Big)^{\big(\frac{-\log(2|y'-y'_0|)}{\log\frac{2}{\delta}}-1\big)}
=\frac{1}{\big(1-\frac{\delta^2}{2}\big)}e^{\beta \log(2|y'-y'_0|)}\\
&
\qquad\qquad=\frac{(2|y'-y'_0|)^\beta}{\big(1-\frac{\delta^2}{2}\big)},
\end{split}\end{equation*}
where
$\beta:=\frac{-\log(1-\frac{\delta^2}{2})}{\log(\frac{2}{\delta})}$.\\
Therefore
\begin{equation*}
|y_n-y_{0n}|\leq\frac{2^{\beta+1}c}{\big(1-\frac{\delta^2}{2}\big)}|y'-y'_0|^\beta,
\end{equation*}
which is the desired result with $\theta:=\frac{2^{\beta+1}c}{(1-\frac{\delta^2}{2})}$.\\
Finally, eventually adding a constant to $\theta$, the result holds also in case $(iii)$.\\
Indeed in this case $|y'-y'_0|^\beta\geq(1/2)^\beta$ and
\begin{equation*}
|y_n-y_{0n}|\leq|y_n|+|y_{0n}|\leq2.
\end{equation*}
So we get the claim provided that $\theta(1/2)^\beta\geq2$.

Notice that, as $y_0$ varies, $\Psi_{y_0,k}$ are Holder continuous functions with Holder modulus of continuity bounded
via the function $\theta t^\beta$. Therefore, if we set
\begin{equation*}
\psi_k(y'):=\sup_{y_0\in A_k\cap\{|x'|\leq1/2\}}\Psi_{y_0,k}(y'),
\end{equation*}
then $\psi_k$ is a Holder continuous function, with Holder modulus of continuity still bounded
via the function $\theta t^\beta$, and $A_k\cap\{|x'|\leq1/2\}$ is above the graph of $\psi_k$.

Arguing in the same way, possibly taking $\theta$ and $\beta$ larger, but still depending only on $\delta$,
we find that, if we define
\begin{equation*}
\Phi_{y_0,k}(y'):=y_{0,n}+2c\Big(1-\frac{\delta^2}{2}\Big)^{j_0}+\theta|y'-y'_0|^\beta,
\end{equation*}
then $A_k\cap\{|x'|\leq1/2\}$ is below the graph of $\Phi_{y_0,k}$. Again, we define
\begin{equation*}
\phi_k(y'):=\inf_{A_k\cap\{|x'|\leq1/2\}}\Phi_{y_0,k}(y'),
\end{equation*}
so that $\phi_k$ is a Holder continuous function, with Holder modulus of continuity bounded via the function $\theta t^\beta$, and
$A_k\cap\{|x'|\leq1/2\}$ is below the graph of $\phi_k$.

Thus $A_k\cap\{|x'|\leq1/2\}$ lies between the graphs of $\psi_k$ and $\phi_k$ for every $k\geq k(j_0)$
and, by construction,
\begin{equation}\label{another_flat_eq}
0\leq\phi_k(y')-\psi(y')\leq 4c\Big(1-\frac{\delta^2}{2}\Big)^{j_0}.
\end{equation}
Also, for $j_0$ fixed, by Ascoli-Arzel\'a Theorem, letting $k\to\infty$,
it follows that $\psi_k$ uniformly converges (up to a subsequence) to a Holder function which depends on $j_0$, say
$
\psi_k\longrightarrow w_{j_0}^-
$.\\
Analogously we find a Holder continuous function $w_{j_0}^+$ s.t. $\phi_k\longrightarrow w_{j_0}^+$ uniformly
(up to a subsequence).
Moreover we have by construction that $w_{j_0}^-\leq w_{j_0}^+$ and
that
\begin{equation}\label{nomoreeqplease}\begin{split}
A_{k_i}\cap\{|x'|\leq1/2\}\textrm{ lies between}\\
\textrm{the graphs of }w_{j_0}^--\epsilon/2
\textrm{ and }w_{j_0}^++\epsilon/2,
\end{split}\end{equation}
for $k_i$ large enough.

Now we let $j_0\to\infty$.
Notice that by the construction of $\theta$ and $\beta$ above,
the Holder constants of $w_{j_0}^\pm$ depend on $\delta$ but are independent of $j_0$.\\
Therefore by Ascoli-Arzel\'a Theorem we find that there exists a Holder continuous function $u$
s.t. $w_{j_0}^-$ converges uniformly (up to subsequences) to $u$.
By $(\ref{another_flat_eq})$, also $w_{j_0}^+$ uniformly converges to $u$.

From $(\ref{nomoreeqplease})$ we get our claim.

Using $(\ref{contr_osc_ineq})$ we can translate the estimate for the flatness
of $\partial E_{k_i}\cap B_{2^j}$ from an estimate in direction $\nu^{k_i}_j$ to an estimate in direction $e_n$,
for every fixed $j$.\\
In this way we can repeat the above argument in bigger and bigger balls, getting a graph in the $e_n$ direction.

To be more precise, consider $x\in\partial E_k\cap B_{2^j}$. Then
\begin{equation*}\begin{split}
|x\cdot e_n|&\leq|x\cdot\nu_j^k|+2^j|e_n-\nu_j^k|\leq
a_k2^{j(\alpha+1)}+2^jC\Big(\sum_{i=0}^{j-1}2^{\alpha i}\Big)a_k\\
&
\leq Ca_k2^j2^{j\alpha}+C2^ja_k\sum_{i=0}^{j-1}2^{\alpha i}
=Ca_k2^j\sum_{i=0}^j2^{\alpha i}\\
&
=Ca_k2^j\frac{2^{\alpha(j+1)}-1}{2^\alpha-1}
\leq Ca_k2^{j(\alpha+1)},
\end{split}
\end{equation*}
which gives
\begin{equation}\label{oscillation_uglyeq_stop}
\partial E_k\cap B_{2^j}\subset\{|x_n|\leq C a_k2^{j(\alpha+1)}\}.
\end{equation}
We remark that the constant $C$ is independent of $k$ and $j$.

Now we can consider the sets
\begin{equation*}
A^1_{k_i}:=\Big\{\Big(x',\frac{x_n}{a_{k_i}}\Big)\,\big|\,(x',x_n)\in\partial E_{k_i}\cap B_2\Big\}.
\end{equation*}
and repeat the argument above to obtain, in $\{|x'|\leq1\}$, the convergence (up to a subsequence)
to the graph $\{(x',v(x')\}$ of a Holder function $v$, which must coincide with $u$ on $\{|x'|\le1/2\}$.

Proceeding in this way with the sets
\begin{equation*}
A^j_{k_i}:=\Big\{\Big(x',\frac{x_n}{a_{k_i}}\Big)\,\big|\,(x',x_n)\in\partial E_{k_i}\cap B_{2^j}\Big\},
\end{equation*}
we get our claim via a diagonal argument.

Clearly $u(0)=0$, so we are left to prove the growth estimate for $u$.\\
From $(\ref{oscillation_uglyeq_stop})$ we know that for every fixed $j$
\begin{equation*}
A_{k_i}^j\subset\{|x_n|\leq C2^{j(\alpha+1)}\},
\end{equation*}
for every $k_i$. Then, since
\begin{equation*}
A_{k_i}^j\cap B'_{2^{j-1}}\longrightarrow A_\infty\cap B'_{2^{j-1}}=\big\{(x',u(x'))\,|\,|x'|\leq 2^{j-1}\big\}
\end{equation*}
uniformly, we obtain
\begin{equation*}
|u(x')|\leq 2^{\alpha+1}C\,2^{j(\alpha+1)}\,\quad\textrm{in }B'_{2^j}\,,
\end{equation*}
for every $j$.

This implies our growth estimate.
Indeed, let $x'\in\mathbb{R}^{n-1}$. Then, if $|x'|\le1$, we have $|u(x')|\leq C$
and, if
$x'\in B'_{2^{j+1}}\setminus B'_{2^j}$,
for some $j$, we have
\begin{equation*}
\frac{|u(x')|}{1+|x'|^{1+\alpha}}\leq C\frac{2^{(j+1)(\alpha+1)}}{1+2^{j(\alpha+1)}}
\leq 2^{\alpha+1}C\sup_{t\in[0,\infty)}\frac{2^{t(\alpha+1)}}{1+2^{t(\alpha+1)}}<\infty.
\end{equation*}

\end{proof}
\end{lem}

Now we show that the function $u$ found in previous Lemma must be linear.

\begin{lem}
The limit function $u$ satisfies
\begin{equation*}
(-\Delta)^\frac{s+1}{2}u=0\quad\textrm{in }\mathbb{R}^{n-1},
\end{equation*}
in the viscosity sense, and therefore is linear.
\begin{proof}
Assume $\varphi+|x'|^2$ is a smooth tangent
function that touches $u$ by below, say for simplicity at the origin.\\
By construction of $u$ we can find $E$ $s$-minimal and $a>0$ small, s.t. 
$\partial E$ is included in a $a\epsilon$ neighborhood of $\{(x',au(x')\}$ for $|x'|\leq R$
and $\partial E$ is touched by below at $x_0$, with $|x'_0|\leq\epsilon$ by a vertical translation of $a\varphi$.

From the Euler-Lagrange equation we know that
\begin{equation*}
\limsup_{\rho\to0}\frac{1}{a}\int_{\Co B_\rho(x_0)}\frac{\chi_E-\chi_{\Co E}}{|x-x_0|^{n+s}}dx\leq 0.
\end{equation*}

We are going to estimate this integral in terms of the function $u$ by integrating on square cylinders with
center $x_0$, i.e.
\begin{equation*}
D_r:=\{(x',x_n)\,|\,|x'-x'_0|<r,\,|(x-x_0)\cdot e_n|<r\}.
\end{equation*}
For simplicity we forget about the principal values in the following integrals.

We fix $\delta$ small and $R$ large, and we assume $a,\,\epsilon\ll\delta$.\\
Since $E$ contains the subgraph $P$ of a translation of $a\varphi$, using Lemma $\ref{explicit_curv_formula}$ we have
\begin{equation*}
\frac{1}{a}\int_{D_\delta}\frac{\chi_E-\chi_{\Co E}}{|x-x_0|^{n+s}}dx
\geq\frac{1}{a}\int_{D_\delta}\frac{\chi_P-\chi_{\Co P}}{|x-x_0|^{n+s}}dx
\geq -C(\varphi)\delta^{1-s}.
\end{equation*}
Using the flatness hypothesis of $E$ in the balls $B_{2^j}$, we know that
\begin{equation*}
\partial E\cap B_{2R}(x_0)\subset\{|(x-x_0)\cdot e_n|\leq C(R)a\},
\end{equation*}
for some $C(R)>0$, and hence we get by symmetry
\begin{equation*}
\frac{1}{a}\int_{D_R\setminus D_\delta}\frac{\chi_E-\chi_{\Co E}}{|x-x_0|^{n+s}}dx
=
\frac{1}{a}\int_{A}\frac{\chi_E-\chi_{\Co E}}{|x-x_0|^{n+s}}dx,
\end{equation*}
where
\begin{equation*}
A:=(D_R\setminus D_\delta)\cap\{|(x-x_0)\cdot e_n|\leq C(R)a\}.
\end{equation*}
Taking $a$ small enough, we can assume that $C(R)a<\delta/2$.
Also notice that $A\subset \Co B_\delta(x_0)$. Then for every $x\in A$ we get
\begin{equation}\label{uglyeq1}
|x'-x'_0|^2\geq\delta^2-|(x-x_0)\cdot e_n|^2\geq\delta^2-C(R)^2a^2\geq\frac{3}{4}\delta^2.
\end{equation}
Now consider the function
\begin{equation*}
F(t):=\big(|x'-x'_0|^2+t|(x-x_0)\cdot e_n|^2\big)^{-\frac{n+s}{2}},
\end{equation*}
so that
\begin{equation*}
\Big|\frac{1}{|x-x_0|^{n+s}}-\frac{1}{|x'-x'_0|^{n+s}}\Big|=|F(1)-F(0)|=\Big|\int_0^1F'(t)\,dt\Big|.
\end{equation*}
We have
\begin{equation*}
F'(t)=-\frac{n+s}{2}\big(|x'-x'_0|^2+t|(x-x_0)\cdot e_n|^2\big)^{-\frac{n+s+2}{2}}|(x-x_0)\cdot e_n|^2,
\end{equation*}
and hence we find
\begin{equation*}\begin{split}
\Big|\int_0^1F'(t)\,dt\Big|&\leq\frac{n+s}{2}C(R)^2a^2\int_0^1
\frac{dt}{(|x'-x'_0|^2+t|(x-x_0)\cdot e_n|^2)^\frac{n+s+2}{2}}\\
&
\leq\frac{n+s}{2}C(R)^2a^2\int_0^1
\frac{dt}{|x'-x'_0|^{n+s+2}}\leq C(R,\delta)a^2,
\end{split}
\end{equation*}
where we used $(\ref{uglyeq1})$ to obtain the last inequality.

Then, using this and the fact that $\partial E$ is included in a $a\epsilon$ neighborhood of $\{(x',u(x'))\}$ for $|x'|\leq R$, we find
\begin{equation*}\begin{split}
\frac{1}{a}\int_{A}\frac{\chi_E-\chi_{\Co E}}{|x-x_0|^{n+s}}dx&
=
\frac{1}{a}\int_{B'_R\setminus B'_\delta}\frac{a2(u(x')-u(x'_0)+O(\epsilon))}{|x'-x'_0|^{n+s}}dx'+O(a^2)\\
&
=2\int_{B'_R\setminus B'_\delta}\frac{u(x')-u(x'_0)}{|x'-x'_0|^{n+s}}dx'+O(\epsilon)+O(a^2).
\end{split}
\end{equation*}

We are left to estimate the contribution coming fom $\Co D_R$.\\
Let $a=2^{-k\alpha}$.
We argue as we did in the beginning of the proof of Harnack inequality to estimate $\I_s^{1/2}[E](y)$,
exploiting our flatness hypothesis for $\partial E\cap B_{2^i}$ for $0\leq i\leq k$.
 We have
\begin{equation*}\begin{split}
\frac{1}{a}\int_{\Co D_R}\frac{\chi_E-\chi_{\Co E}}{|x-x_0|^{n+s}}dx&\leq\frac{1}{a}C\Big(\int_{R/2}^{2^k}
\frac{ar^{\alpha+1}r^{n-2}}{r^{n+s}}dr+\int_{2^k}^\infty\frac{r^{n-1}}{r^{n+s}}dr\Big)\\
&
\leq\frac{1}{a}C\Big(a\int_{R/2}^\infty \frac{d}{dr}r^{\alpha-s}dr+2^{-ks}\Big)\\
&
\leq\frac{1}{a}C\big(a\,R^{\alpha-s}+a^{1+\eta}\big)\leq C(R^{\alpha-s}+a^\eta),
\end{split}\end{equation*}
where $\eta=\frac{s-\alpha}{\alpha}$.

Putting these estimates together and letting $\epsilon,\,a\to0$, we obtain from the Euler-Lagrange equation for $E$
\begin{equation*}
\int_{B'_R\setminus B'_\delta}\frac{u(x')-u(0)}{|x'|^{n+s}}dx'\leq C(\delta^{1-s}+R^{\alpha-s}).
\end{equation*}
Letting $\delta\to0$ and $R\to\infty$ shows that $u$ is a viscosity solution in $0$ (we can repeat the same argument
if $u$ is touched from above).

Then Theorem $\ref{liouville_frac}$ guarantees that $u$ is linear, concluding the proof.

\end{proof}
\end{lem}

This gives a contradiction, proving our CLAIM and hence Theorem $\ref{imp_flat_teo1}$.
From Theorem $\ref{imp_flat_teo1}$ we can deduce

\begin{teo}[Regularity]\label{flat_reg_teo1}
Let $\alpha\in(0,s)$. There exists $\epsilon_0=\epsilon_0(n,s,\alpha)>0$ s.t. if $E$ is $s$-minimal in $B_1$, with $0\in\partial E$ and
\begin{equation*}
\partial E\cap B_1\subset\{|x_n|\leq\epsilon_0\},
\end{equation*}
then
$\partial E\cap B_{1/2}$ is a $C^{1,\alpha}$ surface.
\begin{proof}
Let $k_0$ be from Theorem $\ref{imp_flat_teo1}$.
If $\epsilon_0<2^{-k_0(\alpha+1)}$, then $E$ satisfies the hypothesis of the Theorem, with $\nu_i=e_n$
for every $i\in\{0,\dots,k_0\}$, and hence there exist $\nu_i\in\mathbb{S}^{n-1}$ for every $i$, s.t.
\begin{equation*}
\partial E\cap B_{2^{-i}}\subset\{|x\cdot\nu_i|\leq2^{-i(\alpha+1)}\}.
\end{equation*}
This implies
\begin{equation*}
|\nu_i-\nu_{i+1}|\leq C2^{-i\alpha},
\end{equation*}
with $C>0$ independent of $i$, and hence
\begin{equation*}
\nu_i\longrightarrow\nu(0),
\end{equation*}
for some $\nu(0)\in\mathbb{S}^{n-1}$. Moreover we easily get by induction
\begin{equation*}
|\nu_i-\nu(0)|\leq 2C\,2^{-i\alpha}.
\end{equation*}
Thus, if $x\in\partial E\cap B_{2^{-i}}$,
\begin{equation*}
|x\cdot\nu(0)|\leq|x\cdot\nu_i|+|x|\,|\nu_i-\nu(0)|\leq C2^{-i(\alpha+1)},
\end{equation*}
and hence
\begin{equation*}
\partial E\cap B_{2^{-i}}\subset\big\{|x\cdot\nu(0)|\leq C2^{-i(\alpha+1)}\big\},
\end{equation*}
for every $i$.

This implies that $\partial E$ is a differentiable surface in $0$, with normal $\nu(0)$.\\
If we take $\epsilon_0$ smaller, say $\epsilon_0<\frac{1}{4}2^{-k_0(1+\alpha)}$, then, after traslating $E$,
we can repeat the same argument at every point $x_0\in\partial E\cap B_{1/2}$ and get
that $\partial E\cap B_{1/2}$ is actually a $C^{1,\alpha}$ surface.

\end{proof}
\end{teo}

Now we show that if an $s$-minimal set $E$ has an interior tangent ball in some point, say $B_r(-re_n)$ in $0\in\partial E$,
with $r$ big,
then $\partial E\cap B_1$ must lie below $\{x_n=1/2\}$.

\begin{lem}
There exists $R_0=R_0(n,s)>0$ s.t. the following result holds.
Let $E\subset\R$ be $s$-minimal in $B_2$, with $0\in\partial E$.
If
%\begin{equation*}
$B_R(-R\,e_n)\subset E,$
%\quad\textrm{and}\quad\partial B_R(-R\,e_n)\cap\partial E=\{0\},
%\end{equation*}
for some $R\geq R_0$,
then
\begin{equation*}
\partial E\cap B_1\subset\{x_n\leq1/2\}.
\end{equation*}
\begin{proof}
Suppose the claim is false. Then there is a point $y\in\partial E\cap B_1$ with $y_n>1/2$.
Since $B_{1/4}(y)\subset B_2$,
the Clean ball condition guarantees the existence of a ball
\begin{equation*}
B_{\frac{1}{4}c}(p)\subset E\cap B_\frac{1}{4}(y).
\end{equation*}
Moreover, since $\partial E$ has an interior tangent ball at 0, we have
\begin{equation*}
\limsup_{\delta\to0}\I_s^\delta[E](0)\leq0.
\end{equation*}
Let $D:=B_R(-R\,e_n)\cup B_R(R\,e_n)$ and let $K$ be the convex envelope of $D$; notice that $B_R\subset K$.
We can split
\begin{equation*}\begin{split}
P.V.\int_{\R}&\frac{\chi_E(z)-\chi_{\Co E}(z)}{|z|^{n+s}}dz=\int_{\Co K}\frac{\chi_E(z)-\chi_{\Co E}(z)}{|z|^{n+s}}dz\\
&+P.V:\int_{K\setminus D}\frac{\chi_E(z)-\chi_{\Co E}(z)}{|z|^{n+s}}dz
+P.V.\int_D\frac{\chi_E(z)-\chi_{\Co E}(z)}{|z|^{n+s}}dz\\
&=:I_1+I_2+I_3.
\end{split}
\end{equation*}
We can
bound
\begin{equation*}
|I_1|\leq\int_{\Co B_R}\frac{1}{|z|^{n+s}}dz\leq C_1(n,s)R^{-s}.
\end{equation*}
As for $I_2$, we have
\begin{equation*}\begin{split}
|I_2|&\leq P.V.\int_{K\setminus D}\frac{1}{|z|^{n+s}}dz\\
&
=P.V.\int_{(K\setminus D)\setminus B_{R/2}}\frac{1}{|z|^{n+s}}dz
+P.V.\int_{(K\setminus D)\cap B_{R/2}}\frac{1}{|z|^{n+s}}dz\\
&
\leq C_1(n,s)\Big(\frac{R}{2}\Big)^{-s}+P.V.\int_{(K\setminus D)\cap B_{R/2}}\frac{1}{|z|^{n+s}}dz.
\end{split}
\end{equation*}
If we let $F(h):=R-(R^2-h^2)^{1/2}$, then
\begin{equation*}
F'(0)=0\quad\textrm{and}\quad F''(h)=\frac{R^2}{(R^2-h^2)^{3/2}}\leq\Big(\frac{4}{3}\Big)^\frac{3}{2}\frac{1}{R},
\end{equation*}
for every $h\in[0,R/2]$,
and hence, arguing as in Lemma $\ref{explicit_curv_formula}$, we get
\begin{equation*}\begin{split}
P.V.\int_{(K\setminus D)\cap B_{R/2}}&\frac{1}{|z|^{n+s}}dz
=2\int_{\mathbb{S}^{n-2}}d\mathcal{H}^{n-2}\int_0^{R/2}\Big(\int_0^\frac{F(\rho)}{\rho}(1+t^2)^{-\frac{n+s}{2}}\Big)
\frac{d\rho}{\rho^{s+1}}\\
&
\leq C\frac{1}{R}\int_0^{R/2}\frac{d}{d\rho}\rho^{1-s}\,d\rho\leq CR^{-s}.
\end{split}
\end{equation*}
Therefore
\begin{equation*}
|I_2|\leq C_2(n,s)R^{-s}.
\end{equation*}
We're left to estimate $I_3$. Notice that $D$ is symmetric with respect to $\{x_n=0\}$ and $B_R(-R\,e_n)\subset E$
by hypothesis. Moreover the smal ball $B_{c/4}(p)$ is contained in $E$ and $B_{c/4}(p)\subset B_R(R\,e_n)$.\\
Roughly speaking, the contribution coming from any point $z\in\Co E\cap D$ is canceled by that of the point $-z\in E\cap B_R(-R\, e_n)$
and we are left with (at least) the contribution coming from $B_{c/4}(p)$, which is positive. That is
\begin{equation*}
I_3=P.V.\int_D\frac{\chi_E(z)-\chi_{\Co E}(z)}{|z|^{n+s}}dz
\geq \int_{B_{\frac{1}{4}c}(p)}\frac{1}{|z|^{n+s}}dz.
\end{equation*}
Now notice that for every $z\in B_{c/4}(p)$ we have
\begin{equation*}
|z|\leq |z-p|+|p-y|+|y|\leq \frac{1}{4}c+\frac{1}{4}+1\leq2,
\end{equation*}
and hence we obtain
\begin{equation*}
I_3\geq \frac{1}{2^{n+s}}\,\omega_n\frac{1}{4^n}c^n=:C_3(n,s).
\end{equation*}
We remark that this last estimate does not depend on $R$ nor on our set $E$ nor the point $y\in\partial E\cap B_1$ lying in the strip
$\{1/2<y\leq1\}$.

Therefore
\begin{equation*}
0\geq\limsup_{\delta\to0}\I_s^\delta[E](0)\geq\liminf_{\delta\to0}\I_s^\delta[E](0)
\geq C_3-\big(C_1+C_2\big)R^{-s}>0,
\end{equation*}
provided $R$ is big enough, giving a contradiction.

\end{proof}
\end{lem}

\begin{rmk}
For any fixed $\beta\in(0,1/2]$, we can repeat the same argument 
to show that $\partial E\cap B_1$ lies below the plane $\{x_n=\beta\}$,
provided that $E$ has an interior tangent ball in 0, with radius $R\geq R(n,s,\beta)$.\\
In this case the constant $C_3$ appearing in the proof becomes
\begin{equation*}
C_3=C_3(n,s,\beta)\sim\beta^n,
\end{equation*}
and hence
\begin{equation*}
\lim_{\beta\to0}R(n,s,\beta)=\infty.
\end{equation*}
\end{rmk}

As a consequence we see that if $E$ has an interior tangent ball in some point $x_0$
then
we can control the flatness of $\partial E$ in a small enough neighborhood of $x_0$.

Therefore, using this Lemma we can prove the following consequence of the Regularity Theorem
\begin{coroll}\label{int_tang_flat_reg}
Let $E\subset\R$ be $s$-minimal in $\Omega$ and let $x_0\in\partial E\cap\Omega$.\\
If $E$ has an interior tangent ball $B_r(p)\subset E$ in $x_0$, then
$\partial E$ is a $C^{1,\alpha}$ surface in a neighborhood of $x_0$.
\begin{proof}
After a traslation and a rotation, we can suppose $x_0=0$ and
\begin{equation*}
B_r(-re_n)\subset E.
\end{equation*}
If we dilate everything by a factor $\lambda':=\lambda/r>0$ we obtain
\begin{equation*}
B_\lambda(-\lambda e_n)\subset\lambda' E.
\end{equation*}
Clearly $0\in\partial(\lambda' E)$ and, taking $\lambda$ big enough, the set $\lambda' E$
is $s$-minimal in $B_2\subset\lambda'\Omega$.
Using previous Lemma and the Remark above, we know that if $\lambda\geq R(n,s,\epsilon_0/2)$,
then $\partial(\lambda' E)\cap B_1\subset\{x_n\leq\epsilon_0/2\}$.

Moreover $\partial(\lambda' E)\cap B_1$ lies above the ball $B_\lambda(-\lambda e_n)$.
Thus, eventually taking a bigger $\lambda$, we have also
$\partial(\lambda' E)\cap B_1\subset\{x_n\geq-\epsilon_0/2\}$.\\
Now the set $\lambda'E$ satisfies the hypothesis of the Regularity Theorem and hence
$\partial (\lambda' E)\cap B_{1/2}$ is a $C^{1,\alpha}$ surface. Scaling back concludes the proof.

\end{proof}
\end{coroll}

Considering $\Co E$ in place of $E$ we see that the same holds if we have an exterior tangent ball.

\end{section}

\begin{section}{Monotonicity Formula}

In this section we prove a monotonicity formula for a quantity related to the fractional perimeter of an $s$-minimal set.\\
This formula can be seen as an extension to the fractional framework of the classical monotonicity formula which holds true for minimal surfaces (see the end of Chapter 1).

However in order to define this quantity we need to consider an appropriate extension function in one extra variable.

\begin{subsection}{Intermezzo About the Fractional Laplacian: The Extension Problem}

For all the details about the extension problem we refer to \cite{extension}. See also \cite{Sire}.

For every $s\in(0,1)$ we define the weighted $L^1$-space
\begin{equation*}
L_\frac{s}{2}:=\left\{u:\R\longrightarrow\mathbb{R}\,\Big|\,\int_{\R}\frac{|u(y)|}{(1+|y|^2)^\frac{n+s}{2}}\,dy<\infty\right\}.
\end{equation*}
For a function $u\in L_{s/2}$ we consider the extension $\tilde{u}:\R\times[0,\infty)\longrightarrow\mathbb{R}$,
which solves
\begin{equation}\label{extension}
\left\{\begin{array}{cc}
\textrm{div}(z^{1-s}\nabla\tilde{u})=0&\textrm{in }\mathbb{R}^{n+1}_+,\\
\tilde{u}=u&\textrm{on }\{z=0\},
\end{array}\right.
\end{equation}
where
\begin{equation*}
\mathbb{R}^{n+1}_+=\{(x,z)\in\mathbb{R}^{n+1}\,|\,x\in\R,\,z>0\}.
\end{equation*}

\begin{rmk}
Actually such an extension can be defined in the same way for every $s\in(0,2)$ and the case $s=1$ is well known, but we are interested only in the range $s\in(0,1)$.
\end{rmk}

Let $a:=1-s$. We use capital letters, like $X$, to
denote points in $\mathbb{R}^{n+1}$.
\begin{rmk}
It is clear that the first equation in $(\ref{extension})$ is the Euler-Lagrange equation for the functional
\begin{equation}\label{extension_energy}
\int_{\{z>0\}}|\nabla\tilde{u}|^2z^a\,dX,
\end{equation}
and it can be rewritten as
\begin{equation*}
\Delta_x\tilde{u}+\frac{a}{z}\tilde{u}_z+\tilde{u}_{zz}=0,
\end{equation*}
where $\Delta_x$ denotes the Laplacian in the first $n$ variables and the pedice $z$ denotes derivation in the last variable.
\end{rmk}

The solution $\tilde{u}$ to $(\ref{extension})$ can be explicitly computed via the Poisson formula
\begin{equation*}
\tilde{u}(\cdot,z)=P(\cdot,z)\ast u,\qquad\textrm{i.e.}\quad\tilde{u}(x,z)=\int_{\R}P(x-\xi,z)u(\xi)\,d\xi,
\end{equation*}
where the Poisson kernel $P$ is
\begin{equation*}
P(x,z):=c_1(n,a)\frac{z^{1-a}}{\left(|x|^2+z^2\right)^\frac{n+1-a}{2}}.
\end{equation*}
If we define
\begin{equation*}
H(x):=c_1\frac{1}{(1+|x|^2)^\frac{n+s}{2}},
\end{equation*}
then we have
\begin{equation*}
P(x,1)=H(x)\qquad\textrm{and}\quad P(x,z)=\frac{1}{z^n}H\Big(\frac{x}{z}\Big).
\end{equation*}
In particular
\begin{equation*}
\int_{\R}P(x,z)\,dx=\frac{1}{z^n}\int_{\R}H\Big(\frac{x}{z}\Big)\,dx=\int_{\R}H(\xi)\,d\xi,
\end{equation*}
for every $z>0$. Then the constant $c_1$ is chosen in such a way that
\begin{equation*}
\int_{\R}H(x)\,dx=1.
\end{equation*}

The extension $\tilde{u}$ is related to the $\frac{s}{2}$-fractional Laplacian of $u$ via the formula
\begin{equation}\label{frac_ext_trace}
\lim_{z\to0}-z^a\tilde{u}_z(\cdot,z)=c_2(n,a)(-\Delta)^\frac{s}{2}u,
\end{equation}
which holds in the distributional sense, i.e. $\tilde{u}$ is a weak solution of the Neumann problem
\begin{equation}
\left\{\begin{array}{cc}
\textrm{div}(z^a\nabla\tilde{u})=0&\textrm{in }\mathbb{R}^{n+1}_+,\\
-z^a\frac{\partial\tilde{u}}{\partial z}=c_2(n,a)(-\Delta)^\frac{s}{2}u&\textrm{on }\partial\mathbb{R}^{n+1}_+.
\end{array}\right.
\end{equation}
Moreover it can be shown that
\begin{equation*}
\int_{\mathbb{R}^{n+1}_+}|\nabla\tilde{u}|^2z^a\,dX=c_3(n,a)[u]_{H^\frac{s}{2}(\R)}^2,
\end{equation*}
for every $u\in H^\frac{s}{2}(\R)$ with compact support.\\

Now we consider the local contribution of the $H^\frac{s}{2}$-seminorm of $u$ in the ball $B_r$, i.e.
\begin{equation*}
\J_r(u):=\int_{B_r}\int_{B_r}\frac{|u(x)-u(y)|^2}{\kers}\,dx\,dy+2\int_{B_r}\int_{\Co B_r}\frac{|u(x)-u(y)|^2}{\kers}\,dx\,dy.
\end{equation*}
Notice that the first term is just $[u]_{H^\frac{s}{2}(B_r)}^2$ and, if $u\in H^\frac{s}{2}(\R)$,
\begin{equation*}
[u]_{H^\frac{s}{2}(\R)}^2=\J_r(u)+[u]_{H^\frac{s}{2}(\Co B_r)}^2.
\end{equation*}
In particular, if $u,\,v\in H^\frac{s}{2}(\R)$ and $u=v$ outside $B_r$, then
\begin{equation*}
[u]_{H^\frac{s}{2}(\R)}^2-[v]_{H^\frac{s}{2}(\R)}^2=
\J_r(u)-\J_r(v).
\end{equation*}

\begin{rmk}
Notice that the functional $\J_r$ is simply the extension to generic functions of the fractional perimeter in the ball $B_r$, i.e.
\begin{equation}
\J_{B_r}(E)=P_s(E,B_r)=\frac{1}{2}\J_r(\chi_E).
\end{equation}
In particular, if $u=\chi_E-\chi_{\Co E}$, then
\begin{equation*}
\J_r(u)=8P_s(E,B_r).
\end{equation*}
\end{rmk}

Now we prove some estimates which relate our functional $\J_1(u)$ with the energy $(\ref{extension_energy})$ of the extension
$\tilde{u}$.
\begin{prop}
Let $\Omega\subset\mathbb{R}^{n+1}$ be a bounded open set with Lipschitz boundary and denote
\begin{equation*}
\Omega_0:=\Omega\cap\{z=0\}\subset\R,\qquad\Omega_+:=\Omega\cap\{z>0\}.
\end{equation*}

(a) If $\Omega_0\subset\subset B_1$ then
\begin{equation}\label{local_energy1}
\int_{\Omega_+}|\nabla\tilde{u}|^2z^a\,dX\leq C\J_1(u),
\end{equation}
with $C$ depending on $\Omega$.

(b) If $B_1\subset\subset\Omega_0$ and $u$ is bounded in $\R$ then
\begin{equation*}
\J_1(u)\leq C\left(1+\int_{\Omega_+}|\nabla\tilde{u}|^2z^a\,dX\right),
\end{equation*}
with $C$ depending on $\Omega$ and $\|u\|_{L^\infty(\R)}$.

\begin{proof}

(a) We can assume without loss of generality that $\int_{B_1}u=0$. Then
\begin{equation*}
2\int_{B_1}\int_{\R}\frac{u(x)u(y)}{(1+|x|^2)^\frac{n+s}{2}}\,dx\,dy=
2\left(\int_{B_1}u(y)\,dy\right)\left(\int_{\R}\frac{u(x)}{(1+|x|^2)^\frac{n+s}{2}}\,dx\right)=0,
\end{equation*}
and hence
\begin{equation*}\begin{split}
\int_{\R}\frac{|u(x)|^2}{(1+|x|^2)^\frac{n+s}{2}}\,dx&
=\frac{1}{|B_1|}\int_{B_1}\int_{\R}\frac{|u(x)|^2}{(1+|x|^2)^\frac{n+s}{2}}\,dx\,dy\\
&
\leq\frac{1}{|B_1|}\int_{B_1}\int_{\R}\frac{|u(x)-u(y)|^2}{(1+|x|^2)^\frac{n+s}{2}}\,dx\,dy.
\end{split}\end{equation*}
For every $y\in B_1$ we have $|x-y|\leq1+|x|$ and hence
\begin{equation*}
|x-y|^{n+s}\leq(1+|x|)^{n+s}\leq C(1+|x|^2)^\frac{n+s}{2}.
\end{equation*}
Therefore
\begin{equation*}\begin{split}
\frac{1}{|B_1|}\int_{B_1}\int_{\R}\frac{|u(x)-u(y)|^2}{(1+|x|^2)^\frac{n+s}{2}}\,dx\,dy&
\leq C\int_{B_1}\int_{\R}\frac{|u(x)-u(y)|^2}{\kers}\,dx\,dy\\
&
\leq C\J_1(u),
\end{split}
\end{equation*}
and
\begin{equation*}
\int_{\R}\frac{|u(x)|^2}{(1+|x|^2)^\frac{n+s}{2}}\,dx\leq C\J_1(u).
\end{equation*}
Thus, by Holder inequality
\begin{equation*}\begin{split}
\int_{\R}\frac{|u(x)|}{(1+|x|^2)^\frac{n+s}{2}}\,dx&
=\int_{\R}\frac{|u(x)|}{(1+|x|^2)^\frac{n+s}{4}}\,\frac{1}{(1+|x|^2)^\frac{n+s}{4}}\,dx\\
&
\leq\left(\int_{\R}\frac{|u(x)|^2}{(1+|x|^2)^\frac{n+s}{2}}\,dx\right)^\frac{1}{2}
\left(\int_{\R}\frac{1}{(1+|x|^2)^\frac{n+s}{2}}\,dx\right)^\frac{1}{2}\\
&
\leq C\J_1(u)^\frac{1}{2}.
\end{split}
\end{equation*}

Now let $\varphi\in C_c^\infty(\R)$ be a smooth cutoff function s.t. $\varphi=1$ in $N_b(\Omega_0)$,
the $b$-neighborhood of $\Omega_0$, for some $b\in(0,1)$ small enough to have $N_b(\Omega_0)\subset\subset B_1$,
and $\supp\varphi\subset B_1$, .
We write
\begin{equation*}
u=\varphi u+(1-\varphi)u=:u_1+u_2.
\end{equation*}
Clearly $\tilde{u}=\tilde{u}_1+\tilde{u}_2$. Since $u_1$ is compactly supported in $B_1$, we have
\begin{equation*}
\int_{\mathbb{R}_+^{n+1}}|\nabla\tilde{u}_1|^2z^a\,dX=C[u_1]_{H^\frac{s}{2}}^2=C\J_1(u_1)\leq C\J_1(u).
\end{equation*}

On the other hand, it can be shown that for every $(x,z)\in\Omega_+$
\begin{equation}\label{punctual_gradient_ext_estimate}
z^a|\nabla\tilde{u}_2(x,z)|\leq C\int_{\R}\frac{|u_2(y)|}{(1+|y|^2)^\frac{n+s}{2}}\,dy\leq C\J_1(u)^\frac{1}{2},
\end{equation}
and hence
\begin{equation*}
\int_{\Omega_+}|\nabla\tilde{u}_2|^2z^a\,dX
=\int_{\Omega_+}|\nabla\tilde{u}_2|^2z^{2a}\,\frac{1}{z^a}\,dX
\leq C\J_1(u)\int_{\Omega_+}\frac{1}{z^a}\,dX.
\end{equation*}
Since $\Omega_+$ is bounded, it is contained in the cylinder $C_R$,
\begin{equation*}
\Omega_+\subset C_R:=\{(x,z)\in\mathbb{R}^{n+1}\,|\,0\leq z<R,\,|x|<R\},
\end{equation*}
for some $R=R(\Omega)$ big enough. Thus
\begin{equation*}
\int_{\Omega_+}\frac{1}{z^a}\,dX\leq\int_{C_R}\frac{1}{z^a}\,dX
=\frac{1}{1-a}\int_{B_R}dx\int_0^R\frac{d}{dz}z^{1-a}\,dz
=\frac{\omega_n}{1-a}R^{n+1-a},
\end{equation*}
and
\begin{equation*}
\int_{\Omega_+}|\nabla\tilde{u}_2|^2z^a\,dX\leq C\J_1(u).
\end{equation*}
Therefore
\begin{equation*}\begin{split}
\int_{\Omega_+}|\nabla\tilde{u}|^2z^a\,dX&=\int_{\Omega_+}|\nabla\tilde{u}_1+\nabla\tilde{u}_2|^2z^a\,dX\\
&
\leq\int_{\Omega_+}|\nabla\tilde{u}_1|^2z^a\,dX+\int_{\Omega_+}|\nabla\tilde{u}_2|^2z^a\,dX\\
&
\qquad\qquad\qquad
+2\int_{\Omega_+}|\nabla\tilde{u}_1\cdot\nabla\tilde{u}_2|z^a\,dX\\
&
\leq2C\J_1(u)+2\int_{\Omega_+}|\nabla\tilde{u}_1|z^\frac{a}{2}\,|\nabla\tilde{u}_2|z^\frac{a}{2}\,dX\\
&
\leq2C\J_1(u)+2\left(\int_{\Omega_+}|\nabla\tilde{u}_1|^2z^a\,dX\right)^\frac{1}{2}\left(\int_{\Omega_+}|\nabla\tilde{u}_2|^2z^a\,dX\right)^\frac{1}{2}\\
&
\leq C\J_1(u).
\end{split}\end{equation*}

%We are left to prove the first inequality in $(\ref{punctual_gradient_ext_estimate})$.\\
%Notice that
%$u_2=0$ in $N_b(\Omega_0)=:F$ and hence the Poisson formula gives
%\begin{equation*}
%\tilde{u}_2(x,z)=c_1\int_{\R\setminus F}\frac{u_2(\xi)z^{1-a}}{(|x-\xi|^2+z^2)^\frac{n+1-a}{2}}d\xi.
%\end{equation*}
%Now we can estimate the derivatives in the $x$ coordinates as
%\begin{equation}
%|z^a\nabla_x\tilde{u}_2(x,z)|\leq C
%\int_{\R\setminus F}\frac{|u_2(\xi)|\,z\,|x-\xi|}{(|x-\xi|^2+z^2)^\frac{n+s+2}{2}}d\xi.
%\end{equation}
%Now, if $(x,z)\in\Omega_+$, since $d(\xi,\Omega_0)\geq b$,
%and $\Omega$ is bounded, eventually taking a smaller $b$,
%we have $|x-\xi|\geq b$ for every $\xi\in\R\setminus F$.

(b) Since $u$ is bounded, we have
\begin{equation*}\begin{split}
\int_{B_1}\int_{\Co B_1}\frac{|u(x)-u(y)|^2}{\kers}\,dx\,dy&\leq4\|u\|_{L^\infty(\R)}^2\int_{B_1}\int_{\Co B_1}\frac{1}{\kers}\,dx\,dy\\
&
=4P_s(B_1)\|u\|^2_{L^\infty(\R)}.
\end{split}
\end{equation*}

Let $\psi\in C_c^\infty(\mathbb{R}^{n+1})$ be a smooth cutoff function s.t. $\psi=1$ in $B_1$ and $\supp\psi\subset\Omega$,
and define
\begin{equation*}
v(x):=\psi(x,0)u(x).
\end{equation*}
Notice that, since $u=v$ in $B_1$,
\begin{equation*}
\int_{B_1}\int_{B_1}\frac{|u(x)-u(y)|^2}{\kers}\,dx\,dy=\int_{B_1}\int_{B_1}\frac{|v(x)-v(y)|^2}{\kers}\,dx\,dy\leq\J_1(v).
\end{equation*}
Moreover, since $\supp v\subset\Omega_0$ is compact,
\begin{equation*}
\int_{\mathbb{R}^{n+1}_+}|\nabla\tilde{v}|^2z^a\,dX=C[v]_{H^\frac{s}{2}(\R)}^2\geq C\J_1(v).
\end{equation*}
Since the function $\tilde{v}$ minimizes
\begin{equation*}
\int_{\mathbb{R}^{n+1}_+}|\nabla\tilde{v}|^2z^a\,dX=\inf\left\{\int_{\mathbb{R}^{n+1}_+}|\nabla w|^2z^a\,dX\,\Big|\,
w(\cdot,0)=v\right\}
\end{equation*}
and $\psi(x,0)\tilde{u}(x,0)=\psi(x,0)u(x)=v(x)$, we have
\begin{equation*}
\int_{\mathbb{R}^{n+1}_+}|\nabla\tilde{v}|^2z^a\,dX
\leq
\int_{\mathbb{R}^{n+1}_+}|\nabla(\psi\tilde{u})|^2z^a\,dX.
\end{equation*}

To conclude it is enough to compute the right hand side. Recalling that $\supp\psi\subset\Omega$,
\begin{equation*}\begin{split}
\int_{\mathbb{R}^{n+1}_+}|\nabla(\psi\tilde{u})|^2z^a\,dX&
=
\int_{\mathbb{R}^{n+1}_+}|(\nabla\psi)\tilde{u}+\psi\nabla\tilde{u}|^2z^a\,dX\\
&
=\int_{\Omega_+}|(\nabla\psi)\tilde{u}|^2z^a\,dX+
\int_{\Omega_+}|\psi\nabla\tilde{u}|^2z^a\,dX\\
&
\qquad\qquad
+\int_{\Omega_+}2(\nabla\psi\cdot\nabla\tilde{u})\,\psi\tilde{u}z^a\,dX.
\end{split}
\end{equation*}
Notice that, since $u$ is bounded, also $\tilde{u}$ is bounded
\begin{equation*}
|\tilde{u}(x,z)|\leq\int_{\R}P(\xi,z)|u(x-\xi)|\,d\xi\leq\|u\|_{L^\infty(\R)}\int_{\R}P(\xi,z)\,d\xi=\|u\|_{L^\infty(\R)}.
\end{equation*}
In particular this gives
\begin{equation*}
\|\tilde{u}\|_{L^\infty(\mathbb{R}^{n+1}_+)}\leq\|u\|_{L^\infty(\R)}.
\end{equation*}
Therefore the first term in the right hand side above is
\begin{equation*}
\int_{\Omega_+}|(\nabla\psi)\tilde{u}|^2z^a\,dX
\leq\mathfrak{L}^{n+1}(\Omega_+)\sup_{\Omega_+}|\nabla\psi|^2\|u\|_{L^\infty(\R)}^2\sup_{\Omega_+}z^a<\infty.
\end{equation*}
The second term gives
\begin{equation*}
\int_{\Omega_+}|\psi\nabla\tilde{u}|^2z^a\,dX\leq\sup_{\Omega_+}|\psi|^2\int_{\Omega_+}|\nabla\tilde{u}|^2z^a\,dX.
\end{equation*}
As for the last term, we have
\begin{equation*}
2(\nabla\psi\cdot\nabla\tilde{u})\,\psi\tilde{u}z^a=\textrm{div}(\tilde{u}^2\psi z^a\nabla\psi)
-\tilde{u}^2\psi z^a\Delta\psi-\tilde{u}^2z^a|\nabla\psi|^2
-\tilde{u}^2\psi\,\psi_z\frac{d}{dz}z^a.
\end{equation*}
Since $\supp\psi\subset\Omega$, using the Gauss-Green formula for the first term gives
\begin{equation*}
\int_{\Omega_+}\textrm{div}(\tilde{u}^2\psi z^a\nabla\psi)\,dX
=\int_{\partial\Omega_+}\tilde{u}^2\psi z^a\nabla\psi\cdot\nu_{\Omega_+}\,d\sigma
=\int_{\Omega_0}u^2\psi z^a\psi_z\,dx.
\end{equation*}
Integrating and taking absolute values gives
\begin{equation*}\begin{split}
\int_{\Omega_+}2(\nabla\psi&\cdot\nabla\tilde{u})\,\psi\tilde{u}z^a\,dX
\leq\left|\int_{\Omega_+}2(\nabla\psi\cdot\nabla\tilde{u})\,\psi\tilde{u}z^a\,dX\right|\\
&
\leq\int_{\Omega_0}|u^2\psi z^a\psi_z|\,dx
+\int_{\Omega_+}|\tilde{u}^2\psi z^a\Delta\psi|\,dX
+\int_{\Omega_+}|\tilde{u}^2z^a|\nabla\psi|^2|\,dX\\
&
\qquad
+\int_{\Omega_+}\Big|\tilde{u}^2\psi\,\psi_z\frac{1}{az^{1-a}}\Big|\,dX
\end{split}\end{equation*}
As we did in (a), we can enclose $\Omega_+\subset C_R$, for $R=R(\Omega)$ big enough; then integrating the last term in
the right hand side above gives
\begin{equation*}\begin{split}
\int_{\Omega_+}\Big|\tilde{u}^2\psi\,\psi_z\frac{1}{az^{1-a}}\Big|\,dX&
\leq\|u\|_{L^\infty(\R)}^2\sup_{\Omega_+}|\psi\,\psi_z|\int_{B_R}dx\int_0^R\frac{z^{a-1}}{a}\,dz\\
&
=\|u\|_{L^\infty(\R)}^2\sup_{\Omega_+}|\psi\,\psi_z|\omega_nR^{n+a}.
\end{split}
\end{equation*}
The other three terms are simply bounded by a constant depending on $\|u\|_{L^\infty(\R)}$, $\Omega$
and $\psi$; for example the second term is
\begin{equation*}
\int_{\Omega_+}|\tilde{u}^2\psi z^a\Delta\psi|\,dX
\leq\|u\|_{L^\infty(\R)}^2\sup_{\Omega_+}|\psi\Delta\psi| R^a\mathfrak{L}^{n+1}(\Omega_+).
\end{equation*}

Putting everything together gives the claim.

\end{proof}
\end{prop}

\begin{rmk}
Let $\Omega\subset\mathbb{R}^{n+1}_+$ be a bounded open set with Lipschitz boundary and let $\bar{v}:\Omega\longrightarrow\mathbb{R}$ be s.t.
\begin{equation*}
\int_\Omega|\nabla\bar{v}|^2z^a\,dX<\infty.
\end{equation*}
Then from Holder's inequality
\begin{equation*}
\int_\Omega|\nabla\bar{v}|\,dX<\infty,
\end{equation*}
and hence we can define the trace of $\bar{v}$ on $\partial\Omega$.\\
In particular if $\bar{v}=\tilde{u}$, the trace of $\tilde{u}$ on $\Omega_0$ is clearly $u$.
\end{rmk}

\begin{rmk}
Assume $\bar{v}$ is compactly supported in the open set $\Omega\subset\mathbb{R}^{n+1}$ and has trace $v$ on $\Omega_0$. Then
\begin{equation*}
\int_{\Omega_+}|\nabla\bar{v}|^2z^a\,dX\geq\int_{\mathbb{R}^{n+1}_+}|\nabla\tilde{v}|^2z^a\,dX.
\end{equation*}
We briefly sketch the proof.
Denote by $\bar{v}_k$ the solution of equation
\begin{equation*}
\textrm{div}(z^a\nabla\bar{v}_k)=0\qquad\textrm{in}\quad \mathcal{B}_k^+,
\end{equation*}
which has trace
$v$ on $\{z=0\}$ and $0$ on $\partial\mathcal{B}_k^+\cap\{z>0\}$,
where $\mathcal{B}_k^+$ denotes the upper half of the $(n+1)$-dimensional ball centered at 0, i.e.
\begin{equation*}
\mathcal{B}_k^+=\left\{(x,z)\in\mathbb{R}^{n+1}_+\,|\,(|x|^2+z^2)^\frac{1}{2}<k\right\}.
\end{equation*}
Extend $\bar{v}_k$ to be 0 outside $\mathcal{B}_k^+$. If $k$ is big enough, so that  $\supp\bar{v}\subset\mathcal{B}_k$, then $\bar{v}$ and $\bar{v}_k$ have the same trace on
$\partial\mathcal{B}_k^+$ and hence
\begin{equation*}\begin{split}
\int_{\mathbb{R}^{n+1}_+}|\nabla\bar{v}_k|^2z^a\,dX&=
\int_{\mathcal{B}^+_k}|\nabla\bar{v}_k|^2z^a\,dX\leq\int_{\mathcal{B}^+_k}|\nabla\bar{v}|^2z^a\,dX\\
&
=\int_{\Omega_+}|\nabla\bar{v}|^2z^a\,dX.
\end{split}\end{equation*}
It can be checked that $\nabla\bar{v}_k$ converges to $\nabla\tilde{v}$ in $L^2(\mathbb{R}^{n+1}_+,z^a\,dx\,dz)$,
so we get the claim letting $k\to\infty$.
\end{rmk}

\begin{lem}
Assume $u,\, v:\R\longrightarrow\mathbb{R}$ are s.t. $\J_1(u),\,\J_1(v)<\infty$ and $u-v$ is compactly supported in $B_1$.
Then
\begin{equation}
\inf_{\Omega,\,\bar{v}}\int_{\Omega_+}\left(|\nabla\bar{v}|^2-|\nabla\tilde{u}|^2\right)z^a\,dX=c_3(n,a)\left(\J_1(v)-\J_1(u)\right),
\end{equation}
where the infimum is taken among all bounded open sets $\Omega\subset\mathbb{R}^{n+1}$ with Lipschitz boundary and $\Omega_0\subset B_1$,
and among all functions $\bar{v}$ s.t. $\bar{v}-\tilde{u}$ is compactly supported in $\Omega$
and the trace of $\bar{v}$ on $\{z=0\}$ equals $v$.

\begin{proof}
First of all notice that, since the trace of $\bar{v}$ is equal to $v$ on the whole of $\{z=0\}$ and $\bar{v}=\tilde{u}$
out of $\Omega_+$, we must have supp$(u-v)\subset\Omega_0$.\\
If $u,\,v\in C^\infty_c(\R)$, then
\begin{equation*}\begin{split}
\inf_{\Omega,\,\bar{v}}\int_{\Omega_+}&\left(|\nabla\bar{v}|^2-|\nabla\tilde{u}|^2\right)z^a\,dX
=\int_{\mathbb{R}^{n+1}_+}|\nabla\tilde{v}|^2z^a\,dX-\int_{\mathbb{R}^{n+1}_+}|\nabla\bar{u}|^2z^a\,dX\\
&
=c_3(n,a)\left([v]_{H^\frac{s}{2}(\R)}^2-[u]_{H^\frac{s}{2}(\R)}^2\right)
=c_3(n,a)\left(\J_1(v)-\J_1(u)\right).
\end{split}\end{equation*}

The first equality is a consequence of previous Remark: fixed an open set $\Omega$ as above and an admissible $\bar{v}$, let
$K:=$supp$(\bar{v}-\tilde{u})$ and define $\bar{w}:=\bar{v}$ in $K$ and 0 outside; then the trace of $\bar{w}$ is $v$ on $\Omega_0$
and $\supp\bar{w}\subset\Omega$, so
\begin{equation*}
\int_{\Omega_+}|\nabla\bar{v}|^2z^a\,dX\geq\int_{K_+}|\nabla\bar{v}|^2z^a\,dX=\int_{\Omega_+}|\nabla\bar{w}|^2z^a\,dX
\geq\int_{\mathbb{R}^{n+1}_+}|\nabla\tilde{v}|^2z^a\,dX.
\end{equation*}
On the other hand, it is clear that taking a sequence of admissible pairs of open sets  $\Omega^k$ converging to $\mathbb{R}^{n+1}_+$ and functions
$\bar{v}_k$ converging to $\tilde{v}$ gives the opposite inequality.\\
The other two equalities
are a consequence of the compact supports of $u$ and $v$ and the hypothesis $u=v$ out of $B_1$, respectively.

In the general case let
\begin{equation*}
\Omega^1\subset\Omega^2\subset\Omega^3\dots,\qquad\bigcup_k\Omega^k=\mathbb{R}^{n+1}\setminus\{(x,0)\,|\, x\in\Co B_1\},
\end{equation*}
and denote $\bar{w}_k$ the solution of the equation
\begin{equation*}
\textrm{div}(z^a\nabla\bar{w}_k)=0\qquad\textrm{in}\quad\Omega_+^k,
\end{equation*}
which has trace $w:=v-u$ on $\Omega_0^k$ and 0 on $\partial\Omega^k\cap\{z>0\}$. We extend $\bar{w}_k$ to be 0 outside
$\Omega^k$. Notice that the function $\tilde{u}+\bar{w}_k$ satisfies the equation
\begin{equation*}
\textrm{div}(z^a\nabla(\tilde{u}+\bar{w}_k))=0\qquad\textrm{in}\quad\Omega_+^k,
\end{equation*}
and has trace $v$ on the whole of $\{z=0\}$.\\
If $\Omega\subset\Omega^k$, then $\bar{v}$ and $\tilde{u}+\bar{w}_k$ have the same trace on $\partial\Omega^k_+$,
equal to $v$ on $\Omega^k_0$ and $\tilde{u}$ on $\partial\Omega^k\cap\{z>0\}$; actually $\bar{v}=\tilde{u}=\tilde{u}+\bar{w}_k$ in $\mathbb{R}^{n+1}_+\setminus\Omega^k$. Therefore
\begin{equation*}\begin{split}
\int_{\mathbb{R}^{n+1}_+}\left(|\nabla\bar{v}|^2-|\nabla\tilde{u}|^2\right)z^a\,dX
\geq\int_{\mathbb{R}^{n+1}_+}\left(|\nabla(\tilde{u}+\bar{w}_k|^2-|\nabla\tilde{u}|^2\right)z^a\,dX\\
=\int_{\mathbb{R}^{n+1}_+}|\nabla\bar{w}_k|^2z^a\,dX
+2\int_{\mathbb{R}^{n+1}_+}z^a\nabla\tilde{u}\cdot\nabla\bar{w}_k\,dX.
\end{split}\end{equation*}
The second term is independent of $k$.
Indeed $\tilde{u}$ satisfies
\begin{equation*}
\textrm{div}(z^a\nabla\tilde{u})=0\qquad\textrm{in}\quad\mathbb{R}^{n+1}_+
\end{equation*}
and $\bar{w}_{k_1}-\bar{w}_{k_2}$ is compactly supported in $\mathbb{R}^{n+1}$ and has trace 0 on $\{z=0\}$; therefore using Gauss-Green and the equality
\begin{equation*}\begin{split}
z^a\nabla\tilde{u}\cdot\nabla(\bar{w}_{k_1}-\bar{w}_{k_2})&=\textrm{div}(z^a(\bar{w}_{k_1}-\bar{w}_{k_2})\nabla\tilde{u})
-(\bar{w}_{k_1}-\bar{w}_{k_2})\textrm{div}(z^a\nabla\tilde{u})\\
&
=\textrm{div}(z^a(\bar{w}_{k_1}-\bar{w}_{k_2})\nabla\tilde{u}),
\end{split}\end{equation*}
we get
\begin{equation*}
\int_{\mathbb{R}^{n+1}_+}z^a\nabla\tilde{u}\cdot\nabla(\bar{w}_{k_1}-\bar{w}_{k_2})\,dX=0,
\end{equation*}
and hence
\begin{equation*}
\int_{\mathbb{R}^{n+1}_+}z^a\nabla\tilde{u}\cdot\nabla\bar{w}_k\,dX
=
\int_{\mathbb{R}^{n+1}_+}z^a\nabla\tilde{u}\cdot\nabla\bar{w}_1\,dX,
\end{equation*}
for every $k$. As in the case of balls in previous Remark, it can be checked that
$\nabla\bar{w}_k$ converges to $\nabla\tilde{w}$ in $L^2(\mathbb{R}^{n+1}_+,z^a\,dx\,dz)$.\\
Thus if we let $k\to\infty$ we find that the infimum equals
\begin{equation*}
\int_{\mathbb{R}^{n+1}_+}\left(|\nabla\tilde{w}|^2+2\nabla\tilde{u}\cdot\nabla\bar{w}_1\right)z^a\,dX=c_3(n,a)\J_1(w)
+2\int_{\mathbb{R}^{n+1}_+}z^a\nabla\tilde{u}\cdot\nabla\bar{w}_1\,dX,
\end{equation*}
since $w$ has compact support in $\R$.

In the particular case $u,\,v\in C_c^\infty(\R)$ we already showed that the inf is $c_3(\J_1(v)-\J_1(u))$ and hence we get
\begin{equation*}
c_3(n,a)\J_1(w)
+2\int_{\mathbb{R}^{n+1}_+}z^a\nabla\tilde{u}\cdot\nabla\bar{w}_1\,dX
=c_3(n,a)\left(\J_1(u+w)-\J_1(u)\right),
\end{equation*}
for every $u,\,w\in C_c^\infty(\R)$. Then by approximation we find that this equality holds for all $u,\,w$ with $\J_1(u),\,\J_1(w)<\infty$,
concluding the proof.

\end{proof}
\end{lem}

As a consequence, if we restrict our attention to functions $v=\chi_F-\chi_{\Co F}$, we obtain the following
\begin{prop}\label{minim_functional}
The set $E$ is $s$-minimal in $B_1$ if and only if the extension $\tilde{u}$ of $u=\chi_E-\chi_{\Co E}$ satisfies
\begin{equation*}
\int_{\Omega_+}|\nabla\bar{v}|^2z^a\,dX\geq\int_{\Omega_+}|\nabla\tilde{u}|^2z^a\,dX,
\end{equation*}
for all bounded open sets $\Omega$ with Lipschitz boundary s.t. $\Omega_0\subset\subset B_1$
and all functions $\bar{v}$ that equal $\tilde{u}$ in a neighborhood of $\partial\Omega$ and take the values $\pm1$ on $\Omega_0$.
\end{prop}

\end{subsection}

\begin{subsection}{Monotonicity Formula}

Finally we are ready to define the promised quantity and prove the monotonicity formula.

Assume $E$ is $s$-minimal in $B_R$. For all $r<R$ we define the functional
\begin{equation}
\Phi_E(r):=\frac{1}{r^{n+a-1}}\int_{\mathcal{B}_r^+}|\nabla\tilde{u}|^2z^a\,dX,
\end{equation}
where
\begin{equation*}
u=\chi_E-\chi_{\Co E}.
\end{equation*}

\begin{lem}[Scale Invariance]
The functional $\Phi_E$ is scale invariant in the sense that the rescaled set $\lambda E$ satisfies
\begin{equation*}
\Phi_{\lambda E}(\lambda r)=\Phi_E(r).
\end{equation*}

\begin{proof}

Let $v:=\chi_{\lambda E}-\chi_{\Co(\lambda E)}$ and notice that $v(x)=u\left(\frac{x}{\lambda}\right)$.\\
Since
\begin{equation*}
P(x,z)=c_1\frac{\lambda^{1-a}}{\lambda^{n+1-a}}
\,\frac{\left(\frac{z}{\lambda}\right)^{1-a}}{\left(\left|\frac{x}{\lambda}\right|^2+\left(\frac{z}{\lambda}\right)^2\right)^\frac{n+1-a}{2}}
=\frac{1}{\lambda^n}
P\Big(\frac{x}{\lambda},\frac{z}{\lambda}\Big),
\end{equation*}
we have
\begin{equation*}
\tilde{v}(x,z)=\int_{\R}P(x-\xi,z)v(\xi)\,d\xi
=\frac{1}{\lambda^n}\int_{\R}P\left(\frac{x-\xi}{\lambda},\frac{z}{\lambda}\right)u\left(\frac{\xi}{\lambda}\right)\,d\xi
=\tilde{u}\left(\frac{x}{\lambda},\frac{z}{\lambda}\right),
\end{equation*}
and hence
\begin{equation*}
\nabla\tilde{v}(x,z)=\frac{1}{\lambda}\nabla\tilde{u}\left(\frac{x}{\lambda},\frac{z}{\lambda}\right).
\end{equation*}
Therefore
\begin{equation*}\begin{split}
\int_{\mathcal{B}_{\lambda r}^+}|\nabla\tilde{v}(x,z)|^2z^a\,dX&
=\frac{1}{\lambda^{2-a}}\int_{\mathcal{B}_{\lambda r}^+}\left|\nabla\tilde{u}\left(\frac{x}{\lambda},\frac{z}{\lambda}\right)\right|^2\left(\frac{z}{\lambda}\right)^a\,dX\\
&
=\lambda^{n+a-1}\int_{\mathcal{B}_r^+}|\nabla\tilde{u}(x,z)|^2z^a\,dX,
\end{split}
\end{equation*}
proving the claim.

\end{proof}
\end{lem}

\begin{lem}\label{bounded_monotonicity_functional}
There exists a constant $C=C(n,a)>0$ s.t.
\begin{equation*}
\Phi_E(r)\leq C
\end{equation*}
for every $r\leq R/2$.

\begin{proof}
From the inequality $(\ref{local_energy1})$, with $\Omega=\mathcal{B}_{1/2}$ we obtain
\begin{equation*}
\Phi_E\Big(\frac{1}{2}\Big)=\frac{1}{2^{n+a-1}}\int_{\mathcal{B}_{1/2}^+}|\nabla\tilde{u}|^2z^a\,dX\leq\frac{C}{2^{n+a-1}}
\J_1(u)=C P_s(E,B_1),
\end{equation*}
where $C=C(n,a)$.

Now, using the scaling invariance of $\Phi_E$ and the scaling of the $s$-perimeter, we get
\begin{equation*}\begin{split}
\Phi_E(r)&=\Phi_{\frac{1}{2r}E}\Big(\frac{1}{2}\Big)\leq C P_s\Big(\frac{1}{2r}E,B_1\Big)=C\frac{1}{(2r)^{n+a-1}}P_s(E,B_{2r})\\
&
\leq C\frac{1}{(2r)^{n+a-1}}P_s(B_{2r})\qquad\left(\textrm{since }E\textrm{ is }s\textrm{-minimal in }B_R\textrm{ and }2r\leq R\right)\\
&
=CP_s(B_1)=:C(n,a).
\end{split}\end{equation*}

\end{proof}
\end{lem}

%\begin{lem}
%If $0\in\partial E$, then for every $r\leq R$
%\begin{equation*}
%\Phi_E(r)\geq c,
%\end{equation*}
%for some $c=c(n,a)>0$ small.
%\begin{proof}
%\end{proof}
%\end{lem}

\begin{teo}[Monotonicity Formula]
Let $E$ be an $s$-minimal set in $B_R$. Then the function $r\longmapsto\Phi_E(r)$ is increasing.
\begin{proof}
Notice that $\Phi_E$ is continuous (see Remark $\ref{continuity_of_functional_mon}$ below) and differentiable at $r$ for almost every $r\in(0,R)$, with
\begin{equation*}\begin{split}
\frac{d}{dr}\Phi_E(r)&=-(n+a-1)\frac{1}{r^{n+a-2}}\int_{\mathcal{B}_r^+}|\nabla\tilde{u}|^2z^a\,dX\\
&
\qquad
+\frac{1}{r^{n+a-1}}\int_{(\partial\mathcal{B}_r)^+}|\nabla\tilde{u}|^2z^a\,d\mathcal{H}^n(X),
\end{split}
\end{equation*}
where
\begin{equation*}
(\partial\mathcal{B}_r)^+:=\partial\mathcal{B}_r\cap\{z>0\}.
\end{equation*}
To prove the claim we show that $\frac{d}{dr}\Phi_E(r)\geq0$. Due to the scale invariance, it is enough to prove the inequality for $r=1$, i.e. that
\begin{equation}\label{eq_some}
\int_{(\partial\mathcal{B}_1)^+}|\nabla\tilde{u}|^2z^a\,d\mathcal{H}^n(X)\geq
(n+a-1)\int_{\mathcal{B}_1^+}|\nabla\tilde{u}|^2z^a\,dX.
\end{equation}

To do so, we define the function $\bar{v}$ in $\mathbb{R}^{n+1}_+$, as
\begin{equation*}
\bar{v}(x,z):=\left\{\begin{array}{cc}\tilde{u}((1+\epsilon)(x,z)),&\textrm{in }\mathcal{B}_{1/(1+\epsilon)}^+,\\
\tilde{u}\big(\frac{(x,z)}{|(x,z)|}\big),&\textrm{in }\mathcal{B}_1^+\setminus\mathcal{B}_{1/(1+\epsilon)}^+,
\end{array}\right.
\quad\textrm{and }\bar{v}:=\tilde{u}\textrm{ in }\Co\mathcal{B}_1^+.
\end{equation*}
In particular the trace $v$ of $\bar{v}$ on $\{z=0\}$ is equal to $\chi_F-\chi_{\Co F}$, for some set $F$ which coincides with $E$ in $\R\setminus B_1$. Therefore the minimality of $E$ implies, thanks to Proposition $\ref{minim_functional}$,
\begin{equation*}
\int_{\mathcal{B}_1^+}|\nabla\bar{v}|^2z^a\,dX\geq\int_{\mathcal{B}_1^+}|\nabla\tilde{u}|^2z^a\,dX.
\end{equation*}
Moreover by construction the function $\bar{v}$ is constant along radial directions in the strip $\mathcal{B}_1^+\setminus\mathcal{B}_{1/(1+\epsilon)}^+$ and hence its gradient there is equal to
\begin{equation*}
\nabla\bar{v}(x,z)=\frac{1}{|(x,z)|}\nabla_\tau\tilde{u}\Big(\frac{(x,z)}{|(x,z)|}\Big),
\end{equation*}
where $\nabla_\tau$ denotes the tangential component of the gradient in $(\partial\mathcal{B}_1)^+$.\\
Thus we obtain
\begin{equation*}\begin{split}
&\int_{\mathcal{B}_1^+}|\nabla\tilde{u}|^2z^a\,dX
\leq
\int_{\mathcal{B}_{1/(1+\epsilon)}^+}|\nabla\bar{v}|^2z^a\,dX
+
\int_{\mathcal{B}_1^+\setminus\mathcal{B}_{1/(1+\epsilon)}^+}|\nabla\bar{v}|^2z^a\,dX\\
&
=\frac{1}{(1+\epsilon)^{n+a-1}}\int_{\mathcal{B}_1^+}|\nabla\tilde{u}|^2z^a\,dX
+
\int_{\mathcal{B}_1^+\setminus\mathcal{B}_{1/(1+\epsilon)}^+}\Big|\nabla_\tau\tilde{u}\Big(\frac{(x,z)}{|(x,z)|}\Big)\Big|^2\frac{z^a}{|(x,z)|^2}\,dX,
\end{split}
\end{equation*}
and hence
\begin{equation*}\begin{split}
\frac{1}{\epsilon}\Big(1&-\frac{1}{(1+\epsilon)^{n+a-1}}\Big)\int_{\mathcal{B}_1^+}|\nabla\tilde{u}|^2z^a\,dX\\
&
\leq
\frac{1}{\epsilon}\int_\frac{1}{1+\epsilon}^1 dt\int_{(\partial\mathcal{B}_t)^+}
\Big|\nabla_\tau\tilde{u}\Big(\frac{(x,z)}{|(x,z)|}\Big)\Big|^2\frac{z^a}{|(x,z)|^2}\,d\mathcal{H}^n(X).
\end{split}
\end{equation*}
Then passing to the limit as $\epsilon\to0$ gives
\begin{equation*}
\int_{(\partial\mathcal{B}_1)^+}|\nabla_\tau\tilde{u}|^2z^a\,d\mathcal{H}^n(X)
\geq(n+a-1)\int_{\mathcal{B}_1^+}|\nabla\tilde{u}|^2z^a\,dX.
\end{equation*}
Thus
\begin{equation}\label{homog_minim_eq}\begin{split}
\int_{(\partial\mathcal{B}_1)^+}&|\nabla\tilde{u}|^2z^a\,d\mathcal{H}^n(X)\\
&
\geq(n+a-1)\int_{\mathcal{B}_1^+}|\nabla\tilde{u}|^2z^a\,dX
+
\int_{(\partial\mathcal{B}_1)^+}|\nabla_\nu\tilde{u}|^2z^a\,d\mathcal{H}^n(X),
\end{split}\end{equation}
which implies $(\ref{eq_some})$, concluding the proof.

\end{proof}
\end{teo}

In particular notice that from $(\ref{homog_minim_eq})$ we obtain
\begin{equation*}
\frac{d}{dr}\Phi_E(r)=0\qquad\Longrightarrow\qquad\nabla_\nu\tilde{u}=0\qquad\textrm{on }(\partial\mathcal{B}_r)^+.
\end{equation*}
As a consequence we have the following
\begin{coroll}
The function $r\longmapsto\Phi_E(r)$ is constant if and only if $\tilde{u}$ is homogeneous of degree 0.
\end{coroll}

\end{subsection}
\end{section}
\begin{section}{Minimal Cones}

We study the blow-up limit $\lambda E$, as $\lambda\to\infty$, of a set $E$ which is $s$-minimal in $B_1$ and s.t. $0\in\partial E$, showing that it is an $s$-minimal cone $C$.

In particular we exploit the improvement of flatness to show that if $C$ is a half-space,
then $\partial E$ is $C^{1,\alpha}$ near 0.\\

We begin with the following technical result

\begin{prop}
Let $E_k\subset\R$ be $s$-minimal in $B_k$ for every $k\in\mathbb{N}$ and suppose $E_k\xrightarrow{loc}E$.
Then the corresponding extensions $\tilde{u}_k$, respectively $\tilde{u}$, satisfy

(i)$\qquad\tilde{u}_k\longrightarrow\tilde{u}$ uniformly on compact sets of $\mathbb{R}^{n+1}_+$,

(ii)$\qquad\nabla\tilde{u}_k\longrightarrow\nabla\tilde{u}$ in $L^2_{loc}(\mathbb{R}^{n+1}_+,z^a\,dxdz)$.\\
In particular $\Phi_{E_k}(r)\longrightarrow\Phi_E(r)$.
\begin{proof}

Notice that the functions $\tilde{u}_k$ are uniformly Lipschitz continuous on each compact set of $\{z>0\}$ (which is easily shown e.g. using the 
Poisson formula).

Consider a subsequence $\tilde{u}_{k_i}$ that converges uniformly on compact sets to a function $\tilde{v}$.
We want to show that $\tilde{v}=\tilde{u}$. The uniform convergence implies that also $\tilde{v}$ satisfies the equation
\begin{equation*}
\textrm{div}(z^a\nabla\tilde{v})=0,\qquad\textrm{in }\mathbb{R}^{n+1}_+,
\end{equation*}
and it is bounded. 
Thus, if we prove that the trace $v$ of $\tilde{v}$ is equal to $u=\chi_E-\chi_{\Co E}$ on $\{z=0\}$, then we get
$\tilde{v}=\tilde{u}$.

Let $r>0$. Fatou's Lemma gives
\begin{equation*}
\int_{\mathcal{B}_r^+}|\nabla\tilde{v}|^2z^a\,dX\leq\liminf_{i\to\infty}
\int_{\mathcal{B}_r^+}|\nabla\tilde{u}_{k_i}|^2z^a\,dX
\leq C r^{n+a-1},
\end{equation*}
the last inequality being a consequence of Lemma $\ref{bounded_monotonicity_functional}$ (notice that for $i$ big enough the set
$E_{k_i}$ is $s$-minimal in $B_{2r}$).\\
Then near the boundary $\{z=0\}$ of $\mathbb{R}^{n+1}_+$, using Holder's inequality we get
\begin{equation*}\begin{split}
\int_{\mathcal{B}_r\cap\{0<z<\delta\}}|\nabla(\tilde{u}_{k_i}-\tilde{v})|\,dX&
=\int_{\mathcal{B}_r\cap\{0<z<\delta\}}z^{-\frac{a}{2}}|\nabla(\tilde{u}_{k_i}-\tilde{v})|z^\frac{a}{2}\,dX\\
&
\leq\Big(\frac{\mathcal{L}^n(B_r)}{1-a}\Big)^\frac{1}{2}\delta^\frac{1-a}{2}
\Big(\int_{\mathcal{B}_r^+}|\nabla(\tilde{u}_{k_i}-\tilde{v})|^2z^a\,dX\Big)^\frac{1}{2}\\
&
\leq C\delta^\frac{1-a}{2},
\end{split}\end{equation*}
with $C$ depending on $r$, but not on $\delta$ or $k_i$. On the other hand $\nabla\tilde{u}_{k_i}$ converges
to $\nabla\tilde{v}$ uniformly on compact sets of $\mathbb{R}^{n+1}_+$. Therefore
\begin{equation*}
\int_{\mathcal{B}_r^+}|\nabla(\tilde{u}_{k_i}-\tilde{v})|\,dX
\leq C\delta^\frac{1-a}{2}+\int_{\mathcal{B}_r\cap\{z\geq\delta\}}|\nabla(\tilde{u}_{k_i}-\tilde{v})|\,dX,
\end{equation*}
with the last term going to 0 as $i\to\infty$. Thus taking the limit gives
\begin{equation*}
\limsup_{i\to\infty}\int_{\mathcal{B}_r^+}|\nabla(\tilde{u}_{k_i}-\tilde{v})|\,dX\leq C\delta^\frac{1-a}{2},
\end{equation*}
for every $\delta>0$ small. Since $\delta$ is arbitrary, we see that $\tilde{u}_{k_i}$ converges to $\tilde{v}$ in $W^{1,1}(\mathcal{B}_r^+)$ and this implies the convergence of the traces $u_{k_i}\longrightarrow v$ in $L^1(B_r)$.\\
Since this holds for every $r>0$, we get $v=u$, as wanted,
proving $(i)$.

Now we prove $(ii)$. It is enough to prove the convergence in $\mathcal{B}_r^+$ for every $r>0$. From inequality $(\ref{local_energy1})$ we get
\begin{equation*}
\limsup_{k\to\infty}\int_{\mathcal{B}_r^+}|\nabla(\tilde{u}_k-\tilde{u})|^2z^a\,dX\leq C\limsup_{k\to\infty}\J_{2r}(u_k-u).
\end{equation*}
We want to show that $\J_{2r}(u_k-u)\longrightarrow0$.

Define the functions
\begin{equation*}
f_k(x,y):=\frac{u_k(x)-u_k(y)}{|x-y|^\frac{n+s}{2}}\chi_{B_{2r}}(x)\big(\chi_{B_{2r}}(y)+\sqrt{2}\chi_{\Co B_{2r}}(x)\big),
\end{equation*}
so that
\begin{equation*}
\|f_k\|_{L^2(\R\times\R)}^2=\J_{2r}(u_k)=8P_s(E_k,B_{2r}).
\end{equation*}
According to Theorem $\ref{nonlocal_compactness}$
\begin{equation*}
\lim_{k\to\infty}P_s(E_k,B_{2r})=P_s(E,B_{2r})
=\frac{1}{8}\|f\|_{L^2(\R\times\R)}^2,
\end{equation*}
with
\begin{equation*}
f(x,y):=\frac{u(x)-u(y)}{|x-y|^\frac{n+s}{2}}\chi_{B_{2r}}(x)\big(\chi_{B_{2r}}(y)+\sqrt{2}\chi_{\Co B_{2r}}(x)\big).
\end{equation*}
Moreover, since $u_k\longrightarrow u$ in $L^1_{loc}(\R)$, from every subsequence of $\{u_k\}$ we can extract a subsequence
$\{u_{k_i}\}$ converging pointwise (almost everywhere) to $u$, and hence also $f_{k_i}$ converges pointwise to $f$.\\
Then for every such subsequence we have the standard implication
\begin{equation*}\begin{split}
f_{k_i}\longrightarrow f\quad&\textrm{a.e. in } \R\times\R,\quad\|f_{k_i}\|_{L^2(\R\times\R)}
\longrightarrow\|f\|_{L^2(\R\times\R)}\\
&
\Longrightarrow\quad f_{k_i}\longrightarrow f\quad\textrm{in }L^2(\R\times\R),
\end{split}
\end{equation*}
and hence
\begin{equation*}
\J_{2r}(u_{k_i}-u)=\|f_{k_i}-f\|_{L^2(\R\times\R)}^2\longrightarrow0.
\end{equation*}

This proves the claim.\\
Indeed, suppose there exists a subsequence of $\{u_k\}$ (we relabel it for simplicity) s.t.
$\J_{2r}(u_k-u)\geq\epsilon$ for every $k$. Then what we have just shown proves that we can extract a subsequence $\{u_{k_i}\}$
s.t. $\J_{2r}(u_{k_i}-u)\longrightarrow0$, giving a contradiction.

\end{proof}
\end{prop}

\begin{rmk}\label{general_seq_conv}
The same Proposition remains true if we consider a sequence $E_k\xrightarrow{loc}E$ of sets $E_k$ $s$-minimal
in $B_{\lambda_k}$, with $\lambda_k\longrightarrow\infty$. Thus in particular when we consider the blow-up sequence
$E_k:=\lambda_k E$ of a set $E$ which is $s$-minimal in $B_1$, with $0\in\partial E$, provided that such a sequence admits a limit. 
\end{rmk}

\begin{rmk}\label{continuity_of_functional_mon}
Exploiting the same argument used in the proof of $(ii)$ we can show that the functional $\Phi_E$ is continuous in $r$.\\
Indeed, let $E$ be $s$-minimal in $B_R$ and take $\tilde{r}\in(0,R)$; we want to show that
\begin{equation*}
\lim_{r\to\tilde{r}}\Phi_E(r)=\Phi_E(\tilde{r}).
%\frac{1}{r^{n+a-1}}\int_{\mathcal{B}_r^+}|\nabla\tilde{u}|^2z^a\,dX=
%\frac{1}{\tilde{r}^{n+a-1}}\int_{\mathcal{B}_{\tilde{r}}^+}|\nabla\tilde{u}|^2z^a\,dX,
\end{equation*}
%where $u=\chi_E-\chi_{\Co E}$.
Using the scaling invariance we have
\begin{equation*}
\Phi_E(r)=\Phi_{\frac{\tilde{r}}{r}E}(\tilde{r}),
\end{equation*}
so it is enough to show that
\begin{equation*}
\nabla\tilde{u}_r\longrightarrow\nabla\tilde{u}\quad\textrm{in}\quad
L^2(\mathcal{B}_{\tilde{r}}^+,z^a\,dxdz),
\end{equation*}
where $u_r:=\chi_{\frac{\tilde{r}}{r}E}-\chi_{\Co(\frac{\tilde{r}}{r}E)}$
and $u=\chi_E-\chi_{\Co E}$.\\
Notice that the set $\frac{\tilde{r}}{r}E$ is $s$-minimal in $B_{\tilde{r}\frac{R}{r}}$ and $\frac{\tilde{r}}{r}E\xrightarrow{loc}E$ as
$r\to\tilde{r}$.\\
Therefore taking a small $\delta>0$, both the sets $E$ and $\frac{\tilde{r}}{r}E$ are $s$-minimal in $B_{\tilde{r}+\epsilon}\subset B_R$, for every $|r-\tilde{r}|\leq\delta$, for some very small $0<\epsilon<R-\tilde{r}$.\\
Now we have
\begin{equation*}
\limsup_{r\to\tilde{r}}\int_{\mathcal{B}_{\tilde{r}}^+}|\nabla(\tilde{u}_r-\tilde{u})|^2z^a\,dX\leq
C \limsup_{r\to\tilde{r}}\J_{\tilde{r}+\epsilon}(u_r-u),
\end{equation*}
and reasoning as above proves the claim.

\end{rmk}

Now we can use the pointwise convergence of the functions $\Phi_{E_k}$ to show that the blow-up limit of $E$ in 0
(if it exists) is a cone.

We recall that a set $C$ is a cone (with vertex in 0) if for every $t>0$ we have $tC=C$.
Moreover we say that $C$ is an $s$-minimal cone if it is locally $s$-minimal in $\R$, meaning that $C$ is $s$-minimal in every ball $B\subset\R$.

\begin{teo}[Blow-up Limit]\label{blowup_teo}
Let $E\subset\R$ be $s$-minimal in $B_1$ with $0\in\partial E$ and let $\lambda_k\longrightarrow\infty$ be a sequence s.t.
\begin{equation*}
\lambda_k E\xrightarrow{loc}C.
\end{equation*}
Then $C$ is an $s$-minimal cone.
\begin{proof}
Theorem $\ref{nonlocal_compactness}$ proves that $C$ is $s$-minimal in every ball $B\subset\R$.\\
Using previous Proposition and Remark $\ref{general_seq_conv}$ we get
\begin{equation*}
\Phi_E\Big(\frac{r}{\lambda_k}\Big)=\Phi_{\lambda_k E}(r)\xrightarrow{k\to\infty}\Phi_C(r),
\end{equation*}
for every $r>0$. Therefore $\Phi_C$ is a constant function, with
\begin{equation*}
\Phi_C(r)=\lim_{t\to0}\Phi_E(t),
\end{equation*}
the existence of the limit being guaranteed by the monotonicity of $\Phi_E$.\\
Since $\Phi_C$ is constant, we conclude that the extension $\tilde{u}_C$, and hence also its trace $u_C=\chi_C-\chi_{\Co C}$, is homogeneous of degree 0,
proving that $C$ is a cone.

\end{proof}
\end{teo}

\begin{defin}
We say that a cone $C$ as in Theorem $\ref{blowup_teo}$ is a tangent cone for $E$ at 0.
\end{defin}

Now we prove the existence of tangent cones

\begin{prop}[Existence of Blow-up Limits]\label{blowup_exist}
Let $E\subset\R$ be $s$-minimal in $B_1$ with $0\in\partial E$ and let $\lambda_k\longrightarrow\infty$. Then there exist
an $s$-minimal cone $C$ and a subsequence $\{\lambda_{k_i}\}$ of $\{\lambda_k\}$ s.t.
\begin{equation*}
\lambda_{k_i}E\xrightarrow{loc}C.
\end{equation*}

\begin{proof}
We want to use Theorem $\ref{compact_embd_th}$ to construct a limit set via a diagonal argument. Then previous Theorem shows that it is an $s$-minimal cone.

Since we are working with subsequences, we can as well assume that $\lambda_k\nearrow\infty$.\\
Let $h\in\mathbb{N}$. Notice that $\lambda_k E$ is $s$-minimal in $B_{\lambda_k}$. Then for $k\geq k(h)$ we have $\lambda_k\geq h$ and hence in particular $\lambda_k E$ is $s$-minimal in $B_h$.
For every such $k$ the minimality implies
\begin{equation*}\begin{split}
[\chi_{\lambda_kE}]_{W^{s,1}(B_h)}&=2P^L_s(\lambda_k E,B_h)\leq2P_s(\lambda_k E,B_h)\\
&
\leq2P_s((\lambda_kE)\setminus B_h,B_h)\\
&
=2\Ll_s((\lambda_kE)\setminus B_h,\Co((\lambda_kE)\setminus B_h)\cap B_h)\\
&
\leq2\Ll_s(\Co B_h,B_h)=2P_s(B_h),
\end{split}\end{equation*}
and clearly
\begin{equation*}
\|\chi_{\lambda_kE}\|_{L^1(B_h)}\leq|B_h|.
\end{equation*}
Therefore
\begin{equation*}
\forall\, h\quad\exists\, k(h)\quad\textrm{s.t.}\quad\|\chi_{\lambda_kE}\|_{W^{s,1}(B_h)}\leq c_h<\infty,\quad\forall\, k\geq k(h).
\end{equation*}
Thus Theorem $\ref{compact_embd_th}$ guarantees the existence of a subsequence $\{\lambda_{k_i}\}$ (with $k_1\geq k(h)$)
s.t.
\begin{equation*}
(\lambda_{k_i}E)\cap B_h\xrightarrow{i\to\infty}E^h,
\end{equation*}
in measure,
for some set $E^h\subset B_h$.

Applying this argument for $h=1$ we get a subsequence $\{\lambda^1_k\}$ of $\{\lambda_k\}$ with
\begin{equation*}
(\lambda_k^1E)\cap B_1\longrightarrow E^1.
\end{equation*}
Applying again the argument in $B_2$, with $\{\lambda_k^1\}$ in place of $\{\lambda_k\}$, we get a subsequence
$\{\lambda^2_k\}$ of $\{\lambda^1_k\}$, with
\begin{equation*}
(\lambda_k^2E)\cap B_2\longrightarrow E^2.
\end{equation*}
Notice that we must have $E^1\subset E^2$ in measure (by the uniqueness of the limit in $B_1$). We can also suppose that $\lambda_1^2>\lambda^1_1$.\\
Proceeding inductively in this way we get a subsequence $\{\lambda_1^k\}$ of $\{\lambda_k\}$ s.t.
\begin{equation*}
(\lambda_1^kE)\cap B_h\xrightarrow{k\to\infty} E^h,\qquad\textrm{for every }h\in\mathbb{N},
\end{equation*}
with $E^h\subset E^{h+1}$. Therefore if we define $C:=\bigcup_hE^h$ we get
\begin{equation*}
\lambda_1^kE\xrightarrow{loc}C,
\end{equation*}
concluding the proof.

\end{proof}\end{prop}

We remark that, as in the classical setting of Caccioppoli sets, two different subsequences might converge to different cones i.e. tangent cones need not be unique.

However this can happen only at singular points: if $\partial E$ is regular in a neighborhood of 0, then
the tangent cone is necessarily the half-space
\begin{equation*}
H^-(0)=\{x\in\R\,|\,x\cdot\nu_E(0)\leq0\}.
\end{equation*}
Moreover, exploiting the improvement of flatness, we can show that if $E$ has a half-space $C$ as tangent cone, then 
$\partial E$ is regular in a neighborhood of 0.
Thus in particular
$C=H^-(0)$ is the unique tangent cone at 0.

\begin{teo}[Regularity]\label{tangent_cone_reg}
Let $E\subset\R$ be $s$-minimal in $B_1$ with $0\in\partial E$. If $E$ has a half-space as a tangent cone,
then $\partial E$ is a $C^{1,\alpha}$ surface in a neighborhood of 0.

\begin{proof}
Let $C$ be the tangent half-space of the hypothesis and let $\lambda_k\nearrow\infty$ be s.t. $\lambda_kE\xrightarrow{loc}C$. Then Corollary $\ref{haus_conv_min}$ implies that for every $\epsilon>0$
\begin{equation*}
\partial(\lambda_k E)\cap B_1\subset N_\epsilon(\partial C)\cap B_1\subset N_\epsilon(\partial C),\qquad\textrm{for }k\geq k(\epsilon).
\end{equation*}
Since $C$ is a half-space, up to rotation we have
\begin{equation*}
N_\epsilon(\partial C)=\{|x\cdot e_n|<\epsilon\}.
\end{equation*}

Taking $\epsilon=\frac{\epsilon_0}{2}$, we see that $\lambda_{k(\epsilon_0)}E$ satisfies the hypothesis of Theorem $\ref{flat_reg_teo1}$,
and hence $\partial(\lambda_{k(\epsilon_0)}E)=\lambda_{k(\epsilon_0)}\partial E$ is a $C^{1,\alpha}$ surface in $B_{1/2}$.\\
Therefore scaling back we see that $\partial E$ is a $C^{1,\alpha}$ surface in $B_{1/2\lambda_{k(\epsilon_0)}}$.

\end{proof}
\end{teo}

\begin{defin}
Let $E\subset\R$ be $s$-minimal in $\Omega$. A point $x_0\in\partial E\cap\Omega$ that has a half-space as a tangent cone
is called a regular point. The points in $\partial E\cap\Omega$ that are not regular are called singular points.
\end{defin}

For a minimal cone $C$ we denote by $\Phi_C$ its energy, i.e. the constant value of the function $\Phi_C(r)$.
We show that half-spaces have minimal energy amongst minimal cones, with a gap separating their energy from those of other cones. Let $\Pi:=\{x_1>0\}$ be a half-space.

\begin{teo}[Energy Gap]
Let $C$ be an $s$-minimal cone. Then
\begin{equation}\label{energy_gap1}
\Phi_C\geq\Phi_\Pi.
\end{equation}
Moreover, if $C$ is not a half-space, then
\begin{equation}\label{energy_gap2}
\Phi_C\geq\Phi_\Pi+\delta_0,
\end{equation}
where $\delta_0>0$ is a constant depending only on $n$ and $s$.

\begin{proof}
We have $0\in\partial C\cap B_1$ and hence the Clean ball condition (Corollary $\ref{clean_ball}$)
guarantees the existence of some small ball $B\subset E\cap B_1$. Sliding $B$ vertically until we first touch $\partial C$
we find a point $x_0\in\partial C$ having an interior tangent ball. Corollary $\ref{int_tang_flat_reg}$ then implies that $\partial C$ is $C^{1,\alpha}$ in a neighborhood of $x_0$ and hence the tangent cone of $C$ at $x_0$ is a half-space. Therefore
\begin{equation*}
\lim_{r\to0}\Phi_{E-x_0}(r)=\Phi_\Pi.
\end{equation*}
On the other hand, since $\frac{1}{k}(C-x_0)=C-\frac{1}{k}x_0$, we obtain
\begin{equation*}
\frac{1}{k}(C-x_0)\xrightarrow{loc}C
\end{equation*}
and hence
\begin{equation*}
\Phi_{C-x_0}(k)\longrightarrow\Phi_C,\quad\textrm{as }k\to\infty.
\end{equation*}
The monotonicity of $\Phi_{C-x_0}$ gives $(\ref{energy_gap1})$. We have equality only when $\Phi_{C-x_0}$ is constant, i.e.
when $C-x_0$ is a cone, thus when $C-x_0$ is a half-space,
which in turn implies that $C$ is a half-space.

To prove $(\ref{energy_gap2})$ we use a compactness argument.\\
Assume by contradiction that there exist minimal cones $C_k$ with
\begin{equation}\label{energy_gap_contr_ass}
\Phi_{C_k}\leq\Phi_\Pi+\frac{1}{k},\quad\textrm{for every }k\in\mathbb{N},
\end{equation}
that are not half-spaces. Arguing as in the proof of Proposition $\ref{blowup_exist}$ we can find a convergent subsequence
\begin{equation*}
C_{k_i}\xrightarrow{loc} C_0,
\end{equation*}
for some minimal cone $C_0$. Since $\Phi_{C_0}\geq\Phi_\Pi$, from $(\ref{energy_gap_contr_ass})$ we get
$\Phi_{C_0}=\Phi_\Pi$ and hence $C_0$ is a half-space.\\
Therefore, using again Corollary $\ref{haus_conv_min}$ we get (up to rotation)
\begin{equation*}
\partial C_{k_i}\cap B_1\subset\{|x\cdot e_n|\leq\epsilon_0\},
\end{equation*}
for all $k_i$ large enough.\\
Then Theorem $\ref{flat_reg_teo1}$ implies that $\partial C_{k_i}$ are $C^{1,\alpha}$ surfaces around 0. Thus we find that
$C_{k_i}$ is a half-space for all large $k_i$, giving a contradiction.

\end{proof}
\end{teo}

\end{section}

\begin{section}{Dimension Reduction}

In this section we adapt Federer classical reduction argument to estimate the size of the singular set of an
$s$-minimal set $E$.

We will need the following result from \cite{cones}, which states that in dimension 2 there are no singular minimal cones
\begin{teo}\label{min2cones}
If $E$ is an $s$-minimal cone in $\mathbb{R}^2$, then $E$ is a half-space.
\end{teo}

We remark that if $\partial C$ is singular in $x_0\not=0$, then also the vertex 0 is a singular point.
Indeed, if $\partial C$ is regular in 0, then $C$ must be a half space, and hence $\partial C$ is a plane, which is regular
in every point.

The dimension reduction result is the following: if $C\subset\R$ is a minimal cone with a singularity in $x_0\in\partial C$, with $x_0\not=0$, then we can find a minimal cone
%$\lambda_k(C-x_0)$ converges to $K\times\mathbb{R}$, where
$K\subset\mathbb{R}^{n-1}$ which is singular in $0$.

%we prove that the Hausdorff dimension of the singular set is at most $n-3$.

Roughly speaking, using dimension reduction we can inductively reduce the dimension of the ambient space
until we are left with a minimal cone which is singular only in the vertex 0 and
from Theorem $\ref{min2cones}$ we see that in dimension 3 singular minimal cones can be singular
only in the vertex 0 (otherwise we could find a singular minimal cone in dimension 2).

As a consequence we will prove (see Corollary 2 in \cite{cones}) that the Hausdorff dimension of the singular set is at most $n-3$.\\

We begin with the following technical result
\begin{lem}\label{interpol}
Let $\bar{w}$ be a bounded function defined in $\mathcal{B}_1^+\subset\mathbb{R}^{n+1}$ s.t. $\bar{w}=0$
in a neighborhood of $\partial\mathcal{B}_1$ and
\begin{equation*}
\int_{\mathcal{B}_1^+}|\nabla\bar{w}|^2z^a\,dX<\infty.
\end{equation*}
There exists a function $\mathcal{W}=\mathcal{W}(x,x_{n+1},z)$ defined in
\begin{equation*}
\mathcal{B}_1^+\times[-1,1]=\{(x,x_{n+1},z)\in\mathbb{R}^{n+2}\,|\,(x,z)\in\mathcal{B}_1^+,\,x_{n+1}\in[-1,1]\},
\end{equation*}
with the following properties
\begin{equation}\label{sharp_switch}\begin{split}
&(i)\quad\mathcal{W}=0,\textrm{ if }x_{n+1}<-\frac{1}{2},\\
&
(ii)\quad\mathcal{W}(x,x_{n+1},z)=\bar{w}(x,z),\textrm{ if }x_{n+1}>\frac{1}{2},\\
&
(iii)\quad\mathcal{W}=0\textrm{ in a neighborhood of }\partial\mathcal{B}_1^+\times[-1,1],\\
&
(iv)\quad\mathcal{W}(x,x_{n+1},0)=\left\{\begin{array}{cc}
0&\textrm{if }x_{n+1}\leq0,\\
\bar{w}(x,0)&\textrm{if }x_{n+1}>0,
\end{array}\right.\\
&
(v)\quad\int_{\mathcal{B}_1^+\times[-1,1]}|\nabla\mathcal{W}|^2z^a\,d\mathcal{X}<\infty,
\end{split}\end{equation}
where $\mathcal{X}=(x,x_{n+1},z)\in\mathbb{R}^{n+2}$.
\begin{proof}
First we assume that $0\leq\bar{w}\leq1$. Moreover we can think $\bar{w}$ is defined in $\mathbb{R}^{n+2}$ and is constant in
the $x_{n+1}$ variable. Let $\pi$ be the extension in $\mathbb{R}^{n+2}$ corresponding to
$\chi_{\{x_{n+1}>0\}}$. Then the function
\begin{equation*}
\mathcal{W}_1:=\min\{\bar{w},\pi\}
\end{equation*}
satisfies the last three points of $(\ref{sharp_switch})$.

Now we modify $\mathcal{W}_1$ so that points $(i)$ and $(ii)$ of $(\ref{sharp_switch})$ also hold.\\
Let $\phi_1$ be a smooth cutoff function on $\mathbb{R}$, with $\phi_1=0$ outside $[-1/2,1/2]$ and
$\phi_1=1$ on $[-1/4,1/4]$. Now define $\phi_2:=1-\phi_1$ on $[0,\infty)$ and
$\phi_2:=0$ on $(-\infty,0)$.
Then
\begin{equation*}
\mathcal{W}:=\phi_1(x_{n+1})\mathcal{W}_1+\phi_2(x_{n+1})\bar{w}
\end{equation*}
satisfies all the properties in $(\ref{sharp_switch})$.\\
The general case follows by repeating this construction for the (scaled) positive and negative parts,
$\bar{w}^+$ and $\bar{w}^-$, and then subtracting the functions we obtain.

\end{proof}
\end{lem}

We will exploit this result in the proof of the following

\begin{teo}
The set $E\subset\R$ is locally $s$-minimal in $\R$ if and only if the set $E\times\mathbb{R}$ is locally $s$-minimal in $\mathbb{R}^{n+1}$.
\begin{proof}
First of all, notice that if $\tilde{u}=\tilde{u}(x,z)$ is the extension in $\mathbb{R}^{n+1}$ for
the function
$\chi_E-\chi_{\Co E}$, then by making $\tilde{u}$ to be constant in the $x_{n+1}$ variable
we obtain the extension in $\mathbb{R}^{n+2}$ corresponding to $E\times\mathbb{R}$.

$\Longrightarrow)$ Assume $E$ is locally $s$-minimal in $\R$.

As we remarked above, the extension in $\mathbb{R}^{n+2}$ corresponding to $E\times\mathbb{R}$ is
the function
$\mathcal{U}(x,x_{n+1},z)=\tilde{u}(x,z)$, which is constant in the $x_{n+1}$ variable. Therefore
$|\nabla\mathcal{U}|=|\nabla_X\mathcal{U}|=|\nabla\tilde{u}|$, where $\nabla_X$ denotes the gradient
in the $(x,z)$ variables.\\
Also notice that for any function $\bar{v}(x,x_{n+1},z)$ we have $|\nabla\bar{v}|^2\geq|\nabla_X\bar{v}|^2$
for every fixed $x_{n+1}$. 
Now suppose that $\bar{v}$ is s.t. supp$(\bar{v}-\mathcal{U})\subset Q$, where $Q$ is a bounded open cube in $\mathbb{R}^{n+2}$,
and the trace of $\bar{v}$ on $\{z=0\}$ takes only the values $\pm1$.

Let $Q_t:=Q\cap\{x_{n+1}=t\}$ and $Q_t^+:=Q_t\cap\{z>0\}$.
Then slicing we get
\begin{equation*}
\int_{Q}|\nabla\bar{v}|^2z^a\,d\mathcal{X}\geq\int\Big(\int_{Q_t}|\nabla_X\bar{v}|^2z^a\,d X\Big)dt,
\end{equation*}
and the minimality of $E$ implies, as in Proposition $\ref{minim_functional}$,
\begin{equation*}
\int_{Q_t^+}|\nabla_X\bar{v}|^2z^a\,d X\geq
\int_{Q_t^+}|\nabla\tilde{u}|^2z^a\,d X=\int_{Q_t^+}|\nabla\mathcal{U}|^2z^a\,d X.
\end{equation*}
Therefore
\begin{equation*}
\int_{Q_+}|\nabla\bar{v}|^2z^a\,d\mathcal{X}
\geq
\int_{Q_+}|\nabla\mathcal{U}|^2z^a\,d\mathcal{X},
\end{equation*}
which implies the minimality of $E\times\mathbb{R}$ in $Q$, again via Proposition $\ref{minim_functional}$,
and hence the local $s$-minimality of $E\times\mathbb{R}$ in $\mathbb{R}^{n+2}$.

$\Longleftarrow)$ Assume $E\times\mathbb{R}$ is locally $s$-minimal in $\mathbb{R}^{n+2}$.

Let $\bar{v}(x,z)$ be s.t. supp$(\bar{v}-\tilde{u})\subset\mathcal{B}_R\subset\mathbb{R}^{n+1}$
and the trace of $\bar{v}$ on $\{z=0\}$ takes only the values $\pm1$. We want to show
\begin{equation}\label{cylindersfun}
\int_{\mathcal{B}_R^+}|\nabla\bar{v}|^2z^a\,dX\geq
\int_{\mathcal{B}_R^+}|\nabla\tilde{u}|^2z^a\,dX.
\end{equation}
We can suppose that the first integral is finite. Moreover the local minimality of $E\times\mathbb{R}$
implies
\begin{equation*}
\int_{-1}^1\Big(\int_{\mathcal{B}_R^+}|\nabla\tilde{u}|^2z^a\,dX\Big)dx_{n+1}
=
\int_{\mathcal{B}_R^+\times[-1,1]}|\nabla\mathcal{U}|^2z^a\,d\mathcal{X}<\infty,
\end{equation*}
so the second integral above is also finite.

Using previous Lemma, we can construct a competitor for $\mathcal{U}$
and obtain $(\ref{cylindersfun})$ by confronting the corresponding energies.\\
Let $c>1$ and let $\mathcal{W}$ be the function obtained in Lemma $\ref{interpol}$ for $\bar{w}:=\tilde{u}-\bar{v}$.
Now we consider the
function $\mathcal{V}(x,x_{n+1},z)$, defined in
$\mathcal{D}:=\mathcal{B}_R^+\times[-(c+1),c+1]$ as
\begin{equation*}
\mathcal{V}(x,x_{n+1},z):=\left\{
\begin{array}{cc}
\bar{v}(x,z),&\textrm{if }|x_{n+1}|\leq c-1,\\
\bar{v}(x,z)+\mathcal{W}(x,|x_{n+1}|-c,z),&\textrm{if }-1<|x_{n+1}|-c\leq1.
\end{array}\right.
\end{equation*}
This function $\mathcal{V}$ is a competitor for $\mathcal{U}$ in $\mathcal{D}_+:=\mathcal{D}\cap\{z>0\}$
in the sense of Proposition $\ref{minim_functional}$.
Moreover
notice that $\mathcal{V}$ is even in the $x_{n+1}$ variable
and the energy is s.t.
\begin{equation*}
\int_{\mathcal{B}_R^+\times[c-1,c+1]}|\nabla\mathcal{V}|^2z^a\,d\mathcal{X}=\Lambda
\end{equation*}
is independent of $c$ and
\begin{equation*}
\int_{\mathcal{B}_R^+\times[0,c-1]}|\nabla\mathcal{V}|^2z^a\,d\mathcal{X}
=(c-1)\int_{\mathcal{B}_R^+}|\nabla\bar{v}|^2z^a\,dX.
\end{equation*}
Therefore
the minimality of $E\times\mathbb{R}$ gives
\begin{equation*}\begin{split}
2(c+1)&\int_{\mathcal{B}_R^+}|\nabla\tilde{u}|^2z^a\,dX=
\int_{-(c+1)}^{c+1}\Big(\int_{\mathcal{B}_R^+}|\nabla\tilde{u}|^2z^a\,dX\Big)dx_{n+1}\\
&
=\int_{\mathcal{D}_+}|\nabla\mathcal{U}|^2z^a\,d\mathcal{X}
\leq
\int_{\mathcal{D}_+}|\nabla\mathcal{V}|^2z^a\,d\mathcal{X}\\
&
=2(c-1)\int_{\mathcal{B}_R^+}|\nabla\bar{v}|^2z^a\,dX+2\Lambda.
\end{split}\end{equation*}
Dividing by $2c$ and letting $c\to\infty$ we get $(\ref{cylindersfun})$, concluding the proof.

\end{proof}
\end{teo}

Now we prove the dimension reduction result.
The idea is to blow-up a singular cone in correspondence of a non-vertex singularity.
The tangent cone thus obtained is a cylinder whose base is a cone which is singular in the vertex.
\begin{teo}[Dimension Reduction]\label{dimeredteo}
Let $C$ be an $s$-minimal cone in $\R$, with $x_0=e_n\in\partial C$ (and vertex in 0).
From any sequence converging to $\infty$ we can extract a subsequence $\lambda_k\to\infty$ s.t.
\begin{equation*}
\lambda_k(C-x_0)\xrightarrow{loc} A\times\mathbb{R},
\end{equation*}
where $A$ is an $s$-minimal cone in $\mathbb{R}^{n-1}$.\\
Moreover, if $x_0$ is a singular point for $\partial C$ then 0 is a singular point for $A$.
\begin{proof}
We already know (Proposition $\ref{blowup_exist}$) that a blow-up limit $D$ exists and that it is an $s$-minimal cone
(Theorem $\ref{blowup_teo}$). Suppose that $D=A\times\mathbb{R}$. Then previous Theorem implies
that $A$ is locally $s$-minimal in $\mathbb{R}^{n-1}$; moreover, since $A\times\mathbb{R}$ is a cone,
also $A$ must be a cone. \\
As for the last claim, we remark that $A\times\mathbb{R}$ is the tangent cone for $C$ at $x_0$. Therefore
if it were regular in 0, the cone $C$ would be regular in $x_0$ (Theorem $\ref{tangent_cone_reg}$).
To conclude, notice that 0 is a regular point of $A\times\mathbb{R}$ if and only if 
0 is regular for $A$.

We are left to show that the blow-up limit $D$ is of the form $D=A\times\mathbb{R}$, i.e. that $D$ is constant in the $e_n$ direction.
To do this, we show that if $x$ is an interior point of $D$, then the whole line $L:=\{x+te_n\,|\,t\in\mathbb{R}\}$
is included in the interior of $D$.

Indeed, let $x$ be an interior point of $D$ i.e. $B_\epsilon(x)\subset D$. Then by uniform density
(exploiting Corollary $\ref{haus_conv_min}$) we have
\begin{equation*}
B_{\epsilon/2}(x)\subset C_k:=\lambda_k(C-x_0),
\end{equation*}
for all $k$ big enough.

Notice that $C_k=C-\lambda_ke_n$ is a cone with vertex in $-\lambda_ke_n$. Let $T_k$ be the cone with vertex in
$-\lambda_ke_n$ generated by $B_{\epsilon/2}(x)$. Then $T_k\subset C_k$ and
\begin{equation*}
T_k\xrightarrow{loc}\bigcup_{t\in\mathbb{R}}\big(B_{\epsilon/2}(x)+te_n\big)=N_{\epsilon/2}(L).
\end{equation*}
As a consequence we get $N_{\epsilon/2}(L)\subset D$, as wanted, concluding the proof.

\end{proof}\end{teo}

\begin{rmk}
Since there are no singular $s$-minimal cones in dimension 2, previous Theorem implies that $s$-minimal cones in $\mathbb{R}^3$
can have at most one singularity, in the origin.
\end{rmk}

Finally we are ready to estimate the dimension of the singular set.\\
The argument is the same as the one for classical minimal surfaces.

We recall some definitions and results about Hausdorff measures which we will need in the proof. For the details we refer
to e.g. \cite{Maggi} or Chapter 11 of \cite{Giusti}.

Let $E\subset\R$, $d\in[0,\infty)$ and $\delta\in(0,\infty]$. Then
\begin{equation*}
\h_\delta^d(E):=\frac{\omega_d}{2^d}\inf\Big\{\sum_{j=1}^\infty (\textrm{diam }S_j)^d\,\big|\,
E\subset\bigcup_{j=1}^\infty S_j,\,\textrm{diam }S_j<\delta\Big\},
\end{equation*}
and
\begin{equation*}
\h^d(E):=\lim_{\delta\to0}\h_\delta^d(E)=\sup_\delta\h^d_\delta(E).
\end{equation*}
It is easy to show that
\begin{equation}\label{hauseq1}
\h^d_\infty(E)=0\quad\Longleftrightarrow\quad\h^d(E)=0,
\end{equation}
and
\begin{equation}\label{hauseq2}
\h^d_\infty(E)=0\quad\Longrightarrow\quad\h^{d+1}_\infty(E\times\mathbb{R})=0.
\end{equation}
We also recall the following density property:
if $E\subset\R$ and $d>0$, then
\begin{equation}\label{hausdenseq}
\limsup_{r\to0}\frac{\h_\infty^d(E\cap B_r(x))}{\omega_dr^d}\geq\frac{1}{2^d},\qquad\textrm{for }\h^d\textrm{-a.e. }x\in E.
\end{equation}

\begin{teo}[Dimension of the Singular Set]\label{singsetdim}
Let $E\subset\R$ be $s$-minimal in $\Omega$.
The singular set $\Sigma_E\subset\partial E\cap\Omega$ has Hausdorff dimension at most $n-3$, i.e.
\begin{equation*}
\h^d(\Sigma_E)=0,\qquad\textrm{for every }d>n-3.
\end{equation*}

%We begin by showing the following
%\begin{equation}\label{claim1reduction}
%\h^d(\Sigma_C)=0\quad\textrm{for all }s\textrm{-minimal cones }C\quad\Longrightarrow\quad\h^d(\Sigma_E)=0.
%\end{equation}

%Assume that $\h^d(\Sigma_C)=0$ for every $s$-minimal cone $C\subset\R$.

\end{teo}

The main part of the proof is the following
\begin{prop}
Let $\{E_h\}$ and $E$ be $s$-minimal sets in $\Omega$ s.t. $E_h\xrightarrow{loc}E$.
Then, for every compact set $K\subset\Omega$
and every $\epsilon>0$,
\begin{equation}\label{subsingular}
\Sigma_{E_h}\cap K\subset N_\epsilon(\Sigma_E\cap K),
\end{equation}
for $h$ sufficiently large. Therefore
\begin{equation}\label{estimatesizefund}
\h^d_\infty(\Sigma_E\cap K)\geq\limsup_{h\to\infty}\h^d_\infty(\Sigma_{E_h}\cap K).
\end{equation}
\begin{proof}
We remark that $(\ref{estimatesizefund})$ is an immediate consequence of $(\ref{subsingular})$.

Suppose $(\ref{subsingular})$ is false. Then (up to a subsequence)
for every $h$ there exists $x_h\in\Sigma_{E_h}\cap K$ s.t. $d(x_h,\Sigma_E\cap K)\geq\epsilon$.
Since $K$ is compact, we can suppose $x_k\longrightarrow x_0\in K$ and
using Corollary $\ref{haus_conv_min}$ we find $x_0\in\partial E$.\\
We only need to show that $x_0\in\Sigma_E$. Then for $h$ big enough we have
\begin{equation*}
x_h\in B_\epsilon(x_0)\subset N_\epsilon(\Sigma_E\cap K),
\end{equation*}
giving a contradiction.

Up to considering translated sets and $\Omega'\subset\subset\Omega$ s.t. $K\subset\Omega'$,
we can suppose that $x_h=0=x_0$ for every $h$.

If $0$ is a regular point for $E$, then (up to dilation and rotation) we have
$\partial E\cap B_{3/2}\subset\{|x_n|<\epsilon_0/2\}$.\\
But then Corollary $\ref{haus_conv_min}$ implies
\begin{equation*}
\partial E_h\cap B_1\subset\{|x_n|\leq\epsilon_0\},
\end{equation*}
for all $h$ large enough and hence $0$ is a regular point for $E_h$ thanks to Theorem $\ref{flat_reg_teo1}$.
This gives a contradiction, concluding the proof.

\end{proof}
\end{prop}

\begin{proof}[Proof of Theorem $\ref{singsetdim}$]
We begin by proving the following
\begin{equation}\label{firststepproof}
\h^d(\Sigma_E)>0\quad\Longrightarrow\quad\exists C\subset\R\,s\textrm{-minimal cone s.t. }\h^d(\Sigma_C)>0.
\end{equation}

Since $\h^d(\Sigma_E)>0$, using $(\ref{hausdenseq})$ we know that there exist $x\in\Sigma_E$ and a sequence $r_h\to0$
s.t.
\begin{equation}\label{dunno1}
\h^s_\infty(\Sigma_E\cap B_{r_h}(x))\geq\frac{\omega_dr_h^d}{2^{d+1}}.
\end{equation}
We can suppose that $x=0$. Up to a subsequence we know that $E_h:=r_h^{-1}E$ converges
locally in $\R$ to an $s$-minimal cone $C$. Moreover each $E_h$ is $s$-minimal in $r_h^{-1}\Omega\supset\Omega$.
Also notice that clearly
\begin{equation*}
\Sigma_{E_h}=r_h^{-1}\Sigma_E,
\end{equation*}
and hence scaling in $(\ref{dunno1})$ we get
\begin{equation*}
\h_\infty^d(\Sigma_{E_h}\cap B_1)=r_h^{-d}\h_\infty^d(\Sigma_E\cap B_{r_h})\geq\frac{\omega_d}{2^{d+1}}.
\end{equation*}
Passing to the limit in $h$ and using $(\ref{estimatesizefund})$ we get $\h^d_\infty(\Sigma_C)>0$.
Thanks to $(\ref{hauseq1})$, this proves $(\ref{firststepproof})$.\\

Now, if $C\subset\R$ is a singular $s$-minimal cone s.t. $\h^d(\Sigma_C)>0$, we can repeat
the same argument blowing-up near some $x\in\Sigma_C$, $x\not=0$.
By Theorem $\ref{dimeredteo}$ we know that the limiting cone is of the form $A\times\mathbb{R}$,
with $A\subset\mathbb{R}^{n-1}$ a singular $s$-minimal cone. From the preceeding discussion we know that
\begin{equation*}
\h^d(\Sigma_{A\times\mathbb{R}})>0
\end{equation*}
and hence, using $(\ref{hauseq2})$,
\begin{equation*}
\h^{d-1}(\Sigma_A)>0.
\end{equation*}

Repeating this argument, we get the existence of a singular $s$-minimal cone $C_k\subset\mathbb{R}^{n-k}$,
with $\h^{d-k}(\Sigma_{C_k})>0$,
for every $k<d$.

Since, as remarked above, $s$-minimal cones in $\mathbb{R}^3$ can have at most one singular point,
we conclude that $d\leq n-3$.

\end{proof}

As a consequence of this estimate, since we know that $\partial E$ is $C^{1,\alpha}$ in a neighborhood
of any regular point, we immediately get the following
\begin{coroll}
Let $E\subset\R$ be $s$-minimal in $\Omega$. Then $\partial E\cap\Omega$ has Hausdorff dimension
at most $n-1$, i.e.
\begin{equation*}
\h^d(\partial E\cap\Omega)=0,\qquad\textrm{for every }d>n-1.
\end{equation*}
\end{coroll}

\end{section}

\begin{section}{Further Results}

In this last section we collect some more results about the regularity of nonlocal minimal surfaces.
For the details and the proofs we refer to the corresponding articles.\\

We begin by saying something more about the dimension of the singular set.\\
Suppose we know that there are no singular $s$-minimal cones in $\mathbb{R}^m$. Then we can apply the same argument
used in the proof of Theorem $\ref{singsetdim}$
to prove that
the singular set $\Sigma_E$ has Hausdorff dimension at most $n-(m+1)$.

In the classical case it is well known that there are no minimal cones in $\mathbb{R}^m$ for any $m\leq7$.
Therefore, using the dimension reduction argument we obtain Theorem $\ref{local_min_reg}$, which completely characterizes the regularity of a classical minimal surface.

However in the nonlocal case it is not known, for a general $s$, wether or not there exist singular $s$-minimal cones in dimension $m\geq3$.

%Therefore it is not known if we can reproduce Theorem $\ref{local_min_reg}$.

In \cite{unifor} the authors exploited uniform estimates as $s\to1$ to prove that when $s$
is sufficiently close to 1, they don't exist and hence the singular set has the same dimension as in the local case.

\begin{teo}
Let $n\leq7$. There exists $\epsilon_0>0$ s.t. if $s\in(1,\epsilon_0,1)$, then any $s$-minimal cone is a half-space.
\end{teo}

\begin{teo}
There exists $\epsilon_0>0$ s.t. if $s\in(1-\epsilon_0,1)$, then

$(i)\quad$ if $n\leq 7$ then the boundary of any $s$-minimal set is locally a $C^{1,\alpha}$-hypersurface,

$(ii)\quad$ if $n=8$ then the boundary of any $s$-minimal set is locally a $C^{1,\alpha}$-hypersurface,
except at most at countably many isolated points,

$(iii)\quad$ if $n\geq9$ then the boundary of any $s$-minimal set is locally a $C^{1,\alpha}$-hypersurface
outside a closed set $\Sigma_E$, with $\h^d(\Sigma_E)=0$ for any $d>n-8$.
\end{teo}

On the other hand in \cite{DelPino} the authors studied a particular kind of cones, the Lawson cones
\begin{equation*}
C_\alpha=\{(u,v)\in\mathbb{R}^m\times\mathbb{R}^n\,|\,|v|=\alpha|u|\},
\end{equation*}
with $m,\,n\geq1$,
and proved that there is a unique $\alpha=\alpha(s,m,n)>0$
s.t. $C_\alpha$ is $s$-minimal in $\mathbb{R}^{m+n}$. We denote $C_m^n(s)$ such a cone.

Unlike what happens in the classical case, when $n\geq3$ a nontrivial $s$-minimal cone, $C^{n-1}_1(s)$, does indeed exist.

Moreover when $s$ is sufficiently close to 0, all Lawson cones $C_m^n(s)$ are shown to be unstable
if $N:=n+m\leq6$ and stable if $N=7$.

This suggests that a regularity theory up to a singular set
of dimension $N-7$ should be the best possible for a general $s$.\\

In \cite{ShortReg} the authors improved the regularity in the neighborhood of a regular point, showing
that Lipschitz regularity actually implies $C^\infty$ regularity.

\begin{teo}
Let $n\geq2$ and let $E$ be $s$-minimal in $B_1$.
If $\partial E\cap B_1$ is locally Lipschitz, then $\partial E\cap B_1$ is $C^\infty$.
\end{teo}
This improves Theorem $\ref{tangent_cone_reg}$.\\
In particular this guarantees that the $s$-fractional mean curvature $\I_s[E](x)$ is well defined in
every regular point $x\in(\partial E\setminus\Sigma_E)\cap\Omega$ and the Euler-Lagrange equation is satisfied in
the classical sense
%(not only in the viscosity sense)
in every such point.\\
Since $\h^{n-1}(\Sigma_E\cap\Omega)=0$, we see that if
$E\subset\R$ is $s$-minimal in $\Omega$, then
\begin{equation}
\I_s[E](x)=0,\quad\h^{n-1}\textrm{-a.e. }x\in\partial E\cap\Omega.
\end{equation}

In \cite{ShortReg} the authors also proved a Bernstein-type result
\begin{teo}
Let $n\geq2$ and  let $E=\{(x',x_n)\in\mathbb{R}^{n-1}\times\mathbb{R}\,|\,x_n<u(x')\}$
be an $s$-minimal graph. If there are no singular $s$-minimal cones in dimension $n-1$, then
$u$ is an affine function (thus $E$ is a half-space).
\end{teo}

Another interesting property concerning $s$-minimal sets and graphs is given in the recent paper \cite{graph}.
Roughly speaking, an $s$-minimal set which is a subgraph outside a cylinder is actually a subgraph in the whole space.

\begin{teo}
Let $n\geq2$ and let $\Omega_0\subset\mathbb{R}^{n-1}$ be a bounded open set with $\partial \Omega_0$ of class $C^{1,1}$.
Let $\Omega:=\Omega_0\times\mathbb{R}$ and let $E\subset\R$ be $s$-minimal in $\Omega$.
Assume that
\begin{equation*}
E\setminus\Omega=\{x_n<u(x')\,|\,x'\in\mathbb{R}^{n-1}\setminus\Omega_0\},
\end{equation*}
for some continuous $u:\mathbb{R}^{n-1}\longrightarrow\mathbb{R}$. Then
\begin{equation*}
E\cap\Omega=\{x_n<v(x')\,|\,x'\in\Omega_0\},
\end{equation*}
for some $v:\mathbb{R}^{n-1}\longrightarrow\mathbb{R}$.
\end{teo}

A delicate point is the fact that in general $s$-minimal sets are not regular (not even continuous)
up to the boundary.
Indeed boundary stickiness phenomena may occur.

The boundary behavior is studied in another recent paper by the same authors, \cite{bdary}.
We briefly describe one of the many results obtained in this article, an example of stickiness phenomenum.

For any $\delta>0$, let
\begin{equation*}
K_\delta:=(B_{1+\delta}\setminus B_1)\cap\{x_n<0\},
\end{equation*}
a small half-ring.
Define $E_\delta$ to be the set minimizing $P_s(F,B_1)$ among all sets $F\subset\R$ s.t.
$F\setminus B_1=K_\delta$.

We remark that in the local framework the set minimizing the perimeter, among all sets having $K_\delta$
as boundary value at $\partial B_1$, is always the flat set $\{x_n<0\}\cap B_1$,
independently of $\delta$.

However in the nonlocal framework this changes dramatically, since nonlocal minimizers stick to the boundary $\partial B_1$,
provided $\delta$ is suitably small.
To be more precise,
\begin{teo}
There exists $\delta_0=\delta_0(n,s)>0$ s.t. for any $\delta\in(0,\delta_0]$ we have
\begin{equation*}
E_\delta=K_\delta.
\end{equation*}
\end{teo}
As described in the article, these rather surprising sticking effects have some (at least vague) heuristic explanations.

For example, to see that
\begin{equation*}
E:=(\{x_n<0\}\cap B_1)\cup K_\delta=\{x_n<0\}\cap B_{1+\delta}
\end{equation*}
cannot be our nonlocal minimizer, we can look at $\I_s[E](0)$.\\
%, the $s$-fractional mean curvature in 0.\\
There is no contribution coming from inside $B_{1+\delta}$ because of symmetry,
\begin{equation*}
P.V.\int_{B_{1+\delta}}\frac{\chi_E(y)-\chi_{\Co E}(y)}{|y|^{n+s}}dy=0.
\end{equation*}
On the other hand, outside $B_{1+\delta}$
we have no contribution coming from $E$,
\begin{equation*}
\int_{\Co B_{1+\delta}}\frac{\chi_E(y)-\chi_{\Co E}(y)}{|y|^{n+s}}dy=
-\int_{\Co B_{1+\delta}}\frac{1}{|y|^{n+s}}dy=-\frac{n\,\omega_n}{s}(1+\delta)^{-s}.
\end{equation*}
Since $\I_s[E](0)<0$, $E$ cannot be the $s$-minimal set we are looking for.\\
Now the idea is that, in order to compensate the contribution coming from outside $B_{1+\delta}$
(which is the same for every competitor),
our set $E$ has to bend near 0, becoming convex.

However when $\delta$ is very small this bending is not enough to compensate the other contribution and
the set $E$ has to stick to the half-ring $K_\delta$ in order to satisfy the Euler-Lagrange equation.

\end{section}

\end{chapter}

%%%%%%%%%%%%%%%%%%%%%%%%%%%%%%%%%%%%%%%%
%%%%%%%%%%%%%%%%%%%%%%%%%%%%%%%%%%%%%%%%
%%%%%%%%%%%%%%%%%%%%%%%%%%%%%%%%%%%%%%%%

%%%%%%%%%%%%%%%%%%%%%%%%%%%%%%%%%%%%%%%%
%%%%%%%%%%%%%%%%%%%%%%%%%%%%%%%%%%%%%%%%
%%%%%%%%%%%%%%%%%%%%%%%%%%%%%%%%%%%%%%%%

%%%%%%%%%%%%%%%%%%%%%%%%%%%%%%%%%%%%%%%%
%%%%%%%%%%%%%%%%%%%%%%%%%%%%%%%%%%%%%%%%
%%%%%%%%%%%%%%%%%%%%%%%%%%%%%%%%%%%%%%%%

%%%%%%%%%%%%%%%%%%%%%%%%%%%%%%%%%%%%%%%%
%%%%%%%%%%%%%%%%%%%%%%%%%%%%%%%%%%%%%%%%
%%%%%%%%%%%%%%%%%%%%%%%%%%%%%%%%%%%%%%%%

%%%%%%%%%%%%%%%%%%%%%%%%%%%%%%%%%%%%%%%%
%%%%%%%%%%%%%%%%%%%%%%%%%%%%%%%%%%%%%%%%
%%%%%%%%%%%%%%%%%%%%%%%%%%%%%%%%%%%%%%%%

%%%%%%%%%%%%%%%%%%%%%%%%%%%%%%%%%%%%%%%%
%%%%%%%%%%%%%%%%%%%%%%%%%%%%%%%%%%%%%%%%
%%%%%%%%%%%%%%%%%%%%%%%%%%%%%%%%%%%%%%%%

%%%%%%%%%%%%%%%%%%%%%%%%%%%%%%%%%%%%%%%%
%%%%%%%%%%%%%%%%%%%%%%%%%%%%%%%%%%%%%%%%
%%%%%%%%%%%%%%%%%%%%%%%%%%%%%%%%%%%%%%%%

%%%%%%%%%%%%%%%%%%%%%%%%%%%%%%%%%%%%%%%%
%%%%%%%%%%%%%%%%%%%%%%%%%%%%%%%%%%%%%%%%
%%%%%%%%%%%%%%%%%%%%%%%%%%%%%%%%%%%%%%%%

%%%%%%%%%%%%%%%%%%%%%%%%%%%%%%%%%%%%%%%%
%%%%%%%%%%%%%%%%%%%%%%%%%%%%%%%%%%%%%%%%
%%%%%%%%%%%%%%%%%%%%%%%%%%%%%%%%%%%%%%%%

\appendix
\chapter{Viscosity Solutions}

\begin{section}{Definitions and Basic Properties}
We recall the definition of viscosity solutions and some of their properties in the classical context of second order (degenerate) elliptic equations.
For a complete introduction to the subject, see \cite{Visco}. For a quick introduction see \cite{Ambrosio}. For more details, in particular regarding the regularity theory, see also \cite{Fully}.

We limit ourselves to the local theory, the general nonlocal case needing more technical assumptions, in particular because of the singularity of the kernels appearing in the operators. For the appropriate definitions and properties in the nonlocal framework see \cite{regularity}.\\
Anyway in the next chapter we will say something also about nonlocal viscosity solutions, but only in the case of the fractional laplacian.\\

Consider a function
\begin{equation*}
F:\R\times\mathbb{R}\times\R\times\mathcal{S}(n)\longrightarrow\mathbb{R},
\end{equation*}
where $\mathcal{S}(n)$ denotes the space of symmetric $n\times n$ matrices, equipped with the usual order.
We are looking for a solution to the PDE
\begin{equation*}
F(x,u,\nabla u,D^2 u)=0.
\end{equation*}
We suppose that $F$ satisfies the following two conditions:
\begin{equation}
t\leq s\qquad\Longrightarrow\qquad F(x,t,p,X)\leq F(x,s,p,X),
\end{equation}
i.e. $F$
is proper, and
\begin{equation}
Y\leq X\qquad\Longrightarrow\qquad F(x,t,p,X)\leq F(x,t,p,Y),
\end{equation}
i.e. $F$ is degenerate elliptic.

\begin{ese}
The easiest example is that of degenerate elliptic linear equations, i.e. PDE's of the form
\begin{equation*}
-\sum_{i,j}a_{i,j}(x)\frac{\partial^2u}{\partial x_i\partial x_j}+\sum_i b_i(x)\frac{\partial u}{\partial x_i}+c(x)u=f(x),
\end{equation*}
where $A(x)=\{a_{i,j}(x)\}\in\mathcal{S}(n)$ and $A(x)\geq0$. The corresponding $F$ is then
\begin{equation*}
F(x,t,p,X):=-\textrm{trace}(A(x)X)+b(x)\cdot p+c(x)t-f(x),
\end{equation*}
and the condition $A(x)\geq0$ guarantees that $F$ is degenerate elliptic. In order for $F$ to be proper we must require
$c(x)\geq0$.\\
If there exist constants $\lambda,\,\Lambda>0$ s.t. $\lambda I_n\leq A(x)\leq\Lambda I_n$ for all $x$, where $I_n$ is the identity matrix, we say as usual that $F$ is uniformly elliptic.
\end{ese}

\begin{ese}
A second example, more interesting for us, is that of quasilinear elliptic equations in divergence form
\begin{equation*}
%-\sum_i\frac{\partial}{\partial x_i}(a_i(x,\nabla u))
-\textrm{div}(a(x,\nabla u))
+b(x,u,\nabla u)=0.
\end{equation*}
The usual notion of ellipticity for such an equation is the monotonicity of the vector field $a(x,p)$ in $p$ as a mapping $a(x,\cdot):\R\to\R$, i.e.
\begin{equation*}
(a(x,p)-a(x,p'))\cdot(p-p')>0,\qquad\textrm{for every }p\not=p'\in\R.
\end{equation*}
Supposing we have enough regularity to carry out the differentiation, we can rewrite the equation in nondivergence form as
\begin{equation*}
-\sum_{i,j}\frac{\partial a_i}{\partial p_j}(x,\nabla u)\frac{\partial^2 u}{\partial x_i\partial x_j}+b(x,u,\nabla u)-\sum_i\frac{\partial a_i}{\partial x_i}(x,\nabla u)=0
\end{equation*}
and hence
\begin{equation*}
F(x,t,p,X)=-\textrm{trace}(D_pa(x,p)X)+b(x,t,p)-\sum_i\frac{\partial a_i}{\partial x_i}(x,p).
\end{equation*}
The monotonicity of $a$ in $p$ guarantees that $F$ is degenerate elliptic and if we ask $b$ to be nondecreasing in $t$ then $F$ is also proper.\\
The particular case we are interested in is the minimal surface equation $(\ref{min_surf_eq})$, which is given by
\begin{equation*}
a(x,p)=\frac{p}{\sqrt{1+|p|^2}},\qquad b=0,
\end{equation*}
and hence, carrying out the calculation,
\begin{equation*}
F(x,t,p,X)=-\frac{1}{\sqrt{1+|p|^2}}\textrm{trace}(X)+\frac{1}{({1+|p|^2})^{\frac{3}{2}}}\textrm{trace}((p\otimes p) X).
\end{equation*}
We remark that another important example belonging to this class is the so called $m$-Laplace equation
\begin{equation*}
-\textrm{div}(|\nabla u|^{m-2}\nabla u)=0,\qquad m\in(1,\infty).
\end{equation*}
\end{ese}

For much more examples see \cite{Visco}.\\

Before giving the definition of viscosity solution, we show where it comes from.
Suppose $F$ is proper, degenerate elliptic and suppose also it is continuous.\\
Suppose we already have a classical subsolution of $F=0$, i.e. $u\in C^2(\R)$ s.t.
\begin{equation*}
F(x,u(x),\nabla u(x),D^2 u(x))\leq0\qquad\textrm{for every }x\in\R.
\end{equation*}
Now, if we have a function $\varphi\in C^2(\R)$ s.t. $x_0$ is a local maximum for $u-\varphi$, then from calculus we know that
$\nabla u(x_0)=\nabla\varphi(x_0)$ and
$D^2u(x_0)\leq D^2\varphi(x_0)$.
Therefore by degenerate ellipticity
\begin{equation*}
F(x_0,u(x_0),\nabla \varphi(x_0),D^2 \varphi(x_0))\leq F(x_0,u(x_0),\nabla u(x_0),D^2 u(x_0))\leq0.
\end{equation*}
Notice that the first term in the inequality does not depend on any derivative of $u$.
This suggests the following definition of viscosity solution, based on an appropriate class of test functions.
For simplicity we restrict to open sets; for the general definitions see \cite{Visco}.
\begin{defin}
Suppose $F$ is proper and degenerate elliptic and let $\Omega\subset\R$ be open. A viscosity subsolution of $F=0$ on $\Omega$
is an upper semicontinuous function $u:\Omega\longrightarrow\mathbb{R}$ s.t.
\begin{equation*}
F(x_0,u(x_0),\nabla \varphi(x_0),D^2 \varphi(x_0))\leq0,
\end{equation*}
for every $x_0\in\Omega$ and every $\varphi\in C^2(\Omega)$ s.t. $x_0$ is a local maximum for $u-\varphi$.

In the same way, a viscosity supersolution of $F=0$ on $\Omega$ is a lower semicontinuous function $u:\Omega\longrightarrow\mathbb{R}$ s.t.
\begin{equation*}
F(x_0,u(x_0),\nabla \varphi(x_0),D^2 \varphi(x_0))\geq0,
\end{equation*}
for every $x_0\in\Omega$ and every $\varphi\in C^2(\Omega)$ s.t. $x_0$ is a local minimum for $u-\varphi$.\\
A function $u:\Omega\longrightarrow\mathbb{R}$ is a viscosity solution of $F=0$ if it is both a viscosity subsolution and supersolution.
\end{defin}

Notice in particular that a viscosity solution, being both upper and lower semicontinuous, is continuous.\\

Since only the derivatives of $\varphi$ are involved, we can assume that the local maximum (minimum) of $u-\varphi$ is $0$, i.e. that
\begin{equation*}
u(x_0)=\varphi(x_0).
\end{equation*}
Moreover, up to considering a perturbation of $\varphi$, it is not restrictive to suppose the maximum (minimum) to be strict, i.e.
\begin{equation*}
u(x)<\varphi(x)\qquad (u(x)>\varphi(x)\quad\textrm{respectively}),
\end{equation*}
for every $x\not=x_0$ in a neighborhood of $x_0$.\\
From the geometric point of view this means that the graph of $\varphi$ touches the graph of $u$ from above (below) at $x_0$
and that $(x_0,u(x_0))$ is the only contact point (in a neighborhood).\\

Exploiting Taylor expansions we can give equivalent definitions based on the notions of subjets and superjets.
Roughly speaking, instead of taking general $C^2$ tangent graphs, we can restrict ourselves to (small perturbations of) tangent paraboloids.
\begin{defin}
Let $\Omega$ be an open set and let $u:\Omega\longrightarrow\mathbb{R}$ be upper semicontinuous. The second order superjet of $u$ at $x_0\in\Omega$ is the collection $J_\Omega^{2,+}u(x_0)$ of all pairs $(p,X)\in\R\times\mathcal{S}(n)$ s.t.
\begin{equation*}
u(x)\leq u(x_0)+\langle p,x-x_0\rangle+\frac{1}{2}\langle X(x-x_0),x-x_0\rangle+o(|x-x_0|^2),\quad\textrm{as }x\to x_0.
\end{equation*}
In the same way we can define the second order subjet $J_\Omega^{2,-}u(x_0)$ of a lower semicontinuous function $u$, as the
collection of pairs $(p,X)\in\R\times\mathcal{S}(n)$ which satisfy the opposite inequality.
\end{defin}
Notice that since this is a local notion and we are considering only open sets $\Omega$, the subjet and the superjet do not really depend on $\Omega$.
\begin{rmk}
It is clear that $u$ is twice differentiable at $x_0$ if and only if $J^2u(x_0):=J^{2,+}u(x_0)\cap J^{2,-}u(x_0)$
is nonempty, in which case
\begin{equation*}
J^2u(x_0)=\{(\nabla u(x_0),D^2u(x_0))\}.
\end{equation*}
\end{rmk}
\noindent
It is easy to verify the following
\begin{prop}
Let $\Omega$ be an open set and let $u:\Omega\longrightarrow\mathbb{R}$ be an upper semicontinuous function. Then
$u$ is a viscosity subsolution of $F=0$ if and only if
\begin{equation*}
F(x,u(x),p,X)\leq0\quad\textrm{for all}\quad x\in\Omega\quad\textrm{and}\quad(p,X)\in J_\Omega^{2,+}u(x).
\end{equation*}
Similarly for viscosity supersolutions.
\end{prop}

One of the basic existence theorems for viscosity solutions is based on Perron's method.
\begin{teo}[Perron]
Let $f,\, g$ be respectively a subsolution and a supersolution of $F=0$ in an open set $\Omega$, s.t.
\begin{equation*}
f(x)\leq g(x),\quad f(x)>-\infty,\quad g(x)<\infty\quad\forall x\in\Omega.
\end{equation*}
Then there exists a viscosity solution $u$ of $F=0$ in $\Omega$ satisfying
\begin{equation*}
f(x)\leq u(x)\leq g(x)\qquad\quad\forall x\in\Omega.
\end{equation*}
\end{teo}
An important observation for the proof is that the sup of any family
of viscosity subsolutions is again a viscosity subsolution (wherever it is finite).\\
\begin{rmk}
Under some additional assumptions on $F$, namely
\begin{equation*}
\gamma(t-s)\leq F(x,t,p,X)-F(x,s,p,X)\qquad\textrm{for }t\geq s,
\end{equation*}
and an appropriate growth condition with respect to $p$ (for the details see the [User's Guide]) it is possible to prove
a comparison principle.\\
Let $\Omega$ be a bounded open set. Let $u$ be a subsolution of $F=0$, upper semicontinuous in $\overline{\Omega}$,
and $v$ be a supersolution, lower semicontinuous in $\overline{\Omega}$. If $u\leq v$ on $\partial\Omega$,
then $u\leq v$ on $\overline{\Omega}$.

In particular, when $F$ satisfies this comparison principle, if we impose some boundary condition $w=f$ on $\partial\Omega$, then
if $u$ is a viscosity subsolution and $v$ is a viscosity supersolution, we have $u\leq v$.
\end{rmk}

\end{section}
\begin{section}{Harnack Inequality}

We conclude this chapter by stating a regularity result for a class of fully nonlinear elliptic equations.
For the proof see \cite{Savin}.

We consider
\begin{equation*}
F:B_1\times\mathbb{R}\times\R\times\mathcal{S}(n)\longrightarrow\mathbb{R}
\end{equation*}
degenerate elliptic and satisfying the following hypothesis for $|p|\leq\delta,\,|t|\leq\delta$\\

(1) $F(x,t,p,\cdot)$ is uniformly elliptic in a $\delta$-neighborhood of the origin, with ellipticity constants $\Lambda\geq\lambda>0$
i.e.
\begin{equation*}
\lambda|Y|\leq F(x,t,p,X)-F(x,t,p,X+Y)\leq\Lambda|Y|,\quad\textrm{if }Y\geq0,\,|X|,\,|Y|\leq\delta,
\end{equation*}

(2) constants are solutions of $F=0$, i.e.
\begin{equation*}
F(x,t,0,0)=0,
\end{equation*}
and $F$ is Lipschitz in $p$
\begin{equation*}
|F(x,t,p,X)-F(x,t,p',X)|\leq L\,|p-p'|,\qquad|X|,\,|p|,\,|p'|\leq\delta.
\end{equation*}

\begin{rmk}
Notice that the minimal surface equation satisfies the above hypothesis.
\end{rmk}
For such an $F$ we can prove the following Harnack inequality for flat solutions
\begin{teo}[Harnack Inequality]\label{Harnack_vi}
Let $F$ be as above. There exist universal constants $c_0$ small and $C_0$ large s.t., if $u:B_1\longrightarrow\mathbb{R}$
is a viscosity solution of $F=0$, with
\begin{equation*}
0\leq u\leq c_0\delta\qquad\textrm{in }B_1,
\end{equation*}
then
\begin{equation*}
u\leq C_0 u(0)\qquad\textrm{in }B_{1/2}.
\end{equation*}
\end{teo}
As a consequence we obtain that the oscillation of a flat solution decreases in the interior
\begin{coroll}
Let $u$ be a viscosity solution of $F=0$, with
\begin{equation*}
u(0)=0,\qquad\|u\|_{L^\infty(B_1)}\leq c_0\delta.
\end{equation*}
Then
\begin{equation*}
\|u\|_{L^\infty(B_{1/2})}\leq(1-\nu)\|u\|_{L^\infty(B_1)},
\end{equation*}
for some small universal constant $\nu>0$.
\end{coroll}

Adapting the technique used in the proofs, one can then obtain the theorems for the regularity of minimal surfaces stated at the end of Chapter 1.\\

Assume $F$ to be more regular, i.e. $F\in C^2$ with $|D^2F|\leq K$ in a $\delta$-neighborhood of the set
$\{(x,0,0,0)\,|\,x\in B_1\}$ and suppose that instead of (2) $F$ satisfies\\

(3) 0 is a solution of $F=0$, i.e. $F(x,0,0,0)=0$.\\

\noindent
Then we can prove the following regularity result
\begin{teo}
There exists a constant $c_1$ (depending on $F$) s.t. if $u$ is a viscosity solution of $F=0$, with
\begin{equation*}
\|u\|_{L^\infty(B_1)}\leq c_1,
\end{equation*}
then $u\in C^{2,\alpha}(B_{1/2})$, and
\begin{equation*}
\|u\|_{C^{2,\alpha}(B_{1/2})}\leq\delta.
\end{equation*}
\end{teo}

\end{section}

%%%%%%%%%%%%%%%%%%%%%%%%%%%%%%%%%%%%%%%%
%%%%%%%%%%%%%%%%%%%%%%%%%%%%%%%%%%%%%%%%
%%%%%%%%%%%%%%%%%%%%%%%%%%%%%%%%%%%%%%%%

%%%%%%%%%%%%%%%%%%%%%%%%%%%%%%%%%%%%%%%%
%%%%%%%%%%%%%%%%%%%%%%%%%%%%%%%%%%%%%%%%
%%%%%%%%%%%%%%%%%%%%%%%%%%%%%%%%%%%%%%%%

%%%%%%%%%%%%%%%%%%%%%%%%%%%%%%%%%%%%%%%%
%%%%%%%%%%%%%%%%%%%%%%%%%%%%%%%%%%%%%%%%
%%%%%%%%%%%%%%%%%%%%%%%%%%%%%%%%%%%%%%%%

%%%%%%%%%%%%%%%%%%%%%%%%%%%%%%%%%%%%%%%%
%%%%%%%%%%%%%%%%%%%%%%%%%%%%%%%%%%%%%%%%
%%%%%%%%%%%%%%%%%%%%%%%%%%%%%%%%%%%%%%%%

%%%%%%%%%%%%%%%%%%%%%%%%%%%%%%%%%%%%%%%%
%%%%%%%%%%%%%%%%%%%%%%%%%%%%%%%%%%%%%%%%
%%%%%%%%%%%%%%%%%%%%%%%%%%%%%%%%%%%%%%%%

%%%%%%%%%%%%%%%%%%%%%%%%%%%%%%%%%%%%%%%%
%%%%%%%%%%%%%%%%%%%%%%%%%%%%%%%%%%%%%%%%
%%%%%%%%%%%%%%%%%%%%%%%%%%%%%%%%%%%%%%%%

%%%%%%%%%%%%%%%%%%%%%%%%%%%%%%%%%%%%%%%%
%%%%%%%%%%%%%%%%%%%%%%%%%%%%%%%%%%%%%%%%
%%%%%%%%%%%%%%%%%%%%%%%%%%%%%%%%%%%%%%%%

%%%%%%%%%%%%%%%%%%%%%%%%%%%%%%%%%%%%%%%%
%%%%%%%%%%%%%%%%%%%%%%%%%%%%%%%%%%%%%%%%
%%%%%%%%%%%%%%%%%%%%%%%%%%%%%%%%%%%%%%%%

%%%%%%%%%%%%%%%%%%%%%%%%%%%%%%%%%%%%%%%%
%%%%%%%%%%%%%%%%%%%%%%%%%%%%%%%%%%%%%%%%
%%%%%%%%%%%%%%%%%%%%%%%%%%%%%%%%%%%%%%%%

\chapter{Fractional Laplacian}

\begin{section}{Definitions}
For all the details we remand to the \cite{HitGuide} and the references cited therein.

There are several equivalent ways to define the fractional Laplacian.

\begin{defin}
Let $s\in(0,1)$. For any $u\in\mathcal{S}(\R)$, the $s$-fractional Laplacian of $u$ in $x$ is defined as
\begin{equation}
(-\Delta)^su(x):=C(n,s)\textrm{P.V.}\int_{\R}\frac{u(x)-u(y)}{|x-y|^{n+2s}}\,dy,
\end{equation}
where the constant $C(n,s)$ is defined as
\begin{equation*}
C(n,s):=\left(\int_{\R}\frac{1-\cos(\zeta_1)}{|\zeta|^{n+2s}}\,d\zeta\right)^{-1}.
\end{equation*}
\end{defin}

Notice that, for every $\delta>0$ we have
\begin{equation*}
\int_{\Co B_\delta(x)}\frac{|u(x)-u(y)|}{|x-y|^{n+2s}}\,dy\leq2\|u\|_{L^\infty(\R)}\int_{\Co B_\delta(x)}\frac{1}{|x-y|^{n+2s}}\,dy<\infty,
\end{equation*}
for every $s\in(0,1)$. On the other hand
\begin{equation*}
|u(x)-u(y)|\leq C|x-y|,
\end{equation*}
so that, if $s\in(0,1/2)$,
\begin{equation*}
\int_{B_\delta(x)}\frac{|u(x)-u(y)|}{|x-y|^{n+2s}}\,dy\leq C\int_{B_\delta(x)}\frac{1}{|x-y|^{n+2s-1}}\,dy<\infty.
\end{equation*}
Therefore there is no real need to consider the principal value for $s\in(0,1/2)$.

We can rewrite the singular integral as a weighted second order differential quotient, exploiting the symmetry of the kernel.
\begin{prop}
Let $s\in(0,1)$. For any $u\in\mathcal{S}(\R)$
\begin{equation*}
(-\Delta)^su(x)=-\frac{1}{2}C(n,s)\int_{\R}\frac{u(x+y)+u(x-y)-2u(x)}{|y|^{n+2s}}\,dy,
\end{equation*}
for every $x\in\R$.
\end{prop}
Notice that there is no need to take the principal value in the right hand side, since for every $s\in(0,1)$
\begin{equation*}
\frac{|u(x+y)+u(x-y)-2u(x)|}{|y|^{n+2s}}\leq\frac{\|D^2u\|_{L^\infty(\R)}}{|y|^{n+2s-2}},
\end{equation*}
which is integrable near $0$.\\

The fractional Laplacian $(-\Delta)^s$ can be equivalently defined using the Fourier transform $\mathcal{F}$, as a pseudo-differential operator of symbol $|\xi|^{2s}$.

\begin{prop}
Let $s\in(0,1)$ and let $(-\Delta)^s:\mathcal{S}(\R)\longrightarrow L^2(\R)$ be the fractional Laplacian operator defined above.
Then for any $u\in\mathcal{S}(\R)$,
\begin{equation*}
(-\Delta)^su=\mathcal{F}^{-1}(|\xi|^{2s}(\mathcal{F}u)).
\end{equation*}
\end{prop}

The constant $C(n,s)$ was defined in such a way that there is no constant appearing in the above equality.\\
Another interesting observation is the following
\begin{prop}
Let $s\in(0,1)$. For every $u\in H^s(\R)$
\begin{equation*}
[u]^2_{H^s(\R)}=2C(n,s)^{-1}\int_{\R}|\xi|^{2s}|\mathcal{F}u(\xi)|^2\,d\xi.
\end{equation*}
\end{prop}

As a consequence we get
\begin{prop}
Let $s\in(0,1)$. For every $u\in H^s(\R)$,
\begin{equation}
[u]^2_{H^s(\R)}=2C(n,s)^{-1}\|(-\Delta)^\frac{s}{2}u\|^2_{L^2(\R)}.
\end{equation}
\begin{proof}
Using the Plancherel formula we get
\begin{equation*}
\|(-\Delta)^\frac{s}{2}u\|^2_{L^2(\R)}=\|\mathcal{F}(-\Delta)^\frac{s}{2}u\|^2_{L^2(\R)},
\end{equation*}
and hence
\begin{equation*}
\|(-\Delta)^\frac{s}{2}u\|^2_{L^2(\R)}
=\||\xi|^s\mathcal{F}u\|_{L^2(\R)}^2
=\frac{1}{2}C(n,s)[u]^2_{H^s(\R)}.
\end{equation*}

\end{proof}
\end{prop}

In particular this implies that the operator in the Euler-Lagrange equation for the minimization of the $H^s$ seminorm is the $\frac{s}{2}$-fractional Laplacian.\\
%\begin{equation*}
%(-\Delta)^\frac{s}{2}u=0.
%\end{equation*}

In order for the fractional Laplacian to be well defined in some point $x_0$, we can weaken the requests on the function $u:\R\longrightarrow\mathbb{R}$ in such a way that we still have integrability away from $x_0$,  and ask $u$ to be regular enough in a neighborhood of $x_0$.
For example, notice that the condition
\begin{equation}\label{frac_lap_cond}
\int_{\R}\frac{|u(y)|}{(1+|y|^2)^\frac{n+2s}{2}}\,dy<\infty,
\end{equation}
guarantees that for every $x_0\in\R$
\begin{equation*}
\int_{\Co B_\delta(x_0)}\frac{|u(x_0)-u(y)|}{|x_0-y|^{n+2s}}\,dy<\infty,
\end{equation*}
for every $\delta>0$.\\
Therefore if $u$ is $C^2$ in some neighborhood of $x_0$, then $(-\Delta)^su(x_0)$ is well defined.
Actually, it is enough to require $u$ to be $C^{2s+\epsilon}$, if $s\in(0,1/2)$, or $C^{1,2s-1+\epsilon}$, if $s\in[1/2,1)$, for
some $\epsilon>0$,
in a neighborhood of $x_0$.
%Now suppose $\varphi\in C^2(B_r(x_0))$ touches $u$ from above in $x_0$, i.e.
%\begin{equation*}
%\varphi(x_0)=u(x_0)\qquad\textrm{and}\qquad\varphi(y)>u(y),\quad\forall y\in B_r(x_0)\setminus\{x_0\}.
%\end{equation*}
%and define the function $v$ as
%\begin{equation*}
%v(y):=\left\{ \begin{array}{cc}
%\varphi(y), & y\in B_{r/2}(x_0),\\
%u(y), & y\in\C B_{r/2}(x_0).
%\end{array}\right.
%\end{equation*}

%\noindent
%Then for every $y\in B_{r/2}$ we have
%\begin{equation*}\begin{split}
%\frac{u(x_0+y)+u(x_0-y)-2u(x_0)}{|y|^{n+2s}}&\leq\frac{\varphi(x_0+y)+\varphi(x_0-y)-2\varphi(x_0)}{|y|^{n+2s}}\\
%&
%\leq\frac{\|D^2\varphi\|_{L^\infty(B_{r/2})}}{|y|^{n+2s-2}},
%\end{split}
%\end{equation*}
%which is integrable near 0.

In any case it is proved in \cite{obstacle} that the fractional Laplacian $(-\Delta)^s$ can be defined in the distributional sense
for every function $u$ in the weighted $L^1$-space
\begin{equation}
L_s:=\left\{u:\R\longrightarrow\mathbb{R}\,\Big|\,\int_{\R}\frac{|u(y)|}{(1+|y|^2)^\frac{n+2s}{2}}\,dy<\infty\right\}.
\end{equation}

\begin{rmk}\label{frac_growth}
Notice that when $s>1/2$, the condition $(\ref{frac_lap_cond})$ is satisfied in particular by functions $u\in L^1_{loc}(\R)$ that grow at infinity at most like $|y|^{1+\alpha}$, for some $\alpha\in(0,2s-1)$.
\end{rmk}

\begin{rmk}
We also remark that $L_s\subset L^1_{loc}(\R)$.\\
Indeed, if $\Omega$ is a bounded open set and $u\in L_s$, then
\begin{equation*}\begin{split}
\int_\Omega|u|\,dx&=\int_\Omega\frac{|u(x)|}{(1+|x|^2)^\frac{n+2s}{2}}(1+|x|^2)^\frac{n+2s}{2}\,dx\\
&
\leq\max_{x\in\overline{\Omega}}(1+|x|^2)^\frac{n+2s}{2}
\int_{\R}\frac{|u(x)|}{(1+|x|^2)^\frac{n+2s}{2}}\,dx<\infty.
\end{split}
\end{equation*}
\end{rmk}

In \cite{obstacle} the following result is proved.
\begin{prop}
Let $s\in(0,1)$ and let $\Omega$ be an open set. If $u\in L_s$ is s.t. either
\begin{equation*}
u\in C^{2s+\epsilon}(\Omega),\qquad\textrm{if}\quad s<\frac{1}{2},
\end{equation*}
or
\begin{equation*}
u\in C^{1,2s-1+\epsilon}(\Omega),\qquad\textrm{if}\quad s\geq\frac{1}{2},
\end{equation*}
for some $\epsilon>0$, then
\begin{equation*}
(-\Delta)^su\in C(\Omega).
\end{equation*}
\end{prop}

For more details about the regularity of $(-\Delta)^su$ see \cite{obstacle}.

\end{section}
\begin{section}{Liouville-type Theorem}

In this section we show that if $u$ is a viscosity solution of the equation
\begin{equation*}
(-\Delta)^su=0,\qquad\textrm{in}\quad B_1,
\end{equation*}
in the sense explained below,
%i.e. $u$ is $s$-harmonic in $B_1$,
with $s>1/2$,
and if $u$ grows at infinity at most like $|x|^{1+\alpha}$, then it is linear.
\begin{rmk}\label{rmk_visc}
Notice that if $u(x_0)=v(x_0)$ and $u<v$ on $\R\setminus\{x_0\}$, then
$(-\Delta)^su(x_0)\geq (-\Delta)^sv(x_0)$, provided these are well defined.\\
In particular, if $(-\Delta)^su(x_0)\leq f(x_0)$, then $(-\Delta)^sv(x_0)\leq f(x_0)$ for every regular enough $v$ touching
$u$ from above at $x_0$.
\end{rmk}

The above remark motivates the following definition of viscosity solution in an open set $\Omega$.
\begin{defin}
Let $s\in(0,1)$.
A function $u:\R\longrightarrow\mathbb{R}$, s.t. $u\in L_s$ and $u$ is upper (lower) semicontinuous in $\overline{\Omega}$,
is a viscosity subsolution (supersolution) to $(-\Delta)^s=f$ if for every $x\in\Omega$ and every $\varphi\in C^2(\overline{N})$,
touching $u$ from above (below) at $x$ in some neighborhood $N$ of $x$ in $\Omega$, i.e.
\begin{equation*}
\varphi(x)=u(x)\quad\textrm{and}\quad \varphi(y)>u(y)\quad(\varphi(y)<u(y)),\qquad\forall y\in N\setminus\{x\},
\end{equation*}
then if we define
\begin{equation*}
v:=\left\{\begin{array}{cc}
\varphi,&\textrm{in }N,\\
u,&\textrm{in }\Co N,
\end{array}\right.
\end{equation*}
we have
\begin{equation*}
(-\Delta)^sv(x)\leq f(x)\qquad((-\Delta)^sv(x)\geq f(x)).
\end{equation*}
The function $u$ is a viscosity solution of $(-\Delta)^su=f$ if it is both a subsolution and a supersolution.
\end{defin}

The following Lemma shows in particular that the existence of a tangent regular function to a subsolution (supersolution) $u$,
in some point $x$ is enough to guarantee that $(-\Delta)u(x)$ is well defined.

\begin{lem}
Let $s\in(0,1)$ and let $u$ be a viscosity subsolution to $(-\Delta)^su=f$ in $\Omega$.
If $\varphi$ is a $C^2$ function touching $u$ from above at $x\in\Omega$, then
$(-\Delta)^su(x)$ is well defined and
\begin{equation*}
(-\Delta)^su(x)\leq f(x).
\end{equation*}
Similarly for viscosity supersolutions.
\end{lem}
The proof can be found for more general operators in \cite{regularity}.\\
In particular this Lemma and Remark $\ref{rmk_visc}$ give the following
\begin{prop}
Let $s\in(0,1)$. Let $u\in L_s$ be upper semicontinuous in $\overline{\Omega}$.
Then $u$ is a viscosity subsolution of $(-\Delta)^su=f$ in $\Omega$ if and only if
$(-\Delta)^su(x)$ is defined and
\begin{equation*}
(-\Delta)^su(x)\leq f(x),
\end{equation*}
in every point $x\in\Omega$ where $u$ admits a $C^2$ tangent function by above.\\
Similarly for viscosity supersolutions.
\end{prop}

In \cite{regularity} it is also proved that if $u$ is a bounded function s.t.
\begin{equation*}
(-\Delta)^su=f\qquad\textrm{in } B_1
\end{equation*}
in the viscosity sense, then
\begin{equation}\label{caf_sil_reg}
\|u\|_{C^{2,\gamma}(B_{1/2})}\leq C\,\left(\|f\|_{C^{1,1}(B_1)}+\|u\|_{L^\infty(\R)}\right),
\end{equation}
for $\gamma,\,C$ depending only on $n$ and $s$.

\begin{teo}\label{liouville_frac}
Let $s>1/2$, and assume
\begin{equation*}
|u(x)|\leq1+|x|^{1+\alpha},\qquad0<\alpha<2s-1,
\end{equation*}
and
\begin{equation*}
(-\Delta)^su=0\qquad\textrm{in }\R
\end{equation*}
in the viscosity sense. Then $u$ is linear.
\begin{proof}
If we cut $u$, defining
\begin{equation*}
v:=u\chi_{B_2},
\end{equation*}
then $v$ satisfies in the viscosity sense the equation
\begin{equation*}
(-\Delta)^sv=f\qquad\textrm{in }B_1,
\end{equation*}
where the function $f$ is defined as
\begin{equation*}
f(x)=\int_{\Co B_2}\frac{u(y)}{|x-y|^{n+2s}}\,dy\qquad\textrm{for every }x\in B_{3/2}.
\end{equation*}
Notice that $f$ is smooth in $B_{3/2}$. Therefore, exploiting our growth hypothesis on $u$, we have
\begin{equation*}
\|f\|_{C^{1,1}(B_1)},\quad\|v\|_{L^\infty(\R)}\leq C_\alpha,
\end{equation*}
and hence, since $u=v$ in $B_{1/2}$, we deduce from $(\ref{caf_sil_reg})$
\begin{equation*}
\|u\|_{C^{1,1}(B_{1/2})}\leq C_\alpha.
\end{equation*}
Notice that the rescaled functions $u_R$,
\begin{equation*}
u_R(x):=\frac{u(Rx)}{R^{1+\alpha}},\qquad R\geq1,
\end{equation*}
still satisfy the equation
\begin{equation*}
(-\Delta)^su_R=0\qquad\textrm{in }\R,
\end{equation*}
in the viscosity sense. Moreover
\begin{equation*}
|u_R(x)|\leq\frac{1+R^{1+\alpha}|x|^{1+\alpha}}{R^{1+\alpha}}\leq1+|x|^{1+\alpha},\qquad\forall\,x\in\R,
\end{equation*}
so that $u_R$ satisfies the same hypothesis as $u$ for every $R\geq1$.\\
In paticular we get the same estimate as above,
\begin{equation*}
\|u_R\|_{C^{1,1}(B_{1/2})}\leq C_\alpha,
\end{equation*}
for all $R\geq1$. This gives
\begin{equation*}
|\nabla u(Rx)-\nabla u(0)|=R^\alpha|\nabla u_R(x)-\nabla u_R(0)|
\leq C_\alpha R^\alpha|x|=C_\alpha R^{\alpha-1}|Rx|,
\end{equation*}
if $|x|\leq1/2$, which implies that $\nabla u(x_0)=\nabla u(0)$ for every $x_0\in\R$.\\
Indeed, if $|x_0|\leq1/2$, then $\left|\frac{x_0}{R}\right|\leq\frac{1}{2}$ for every $R\geq1$ and
\begin{equation*}
|\nabla u(x_0)-\nabla u(0)|=\left|\nabla u\Big(R\frac{x_0}{R}\Big)-\nabla u(0)\right|\leq C_\alpha R^{\alpha-1}|x_0|,
\end{equation*}
which tends to 0 as $R\to\infty$. On the other hand, if $|x_0|>1/2$, then
\begin{equation*}
\left|\frac{x_0}{r2|x_0|}\right|\leq\frac{1}{2},\qquad\textrm{for every }r\geq1,
\end{equation*}
and
\begin{equation*}
|\nabla u(x_0)-\nabla u(0)|=\left|\nabla u\Big(2r|x_0|)\frac{x_0}{2r|x_0|}\Big)-\nabla u(0)\right|\leq C_\alpha (2r|x_0|)^{\alpha-1}|x_0|,
\end{equation*}
which again tends to 0 as $r\to\infty$.

\end{proof}
\end{teo}

\end{section}

%%%%%%%%%%%%%%%%%%%%%%%%%%%%%%%%%%%%%%%%
%%%%%%%%%%%%%%%%%%%%%%%%%%%%%%%%%%%%%%%%
%%%%%%%%%%%%%%%%%%%%%%%%%%%%%%%%%%%%%%%%

%%%%%%%%%%%%%%%%%%%%%%%%%%%%%%%%%%%%%%%%
%%%%%%%%%%%%%%%%%%%%%%%%%%%%%%%%%%%%%%%%
%%%%%%%%%%%%%%%%%%%%%%%%%%%%%%%%%%%%%%%%

%%%%%%%%%%%%%%%%%%%%%%%%%%%%%%%%%%%%%%%%
%%%%%%%%%%%%%%%%%%%%%%%%%%%%%%%%%%%%%%%%
%%%%%%%%%%%%%%%%%%%%%%%%%%%%%%%%%%%%%%%%

%%%%%%%%%%%%%%%%%%%%%%%%%%%%%%%%%%%%%%%%
%%%%%%%%%%%%%%%%%%%%%%%%%%%%%%%%%%%%%%%%
%%%%%%%%%%%%%%%%%%%%%%%%%%%%%%%%%%%%%%%%

%%%%%%%%%%%%%%%%%%%%%%%%%%%%%%%%%%%%%%%%
%%%%%%%%%%%%%%%%%%%%%%%%%%%%%%%%%%%%%%%%
%%%%%%%%%%%%%%%%%%%%%%%%%%%%%%%%%%%%%%%%

%%%%%%%%%%%%%%%%%%%%%%%%%%%%%%%%%%%%%%%%
%%%%%%%%%%%%%%%%%%%%%%%%%%%%%%%%%%%%%%%%
%%%%%%%%%%%%%%%%%%%%%%%%%%%%%%%%%%%%%%%%

%%%%%%%%%%%%%%%%%%%%%%%%%%%%%%%%%%%%%%%%
%%%%%%%%%%%%%%%%%%%%%%%%%%%%%%%%%%%%%%%%
%%%%%%%%%%%%%%%%%%%%%%%%%%%%%%%%%%%%%%%%

%%%%%%%%%%%%%%%%%%%%%%%%%%%%%%%%%%%%%%%%
%%%%%%%%%%%%%%%%%%%%%%%%%%%%%%%%%%%%%%%%
%%%%%%%%%%%%%%%%%%%%%%%%%%%%%%%%%%%%%%%%

%%%%%%%%%%%%%%%%%%%%%%%%%%%%%%%%%%%%%%%%
%%%%%%%%%%%%%%%%%%%%%%%%%%%%%%%%%%%%%%%%
%%%%%%%%%%%%%%%%%%%%%%%%%%%%%%%%%%%%%%%%

\chapter{Distance Function}

We briefly recall the main properties of the (signed) distance function. For the details we refer to \cite{Ambrosio} and \cite{Bellettini}.

\begin{defin}
Let $\emptyset\not=E\subset\R$. We define the distance function from $E$
\begin{equation}
d_E(x)=d(x,E):=\inf_{y\in E}|x-y|,\qquad\textrm{for }x\in\R.
\end{equation}
The signed distance function from $\partial E$, negative inside $E$, is then defined as
\begin{equation}
\bar{d}_E(x)=\bar{d}(x,E):=d(x,E)-d(x,\Co E).
\end{equation}
\end{defin}

We begin with some easy remarks.\\
$(1)\qquad$ We have
\begin{equation*}
d(x,E)=d(x,\bar{E}),\qquad\textrm{for every }x\in\R,
\end{equation*}
so we may as well assume the set $E$ to be closed.\\
In particular
\begin{equation*}
\bar{E}=\{x\in\R\,|\,d(x,E)=0\}=\{d_E=0\}.
\end{equation*}
Moreover
\begin{equation*}
d(x,E)=d(x,\partial E),\qquad\textrm{for any }x\not\in\textrm{Int}(E).
\end{equation*}
$(2)\qquad$ By definition
\begin{equation}
\bar{d}_E(x)=\left\{\begin{array}{cc}-d(x,\partial E), &\textrm{if }x\in E,\\
d(x,\partial E),&\textrm{if }x\in\Co E,\end{array}\right.
\end{equation}
and hence in particular
\begin{equation*}
\bar{d}(x,\Co E)=-\bar{d}(x,E)\qquad\textrm{for every }x\in\R.
\end{equation*}
Moreover
\begin{equation*}
\bar{E}=\{\bar{d}_E\leq0\},\quad\textrm{ and}\quad\partial E=\{\bar{d}_E=0\}.
\end{equation*}
$(3)\qquad$ If $E\subset F$, then
\begin{equation*}
\bar{d}_E(x)\geq\bar{d}_F(x),\qquad\textrm{for every }x\in\R.
\end{equation*}
$(4)\qquad$ The functions $d_E$ and $\bar{d}_E$ are Lipschitz continuous, with constant $\leq1$.\\
In particular by Rademacher's Theorem they are almost everywhere differentiable in $\R$, with
$|\nabla d_E|\leq1$ and $|\nabla\bar{d}_E|\leq1$, wherever they are defined.

To be more precise we have
\begin{teo}[Differentiability of Distance]
Let $\emptyset\not= E\subset\R$ be a closed set and let $x\in\Co E$. The distance function $d_E$ is differentiable at $x$ if and only if there exists a unique $y\in E$ s.t. $d_E(x)=|x-y|$. In this case
\begin{equation}
\nabla d_E(x)=\frac{x-y}{|x-y|}=\frac{x-y}{d_E(x)}.
\end{equation}
In particular, $|\nabla d_E|=1$ at any differentiable point of $\Co E$.
\begin{proof}
$\Rightarrow)$ Suppose $d_E$ is differentiable at $x$ and let $y\in E$ be any least distance point i.e. $d_E(x)=|x-y|$.
Fix any $z\in\R$; then
\begin{equation*}
d_E(x+\epsilon\,z)=d_E(x)+\epsilon\,\nabla d_E(x)\cdot z+ o(\epsilon),
\end{equation*}
as $\epsilon\to0$, and hence squaring we get
\begin{equation*}
d_E^2(x+\epsilon\,z)\geq d_E^2(x)+2\epsilon\,d_E(x)\nabla d_E(x)\cdot z+ o(\epsilon).
\end{equation*}
Moreover, since $y\in E$,
\begin{equation*}
d_E^2(x+\epsilon\,z)\leq|x+\epsilon\,z-y|^2=d_E^2(x)+2\epsilon(x-y)\cdot z+o(\epsilon).
\end{equation*}
Since these hold for every $z\in\R$, we get
\begin{equation*}
d_E(x)\nabla d_E(x)=x-y,
\end{equation*}
and hence in particular
\begin{equation*}
y=x-d_E(x)\nabla d_E(x)
\end{equation*}
is uniquely determined.

$\Leftarrow)$ This implication is more complicated. We briefly sketch the idea.\\
First of all it can be shown that $d_E$ is a locally semiconcave function in the open set $\Co E$, meaning that
for every open set $\Omega\subset\subset\Co E$ there exists some constant $C\geq0$ s.t. the function
\begin{equation*}
x\longmapsto d_E(x)-\frac{C}{2}|x|^2
\end{equation*}
is concave in $\Omega$. Now we consider the first order super jet at a point $x\in\Co E$,
\begin{equation*}
J^{1,+}d_E(x):=\{p\in\R\,|\,d_E(y)\leq d_E(x)+p\cdot(y-x)+o(|y-x|)\}.
\end{equation*}
Since $d_E$ is locally semiconcave in $\Co E$, it can be shown that $J^{1,+}d_E(x)\not=\emptyset$ for every $x\in\Co E$, and 
we can represent the first order super jet as the closed convex hull of the set
\begin{equation*}
\nabla_*d_E(x):=\Big\{p\in\R\,|\,p=\lim_{k\to\infty}\nabla d_E(x_k),\,x_k\to x\Big\},
\end{equation*}
of reachable gradients. In particular
$d_E$ is differentiable in a point $x\in\Co E$ if and only if
$J^{1,+}d_E(x)$ is a singleton.

We argue by contradiction.

If we suppose that $d_E$ is not differentiable in $x\in\Co E$, then there exist $p_1\not=p_2\in J^{1,+}d_E(x)$ and we can find 
two sequences $\{x^i_k\}_{k\in\mathbb{N}}\subset\Co E$, $i=1,\,2$, both converging to $x$, and s.t.
$d_E$ is differentiable in every $x_k^i$, with
\begin{equation*}
p_i=\lim_{k\to\infty}\nabla d_E(x^i_k),\qquad\textrm{for }i=1,\,2.
\end{equation*}
Now we define
\begin{equation*}
y^i_k:=x_k^i-d_E(x_k^i)\nabla d_E(x_k^i)\in E,
\end{equation*}
the least distance point from $x_k^i$ (see above). Then
\begin{equation*}
\lim_{k\to\infty}y^i_k=x-d_E(x)p_i=:y^i,\qquad\textrm{for }i=1,\,2,
\end{equation*}
and, since $p_1\not=p_2$, we have $y^1\not=y^2$.\\
Using what we proved above we know that $|\nabla d_E(x_k^i)|=1$ and hence also $|p_i|=1$.
Therefore
\begin{equation*}
d_E(x)=|x-y^i|,\qquad\textrm{for }i=1,\,2.
\end{equation*}
Moreover, since $y_k^i\in E$ and $E$ is closed, also $y^i\in E$.\\
We have thus found two distinct least distance point for $x$, concluding the proof.

\end{proof}
\end{teo}

Clearly, by definition of $\bar{d}_E$, we have a similar result for the signed distance function in the complementary of $\partial E$.

\begin{rmk}
We have showed in the proof that if $d_E$ is differentiable in $x\in\Co E$, then the nearest point on $E$ is
\begin{equation*}
y=x-d_E(x)\nabla d_E(x).
\end{equation*}
Moreover, using the triangle inequality, we can show that $y$ is the unique nearest point for every $z\in S\setminus\{y\}$, where $S$ is the segment joining $x$ to $y$. In particular $d_E$ is differentiable in every such $z$ and $\nabla d_E(z)=\nabla d_E(x)$.
\end{rmk}

The smoothness of the boundary $\partial E$ and of the signed distance function $\bar{d}_E$ are strictly related.
We denote
\begin{equation*}
N_\rho(\partial E):=\{-\rho<\bar{d}_E<\rho\},
\end{equation*}
the open tubular $\rho$-neighborhood of $\partial E$.\\
If for some $\rho>0$ we have $\bar{d}_E\in C^k(N_\rho(\partial E))$, with $k\geq2$, then, since $|\nabla\bar{d}_E|=1$,
the Constant Rank Theorem implies that $\partial E=\{\bar{d}_E=0\}$ is a $C^k$ surface.

On the other hand we have

\begin{teo}[Regularity of Signed Distance]
Let $\Omega\subset\R$ be a bounded open set with $C^k$ boundary, for some $2\leq k\leq\infty$. Then
$\bar{d}_\Omega\in C^k(N_\rho(\partial\Omega))$, for some $\rho>0$. Moreover the outer unit normal to $\Omega$
is given by $\nabla\bar{d}_\Omega$.
\begin{proof}
The proof (see \cite{GilTru}) is done using Dini's implicit function Theorem.\\
We explain the argument.

Our hypothesis on $\Omega$ guarantee that $\partial\Omega$ satisfies a uniform interior and exterior ball condition, meaning that for every $y_0\in\partial\Omega$ there exist two balls $B_1,\,B_2$ depending on $y_0$ s.t.
$\bar{B}_1\cap\Co\Omega=\{y_0\}$ and $\bar{B}_2\cap\bar{\Omega}=\{y_0\}$, with the radii of the balls bounded from below by a constant, say $\rho>0$ (which does not depend on $y_0$).

As a consequence, for every point $x\in N_\rho(\partial\Omega)$ there exists a unique point $y=y(x)\in\partial\Omega$ of minimal distance, i.e. with $d(x,\partial\Omega)=|x-y|$, and these points are related by the equation
\begin{equation}
x=y(x)+\bar{d}_\Omega(x)\nu_\Omega(y(x)).
\end{equation}
Moreover previous Theorem implies that $\bar{d}_\Omega$ is differentiable in
$N_\rho(\partial\Omega)\setminus\partial\Omega$ and
\begin{equation}
y(x)=x-\bar{d}_\Omega(x)\nabla\bar{d}_\Omega(x).
\end{equation}
If we show that $y$ is a $C^{k-1}$ function of $x$, then we have for every $x\in N_\rho(\partial\Omega)\setminus\partial\Omega$
\begin{equation*}
\nabla\bar{d}_\Omega(x)=\nu_\Omega(y(x)),
\end{equation*}
which is a $C^{k-1}$ function, and hence $\bar{d}_\Omega\in C^k(N_\rho(\partial\Omega)\setminus\partial\Omega)$.

To be more precise (and to get the regularity in the whole of $N_\rho(\partial\Omega)$), we consider a point $y_0\in\partial\Omega$
and the
$C^{k-1}$ function defined by
\begin{equation*}
\Phi:(B_s(y_0)\cap\partial\Omega)\times\mathbb{R}\longrightarrow\R,\qquad\Phi(x,a):=x+a\,\nu_\Omega(x),
\end{equation*}
which is easily seen to be locally invertible in $(y_0,0)$. Then one relates the inverse function
\begin{equation*}
\Psi:B_\sigma(y_0)\subset V\xrightarrow{\sim}(B_r(y_0)\cap\partial\Omega)\times(-r,r),\qquad\Psi(x)=(y(x),a(x)),
\end{equation*}
with the distance function, showing that $y(x)$ is the minimal distance point from $x$ and $a(x)=\bar{d}_\Omega(x)$.
Since the function $\Psi$ is $C^{k-1}$, also $\bar{d}_\Omega=a$ is $C^{k-1}$ in $B_\sigma(y_0)$ and the argument above shows that it actually is $C^k$. Using the compactness of $\partial\Omega$ we conclude the proof.

\end{proof}
\end{teo}

Before showing an immediate Corollary, which allows us to characterize bounded open sets with regular boundary equivalently in terms of the distance function or as superlevel sets, we recall the following consequence of the smooth
Urysohn's Lemma

\begin{teo}\label{urysohn}
Let $K\subset\R$ be a compact set and let $U\subset\R$ be an open set s.t. $K\subset U$.
Then $\exists f\in C^\infty_c(\R)$ with $0\leq f\leq1$, $f\big|_K=1$ and supp $f\subset U$.
\end{teo}

\begin{coroll}
Let $\Omega\subset\R$ be a bounded open set and let $k\geq2$. The following are equivalent:

(i)$\qquad\partial\Omega$ is a $C^k$-hypersurface (without boundary),\\

(ii)$\qquad\exists\rho>0$ s.t. $\bar{d}_\Omega\in C^k(N_\rho(\partial\Omega))$,\\

(iii)$\quad\quad\exists\varphi\in C_c^2(\R)$ s.t.
\begin{equation*}
\Omega=\{\varphi\geq1/2\},\qquad\partial\Omega=\{\varphi=1/2\},
\end{equation*}
and
\begin{equation*}
\nabla\varphi\not=0\qquad\textrm{in}\quad\{1/8\leq\varphi\leq7/8\}.
\end{equation*}
\begin{proof}
We have already showed $(i)\Leftrightarrow(ii)$ and $(iii)\Rightarrow(i)$ is an immediate consequence of the constant rank
Theorem.\\
We only need to prove the implication $(ii)\Rightarrow(iii)$.

Suppose $\bar{d}_\Omega\in C^k(N_{2\rho}(\partial\Omega))$. Then Whitney's extension Theorem guarantees that there exists
a function
$\psi\in C^2(\R)$ s.t. $\psi=-\bar{d}_\Omega$ in
the closed tubular neighborhood
$\bar{N}_\rho(\partial\Omega)=\{-\rho\leq\bar{d}_\Omega\leq\rho\}$ and $\psi=1$ in
$\{\bar{d}_\Omega\leq-3\rho\}\subset\Omega$.\\
We need to check that $\psi$ stays strictly above 0 in the strip $\{-3\rho\leq\bar{d}_\Omega\leq-\rho\}$.

Suppose it doesn't and define $K:=\{\psi\leq0\}\cap\Omega$, which is compact, and let  $\beta:=\min_K\psi\leq0$.\\
Since $\psi$ is continuous, we have by construction
\begin{equation*}
K\subset\{-3\rho+\delta<\bar{d}_\Omega<-\rho-\delta\}=:U,
\end{equation*}
for some $\delta>0$ small.\\
Using Theorem $\ref{urysohn}$ we can add a bump function to $\psi$ to ensure it stays strictly above 0 inside $\Omega$.
Indeed there exists a smooth function $b:\R\longrightarrow[0,|\beta|+1]$ s.t.
$b=|\beta|+1$ in $K$, and supp $\psi\subset U$.\\
Then $\tilde\psi:=\psi+b\in C^2(\R)$ is s.t. $\tilde{\psi}=-\bar{d}_\Omega$ in
$\bar{N}_\rho(\partial\Omega)$ and $\tilde{\psi}>0$ in $\Omega$.

Now let $\varphi(x):=c(x)\big(\frac{3}{8\rho}\tilde{\psi}(x)+\frac{1}{2}\big)$,
where $c$ is a smooth cutoff funtion s.t. $c=1$ in
$\{\bar{d}_\Omega\leq2\rho\}$, $0\leq c\leq1$, and supp $c\subset\{\bar{d}_\Omega<3\rho\}$.\\
Eventually adding a bump function supported in $\{\rho<\bar{d}_\Omega<3\rho\}$, we have
$\varphi<\frac{1}{2}$ in $\Co\Omega\setminus\partial\Omega$.

Then $\varphi\in C_c^2(\R)$ with $\varphi=-\frac{3}{8\rho}\bar{d}_\Omega+\frac{1}{2}$ in
\begin{equation*}
\{1/8\leq\varphi\leq7/8\}=\{-\rho\leq\bar{d}_\Omega\leq\rho\},
\end{equation*}
so that $\nabla\varphi\not=0$ there, and
\begin{equation*}
\Omega=\{\varphi\geq1/2\},\qquad\partial\Omega=\{\varphi=1/2\},
\end{equation*}
by construction, as wanted.

\end{proof}
\end{coroll}

%%%%%%%%%%%%%%%%%%%%%%%%%%%%%%%%%%%%%%%%%%%%%%%%%%
%%%%%%%%%%%%%%%%%%%%%%%%%%%%%%%%%%%%%%%%%%%%%%%%%%
%%%%%%%%%%%%%%%%%%%%%%%%%%%%%%%%%%%%%%%%%%%%%%%%%%
%%%%%%%%%%%%%%%%%%%%%%%%%%%%%%%%%%%%%%%%%%%%%%%%%%
%%%%%%%%%%%%%%%%%%%%%%%%%%%%%%%%%%%%%%%%%%%%%%%%%%

\clearpage
\null
\thispagestyle{plain}
\clearpage

%%%%%%%%%%%%%%%%%%%%%%%%%%%%%%%%%%%%%%%%%%%%%%%%%%%%
%%%%%%%%%%%%%%%%%%%%%%%%%%%%%%%%%%%%%%%%%%%%%%%%%%%%
%%%%%%%%%%%%%%%%%%%%%%%%%%%%%%%%%%%%%%%%%%%%%%%%%%%%
%%%%%%%%%%%%%%%%%%%%%%%%%%%%%%%%%%%%%%%%%%%%%%%%%%%%


\begin{thebibliography}{90}

\addcontentsline{toc}{chapter}{Bibliography}






\bibitem{curvature} N. Abatangelo and E. Valdinoci, {\it A notion of nonlocal curvature}. Numer. Funct.
Anal. Optim., 35(7-9):793$-$815 (2014).

\bibitem{Ambrosio} L. Ambrosio and N. Dancer, {\it Calculus of variations and partial differential equations}. Springer-Verlag,
Berlin (2000).


\bibitem{Gamma} L. Ambrosio, G. De Philippis and L. Martinazzi, {\it Gamma-convergence of nonlocal perimeter functionals}.
Manuscripta Math. 134, no. 3-4, 377$-$403 (2011).

\bibitem{Bellettini} G. Bellettini, {\it Lecture notes on mean curvature flows, barriers and singular perturbations}. Edizioni della Scuola Normale 13, Pisa (2013).

\bibitem{BBM} J. Bourgain, H. Brezis and P. Mironescu, {\it Limiting embedding theorems for $W^{s,p}$ when $s\to1$ and applications}. J.
Anal. Math. 87, 77$-$101 (2002).

\bibitem{Sire} X. Cabr\'e and Y. Sire, {\it Nonlinear equations for fractional Laplacians I: Regularity,
maximum principles and Hamiltonian estimates}. Ann. Inst. H. Poincar\'e Anal.
Non Lin\'eaire, to appear.

\bibitem{Fully} L. Caffarelli and X. Cabr\'e, {\it Fully nonlinear elliptic equations}. Colloquium
Publications 43, American Mathematical Society, Providence, RI (1995).

\bibitem{CC} L. Caffarelli and A. Cordoba, {\it An elementary regularity theory of minimal surfaces}. Differential
Integral Equations 6, no. 1, 1$-$13 (1993).

\bibitem{CRS}  L. Caffarelli, J.-M. Roquejoffre and O. Savin, {\it Nonlocal minimal surfaces}. Comm. pure Appl. Math. 63, no. 9, 1111$-$1144 (2010).

\bibitem{extension} L. Caffarelli and L. Silvestre, {\it An extension problem related to the fractional laplacian}.
Comm. Partial Differential Equations, 32(7-9):1245$-$1260 (2007).

\bibitem{regularity} L. Caffarelli and L. Silvestre, {\it Regularity theory for fully nonlinear integro-differential
equations}. Comm. Pure Appl. Math. 62, 597$-$638 (2009).

\bibitem{Cso} L. Caffarelli and A. Souganidis, {\it Convergence of nonlocal threshold dynamics approximations to front propagation}. Arch. Rational Mech. Anal., 195 (1): 1$-$23, (2010).


\bibitem{unifor} L. Caffarelli and E. Valdinoci, {\it Regularity properties of nonlocal minimal surfaces via
limiting arguments}. Adv. Math. 248, 843$-$871 (2013).

\bibitem{cafenr}  L. Caffarelli and E. Valdinoci, {\it Uniform estimates and limiting arguments for nonlocal minimal surfaces}. Calc. Var. Partial
Differential Equations 41, no. 1-2, 203$-$240 (2011).

\bibitem{CMP} A. Chambolle, M. Morini and M. Ponsiglione, {\it Nonlocal curvature flows}. Preprint,
 arXiv:1409.1109 (2014).

\bibitem{matteo} M. Cozzi, {\it On the variation of the fractional mean curvature under the effect of $C^{1,\alpha}$ perturbations}. To appear in Discrete
Contin. Dyn. Syst. 35, no. 12 5769$-$5786 (2015).

\bibitem{Visco} M.G. Crandall, H. Ishii and P.L. Lions, {\it User's guide to viscosity solutions of second order partial differential equations}. Bull. Amer. Math. Soc., 27, 1$-$67 (1992).

\bibitem{Davila} J. Davila, {\it On an open question about functions of bounded variation}.
Calc. Var. Partial Differential Equations,15
no. 4, 519$-$527 (2002).

\bibitem{DelPino} J. Davila, M. del Pino and J. Wei, {\it Nonlocal Minimal Lawson Cones}, preprint.

\bibitem{HitGuide} E. Di Nezza, G. Palatucci and E. Valdinoci, {\it Hitchhiker’s guide to the
fractional Sobolev spaces}. Bull. Sci. Math., 136(5):521$-$573 (2012).

\bibitem{asymptzero}  S. Dipierro, A. Figalli, G. Palatucci and E. Valdinoci, {\it Asymptotics of the $s$-perimeter as $s\to0$}. Discrete Contin. Dyn. Syst. 33, no. 7, 2777$-$2790 (2013).

\bibitem{bdary} S. Dipierro, O. Savin and E. Valdinoci, {\it Boundary behavior of nonlocal minimal
surfaces}. (2015).

\bibitem{graph} S. Dipierro, O. Savin and E. Valdinoci, {\it Graph properties for nonlocal minimal surfaces}. (2015).

\bibitem{LipApprox} P. Doktor, {\it Approximation of domains with Lipschitzian boundary}.
Cas. Pest. Mat.  101, 237$-$255 (1976).

\bibitem{Falconer} K.J. Falconer, {\it Fractal geometry: mathematical foundations and applications}. John Wiley and Sons
(1990).

\bibitem{ShortReg} A. Figalli and E. Valdinoci, {\it Regularity and Bernstein-type results for nonlocal minimal surfaces}. To appear in J.
Reine Angew. Math.

\bibitem{GiaSou} M. Giaquinta, G. Modica and J. Soucek, {\it Cartesian currents in the calculus of variations I: cartesian
currents}, Ergeb. Math. Grenzgeb. no. 37, Springer-Verlag, Berlin (1998).

\bibitem{GilTru} D. Gilbarg and N. S. Trudinger, {\it Elliptic partial differential equations
of second order}. Springer-Verlag, 2nd edition, Berlin (1983).

\bibitem{Giusti} E. Giusti, {\it Minimal surfaces and functions of bounded variation}. Monographs in Mathematics, 80.
Birkhauser Verlag, Basel (1984).

\bibitem{Maggi} F. Maggi, {\it Sets of finite perimeter and geometric variational problems}. Cambridge Stud. Adv.
Math. 135, Cambridge Univ. Press, Cambridge (2012).

\bibitem{Mattila} P. Mattila, {\it Geometry of sets and measures in Euclidean spaces}. Cambridge Stud. Adv.
Math. 44, Cambridge Univ. Press, Cambridge (1995).

\bibitem{Savin} O. Savin, {\it Small perturbation solutions for elliptic equations}. Comm. Partial Differential
Equations 32, 557$-$578 (2007).

\bibitem{Phase} O. Savin and E. Valdinoci, {\it $\Gamma$-convergence for nonlocal phase transitions}. Ann. Inst. H.
Poincar\'e Anal. Non Lin\'eaire, 29, 479$-$500 (2012).

\bibitem{cones} O. Savin and E. Valdinoci, {\it Regularity of nonlocal minimal cones in dimension 2}. Calc. Var. Partial Differential
Equations 48, no. 1-2, 33$-$39 (2013).

\bibitem{obstacle} L. Silvestre, {\it Regularity of the obstacle problem for a fractional power of the
Laplace operator}. Comm. Pure Appl. Math. 60, 67$-$112 (2007).

\bibitem{Visintin} A. Visintin. {\it Generalized coarea formula and fractal sets}. Japan J. Indust. Appl. Math.,
8(2):175$-$201 (1991).










\end{thebibliography}
\end{document}